\documentclass{article}

\usepackage{amsfonts} % Includes
\usepackage{amsmath} % all the
\usepackage{amssymb} % AMS
\usepackage{amsthm} % stuff!
	\newtheorem{thm}{Theorem}[section]

	\newtheorem{cor}[thm]{Corollary}
	\newtheorem{defn}[thm]{Definition}
	\newtheorem{lem}[thm]{Lemma}
	
	\newtheorem{prop}[thm]{Proposition}

\usepackage{comment}

\usepackage[margin=3cm]{geometry} % Geometry package good for setting margins

\usepackage{graphicx} % needed for graphics

\usepackage{hyperref} % hyperlinks in document when building pdf
\usepackage{makecell} % for split table cells
\usepackage{multicol}

\usepackage{tikz}
\usetikzlibrary{3d, calc, arrows, decorations.markings, decorations.pathmorphing, positioning, decorations.pathreplacing}
\makeatletter
\tikzoption{canvas is xy plane at z}[]{%
  \def\tikz@plane@origin{\pgfpointxyz{0}{0}{#1}}%
  \def\tikz@plane@x{\pgfpointxyz{1}{0}{#1}}%
  \def\tikz@plane@y{\pgfpointxyz{0}{1}{#1}}%
  \tikz@canvas@is@plane
}
\makeatother
\tikzset{xyp/.style={canvas is xy plane at z=#1}}
\tikzset{xzp/.style={canvas is xz plane at y=#1}}
\tikzset{yzp/.style={canvas is yz plane at x=#1}}
\tikzset{mypersp/.style={x={(0:1cm)},y={(55:0.75cm)},z={(90:1cm)}}}

\newcommand{\drawcube}{
\draw [xzp=0] (0,0) -- (0,1) -- (1,1) -- (1,0) -- cycle;
\draw [yzp=1] (0,0) -- (0,1) -- (1,1) -- (1,0) -- cycle;
\draw [xyp=1] (0,0) -- (0,1) -- (1,1) -- (1,0) -- cycle;
\draw [densely dotted, yzp=0] (0,0) -- (0,1) -- (1,1) -- (1,0) -- cycle;
\draw [densely dotted, xzp=1] (0,0) -- (0,1) -- (1,1) -- (1,0) -- cycle;
\draw [densely dotted, xyp=0] (0,0) -- (0,1) -- (1,1) -- (1,0) -- cycle;
}

\newcommand{\rightclean}
{
\draw [red, ultra thick, yzp=1] (0.5,0) arc (180:90:0.5);
\draw [red, ultra thick, yzp=1] (0.5,1) arc (0:-90:0.5);
}
\newcommand{\frontclean}
{
\draw [red, ultra thick, xzp=0] (0.5,0) arc (180:90:0.5);
\draw [red, ultra thick, xzp=0] (0.5,1) arc (0:-90:0.5);
}

\newcommand{\leftclean}
{
\draw [red, ultra thick, densely dotted, yzp=0] (0.5,0) arc (0:90:0.5);
\draw [red, ultra thick, densely dotted, yzp=0] (0.5,1) arc (-180:-90:0.5);
}

\newcommand{\backclean}
{
\draw [red, ultra thick, densely dotted, xzp=1] (0.5,0) arc (0:90:0.5);
\draw [red, ultra thick, densely dotted, xzp=1] (0.5,1) arc (-180:-90:0.5);
}

\newcommand{\topoff}
{
\draw [red, ultra thick, densely dotted, xyp=1] (0.5,0) arc (0:90:0.5);
\draw [red, ultra thick, densely dotted, xyp=1] (0.5,1) arc (-180:-90:0.5);
}
\newcommand{\topofftwo}
{
\draw [red, ultra thick, densely dotted, xyp=2] (0.5,0) arc (0:90:0.5);
\draw [red, ultra thick, densely dotted, xyp=2] (0.5,1) arc (-180:-90:0.5);
}

\newcommand{\topon}
{
\draw [red, ultra thick, xyp=1] (0.5,0) arc (180:90:0.5);
\draw [red, ultra thick, xyp=1] (0.5,1) arc (0:-90:0.5);
}
\newcommand{\toponthree}
{
\draw [red, ultra thick, xyp=3] (0.5,0) arc (180:90:0.5);
\draw [red, ultra thick, xyp=3] (0.5,1) arc (0:-90:0.5);
}
\newcommand{\bottomon}
{
\draw [red, ultra thick, densely dotted, xyp=0] (0.5,0) arc (180:90:0.5);
\draw [red, ultra thick, densely dotted, xyp=0] (0.5,1) arc (0:-90:0.5);
}
\newcommand{\bottomoff}
{
\draw [red, ultra thick, densely dotted, xyp=0] (0.5,0) arc (0:90:0.5);
\draw [red, ultra thick, densely dotted, xyp=0] (0.5,1) arc (-180:-90:0.5);
}

\newcommand{\strandbackgroundshading}
{
\draw [draw=none, fill=gray!10!white] (0,0) -- (0,0.5) -- (1,0.5) -- (1,0) -- cycle;
\draw [draw=none, fill=gray!10!white] (0,1) -- (0,1.5) -- (1,1.5) -- (1,1) -- cycle;
}
\newcommand{\strandsetup}{
\draw [ultra thick, color=black!50!green] (0,0) -- (0,0.5);
\draw [ultra thick, color=black!50!green] (0,1) -- (0,1.5);
\draw (0,0.25) to [bend left=90] (0,1.25);
\draw (-0.15,0.15) node {$q$};
\draw (-0.15,1.35) node {$p$};
}
\newcommand{\strandsetupn}[2]{
\draw [ultra thick, color=black!50!green] (0,0) -- (0,0.5);
\draw [ultra thick, color=black!50!green] (0,1) -- (0,1.5);
\draw (0,0.25) to [bend left=90] (0,1.25);
\draw (-0.2,0.15) node {$#1$};
\draw (-0.2,1.35) node {$#2$};
}
\newcommand{\cubestrandsetup}
{
\draw [ultra thick, color=black!50!green] (0,0) -- (0,0.5);
\draw [ultra thick, color=black!50!green] (0,1) -- (0,1.5);
\draw (0,0.25) to [bend left=90] (0,1.25);
\draw (-0.15,0.15) node {$v$};
\draw (-0.15,1.35) node {$w$};
}
\newcommand{\lefton}
{
\draw [color=black!50!green, fill=black!50!green] (0,0.25) circle (2 pt);
\draw [color=black!50!green, fill=black!50!green] (0,1.25) circle (2 pt);
}
\newcommand{\leftoff}
{
\draw [color=black!50!green, fill=none] (0,0.25) circle (2 pt);
\draw [color=black!50!green, fill=none] (0,1.25) circle (2 pt);
}
\newcommand{\righton}
{
\draw [color=black!50!green, fill=black!50!green] (1,0.25) circle (2 pt);
\draw [color=black!50!green, fill=black!50!green] (1,1.25) circle (2 pt);
}
\newcommand{\rightoff}
{
\draw [color=black!50!green, fill=none] (1,0.25) circle (2 pt);
\draw [color=black!50!green, fill=none] (1,1.25) circle (2 pt);
}
\newcommand{\aftervused}
{
\draw [draw=none, fill=gray!50!white] (0,0.25) -- (0,0.5) -- (1,0.5) -- (1,0.25) -- cycle;
}
\newcommand{\beforevused}
{
\draw [draw=none, fill=gray!50!white] (0,0.25) -- (0,0) -- (1,0) -- (1,0.25) -- cycle;
}
\newcommand{\afterwused}
{
\draw [draw=none, fill=gray!50!white] (0,1.25) -- (0,1.5) -- (1,1.5) -- (1,1.25) -- cycle;
}
\newcommand{\beforewused}
{
\draw [draw=none, fill=gray!50!white] (0,1.25) -- (0,1) -- (1,1) -- (1,1.25) -- cycle;
}
\newcommand{\usea}
{
\draw [ultra thick] (0,1.25) to [out=0, in=-90] (0.5,1.5);
}
\newcommand{\useb}
{
\draw [ultra thick] (0.5,1) to [out=90, in=180] (1,1.25);
}
\newcommand{\useab}
{
\draw [ultra thick]  (0.4,1) to [out=90, in=-90] (0.6,1.5);
}
\newcommand{\usec}
{
\draw [ultra thick] (0,0.25) to [out=0, in=-90] (0.5,0.5);
}
\newcommand{\used}
{
\draw [ultra thick] (0.5,0) to [out=90, in=180] (1,0.25);
}
\newcommand{\usecd}
{
\draw [ultra thick]  (0.4,0) to [out=90, in=-90] (0.6,0.5);
}
\newcommand{\dothorizontals}
{
\draw [ultra thick, dotted]  (0,0.25) -- (1,0.25);
\draw [ultra thick, dotted]  (0,1.25) -- (1,1.25);
}

\newcommand{\tstrandsetup}[1]{
\draw [ultra thick, color=black!50!green] ($ ( #1 , 0 ) $) -- ($ ( #1 , 0.5) $);
\draw [ultra thick, color=black!50!green] ($ ( #1 , 1) $) -- ($ (#1 , 1.5) $);
}
\newcommand{\tstrandbackgroundshading}[1]
{
\draw [draw=none, fill=gray!10!white] ($ (#1,0) $) -- ($ (#1,0.5) $) -- ($ (#1 + 1, 0.5) $) -- ($ (#1 + 1, 0) $) -- cycle;
\draw [draw=none, fill=gray!10!white] ($ (#1,1) $) -- ($ (#1,1.5) $) -- ($ (#1 + 1, 1.5) $) -- ($ (#1 + 1, 1) $) -- cycle;
}
\newcommand{\trighton}[1]
{
\draw [color=black!50!green, fill=black!50!green] ($ (#1+1,0.25) $) circle (2 pt);
\draw [color=black!50!green, fill=black!50!green] ($ (#1+1,1.25) $) circle (2 pt);
}
\newcommand{\trightoff}[1]
{
\draw [color=black!50!green, fill=none] ($ (#1+1,0.25) $) circle (2 pt);
\draw [color=black!50!green, fill=none] ($ (#1+1,1.25) $) circle (2 pt);
}
\newcommand{\taftervused}[1]
{
\draw [draw=none, fill=gray!50!white] ($ (#1,0.25) $) -- ($ (#1,0.5) $) -- ($ (#1+1,0.5) $) -- ($ (#1+1,0.25) $) -- cycle;
}
\newcommand{\tbeforevused}[1]
{
\draw [draw=none, fill=gray!50!white] ($ (#1,0.25) $) -- ($ (#1,0) $) -- ($ (#1+1,0) $) -- ($ (#1+1,0.25) $) -- cycle;
}
\newcommand{\tafterwused}[1]
{
\draw [draw=none, fill=gray!50!white] ($ (#1,1.25) $) -- ($ (#1,1.5) $) -- ($ (#1+1,1.5) $) -- ($ (#1+1,1.25) $) -- cycle;
}
\newcommand{\tbeforewused}[1]
{
\draw [draw=none, fill=gray!50!white] ($ (#1,1.25) $) -- ($ (#1,1) $) -- ($ (#1+1,1) $) -- ($ (#1+1,1.25) $) -- cycle;
}
\newcommand{\tusea}[1]
{
\draw [ultra thick] ($ (#1,1.25) $) to [out=0, in=-90] ($ (#1+0.5,1.5) $);
}
\newcommand{\tuseb}[1]
{
\draw [ultra thick] ($ (#1+0.5,1) $) to [out=90, in=180] ($ (#1+1,1.25) $);
}
\newcommand{\tuseab}[1]
{
\draw [ultra thick] ($ (#1+0.4,1) $) to [out=90, in=-90] ($ (#1+0.6,1.5) $);
}
\newcommand{\tusec}[1]
{
\draw [ultra thick] ($ (#1,0.25) $) to [out=0, in=-90] ($ (#1+0.5,0.5) $);
}
\newcommand{\tused}[1]
{
\draw [ultra thick] ($ (#1+0.5,0) $) to [out=90, in=180] ($ (#1+1,0.25) $);
}
\newcommand{\tusecd}[1]
{
\draw [ultra thick] ($ (#1+0.4,0) $) to [out=90, in=-90] ($ (#1+0.6,0.5) $);
}
\newcommand{\tdothorizontals}[1]
{
\draw [ultra thick, dotted]  ($ (#1,0.25) $) -- ($ (#1+1,0.25) $);
\draw [ultra thick, dotted]  ($ (#1,1.25) $) -- ($ (#1+1,1.25) $);
}

\newcommand{\To}{\longrightarrow}

\newcommand{\A}{\mathcal{A}}
\newcommand{\AAbar}{\overline{\mathcal{A}}}
\newcommand{\Abar}{\overline{A}}

\newcommand{\CC}{\mathcal{C}}

\newcommand{\fbar}{\overline{f}}

\newcommand{\F}{\mathcal{F}}

\newcommand{\HH}{\mathcal{H}}

\newcommand{\T}{\mathcal{T}}

\newcommand{\U}{\mathcal{U}}
\newcommand{\Ubar}{\overline{U}}

\newcommand{\Z}{\mathbb{Z}}
\newcommand{\ZZ}{\mathcal{Z}}

\newcommand{\on}{\bullet}
\newcommand{\off}{\circ}

\DeclareMathOperator{\inv}{inv}

\newcolumntype{x}[1]{>{\centering\arraybackslash}p{#1}}
\newcommand\diag[4]{%
  \multicolumn{1}{p{#2}|}{\hskip-\tabcolsep
  $\vcenter{\begin{tikzpicture}[baseline=0,anchor=south west,inner sep=#1]
  \path[use as bounding box] (0,0) rectangle (#2+2\tabcolsep,\baselineskip);
  \node[minimum width={#2+2\tabcolsep},minimum height=\baselineskip+\extrarowheight] (box) {};
  \draw (box.north west) -- (box.south east);
  \node[anchor=south west] at (box.south west) {#3};
  \node[anchor=north east] at (box.north east) {#4};
 \end{tikzpicture}}$\hskip-\tabcolsep}}

\begin{document}

\title{A-infinity algebras, strand algebras, and contact categories} % Title goes here!

% When using authblk, do it like this
%\author{Daniel V. Mathews} %ND: Please check details for Dan
%\affil{School of Mathematical Sciences,
%Monash University,
%VIC 3800, Australia
%\texttt{daniel.mathews@monash.edu}}

% Otherwise, author goes here
\author{Daniel V. Mathews}

%If a date is wanted, put it here.
%\date{}

\maketitle

\begin{abstract}

In previous work we showed that the contact category algebra of a quadrangulated surface is isomorphic to the homology of a strand algebra from bordered sutured Floer theory. Being isomorphic to the homology of a differential graded algebra, this contact category algebra has an A-infinity structure, allowing us to combine contact structures not just by gluing, but also by higher-order operations.

In this paper we investigate such A-infinity structures and higher order operations on contact structures. We give explicit constructions of such A-infinity structures, and establish some of their properties, including conditions for the vanishing and nonvanishing of A-infinity operations. Along the way we develop several related notions, including a detailed consideration of tensor products of strand diagrams.

\end{abstract}

%If a table of contents is desired.

\tableofcontents

\section{Introduction}

\subsection{Overview}

In previous work \cite{mathews_strand_2016} we demonstrated an isomorphism of two unital $\Z_2$-algebras: the first arising from contact geometry; the second from bordered Floer theory.
\begin{equation}
\label{eqn:main_iso}
CA (\Sigma, Q) \cong H(\A(\ZZ))
\end{equation}
Here $(\Sigma, Q)$ is a \emph{quadrangulated surface}, a useful object in TQFT-type structures in contact geometry \cite{mathews_itsy_2014, mathews_twisty_2014}, and $\ZZ$ is an \emph{arc diagram}, an equivalent object used in bordered sutured Floer theory %to describe the surface bordering a 3-manifold 
% lipshitz_slicing_2008, lipshitz_notes_2012, , zarev_joining_2010
\cite{zarev_bordered_2009}. The left hand side $CA(\Sigma, Q)$ is the algebra of a \emph{contact category}, with objects and morphisms given by certain contact structures on $\Sigma \times [0,1]$. The right hand side $H(\A(\ZZ))$ is the homology of the strand algebra $\A(\ZZ)$, a differential graded algebra (DGA) generated by \emph{strand diagrams} on $\ZZ$, which encode Reeb chords arising as asymptotics of certain holomorphic curves. 
%So the left hand side of equation (\ref{eqn:main_iso}) is built out of contact geometry, and the right hand side is built out of strand diagrams. 
The isomorphism (\ref{eqn:main_iso}) therefore allows us to interpret (homology classes of) strand diagrams as contact structures.

Of particular interest, (\ref{eqn:main_iso}) expresses the contact category algebra as the homology of a DGA. The homology of a DGA is known to have the structure of an $A_\infty$ algebra. This $A_\infty$ structure provides a sequence of higher-order operations $X_n$ on the homology, extending from multiplication $X_2$ and satisfying relations which provide a homotopy-theoretic form of associativity \cite{stasheff_homotopy_1963, stasheff_homotopy_1963-1}.

While $A_\infty$ structures are well known to arise in Floer theory %, from the splittings of homolomorphic discs 
(see e.g. \cite{seidel_fukaya_2008}), it is perhaps surprising that an $A_\infty$ structure should arise directly out of contact structures.
The $A_\infty$ operations allow us to combine contact structures not just by gluing, but also by higher-order operations.
A natural question arises: what are the higher $A_\infty$ operations on contact structures, and what do they mean geometrically?

This paper essentially consists of an investigation of $A_\infty$ structures on this contact category algebra. This investigation is carried out through the use of strand diagrams, which are more general objects, and easier to work with algebraically than contact structures. Therefore, more accurately, this paper consists of an investigation of $A_\infty$ structures on $H(\A(\ZZ))$, from a contact-geometric perspective.

Throughout this paper we work with $\Z_2$ coefficients; signs are always irrelevant.

\subsection{Main results}

Our first main result is the explicit construction of $A_\infty$ structures on $H(\A(\ZZ))$. 

\begin{thm}
\label{thm:first_thm}
A pair ordering of $\ZZ$ can be used to define an explicit $A_\infty$ structure $X$ on $H(\A(\ZZ))$, together with a morphism of $A_\infty$ algebras $f \colon H(\A(\ZZ)) \To \A(\ZZ)$. These consist of maps
\[
X_n \colon H(\A(\ZZ))^{\otimes n} \To H(\A(\ZZ)),
\quad
f_n \colon H(\A(\ZZ))^{\otimes n} \To \A(\ZZ),
\]
where $X$ extends the DGA structure of $H(\A(\ZZ))$, and $\A(\ZZ)$ is regarded as an $A_\infty$ algebra with trivial $n$-ary operations for $n \geq 3$.
% $f$ extends a cycle selection map $f_1 \colon H(\A(\ZZ)) \To \A(\ZZ)$.
\end{thm}
By (\ref{eqn:main_iso}), theorem \ref{thm:first_thm} provides $A_\infty$ structures on the contact category algebra $CA(\Sigma,Q)$.

We will discuss \emph{pair orderings} as we proceed (section \ref{sec:ordering}); they consist of a total order on the matched pairs of $\ZZ$, along with an ordering of the two points in each pair. In fact the full statement (theorem \ref{thm:main_thm}) allows for a slightly more general $A_\infty$ structures, using certain types of ``choice functions" to parametrise the various choices involved in the construction.

The second main result provides necessary conditions under which these $A_\infty$ maps are nontrivial, and under those conditions gives an explicit description of the results. The idea is that certain ``local" conditions at the matched pairs of $\ZZ$ are necessary to obtain nonzero output from the $A_\infty$ maps.
\begin{thm}
\label{thm:second_thm}
Let $M = M_1 \otimes \cdots \otimes M_n \in H(\A(\ZZ))^{\otimes n}$ be a tensor product of nonzero homology classes of strand diagrams. The maps $f_n$ and $X_n$ of theorem \ref{thm:first_thm} have the following properties.
\begin{enumerate}
\item
If $\fbar_n (M) \neq 0$, then $M$ has $l$ twisted and $m$ critical matched pairs, where $l+m \geq n-1$ and $m \leq n-2$, and all other matched pairs are tight. In this case $\fbar_n (M)$ is a sum of strand diagrams, where each diagram $D$ is tight at all matched pairs where $M$ is critical or tight, and has $n-1-m$ crossed and $l+m-n+1$ twisted matched pairs.
\item
If $X_n (M) \neq 0$, then $M$ has precisely $n-2$ critical matched pairs, and all other matched pairs tight. In this case, $X_n (M)$ is the unique homology class of tight diagram with the appropriate gradings.
\end{enumerate}
\end{thm}
All the terminology will be defined in due course. Very roughly, $\fbar_n$ is the projection of $f_n$ into a useful quotient algebra; matched pairs are objects which appear in the arc diagrams on which strand diagrams are drawn; and the words ``tight", ``twisted", and ``critical" are descriptions of types of configurations of strands in strand diagrams (and their homology classes and their tensor products).

The other main results involve the notion of \emph{operation trees}. These will be defined in due course (section \ref{sec:operation_trees}). They consist of rooted plane binary trees with vertices labelled by strand diagrams or contact structures; they encode the way in which contact structures can be combined by the various $A_\infty$ operations. Trees have commonly been used to encode $A_\infty$ operations (e.g. \cite{keller_introduction_1999, kontsevich_homological_2001, seidel_fukaya_2008}).

Certain trees of this type are required to obtain nonzero output from an $A_\infty$ operation.
\begin{prop}
\label{prop:intro_nonzero_trees}
If $X_n (M) \neq 0$ or $\fbar_n (M) \neq 0$, there is a valid distributive operation tree for $M$.
\end{prop}
%A full statement is proposition \ref{prop:nonzero_implies_trees}.

Our final main result gives sufficient conditions on diagrams and trees which ensure a nonzero result; this result is again described explicitly.
\begin{thm} \
\label{thm:third_thm}
\begin{enumerate}
\item
Suppose $M$ has no on-on doubly occupied matched pairs.
If every valid distributive operation tree for $M$ is strictly $f$-distributive, and at least one such tree exists, then $\fbar_n (M) \neq 0$. Moreover, $\fbar_n (M)$ is given by a single diagram $D$, which can be described explicitly.
\item
Suppose $M$ has no twisted or on-on doubly occupied matched pairs. If every valid distributive operation tree for $M$ is strictly $X$-distributive, and at least one such tree exists, then $X_n (M) \neq 0$. Moreover, $X_n (M)$ is given by the homology class of unique tight diagram with appropriate gradings.
\end{enumerate}
\end{thm}
Very roughly, ``on-on" and ``doubly occupied" refer to particular configurations of strand diagrams at a matched pair; an operation tree is ``valid" if the labels are ``non-singular" in an appropriate sense; and it is ``distributive" if the contact structures labelling the tree have their ``twistedness" spread across its various leaves in an appropriate sense.

As we will explain, these results are quite partial. The necessary conditions of theorem \ref{thm:second_thm} are far from sufficient, and the sufficient conditions of theorem \ref{thm:third_thm} are far from necessary. Since there are many $A_\infty$ structures on $H(\A(\ZZ))$, we cannot expect a complete characterisation of diagrams which yield zero and nonzero results; still, we hope these results can be improved. 

As is already clear, there is a \emph{lot} of terminology to define. Simply stating these results requires us to describe precisely many aspects of strand diagrams, and their tensor products and homology classes. We must name this world in order to understand it.

\subsection{Construction of A-infinity structures}
\label{sec:intro_construction}

In a certain sense, the $A_\infty$ structures on $CA(\Sigma, Q)$ or $H(\A(\ZZ))$ are already understood. In the 1980 paper \cite{kadeishvili_homology_1980}, Kadeishvili showed how to define an $A_\infty$ structure on the homology $H$ of any differential graded algebra $A$ (provided $H$ is free, which is always true with $\Z_2$ coefficients). Indeed, in this paper we follow this construction, and theorem \ref{thm:first_thm} can be regarded as fleshing out its details when $A = \A(\ZZ)$. The only thing possibly new in theorem \ref{thm:first_thm} is the level of explicitness in the construction.

We briefly recall some facts about $A_\infty$ algebra; we refer to Keller's \cite{keller_introduction_1999} for an introduction to $A_\infty$ algebra, or to Seidel \cite{seidel_fukaya_2008} for further details. An \emph{$A_\infty$ structure} $m$ on a $\Z$-graded $\Z_2$-module $A$ is a collection of operations $m_n \colon A^{\otimes n} \To A$ for each $n \geq 1$, where each $m_n$ has degree $n-2$. We call $m_n$ the \emph{$n$-ary} or \emph{level $n$} operation. The operations $m_i$ satisfy, for each $n \geq 1$,
% and any $a_1, \ldots, a_n \in A$,
%\[
%\sum_{i+j+k=n} m_{i+1+k} \left( a_1 \otimes \cdots \otimes a_i \otimes m_j \left( a_{i+1} \otimes \cdots \otimes a_{i+j} \right) \otimes a_{i+j+1} \otimes \cdots \otimes a_{i+j+k} \right) = 0,
%\]
%or more concisely,
\[
\sum_{i+j+k=n} m_{i+1+k} \left( 1^{\otimes i} \otimes m_j \otimes 1^{\otimes k} \right) = 0.
\]
This identity for $n=1$ says that $m_1^2 = 0$, so $m_1$ is a differential; then the identity for $n=2$ is the Leibniz rule, with $m_2$ regarded as multiplication. Indeed an $A_\infty$ algebra with all $m_n = 0$ for $n \geq 3$ is precisely a DGA. A \emph{morphism} $f$ of $A_\infty$ algebras $A \To A'$ (where the operations on $A,A'$ are denoted $m_i, m'_i$ respectively) is a collection of $\Z_2$-module homomorphisms $f_n \colon A^{\otimes n} \To A'$, where each $f_n$ has degree $n-1$. We call $f_n$ the \emph{level $n$} map. The maps $f_i$ satisfy, for each $n \geq 1$,
% and any $a_1, \ldots, a_n \in A$,
%\[
%\sum_{i+j+k=n} f_{i+1+k} \left( a_1 \otimes \cdots \otimes a_i \otimes m_j \left( a_{i+1} \otimes \cdots \otimes a_{i+j} \right) \otimes a_{i+j+1} \otimes \cdots \otimes a_{i+j+k} \right)
%=
%\sum_{i_1 + \cdots + i_s = n} m'_s \left( f_{i_1} \left( a_1 \otimes \cdots \otimes a_{i_1} \right) \otimes \cdots \otimes f_{i_s} \left( a_{i_1 + \cdots + i_{s-1} + 1} \otimes \cdots \otimes a_{i_1 + \cdots + i_{s-1} + i_s} \right) \right),
%\]
%or more concisely,
\[
\sum_{i+j+k = n} f_{i+1+k} \left( 1^{\otimes i} \otimes m_j \otimes 1^{\otimes k} \right)
=
\sum_{i_1 + \cdots + i_s = n} m'_s \left( f_{i_1} \otimes f_{i_2} \otimes \cdots \otimes f_{i_s} \right).
\]

Kadeishvili's construction in \cite{kadeishvili_homology_1980} produces an $A_\infty$ structure $X$ on $H$, consisting of operations $X_n \colon H^{\otimes n} \To H$, % of degree $n-2$,
and a morphism $f$ of $A_\infty$ algebras $H \To A$, consisting of maps $f_n \colon H^{\otimes n} \To A$. % of degree $n-1$. 
The DGA $A$ is regarded as an $A_\infty$ algebra with trivial $n$-ary operations for $n \geq 3$. The $A_\infty$ structure constructed on $H$ begins with trivial differential $X_1 = 0$, and $X_2$ is the multiplication on $H$ inherited from $A$. If $H$ is free then there is a map $f_1 \colon H \To A$ (possibly many) which is an isomorphism in homology, sending each homology class to a cycle representative. The constructed $f_n$ can be taken to begin with any such $f_1$.

The construction proceeds inductively, producing maps $U_n \colon H^{\otimes n} \To A$ of degree $n-2$ along the way. First, $U_1 = 0$, $X_1 = 0$, and $f_1 \colon H \To A$ are given. Once $U_i, X_i, f_i$ are defined for $i<n$, we define $U_n$ by
\begin{align}
\label{eqn:Un_def}
U_n \left( a_1 \otimes \cdots \otimes a_n \right) &=
\sum_{j=1}^{n-1} m_2 \left( f_j \left( a_1 \otimes \cdots \otimes a_j \right) \otimes f_{n-j} \left( a_{j+1} \otimes \cdots \otimes a_n \right) \right) \nonumber \\
&\quad + \sum_{k=0}^{n-2} \sum_{j=2}^{n-1} f_{n-j+1} \left( a_1 \otimes \cdots \otimes a_k \otimes X_j \left( a_{k+1} \otimes \cdots \otimes a_{k+j} \right) \otimes \cdots \otimes a_n \right),
\end{align}
and $X_n$ is then simply the homology class of $U_n$,
\begin{equation}
\label{eqn:Xn_def}
X_n = [U_n].
\end{equation}
Since $f_1$ selects cycles, $f_1 X_n$ and $U_n$ differ by a boundary; $f_n$ is then defined by
\begin{equation}
\label{eqn:fn_eqn}
f_1 X_n - U_n = \partial f_n.
\end{equation}
It is then shown that such $f_n$ and $X_n$ have the desired properties.

Clearly, in this construction there is a choice for $f_n$ at each stage, but no choice for $U_n$ or $X_n$. This choice amounts to a choice of inverse for the differential $\partial$.

Our construction, detailed in section \ref{sec:construction}, gives an explicit way to choose an $f_n$ at each stage. This choice is made by maps which we call \emph{creation operators}. We regard the differential in $\A(\ZZ)$ as an ``annihilation operator", destroying crossings between strands by resolving them. Creation operators, on the other hand, insert crossings in a controlled way. The idea is shown in figure \ref{fig:creation_example}. We introduce creation operators in section \ref{sec:cycle_selection_creation}. Creation operators satisfy Heisenberg relations (proposition \ref{prop:creation_operator_Heisenberg}); this amounts to a chain homotopy from the identity to zero. In a certain sense, creation operators are the only operators obeying such Heisenberg relations (proposition \ref{prop:chain_homotopy_classification}); however they only form a very small subspace of the space of operators inverting the differential as required in Kadeishvili's construction (proposition \ref{prop:inverting_vs_creation}). Similar ``creation operators" have been put to use elsewhere in contexts related to contact geometry and Floer homology \cite{mathews_strings_2017, mathews_dimensionally_2015}.

\begin{figure}
\begin{center}
\begin{tikzpicture}[scale=1.5]
\strandbackgroundshading
\afterwused
\beforewused
\strandsetupn{}{}
\lefton
\righton
\usea
\useb
\draw (2,0.75) node {$\To$};
\begin{scope}[xshift=3 cm]
\strandbackgroundshading
\afterwused
\beforewused
\strandsetupn{}{}
\lefton
\righton
	\dothorizontals
\useab
\end{scope}
\end{tikzpicture}

\caption{The action of a creation operator.}
\label{fig:creation_example}
\end{center}
\end{figure}
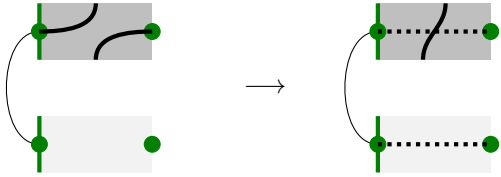

However, there is still choice involved in where to apply creation operators, i.e. where to insert crossings. There is also a choice for the initial cycle selection homomorphism $f_1$. We parametrise such choices through notions of \emph{creation choice functions} and \emph{cycle choice functions} respectively. Our construction in general (theorem \ref{thm:main_thm}) produces an $A_\infty$ structure on $H(\A(\ZZ))$ or $CA(\Sigma, Q)$ from a given cycle choice function and creation choice function. A pair ordering can be used to obtain such choice functions, leading to the formulation of theorem \ref{thm:first_thm}.

In order to define the $A_\infty$ structure $X$ on $H(\A(\ZZ))$, it turns out to be sufficient to work in a particular \emph{quotient} of $\A(\ZZ)$. This simplifies details considerably. We define a two-sided ideal $\F$ in section \ref{sec:ideals}. The maps $\fbar_n$ appearing in theorems \ref{thm:second_thm} and \ref{thm:third_thm} are the images of $f_n$ in the quotient by $\F$. Related ideas appeared in \cite{lipshitz_bimodules_2015}.

Algorithmically, the calculation of an $A_\infty$ map $X_n (M_1 \otimes \cdots \otimes M_n)$, where $M_1, \ldots, M_n$ are homology classes of strand diagrams --- or contact structures --- by the method described above requires the computation of each $f_{j-i+1} (M_i \otimes \cdots \otimes M_j)$ and $X_{j-i+1} (M_i \otimes \cdots \otimes M_j)$, for $1 \leq i \leq j \leq n$. The algorithm therefore has complexity $O(n^2)$ (where we regard each computation of expressions such as (\ref{eqn:Un_def}) as constant time, and the complexity of the arc diagram $\ZZ$, as constant). The contrapositive of theorem \ref{thm:second_thm} provides a set of conditions which imply $X_n (M) = 0$, which are easily checked in constant time. On the other hand, proposition \ref{prop:intro_nonzero_trees} and theorem \ref{thm:third_thm} provide conditions which are much more difficult to check, as the number of operation trees grows much faster with $n$. We regard these results as interesting not because of algorithmic usefulness, but because they perhaps provide some insight into $A_\infty$ operations.

\subsection{Classifications of diagrams, and the many types of twisted}
\label{sec:many_types_twisted}

As already mentioned above, there are many features of strand diagrams which are relevant for our purposes, but which have not been given names in the existing literature. Large parts of this paper, especially sections \ref{sec:algebra_and_anatomy} and \ref{sec:tensor_products}, are devoted to defining and classifying these features, and establishing some of their properties. These are all required for our main theorems. 

Therefore, some of the work here is an exercise in taxonomy. We briefly explain what we need to define and why, and the resulting classifications.

Contact structures naturally come in two types: tight and overtwisted. This dichotomy goes back to Eliashberg's work in the 1980s \cite{eliashberg_classification_1989}. In the present work, consideration of the relationship between strand diagrams and contact structures naturally leads to further distinctions. Roughly speaking, when we look at strand diagrams from a contact-geometric perspective, there are \emph{many} types of ``twisted".

According to the isomorphism (\ref{eqn:main_iso}) of \cite{mathews_strand_2016}, tight contact structures correspond to strand diagrams which are nonzero in homology. Such diagrams are characterised by certain conditions; roughly speaking, they must have an appropriate grading, no crossings, and must not have any matched pair that looks like the left of figure \ref{fig:creation_example}. A strand diagram which fails one or more of these conditions can be regarded as ``overtwisted" in some sense.

The simplest way for a diagram to fail to represent a tight contact structure is by grading: it may lie in a summand of $\A(\ZZ)$ which has no homology. This leads to the notion of \emph{viability} (section \ref{sec:viability}). Only viable strand diagrams can possibly represent contact structures.

It is essential for our purposes to have precise terminology relating to these gradings and summands. We introduce a notion of \emph{H-data}, which combines homological grading and idempotents (section \ref{sec:strand_diagrams}). We also  introduce notions of \emph{on/off} or \emph{1/0} to describe idempotents locally, and \emph{occupation} of various parts of a strand diagram, such as \emph{places} and \emph{steps}, to describe homological grading locally (section \ref{sec:terminology_strand_diagrams}).  Some of this terminology was used in \cite{mathews_strand_2016}.

A viable diagram can still fail to represent a tight contact structures for multiple reasons. In the mildest case, shown in figure \ref{fig:twisted_example}, a strand diagram is the product of strand diagrams corresponding to tight contact structures, but the full contact structure is overtwisted. We define such ``minimally overtwisted" diagrams as \emph{twisted} in section \ref{sec:tight_twisted_crossed_diagrams}.

This contact structure can be described explicitly in terms of contact cubes: see \cite{mathews_strand_2016}. (In fact, stacking only the two relevant cubes yields a tight contact structure; when combined with adjacent cubes however the structure is overtwisted). It can also be described in terms of bypasses. In a future paper we hope to describe the relationship between strands and bypasses systematically.

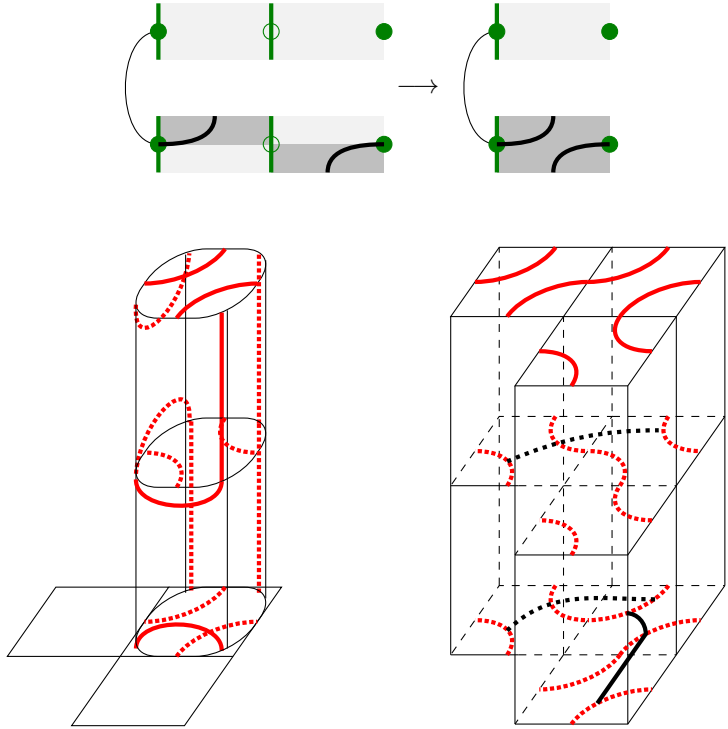
\begin{figure}
\begin{center}
\begin{tikzpicture}[scale=1.5]
\strandbackgroundshading
\tstrandbackgroundshading{1}
\aftervused
\tbeforevused{1}
\strandsetupn{}{}
\tstrandsetup{1}
\lefton
\rightoff
\trighton{1}
\usec
\tused{1}
\draw (2.3,0.75) node {$\To$};
\begin{scope}[xshift=3 cm]
\strandbackgroundshading
\aftervused
\beforevused
\strandsetupn{}{}
\lefton
\righton
\usec
\used
\end{scope}
\end{tikzpicture}
\end{center}

\vspace{0.2 cm}

\begin{center}
\begin{tikzpicture}[mypersp, scale=1.5]
\draw [xyp=3, rounded corners=5mm] (0,0) -- (0,1) -- (1,1) -- (1,0) -- cycle;

\draw [xshift=-1 cm] (0,0) -- (1,0) -- (1,1) -- (0,1) -- cycle;
\draw [xshift=-1 cm] (1,0) -- (2,0) -- (2,1) -- (1,1);
\draw [xshift=-1 cm] (1,0) -- (1,-1) -- (2,-1) -- (2,0);

%\draw [red, ultra thick, densely dotted, yzp=0.14] (0,0.12) to [bend left=90] (1.1,-0.1);
%\draw [red, ultra thick, densely dotted, yzp=0.14] (0,1.12) to [bend right=90] (1.1,0.9);
%\draw [red, ultra thick, xzp=0] (0.9,0.05) -- (0.9,1.05);
%\draw [red, ultra thick, densely dotted, xzp=1] (0.8,-0.05) -- (0.8,0.95);

\draw [red, ultra thick, densely dotted, yzp=0.14] (0,1.62) to [bend left=90] (1.1,1.4);
\draw [red, ultra thick, densely dotted, yzp=0.14] (0,3.12) to [bend right=90] (1.1,2.9);
\draw [red, ultra thick, xzp=0] (0.9,1.55) -- (0.9,3.05);
\draw [red, ultra thick, densely dotted, xzp=1] (0.8,1.45) -- (0.8,2.95);

\draw [xyp=1.5, rounded corners=5mm] (0,0) -- (0,1) -- (1,1) -- (1,0) -- cycle;

\draw [xzp=1] (0.15,-0.05) -- (0.15,2.95);
\draw [xzp=1] (0.86,-0.1) -- (0.86,2.9);
\draw [xzp=0] (0.14,0.12) -- (0.14,3.12);
\draw [xzp=0] (0.95,0.07) -- (0.95,3.07);

%\draw [red, ultra thick, densely dotted, xzp=1] (0.8,0.95) -- (0.8,1.95);
%\draw [red, ultra thick, densely dotted, xzp=1] (0.2,0.97) -- (0.2,1.97);
%\draw [red, ultra thick, xzp=0] (0.14,1.07) to [bend left=90] (0.9,1.05);
%\draw [red, ultra thick, xzp=0] (0.14,2.07) to [bend right=90] (0.9,2.05);

\draw [red, ultra thick, densely dotted, xzp=1] (0.8,-0.05) -- (0.8,1.45);
\draw [red, ultra thick, densely dotted, xzp=1] (0.2,-0.03) -- (0.2,1.47);
\draw [red, ultra thick, xzp=0] (0.14,0.07) to [bend left=90] (0.9,0.05);
\draw [red, ultra thick, xzp=0] (0.14,1.57) to [bend right=90] (0.9,1.55);

\draw [xyp=0, rounded corners=5mm] (0,0) -- (0,1) -- (1,1) -- (1,0) -- cycle;

\draw [red, ultra thick, xyp=3] (0.5,0) arc (180:90:0.5);
\draw [red, ultra thick, xyp=3] (0.5,1) arc (0:-90:0.5);

\draw [red, ultra thick, densely dotted, xyp=1.5] (0.5,0) arc (0:90:0.5);
\draw [red, ultra thick, densely dotted, xyp=1.5] (0.5,1) arc (-180:-90:0.5);

\bottomon
%\draw (1.5,0.75) node {$\To$};
\end{tikzpicture}
\hspace{2 cm}
\begin{tikzpicture}[mypersp, scale=1.5]
\draw[xyp=0, dashed] (0.6,0) -- (1,0) -- (1,1) -- (0,1) -- (0,0);
\draw[xyp=0, dashed] (1,1) -- (2,1);
\draw[xyp=0, dashed] (1,0) -- (2,0);
\draw[xyp=0, dashed] (1,0) -- (1,-1);
\draw[xyp=0] (1,-1) -- (2,-1) -- (2,1);
\draw [xyp=0] (0,0) -- (0.6,0);

\draw[xyp=1.5, dashed] (0.6,0) -- (1,0) -- (1,1) -- (0,1) -- (0,0);
\draw[xyp=1.5, dashed] (1,1) -- (2,1);
\draw[xyp=1.5, dashed] (1,0) -- (2,0);
\draw[xyp=1.5, dashed] (1,0) -- (1,-1);
\draw[xyp=1.5] (1,-1) -- (2,-1) -- (2,1);
\draw[xyp=1.5] (0,0) -- (0.6,0);

\draw[xyp=3] (0,0) -- (1,0) -- (1,1) -- (0,1) -- cycle;
\draw[xyp=3] (1,0) -- (2,0) -- (2,1) -- (1,1);
\draw[xyp=3] (1,0) -- (1,-1) -- (2,-1) -- (2,0);

\draw[xzp=0] (0,0) -- (0,3);
\draw[xzp=0, dashed] (1,0) -- (1,3);
\draw[xzp=0] (2,0) -- (2,3);
\draw[xzp=1] (2,0) -- (2,3);
\draw[xzp=1, dashed] (1,0) -- (1,3);
\draw[xzp=1, dashed] (0,0) -- (0,3);
\draw[xzp=-1] (1,0) -- (1,3);
\draw[xzp=-1] (2,0) -- (2,3);

\draw [red, ultra thick, densely dotted] (0.5,0) arc (0:90:0.5);
\draw [red, ultra thick, densely dotted] (0.5,1) arc (-180:-90:0.5);

\draw [red, ultra thick, densely dotted] (1.5,0) arc (180:90:0.5);
\draw [red, ultra thick, densely dotted] (1.5,1) arc (0:-90:0.5);

\draw [red, ultra thick, densely dotted] (1.5,0) arc (0:-90:0.5);
\draw [red, ultra thick, densely dotted] (1.5,-1) arc (180:90:0.5);

\draw [dotted, ultra thick] (0.35, 0.35) to [out=45, in=180] (1.5, 0.8);
\draw [ultra thick] (1.3,0.6) to [out=0, in=90] (1.6,0.3) to (1.6,-0.7);

\draw [xyp=1.5, red, ultra thick, densely dotted] (1.5,-1) arc (0:90:0.5);
\draw [xyp=1.5, red, ultra thick, densely dotted] (2,-0.5) arc (-90:-180:0.5);

\draw [xyp=1.5, red, ultra thick, densely dotted] (0.5,0) arc (0:90:0.5);
\draw [xyp=1.5, red, ultra thick, densely dotted] (0.5,1) arc (-180:-90:0.5);

\draw [xyp=1.5, red, ultra thick, densely dotted] (1.5,0) arc (0:90:0.5);
\draw [xyp=1.5, red, ultra thick, densely dotted] (1.5,1) arc (-180:-90:0.5);

\draw [xyp=1.5, dotted, ultra thick] (0.35, 0.35) to [out=45, in=180] (1.5, 0.8);

\draw [xyp=3, red, ultra thick] (1.5,-1) arc (0:90:0.5);
\draw [xyp=3, red, ultra thick] (2,-0.5) arc (-90:-180:0.5);

\draw [xyp=3, red, ultra thick] (0.5,0) arc (180:90:0.5);
\draw [xyp=3, red, ultra thick] (0.5,1) arc (0:-90:0.5);

\draw [xyp=3, red, ultra thick] (1.5,0) arc (180:90:0.5);
\draw [xyp=3, red, ultra thick] (1.5,1) arc (0:-90:0.5);
\end{tikzpicture}

\caption{Top: A twisted diagram at a matched pair, the product of two tight diagrams. Below left: the corresponding contact cubes. Stacking the two contact cubes corresponding to the matched pair shown yields a contact structure which remains tight, but combined with adjacent cubes the contact structure is overtwisted. Below right: the same contact structure described in terms of bypasses. A bypass is first attached to the bottom dividing set along the solid arc, yielding the intermediate dividing set; then a bypass is added along the dotted attaching arc. The overtwisted disc can be seen by attaching the latter bypass first, as seen on the bottom dividing set.}
\label{fig:twisted_example}
\end{center}
\end{figure}

Viable strand diagrams can also fail to represent tight contact structures because they have \emph{crossings}. Thus, the natural tight/overtwisted classification of contact structures naturally becomes a 3-fold classification of viable strand diagrams into tight/twisted/crossed. This classification is, in a precise sense,  (lemma \ref{lem:tightness_degeneracy}), in ascending order of degeneracy.

When we proceed to \emph{tensor products} of diagrams in section \ref{sec:tensor_products}, there is again a natural notion of viability (section \ref{sec:anatomy_tensor_products}). Diagrams can represent contact structures, and their tensor products can be regarded as ``stacked" contact structures on $\Sigma \times [0,1]$. Viability then incorporates the natural contact-geometric condition that such stacked structures agree along their common boundaries. %See figure \ref{fig:tensor_product_cubes} for an example.

Tensor products of diagrams again have a natural ``tight/twisted" classification, (section \ref{sec:tightness_tensor_products}), but now there are \emph{six} types, which we call \emph{tight}, \emph{sublime}, \emph{twisted}, \emph{crossed}, \emph{critical}, and \emph{singular}, again in an ascending scale of degeneracy.

When we then arrive at tensor products of \emph{homology classes} of diagrams in section \ref{sec:tensor_product_homology}, only \emph{four} of these types of tightness/twistedness remain.

Throughout, it is necessary to consider strand diagrams \emph{locally} at matched pairs; this corresponds to considering contact structures locally at individual cubes of a cubulated contact structure. Indeed, we show that $H(\A(\ZZ))$ decomposes into a \emph{tensor product} over matched pairs; and we have \emph{local} strand algebras, each with their \emph{local} homology at each matched pair.

Here again, we encounter phenomena not yet given a name in the literature. The observed local diagrams, described as ``fragments" of strand diagrams in \cite{mathews_strand_2016}, are not strand diagrams in the usual sense of bordered Floer theory (e.g. \cite{lipshitz_notes_2012}) or bordered sutured Floer theory (e.g. \cite{zarev_bordered_2009}), since strands may ``run off the top of an arc". Therefore, before we can even start our investigations, we must broaden the usual definition of strand diagrams. In section \ref{sec:algebra_and_anatomy} we therefore introduce a  notion of \emph{augmented strand diagram}.

Tensor products of strand diagrams (i.e. elements of $\A(\ZZ)^{\otimes n}$), or their homology classes (i.e. elements of $H(\A(\ZZ))^{\otimes n}$) thus have a tensor decomposition over matched pairs of $\ZZ$, into local diagrams, in addition to their obvious decomposition into tensor factors. We regard these two types of decomposition as ``vertical" and ``horizontal" respectively, and draw pictures accordingly. Contact-geometrically these two types of tensor decomposition correspond to two types of geometric decompositions of stacked contact structures. An element of $CA(\Sigma,Q)^{\otimes n} \cong H(\A(\ZZ))^{\otimes n}$ can be regarded as a stacking of $n$ cubulated contact structures on $\Sigma \times [0,1]$: this can be cut ``horizontally" into $n$ slices, each containing a contact structure on $\Sigma \times [0,1]$; or it can be cut ``vertically" to obtain stacked contact structures on $\square \times [0,1]$, over each square $\square$ of the quadrangulation.

We give a complete classification of viable local strand diagrams in section \ref{sec:terminology_strand_diagrams}, summarised in table \ref{tbl:local_diagrams}. We also give a complete classification of viable local tensor products of strand diagrams in section \ref{sec:enumerate_local_tensor_products}, summarised in table \ref{tbl:local_tensor_products}. We show (proposition \ref{prop:viable_tensor_product_classification}) that any viable tensor product of diagrams, observed locally at a single matched pair, must appear as one of the tensor products in the table, up to a notion of \emph{extension} and \emph{contraction}, which provide ways, trivial in a contact-geometric sense, to grow or shrink a tensor product. This also yields (proposition \ref{prop:homology_tensor_product_classification}) a complete classification of viable local tensor products of homology classes of strand diagrams, in section \ref{sec:tensor_product_homology}.

Having made such definitions and classifications, we also establish several of their basic properties. In order to prove our main theorems, we need to know facts such as which types of tightness/twistedness can occur within others, as ``sub-tensor-products", ``vertically" or ``horizontally". We consider these and several more questions as we proceed.

\subsection{Contact meaning of A-infinity operations}
\label{sec:intro_contact_meaning}

We now attempt to give some idea of what the $A_\infty$ operations $X_n$ ``mean" in terms of contact geometry. For details and background on the precise correspondence between contact structures and strand diagrams, we refer to our previous paper \cite{mathews_strand_2016}. We also intend to expand on the contact-geometric meaning of strand diagrams, particularly in terms of bypasses, in a future paper.

As discussed in \cite{mathews_strand_2016}, a strand diagram $D$ on an arc diagram $\ZZ$ with appropriate grading (each step of $\ZZ$ covered at most once; no crossings) can be interpreted as a contact structure on $\Sigma \times [0,1]$. Each matched pair of $\ZZ$ corresponds to a square of the quadrangulation $Q$, or a cube in the cubulation $Q \times [0,1]$ of $\Sigma \times [0,1]$.

A strand diagram $D$ containing a single moving strand going from one point (``place" in our terminology) of $\ZZ$ to the next can be regarded as a \emph{bypass}: in passing from one strand to the next, the strand affects two places, and the corresponding contact structure is a \emph{bypass addition}, where the bypass is placed along the two cubes. Bypass addition is a basic operation in 3-dimensional contact geometry \cite{honda_classification_2000}, and in a certain sense is the ``simplest" modification one can make to a contact manifold \cite{honda_gluing_2002}. The result is shown in figure \ref{fig:strand_bypass}.

\begin{figure}
\begin{center}
\begin{tikzpicture}[scale=2]
%\strandbackgroundshading
\draw [draw=none, fill=gray!10!white] (0,0) -- (0,1) -- (1,1) -- (1,0) -- cycle;
%\draw [draw=none, fill=gray!10!white] (0,1) -- (0,1.5) -- (1,1.5) -- (1,1) -- cycle;
%\aftervused
\draw [draw=none, fill=gray!50!white] (0,0.25) -- (0,0.75) -- (1,0.75) -- (1,0.25) -- cycle;
%\strandsetup
\draw [ultra thick, color=black!50!green] (0,0) -- (0,1);
\draw [ultra thick, color=black!50!green] (1,0) -- (1,1);
%\draw (-0.15,0.15) node {$q$};
%\draw (-0.15,1.35) node {$p$};
%\lefton
\draw [color=black!50!green, fill=black!50!green] (0,0.25) circle (2 pt);
\draw [color=black!50!green, fill=none] (0,0.75) circle (2 pt);
%\rightoff
\draw [color=black!50!green, fill=none] (1,0.25) circle (2 pt);
\draw [color=black!50!green, fill=black!50!green] (1,0.75) circle (2 pt);
%\usea
\draw [ultra thick] (0,0.25) to [out=0, in=-90] (0.5,0.5);
%\useb
\draw [ultra thick] (0.5,0.5) to [out=90, in=180] (1,0.75);
\end{tikzpicture}
\hspace{2 cm}
\begin{tikzpicture}[mypersp, scale=1.5]
\drawcube
%\drawvertices
\draw [xyp=0, color=black!50!green, fill=black!50!green] (1,0) circle (2pt);
\draw [xyp=0, color=black!50!green, fill=black!50!green] (0,1) circle (2pt);
\draw [xyp=0, color=black!50!green, fill=black!50!green] (2,1) circle (2pt);
\draw [xyp=1, color=black!50!green, fill=black!50!green] (1,0) circle (2pt);
\draw [xyp=1, color=black!50!green, fill=black!50!green] (0,1) circle (2pt);
\draw [xyp=1, color=black!50!green, fill=black!50!green] (2,1) circle (2pt);
%\labelvertices
\leftclean
\backclean
\frontclean
%\rightbyp
\fill [gray!30!white, draw=none, yzp=1, opacity=0.5] (0,0) -- (1,0) -- (1,1) -- (0,1) -- cycle;
\topon
\bottomoff
%\bottomdiagonal
\draw [xyp=0, color=black!50!green, ultra thick] (1,0) -- (0,1);
\begin{scope}[xshift=1 cm]
\drawcube
%\labelvertices
%\leftbyp 
\backclean
\frontclean
\rightclean
\topon
\bottomoff
%\topdiagonal
\draw [xyp=1, color=black!50!green, ultra thick] (0,0) -- (1,1);

\end{scope}
\end{tikzpicture}
\hspace{2 cm}
\begin{tikzpicture}[scale=1]
\draw (0,0) -- (1,0) -- (1,1) -- (0,1) -- cycle;
\draw (1,0) -- (2,0) -- (2,1) -- (1,1);
\draw [red, ultra thick] (0.5,0) arc (0:90:0.5);
\draw [red, ultra thick] (0.5,1) arc (180:270:0.5);
\draw [red, ultra thick] (1.5,0) arc (0:90:0.5);
\draw [red, ultra thick] (1.5,1) arc (180:270:0.5);
\draw [dotted, ultra thick] (0.3535, 0.3535) -- (1.6464, 0.6464);
\draw (1,1.5) node {$\uparrow$};
\begin{scope}[yshift = 2 cm]
\draw (0,0) -- (1,0) -- (1,1) -- (0,1) -- cycle;
\draw (1,0) -- (2,0) -- (2,1) -- (1,1);
\draw [red, ultra thick] (0.5,1) arc (0:-90:0.5);
\draw [red, ultra thick] (0.5,0) arc (180:90:0.5);
\draw [red, ultra thick] (1.5,1) arc (0:-90:0.5);
\draw [red, ultra thick] (1.5,0) arc (180:90:0.5);
\end{scope}
\end{tikzpicture}

\caption{Left: A portion of a strand diagram consisting of a single strand from one place to the next. Centre: The corresponding cubulated contact structure. Right: This contact structure is given by a bypass attachment.}
\label{fig:strand_bypass}
\end{center}
\end{figure}
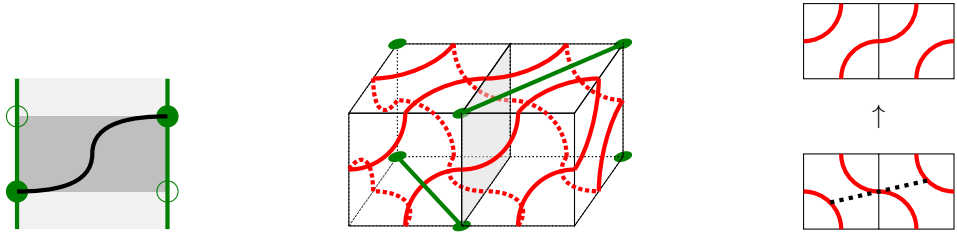

A strand diagram consisting of a longer strand can sometimes be regarded as a product of diagrams with shorter strands, each covering a single step of $\ZZ$ as above. (However, other restrictions may get in the way; for instance arcs of an arc diagram may prevent factorising a longer strand into smaller ones. See e.g. the example of figure 11 of \cite{lipshitz_bimodules_2015}.) 
The corresponding contact structure is given a sequence of bypass additions very closely related to the \emph{bypass systems} of \cite{mathews_chord_2010, mathews_sutured_2011}. See figure \ref{fig:long_strand_bypasses}.

\begin{figure}
\begin{center}
\begin{tikzpicture}[scale=1.5]
%\strandbackgroundshading
\draw [draw=none, fill=gray!10!white] (0,0) -- (0,3) -- (1,3) -- (1,0) -- cycle;
%\draw [draw=none, fill=gray!10!white] (0,1) -- (0,1.5) -- (1,1.5) -- (1,1) -- cycle;
%\aftervused
\draw [draw=none, fill=gray!50!white] (0,0.25) -- (0,1.75) -- (1,1.75) -- (1,0.25) -- cycle;
%\strandsetup
\draw [ultra thick, color=black!50!green] (0,0) -- (0,2);
\draw [ultra thick, color=black!50!green] (1,0) -- (1,2);
%\draw (-0.15,0.15) node {$q$};
%\draw (-0.15,1.35) node {$p$};
%\lefton
\draw [color=black!50!green, fill=black!50!green] (0,0.25) circle (2 pt);
\draw [color=black!50!green, fill=none] (0,0.75) circle (2 pt);
\draw [color=black!50!green, fill=none] (0,1.25) circle (2 pt);
\draw [color=black!50!green, fill=none] (0,1.75) circle (2 pt);
%\rightoff
\draw [color=black!50!green, fill=none] (1,0.25) circle (2 pt);
\draw [color=black!50!green, fill=none] (1,0.75) circle (2 pt);
\draw [color=black!50!green, fill=none] (1,1.25) circle (2 pt);
\draw [color=black!50!green, fill=black!50!green] (1,1.75) circle (2 pt);
%\usea
\draw [ultra thick] (0,0.25) to [out=0, in=180] (1, 1.75);
%\draw [ultra thick] (0,0.25) to [out=0, in=-90] (0.5,0.5);
%\useb
%\draw [ultra thick] (0.5,0.5) to [out=90, in=180] (1,0.75);
\end{tikzpicture}
\hspace{3 cm}
\begin{tikzpicture}[scale=.8]
\draw (0,0) -- (1,0) -- (1,1) -- (0,1) -- cycle;
\draw (1,0) -- (2,0) -- (2,1) -- (1,1);
\draw (0,0) -- (0,-1) -- (0.9,-1) -- (1,0);
\draw (1,0) -- (1.1,-1) -- (2,-1) -- (2,0);
\draw [red, ultra thick] (0.5,0) arc (180:90:0.5);
\draw [red, ultra thick] (0.5,1) arc (0:-90:0.5);

\draw [red, ultra thick] (1.5,0) arc (0:90:0.5);
\draw [red, ultra thick] (1.5,1) arc (180:270:0.5);

\draw [red, ultra thick] (0.5,0) arc (0:-90:0.5);
\draw [red, ultra thick] (0.5,-1) arc (180:95:0.5);

\draw [red, ultra thick] (1.5,0) arc (0:-85:0.5);
\draw [red, ultra thick] (1.5,-1) arc (180:90:0.5);

\draw [dotted, ultra thick] (0.6464, -0.6464) -- (0.3535, 0.6464);
\draw [dotted, ultra thick] (0.1, 0.52) to [out=90, in=180] (0.45, 0.8) to [out=0, in=180] (1.55,0.8);
\draw [dotted, ultra thick] (1.52, 0.9) to [out=0, in=90] (1.9,0.55) to [out=270, in=90] (1.9, -0.52);
\draw (1,1.5) node {$\uparrow$};
\begin{scope}[yshift = 3 cm]
\draw (0,0) -- (1,0) -- (1,1) -- (0,1) -- cycle;
\draw (1,0) -- (2,0) -- (2,1) -- (1,1);
\draw (0,0) -- (0,-1) -- (0.9,-1) -- (1,0);
\draw (1,0) -- (1.1,-1) -- (2,-1) -- (2,0);
\draw [red, ultra thick] (0.5,0) arc (180:90:0.5);
\draw [red, ultra thick] (0.5,1) arc (0:-90:0.5);
\draw [red, ultra thick] (1.5,0) arc (0:90:0.5);
\draw [red, ultra thick] (1.5,1) arc (180:270:0.5);
\draw [red, ultra thick] (0.5,0) arc (180:265:0.5);
\draw [red, ultra thick] (0.5,-1) arc (0:90:0.5);
\draw [red, ultra thick] (1.5,0) arc (180:270:0.5);
\draw [red, ultra thick] (1.5,-1) arc (0:85:0.5);
\end{scope}
\end{tikzpicture}

\caption{Left: A portion of a strand diagram consisting of a single strand.  Right: The corresponding contact structure is given by a sequence of bypass attachments.}
\label{fig:long_strand_bypasses}
\end{center}
\end{figure}
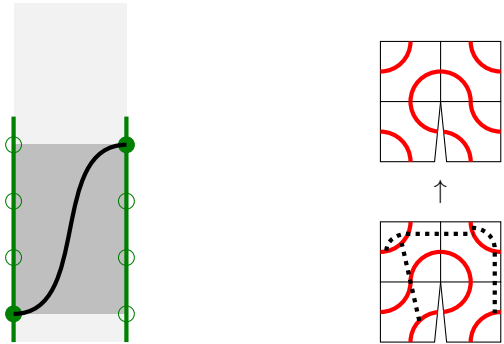

However, when we have a \emph{tensor product} of strand diagrams corresponding to contact structures, the various steps of $\ZZ$ may not be covered in the order in which they would be covered by single strands. If the various diagrams in the tensor product cover the various steps in a matched pair in a ``correct" order, the factors in the tensor product may multiply (using the standard multiplication in $\A(\ZZ)$) to give a diagram which is nonzero in homology. This corresponds to a contact structure built out of bypasses as described above. But if the various diagrams cover the various steps in a different order, then they will not multiply to give something nonzero in homology. Moreover, the \emph{Maslov index} at the matched pair will be lower by $1$ from the ``correct" order.

The simplest example of this phenomenon is shown in figure \ref{fig:twisted_example}. The product of two diagrams, corresponding to tight contact structures, gives an overtwisted contact structure. But if they were multiplied in the opposite order, the result would be tight. For a slightly more complicated example, still ``localised" at a single matched pair, see figure \ref{fig:reordering_bypasses}.

\begin{figure}
\begin{center}

\begin{tikzpicture}[xscale=0.9, yscale=1.5]
\strandbackgroundshading
\tstrandbackgroundshading{1}
\tstrandbackgroundshading{2}
\beforewused
\tafterwused{1}
\tbeforevused{2}
\strandsetupn{}{}
\tstrandsetup{1}
\tstrandsetup{2}
\tstrandsetup{3}
\leftoff
\righton
\trightoff{1}
\trighton{2}
\useb
\tusea{1}
\tused{2}
\end{tikzpicture}
\hspace{2 cm}
\begin{tikzpicture}[xscale=0.9, yscale=1.5]
\strandbackgroundshading
\tstrandbackgroundshading{1}
\tstrandbackgroundshading{2}
\beforevused
\tafterwused{1}
\tbeforewused{2}
\strandsetupn{}{}
\tstrandsetup{1}
\tstrandsetup{2}
\tstrandsetup{3}
\leftoff
\righton
\trightoff{1}
\trighton{2}
\used
\tusea{1}
\tuseb{2}
\end{tikzpicture}

\caption{Left: This tensor product (tight, in our classification) has a nonzero product in $\A(\ZZ)$ or $H(\A(\ZZ))$. Right: This tensor product (critical in our classification) covers the same steps in a different order, and has zero product in $\A(\ZZ)$ or $H(\A(\ZZ))$, but an $A_\infty$ operation may reorder the bypasses and give a nonzero result.}
\label{fig:reordering_bypasses}
\end{center}
\end{figure}
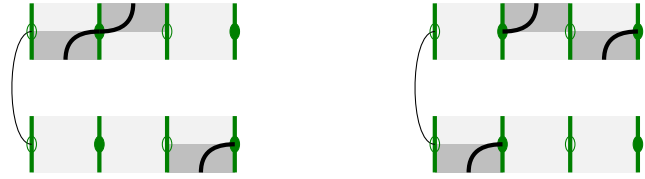

In general, the $A_\infty$ operation $X_n$, when it produces a nonzero result, will effectively reorder the bypasses at $n-2$ matched pairs (since it has grading $n-2$) so as to make their product tight. This is the rough meaning of theorem \ref{thm:second_thm}; the statement is simply an elaboration of this idea, being precise about the various types of tightness/twistedness at each matched pair.

We can say also say a little about how this ``reordering" is achieved. As mentioned above, our construction of the $A_\infty$ structure on $CA(\Sigma, Q)$ or $H(\A(\ZZ))$, following Kadeishvili's method, uses creation operators, whose operation is described locally by figure \ref{fig:creation_example}. A creation operator acts on a local diagram which is \emph{twisted}, i.e. represents a ``minimally overtwisted" contact structure, and makes it \emph{crossed}.

We may then observe a phenomenon which is rather curious from a contact-geometric point of view. Starting from a tensor product which is twisted (or worse), applying a creation operator yields a tensor product of diagrams including crossings --- the most degenerate type of ``twistedness". Yet multiplying out this tensor product may yield a diagram corresponding to a tight contact structure! After multiplication, no crossings remain, nor any twistedness. The result is as if the original diagrams were reordered into the ``correct" order at that matched pair. See figure \ref{fig:A-infinity-reordering} for an example based on the ``badly ordered" tensor product of figure \ref{fig:reordering_bypasses}(right).

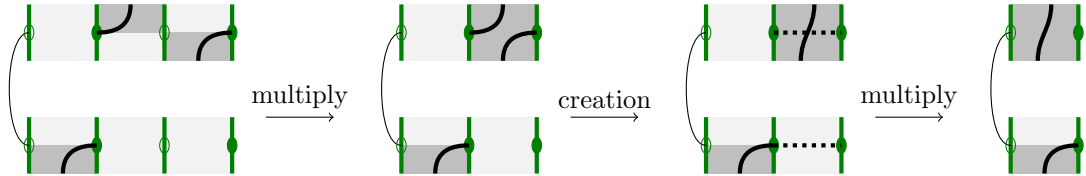
\begin{figure}
\begin{center}

\begin{tikzpicture}[xscale=0.9, yscale=1.5]
\strandbackgroundshading
\tstrandbackgroundshading{1}
\tstrandbackgroundshading{2}
\beforevused
\tafterwused{1}
\tbeforewused{2}
\strandsetupn{}{}
\tstrandsetup{1}
\tstrandsetup{2}
\tstrandsetup{3}
\leftoff
\righton
\trightoff{1}
\trighton{2}
\used
\tusea{1}
\tuseb{2}
\draw [->] (3.5,0.5) -- (4.5,0.5) node [midway, above] {multiply};
\begin{scope}[xshift=5.5 cm]
\strandbackgroundshading
\tstrandbackgroundshading{1}
\beforevused
\tafterwused{1}
\tbeforewused{1}
\strandsetupn{}{}
\tstrandsetup{1}
\tstrandsetup{2}
\leftoff
\righton
\trighton{1}
\used
\tusea{1}
\tuseb{1}
\end{scope}
\draw [->] (8,0.5) -- (9,0.5) node [midway, above] {creation};
\begin{scope}[xshift=10 cm]
\strandbackgroundshading
\tstrandbackgroundshading{1}
\beforevused
\tafterwused{1}
\tbeforewused{1}
\strandsetupn{}{}
\tstrandsetup{1}
\tstrandsetup{2}
\leftoff
\righton
\trighton{1}
\used
\tdothorizontals{1}
\tuseab{1}
\end{scope}
\draw [->] (12.5,0.5) -- (13.5,0.5) node [midway, above] {multiply};
\begin{scope}[xshift=14.5 cm]
\strandbackgroundshading
\beforevused
\afterwused
\beforewused
\strandsetupn{}{}
\tstrandsetup{1}
\leftoff
\righton
\used
\useab
\end{scope}
\end{tikzpicture}

\caption{Mechanics of $A_\infty$ operations, effectively reordering bypasses. Multiplying the last two factors of a critical tensor product yields a twisted diagram. A creation operator turns the twisted diagram into a crossed one, and the tensor product becomes sublime. Multiplication then yields a tight diagram.}
\label{fig:A-infinity-reordering}
\end{center}
\end{figure}

In this way, strand diagrams may pass from being crossed to tight without being twisted along the way. We call this process \emph{sublimation} because of its ``phase-skipping" behaviour. We call a tensor product in which the diagrams are not all tight, but their product is tight, \emph{sublime}.

However, it is not the case that $X_n$ \emph{always} performs reorderings and sublimations in this way; it simply \emph{may} do so. Depending on the various choices involved in the construction, the result may or may not be nonzero on various tensor products. Theorem \ref{thm:second_thm} tells us what the answer must be, if it is nonzero; and gives necessary conditions for it to be nonzero. Theorem \ref{thm:third_thm} does however provide a guarantee that for any $A_\infty$ structures produced by our construction, certain (highly restricted) tensor products always yield a nonzero result.

For lower-level operations, we can say more. We know $X_1 = 0$ and $X_2$ is just multiplication, and we can in fact give an explicit description of $X_3$ (proposition \ref{prop:X3_description}). Beyond that, the multiplicity of choices makes specific statements unwieldy, and theorem \ref{thm:third_thm} is the strongest guarantee of nonzero results that we could find, for now.

For the rest of this paper, we work primarily with strand diagrams. But our approach is heavily influenced by contact geometry, and we regularly comment on the contact-geometric significance of our definitions and results. For these comments, we assume some familiarity with the correspondence between strand algebras and contact structures in \cite{mathews_strand_2016}, and refer there for further details.

\subsection{Relationship to other work}

The strands algebra is a crucial object in bordered Floer theory, appearing in \cite{lipshitz_bordered_2008, lipshitz_slicing_2008, lipshitz_bimodules_2015, lipshitz_notes_2012}. The slightly more general arc diagrams we use here appeared in Zarev's work \cite{zarev_bordered_2009, zarev_joining_2010}. Its homology was explicitly computed in section 4 of \cite{lipshitz_bimodules_2015}. This description was reformulated in \cite{mathews_strand_2016}, where the isomorphism (\ref{eqn:main_iso}) was proved.

The general construction of $A_\infty$ structures on differential graded algebras by Kadeishvili in \cite{kadeishvili_homology_1980} is part of a much larger subject, not one to which the author claims much expertise. There are other methods, such as those of Kontsevich-Soibelman \cite{kontsevich_homological_2001}, Nikolov--Zahariev \cite{nikolov_curved_2013} and Huebschmann \cite{huebschmann_construction_2010}. We do not know of examples where Kadeishvili's construction has been made as absolutely explicit as by the ``creation" operators here. %, but we would not be surprised if they exist. 
In previous work we have found several roles for objects like creation and annihilation operators in contact geometry \cite{mathews_chord_2010, mathews_sutured_2011, mathews_sutured_2013, mathews_contact_2014, mathews_itsy_2014, mathews_twisty_2014, mathews_strings_2017, mathews_dimensionally_2015}.

The various contact-geometric interpretations appearing here derive not only from our previous work \cite{mathews_strand_2016} but also from work on quadrangulated surfaces and their connections to contact geometry, Heegaard Floer theory and TQFT \cite{mathews_itsy_2014, mathews_twisty_2014}. Some of these ideas are also implicit in Zarev's work cited above. Constructions with bypasses go back to Honda's \cite{honda_classification_2000}.

The contact category was introduced by Honda in unpublished work. It has been studied by Cooper \cite{cooper_formal_2015}. Related categorifications have been studied by Tian \cite{tian_categorification_2012, tian_categorification_2014}. The case of discs was considered in our \cite{mathews_chord_2010} and in detail by Honda-Tian in \cite{honda_contact_2016}.

\subsection{Structure of this paper}

As discussed above, there is some work required before we can even properly state our main theorems. First we must define the relevant notions and establish the properties we need.

We begin in section \ref{sec:algebra_and_anatomy} by considering the algebra and anatomy of strand diagrams. We recall existing definitions in section \ref{sec:strand_diagrams}, and generalise them to \emph{augmented} diagrams in section \ref{sec:augmented_diagrams}. We can then define the notion of viability in section \ref{sec:viability}. We consider how augmented diagrams can be cut into local diagrams, and the associated algebra, in section \ref{sec:local_diagrams}. In section \ref{sec:terminology_strand_diagrams} we establish terminology for strand diagrams, including occupation of places and pairs for homological grading, and on/off or 1/0 for idempotents; then (section \ref{sec:tightness}) we define the types of tightness of local diagrams and (section \ref{sec:classification_local_diagrams}) tabulate the various possibilities. In section \ref{sec:local_algebras} we discuss local strand algebras and their homology, and in section \ref{sec:homology_of_strand_algebras} the homology of strand algebras in general. In section \ref{sec:tight_twisted_crossed_diagrams} we define and study the types of tightness for viable augmented diagrams. In section \ref{sec:diagrams_representing_homology} we consider the set of diagrams representing a homology class, and in section \ref{sec:dim_algebras} we calculate the dimensions of various vector spaces related to strand algebras, which we need later. In section \ref{sec:ideals} we introduce the ideal $\F$ and a quotient which simplifies our calculations.

In section \ref{sec:cycle_selection_creation} we then consider objects parametrising the choices involved in constructing $A_\infty$ structures. We discuss cycle selection homomorphisms in section \ref{sec:cycle_selection_homs}. We discuss how different cycle selection maps can differ in section \ref{sec:diff_cycle_selection}. We then introduce creation operators in section \ref{sec:creation_operators}, and discuss how they can invert the differential in section \ref{sec:inverting_differential}. We put them together into global creation operators in section \ref{sec:global_creation}, and discuss how they can be obtained from a pair ordering in section \ref{sec:ordering}.

We then turn to tensor products of strand diagrams in section \ref{sec:tensor_products}. We extend the ``anatomical" notions and terminology for gradings, viability, occupation and idempotents in section \ref{sec:anatomy_tensor_products}. We introduce the six types of tightness/twistedness in section \ref{sec:tightness_tensor_products}. We discuss sub-tensor-products, and the associated notions of extension and contraction, in section \ref{sec:sub_extension_contraction}. We consider the two most curious types of tightness, sublime and singular, in section \ref{sec:sublime_singular}. We can then give a full enumeration of all possible viable local tensor products in section \ref{sec:enumerate_local_tensor_products}. In section \ref{sec:tightness_local_sub} we consider how tightness of tensor products and sub-tensor-products are related. We may then consider tensor products of homology classes of diagrams in section \ref{sec:tensor_product_homology}, discussing their tightness, enumerating the possible local tensor products, and establishing some of their properties. In section \ref{sec:generalised_contraction} we consider a generalised notion of contraction for tensor products of homology classes.

We then have everything we need to construct $A_\infty$ structures explicitly in section \ref{sec:construction}. The construction itself is given in section \ref{sec:constructing_operations}, proving theorem \ref{thm:first_thm}. In section \ref{sec:shorthand} we establish a shorthand notation for tensor products of strand diagrams. In section \ref{sec:low-level_ex} we calculate some examples at low levels of the $A_\infty$ structure.

In section \ref{sec:A-infinity_structures_general} we then discuss some properties of the $A_\infty$ structures we have constructed, and in fact slightly more general $A_\infty$ structures from Kadeishvili's construction. In section \ref{sec:H-data_viability} we discuss how $A_\infty$ operations relate to viability. In section \ref{sec:choices_of_maps} we discuss how the various choices made in Kadeishvili's construction affect the result. Then in section \ref{sec:first_properties_A-infinity} we establish some of the elementary properties of the constructed $A_\infty$ operations, and in section \ref{sec:nontrivial_conditions} prove some necessary conditions for nontrivial $A_\infty$ operations, including those of theorem \ref{thm:second_thm}. In section \ref{sec:low_levels} we establish general properties of the $A_\infty$ maps at levels up to 3.

In a brief section \ref{sec:examples} we calculate some further examples, at levels 3 (section \ref{sec:level_3}) and 4 (section \ref{sec:nec_not_suff}), illustrating some of the complexities which arise.

Finally in section \ref{sec:nontrivial_higher_operations} we consider higher $A_\infty$ operations and when they are nontrivial. We introduce the notion of operation trees in section \ref{sec:operation_trees}, and notions of \emph{validity} and \emph{distributivity} in section \ref{sec:valid_distributive}. In section \ref{sec:join_graft} we discuss some constructions we need on trees (\emph{joining} and \emph{grafting}). Then in section \ref{sec:trees_required} we show how certain trees are required for nonzero results, proving proposition \ref{prop:intro_nonzero_trees}. In section \ref{sec:local_trees} we discuss the operation trees local to a matched pair, and classify them in section \ref{sec:climbing_tree}. In section \ref{sec:strong_validity} we introduce a stronger notion of validity necessary for our results, and after discussing the further operations of \emph{transplantation} and \emph{branch shifts} in section \ref{sec:transplantation}, and introducing a stronger notion of distributivity in section \ref{sec:strict_distributive}, we prove theorem \ref{thm:third_thm} in section \ref{sec:guaranteed_nonzero}.

\subsection{Acknowledgments}

The author thanks Robert Lipshitz for posing the questions that sparked this work, and for helpful conversations. He also thanks Anita for putting up with him. He is supported by Australian Research Council grant DP160103085.

\section{Algebra and anatomy of strand diagrams}
\label{sec:algebra_and_anatomy}

\subsection{Strand diagrams}
\label{sec:strand_diagrams}

We recall the definition of strand diagrams, before proceeding in section \ref{sec:augmented_diagrams} to augment them. We follow our previous paper \cite{mathews_strand_2016}, which in turn is based on Zarev \cite{zarev_bordered_2009}, as well as Lipshitz--Ozsv\'{a}th--Thurston \cite{lipshitz_bordered_2008, lipshitz_bimodules_2015}.  We refer to these papers for further details.

An \emph{arc} diagram consists of a triple $\ZZ = ({\bf Z}, {\bf a}, M)$, where ${\bf Z} = \{ Z_1, \ldots, Z_l \}$ is a set of oriented line segments (intervals), ${\bf a} = (a_1, \ldots, a_{2k})$ is a sequence of distinct points in the interior of the line segments of $\mathbf{Z}$, ordered along the intervals, and $M \colon  {\bf a} \To \{1, 2, \ldots, k\}$ is a 2-to-1 function. Performing oriented surgery on ${\bf Z}$ at all the 0-spheres $M^{-1}(i)$ is required to yield an oriented 1-manifold consisting entirely of arcs (no circles). We say $\ZZ$ is \emph{connected} if the graph obtained from ${\bf Z}$ by identifying each pair $M^{-1}(i)$ is connected.

We call the points of ${\bf a}$ \emph{places}. If $M(a_i) = M(a_j)$ we say $a_i$ and $a_j$ are \emph{twins}; then $a_i, a_j$ form a \emph{matched pair} (or just \emph{pair}). The function $M$ partitions ${\bf a}$ into $k$ such pairs. There is a partial order on ${\bf a}$ where $a_i \leq a_j$ if $a_i, a_j$ lie on the same oriented interval, and are in order along it. 

An \emph{unconstrained strand diagram} over $\ZZ$ is a triple $\mu = (S, T, \phi)$ where $S,T \subseteq \{ a_1, \ldots, a_{2k} \}$ with $|S| = |T|$ and $\phi \colon S \To T$ is a bijection, which is increasing with respect to the partial order on ${\bf a}$ in the sense that $\phi(x) \geq x$ for all $x \in S$. There is a standard way to draw an unconstrained strand diagram in the plane (in fact in $[0,1] \times \mathbf{Z}$), with $|S| = |T|$ \emph{strands}. The strands \emph{begin} at $S$ (drawn at $\{0\} \times S$), \emph{end} at $T$ (drawn at $\{1\} \times T$), and move to the right (in the positive direction along $[0,1]$), never going down, and meeting efficiently without triple crossings. We say $\mu$ \emph{goes from $S$ to $T$}. The points of ${\bf a}$ split ${\bf Z}$ into intervals called \emph{steps}, of two types: \emph{interior} to an interval $Z_i$, and \emph{exterior}, i.e. at the boundary of a $Z_i$. The product $\mu \nu$ of two strand diagrams $\mu = (S,T,\phi)$ and $\nu = (U,V,\psi)$ is given by $(S, V, \psi \circ \phi)$, provided that $T=U$ and the composition $\psi \circ \phi \colon S \to V$ satisfies $\inv(\psi \circ \phi) = \inv(\phi) + \inv (\psi)$; otherwise it is zero. Here $\inv(\mu)$ is the number of inversions, or crossings, in $\mu$. Equivalently, the product $\mu \nu$ is given by concatenating strand diagrams, provided that there are no ``excess inversions", i.e. crossings which can be simplified by a Reidemeister II-type isotopy of strands relative to endpoints. There is a differential $\partial$ which resolves crossings in strand diagrams; $\partial \mu$ is the sum of all strand diagrams obtained from $\mu$ by resolving a crossing so that the number of crossings decreases by exactly $1$. This structure makes the free $\Z_2$-module on strand diagrams over $\ZZ$ into a differential graded algebra over $\Z_2$, which we denote $\widetilde{\A}(\ZZ)$. For each subset $S \subseteq {\bf a}$ there is an idempotent $I(S)$.

A \emph{$\ZZ$-constrained}, or just \emph{constrained}, strand diagram takes into account also the matching $M$ of $\ZZ$. For each $s \subseteq \{1, \ldots, k\}$ we define $I(s) = \sum_S I(S)$, where the sum is over sections $S$ of $s$ under $M$. Here a \emph{section} of $s$ means an $S \subseteq \mathbf{a}$ such that $M|_S$ is a bijection $S \To s$. The $I(s)$ generate a $\Z_2$-subalgebra of $\widetilde{\A}(\ZZ)$. A strand diagram which begins at a section of $s$ and ends at a section of $t$, for $s,t \subseteq \{1, \ldots, k\}$, is said to be \emph{$\ZZ$-constrained}. We say it \emph{begins} at $s$ and \emph{ends} at $t$, or \emph{goes from $s$ to $t$}; we also say $I(s)$, or by abuse of notation just $s$, is the \emph{initial idempotent}, and $I(t)$ or $t$ is the \emph{final idempotent}. Thus a constrained strand diagram begins and ends at subsets of ${\bf a}$ which contain at most one place of each matched pair.	If $I(s) \widetilde{\A}(\ZZ) I(t)$ is nonzero then $|s| = |t| = i$, in which case it is freely generated as a $\Z_2$-module by $\ZZ$-constrained strand diagrams of $i$ strands from $s$ to $t$.

Finally, we symmetrise our strand diagrams with respect to the matched pairs. If $\mu = (S, T, \phi)$ is an unconstrained strand diagram on $\ZZ$ without horizontal strands (i.e. $\phi$ has no fixed points) then we consider adding horizontal strands to $\mu$ at some places $U \subseteq {\bf a} \backslash (S \cup T)$, i.e. adding fixed points to $\phi$ to obtain a function $\phi_U \colon S \cup U \To T \cup U$, which is still a bijection with $\phi(x) \geq x$. We define $a(\mu)$ to be the sum of all strand diagrams that can be obtained from $\mu$ by adding horizontal strands,
\[
a(\mu) = \sum_U \left( S \cup U, T \cup U, \phi_U \right) \in \widetilde{\A}(\ZZ),
\]
and then for each $s,t \subseteq \{1, \ldots, k\}$, $I(s) a(\mu) I(t)$ is either zero, or the sum of all $\ZZ$-constrained strand diagrams from $s$ to $t$ obtained from $\mu$ by adding horizontal strands. (Left-multiplying by $I(s)$ filters for diagrams which start at $s$; right-multiplying by $I(t)$ filters for diagrams which start at $t$; multiplying by both ensures the result is $\ZZ$-constrained.)
Note that if it is possible to add a horizontal strand to $\mu$ at a place $a$ of a matched pair $\{a,a'\}$ to obtain a strand diagram in $I(s) a(\mu) I(t)$, then it is also possible to add a horizontal strand at the twin place $a'$. In this case every diagram in $I(s) a(\mu) I(t)$ contains a horizontal strand at precisely one of $a$ or $a'$; further, for every diagram with a horizontal strand at $a$ appearing in $I(s) a(\mu) I(t)$, the corresponding diagram  with a horizontal strand at $a'$ and otherwise identical will also appear. If there are $j$ such pairs $\{a,a'\}$, then $I(s) a(\mu) I(t)$ is a sum of $2^j$ terms, one for each choice of $a$ or $a'$ in each pair. 

We denote such a sum $I(s) a(\mu) I(t)$ as a single diagram $D$ by drawing the $2j$ horizontal strands involved as dotted and call it a \emph{symmetrised $\ZZ$-constrained strand diagram}. In such a diagram, every horizontal strand is dotted, and dotted horizontal strands come in pairs. So a symmetrised $\ZZ$-constrained strand diagram with $j$ pairs of dotted horizontal strands is in fact a sum of $2^j$ $\ZZ$-constrained strand diagrams. 

The \emph{strand algebra} $\A(\ZZ)$ is the subalgebra of $\widetilde{\A}(\ZZ)$ generated by symmetrised $\ZZ$-constrained strand diagrams. It is preserved by $\partial$ and hence forms a differential graded algebra. This algebra has several gradings. 

The \emph{homological} (also known as the \emph{spin-c grading} or \emph{Alexander} grading), which we abbreviate to \emph{H-grading}, is valued in $H_1 ({\bf Z}, {\bf a})$. Given a strand map $\mu = (S, T, \phi)$ on $\ZZ$, for each $a \in S$, the oriented interval $[a, \phi(a)]$ from $a$ to $\phi(a)$ gives a homology class in $H_1({\bf Z}, {\bf a})$, and the $H$-grading of $\mu$, denoted $h(\mu)$ or just $h$, is the sum of such intervals $[a, \phi(a)]$ over all $a \in S$. In other words, $h$ counts how often each step of $\mathbf{Z}$ is covered. Since horizontal strands cover no steps, a symmetrised constrained diagram $D$ has a well-defined $H$-grading $h(D)$. The H-grading is additive under multiplication of strand diagrams, and $\partial$ preserves $H$-grading. We denote by $\A(\ZZ;h)$ the $\Z_2$-submodule of $\A(\ZZ)$ generated by diagrams with $H$-grading $h$, so we have a direct sum decomposition $\A(\ZZ) = \bigoplus_h \A(\ZZ;h)$.

\begin{defn}
Let $D$ be a (symmetrised constrained) diagram from $s$ to $t$ (where $s,t \subseteq \{1,\ldots,k\}$), with H-grading $h \in H_1({\bf Z}, {\bf a})$. The \emph{H-data} of $D$ is the triple $(h,s,t)$. 
\end{defn}
In other words, the H-data of $D$ consists of its H-grading, i.e. how it covers the steps of $\ZZ$, together with its  initial and final idempotents. Note that $h$ does not in general determine $s$ or $t$. By inspecting $h$ we may deduce that certain strands must start or end at certain points in $\ZZ$: in particular, when the local multiplicity of $h$ changes as we pass from one step to the next, necessarily by $1$ or $-1$, a strand must respectively begin or end. But when the local multiplicity of $h$ does not change from one step to the next, we cannot tell whether strands begin or end. In particular, $h$ does not give any information about horizontal strands. Defining $\A(\ZZ;h,s,t) = I(s) \A(\ZZ;h) I(t)$, we have a decomposition of $\A(\ZZ)$ as a direct sum of $\Z_2$-modules over $H$-data (i.e. both over idempotents $I(s), I(t)$ and over $H$-grading).
\[
\A(\ZZ) = \bigoplus_{h,s,t} \A(\ZZ;h,s,t) = \bigoplus_{h,s,t} I(s) \A(\ZZ;h) I(t)
\]

The \emph{Maslov} grading of $\A(\ZZ)$ is valued in $\frac{1}{2} \Z$. If $\mu$ is a $\ZZ$-constrained strand diagram (not yet symmetrised) from $S$ to $T$ with $H$-grading $h$, then its Maslov grading is
\[
\iota(\mu) = \inv (\mu) - m \left( h , S \right),
\]
where the function
\[
m \colon H_1 ( {\bf Z}, {\bf a}) \times H_0 ({\bf a}) \To \frac{1}{2} \Z
\]
counts local multiplicities of strand diagrams around places. Specifically, for a place $a$ and $h \in H_1  ({\bf Z}, {\bf a})$, $m( h, a )$ is the average of the local multiplicities of $h$ on the steps after and before $a$. Thus $\iota(\mu)$ has a contribution of $+1$ from each crossing; and then non-positive contributions from each place of $S$, depending on the multiplicity of $h$ near that place. It is not difficult to check that all the constrained diagrams in a symmetrised constrained diagram $D$ have the same Maslov grading,
%A \emph{symmetrised} $\ZZ$-constrained strand diagram $D$ is a sum of constrained diagrams which all have the same Maslov grading. To see this, note any two constrained diagrams $\mu, \mu'$ in $D$ differ only by replacing some horizontal strands at some places with horizontal strands at their twin places. Suppose $\mu, \mu'$ from $S, S'$ to $T,T'$ differ by replacing a horizontal strand at $a$ with a horizontal strand at its twin $a'$. If such a replacement is possible then no other strands begin or end at $a$ or $a'$, so the number of crossings $c$ along the horizontal strand at $a$ is equal to the local multiplicity of $h$ above and below $a$; similarly, the number of crossings $c'$ along the horizontal strand at $a'$ is equal to the local multiplicity of $h$ above and below $a'$. The horizontal strand at $a$ makes a contribution of $c$ to $\inv(a)$, and a contribution of $-c$ to $m(h, S)$; the horizontal strand at $a'$ makes a contribution of $c'$ to $\inv(a')$ and a contribution of $-c'$ to $m(h,S')$. Applying this argument wherever diagrams differ, we see $\iota(\mu) = \iota(\mu')$ for any two diagrams $\mu,\mu'$ in $D$, 
so the Maslov grading $\iota(D)$ of $D$ is the grading of any of the constrained diagrams in it.

%***NOT TRUE*** Note that if we fix $H$-data $(h,s,t)$, then all diagrams with this $H$-data have the same contributions from the term $m(h,S)$; so, up to a constant, the Maslov grading is simply the number of crossings in the diagram.

The differential $\partial$ does not affect $H$-grading or idempotents, but lowers the number of crossings in a diagram by $1$ (if the result is nonzero), hence has Maslov degree $-1$. The Maslov index does not respect multiplication in $\A(\ZZ)$; rather, for symmetrised $\ZZ$-constrained strand diagrams $D,D'$ with $H$-gradings $h,h'$ we have
\[
\iota(D D') = \iota(D) + \iota(D') + m( h', \partial h ).
\]

The homology of $\A(\ZZ)$ was described by Lipshitz--Ozsv\'{a}th--Thurston \cite[thm. 9]{lipshitz_bimodules_2015}. As the differential respects H-data $(h,s,t)$, the decomposition $\A(\ZZ) = \bigoplus_{h,s,t} \A(\ZZ;h,s,t)$ descends to homology:
\[
H(\A(\ZZ)) = \bigoplus_{h,s,t} H(\A(\ZZ;h,s,t)).
\]
Lipshitz--Ozsv\'{a}th--Thurston showed that the summand $H(\A(\ZZ;h,s,t))$ is nontrivial if and only if there exists a symmetrised $\ZZ$-constrained strand diagram $D$ with $H$-data $(h,s,t)$ \emph{without crossings}, satisfying two conditions:
\begin{enumerate}
\item
the multiplicity of $h$ on every step of $\mathbf{Z}$ is $0$ or $1$; and
\item
if $\{a,a'\}$ is a matched pair with $a$ in the interior of the support of $h$, and $a'$ not in the interior of the support of $h$, then $a$ does not lie in both $s$ and $t$.
\end{enumerate}
Such a $D$, having no crossings, is obviously a cycle and in fact the homology class of any such $D$ generates $H(\A(\ZZ;h,s,t)) \cong \Z_2$.  We use property (i) extensively as a notion of \emph{viability} in the sequel, from section \ref{sec:viability} onwards. We discuss and reformulate the second condition further in section \ref{sec:local_algebras} below; see also \cite[sec. 3.5--3.7]{mathews_strand_2016}. For such $D$, the expression for the Maslov index simplifies to $\iota(D) = -m(h, S)$, where $S$ is the initial idempotent of any $\ZZ$-constrained strand diagram in $D$.

Since we usually work with a single arc diagram $\ZZ$, we often leave $\ZZ$ implicit and use the shorthand
\[
\A = \A(\ZZ), \quad
\A(h,s,t) = \A(\ZZ;h,s,t), \quad
\HH = H(\A), \quad
\HH (h,s,t) = H(\A(h,s,t)).
\]
The homology $\HH$ inherits multiplication from $\A$ and becomes a differential graded algebra with trivial differential. The point of this paper is to extend this DGA structure to $A_\infty$-structures.

\subsection{Augmented strand diagrams}
\label{sec:augmented_diagrams}

In a symmetrised $\ZZ$-constrained strand diagram, strands run between places in $\mathbf{a} = ( a_1, \ldots, a_{2k} )$. Since the places of $\mathbf{a}$ lie in the interior of the intervals $Z_i$ of $\mathbf{Z}$, no strand ever reaches an endpoint of any interval $Z_i$. In other words, strand diagrams only cover interior steps of $\mathbf{Z}$.

In the sequel however we need to consider strand diagrams where strands cover exterior steps of $\mathbf{Z}$ and reach endpoints of the intervals $Z_i$. We describe this as \emph{flying off} an interval. \emph{Augmented} strand diagrams, which we define presently, extend strand diagrams to allow such behaviour.

To define augmented diagrams formally we again use non-decreasing bijections, but now on sets including the endpoints of each interval. Let the endpoints of the interval $Z_i$ be $-\infty_i$ and $+\infty_i$, at the start and end respectively. A strand flies off the top end of an interval $Z_i$ if some $a_j \neq +\infty_i$ is sent to $+\infty_i$, and a strand flies off the bottom if some $a_j \neq -\infty_i$ satisfies $-\infty_i \mapsto a_j$.  A strand may fly off both ends of an interval if $-\infty_i \mapsto +\infty_i$. We also allow horizontal strands at $\pm \infty_i$, but these present a slight subtlety, discussed below: they simply exist for technical reasons.

Let $\mathbf{a}_{\pm \infty} = \mathbf{a} \cup \{-\infty_1, \ldots, -\infty_l\}_{i=1}^l \cup \{+\infty_1, \ldots, +\infty_l\}_{i=1}^l$.  The points of $\mathbf{a}_{\pm \infty}$ are naturally partially ordered by the total order along each interval, extending the partial order on $\mathbf{a}$. 

An \emph{unconstrained augmented strand diagram} over $\ZZ$ is a triple $(S, T, \phi)$, where $S, T \subseteq \mathbf{a}_{\pm\infty}$ and $\phi \colon S \To T$ is a bijection which is increasing with respect to the partial order on $\mathbf{a}_{\pm\infty}$ in the sense that $\phi(x) \geq x$ for all $x \in S$. Again we say $\mu$ \emph{goes from} $S$ to $T$. The product of two such diagrams is given by composition of bijections $\phi$, if such a composition exists and has no excess inversions, otherwise it is zero. Equivalently, the product is given by by concatenating diagrams, if a concatenation exists and has no excess crossings. A differential is again defined by resolving crossings so that the number of crossings decreases by exactly 1. We define $\widetilde{\A}^{aug}(\ZZ)$ to be the free $\Z_2$-module on unconstrained augmented strand diagrams over $\ZZ$; it is a differential graded algebra over $\Z_2$ with an idempotent $I(S)$ for each $S \subseteq \mathbf{a}_{\pm \infty}$.

A subtlety arises here because if an (unconstrained) augmented strand diagram $\mu$ has a strand (say) flying off the end of an interval to $+\infty_i$, it should still be able to give a nonzero result when composed with another diagram on the right, which does not have any strand at $+\infty_i$. We extend our notion of matching to achieve this effect, but it is no longer a function; rather it is a \emph{partial function} (i.e. partially defined).

To this end, we extend the matching $M \colon \mathbf{a} \To \{1, \ldots, k\}$ to the partial function $M^{aug} \colon \mathbf{a}_{\pm \infty} \To \{1, \ldots, k\}$, which is equal to $M$ on $\mathbf{a}$, and is \emph{not defined} on each $-\infty_i$ or $+\infty_i$. 

Given a set $s \subseteq \{1, \ldots, k\}$, a \emph{section} of $s$ under $M^{aug}$ is then any set $S \subseteq \mathbf{a}_{\pm \infty}$ such that the restriction of $M^{aug}$ to $S$ is a (possibly partially defined) function mapping surjectively and injectively to $s$. Thus a section of $s$ under $M^{aug}$ consists of a section of $s$ under $M$, together with any subset of $\{-\infty_i, +\infty_i\}_{i=1}^l$.

For $s \subseteq \{1, \ldots, k \}$, define $I(s) = \sum_S I(S)$, the sum over sections $S$ of $s$ under $M^{aug}$. The $I(s)$ again generate a subalgebra of $\widetilde{\A}^{aug}(\ZZ)$. An augmented strand diagram which begins at a section of $s$ and ends at a section of $t$, for $s,t \subseteq \{1, \ldots, k\}$, is \emph{$\ZZ$-constrained}; we say it \emph{goes from} $s$ to $t$. If $I(s) \widetilde{\A}^{aug}(\ZZ) I(t)$ is nonzero, then there is at least one section $S$ of $s$ under $M^{aug}$, and at least one section $T$ of $t$ under $M^{aug}$, such that there exists an (unconstrained) augmented strand diagram from $S$ to $T$. Note then that $s$ and $t$ need not have the same size, because $S$ and $T$ can also contain points of the form $+ \infty_i$ or $-\infty_i$.

If $\mu = (S, T, \phi)$ is an unconstrained augmented strand diagram on $\ZZ$ without horizontal strands, we again consider adding horizontal strands to $\mu$ at places $U \subseteq \mathbf{a}_{\pm \infty} \backslash (S \cup T)$ (there can be horizontal strands at $\pm \infty_i$), extending $\phi$ by the identity to $\phi_U \colon S \cup U \To T \cup U$, and defining $a(\mu) = \sum_U (S \cup U, T \cup U, \phi_U) \in \widetilde{\A}^{aug}(\ZZ)$. For each $s,t \subseteq \{1, \ldots, k\}$ then $I(s) a(\mu) I(t)$ is the sum of all $\ZZ$-constrained augmented strand diagrams obtained from $\mu$ by adding horizontal strands, possibly at interval endpoints $\pm \infty_i$. As in the non-augmented case, if one such diagram has $j$ horizontal strands at the places of $\mathbf{a}$, these horizontal strands can be swapped with their twins, resulting in $2^j$ possible arrangements of horizontal strands at these places. Unlike the non-augmented case, for any point of the form $-\infty_i$ or $+\infty_i$ not in $S \cup T$, a horizontal strand can be added at this point. Thus if $| \bigcup_{i=1}^l \{-\infty_i, +\infty_i\} \backslash (S \cup T) | = n$, then there are $2^n$ possible arrangements of horizontal strands at these endpoints. % (one for each subset of $| \bigcup_{i=1}^l \{ -\infty_i, +\infty_i\} \backslash (S \cup T)|$).
\begin{defn}
With notation as above, $I(s) a(\mu) I(t)$ is a sum of $2^{j+n}$ $\ZZ$-constrained augmented strand diagrams.  We can draw such a sum as a single diagram $D$ with $2j$ dotted horizontal strands (leaving the possible horizontal strands at $\pm \infty_i$ implicit) and we call it a \emph{symmetrised $\ZZ$-constrained augmented strand diagram} or just \emph{diagram}.
\end{defn}

Multiplication of diagrams is described as follows. Consider the product of two (symmetrised $\ZZ$-constrained augmented strand) diagrams $D,D'$. If no strand in $D$ or $D'$ flies off an interval, then their product $DD'$ as augmented diagrams is given by concatenating strands, just as for non-augmented diagram. Formally the symmetrised augmented diagram is a sum of $2^n$ diagrams, involving possible horizontal strands at $\pm \infty_i$, but the augmented diagram $DD'$ is drawn identically to the diagram of the the product of non-augmented diagrams. Thus, at least at the level of drawing diagrams, multiplication of augmented diagrams is an extension of multiplication of (non-augmented) diagrams.

If on some interval $Z_i$, both $D$ and $D'$ fly off the top end, then $DD' = 0$. This is because, for any $\ZZ$-constrained augmented strand diagram $(S,T,\phi)$ in $D$, and any such diagram $(S',T',\phi')$ in $D'$, $\phi$ has $+\infty_i$ in its image, but $\phi'$ does not have $+\infty_i$ in its domain, so the functions cannot be composed.  Similarly, if both $D,D'$ fly off the negative end, then $DD' = 0$. If $D$ flies off the top end of $Z_i$ but $D'$ does not, then the composition is well defined there: each $(S,T,\phi)$ in $D$ has $+\infty_i$ in the image of $\phi$; and half of the constrained augmented diagrams $(S',T',\phi')$ in $D'$ have $\phi'$ mapping $+\infty_i \mapsto +\infty_i$ (i.e. a horizontal strand at $+\infty_i$), so such $\phi'$ compose with $\phi$ at $+\infty_i$. If $D'$ flies off the top end of $Z_i$ but $D$ does not, then again composition is well defined: each $(S',T',\phi')$ in $D'$ has $+\infty_i$ in its image, but not in its domain; half the $(S,T,\phi)$ in $D$ do not have $+\infty_i$ in the domain or image; and these $\phi$ and $\phi'$ compose without any problems at $+\infty_i$. Thus, if one of $D,D'$ flies off the top end of $Z_i$ and the other does not, then the product $DD'$ is well defined there. Similarly, if one of $D,D'$ flies off the bottom end of $Z_i$ and the other does not, then the product $DD'$ is well defined there. Thus, roughly, if we can concatenate strands of $D$ and $D'$ into another augmented diagram, with at most one strand flying off any end of any interval, then the product $DD'$ is given by concatenating strands, just as for (non-augmented) strand diagrams. Some examples are shown in figure \ref{fig:aug_mult}
.

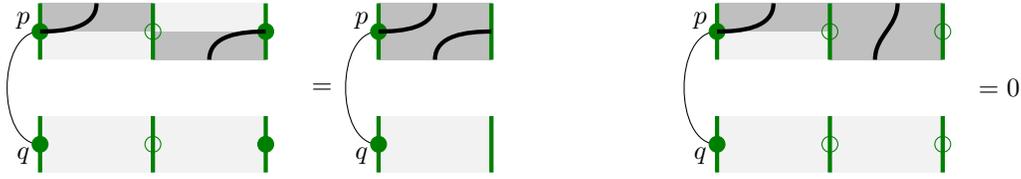
\begin{figure}
\begin{center}
\begin{tikzpicture}[scale=1.5]
\strandbackgroundshading
\tstrandbackgroundshading{1}
\afterwused
\tbeforewused{1}
\strandsetup
\tstrandsetup{1}
\tstrandsetup{2}
\lefton
\rightoff
\trighton{1}
\usea
\tuseb{1}
\draw (2.5,0.75) node {$=$};
\begin{scope}[xshift = 3cm]
\strandbackgroundshading
\afterwused
\beforewused
\strandsetup
\tstrandsetup{1}
\lefton
\usea
\useb
\end{scope}
\begin{scope}[xshift = 6cm]
\strandbackgroundshading
\tstrandbackgroundshading{1}
\afterwused
\tbeforewused{1}
\tafterwused{1}
\strandsetup
\tstrandsetup{1}
\tstrandsetup{2}
\lefton
\rightoff
\trightoff{1}
\usea
\tuseab{1}
\draw (2.5, 0.75) node {$= 0$};
\end{scope}

\end{tikzpicture}
\caption{Multiplication of augmented diagrams.}
\label{fig:aug_mult}
\end{center}
\end{figure}

The \emph{augmented strand algebra} $\A^{aug}(\ZZ)$ is the subalgebra of $\widetilde{\A}^{aug}(\ZZ)$ generated by (symmetrised $\ZZ$-constrained augmented strand) diagrams. It is preserved by $\partial$ and forms a differential graded algebra. There is again an \emph{$H$-grading} $h$ given by taking the sum of oriented intervals $[a, \phi(a)]$, and regarding it as an element of a relative first homology group of the intervals $\mathbf{Z}$. However now the endpoints of $[a, \phi(a)]$ may include the $\pm \infty_i$, so $h \in H_1 (\mathbf{Z}, \mathbf{a}_{\pm \infty} )$. Note that $H_1 (\mathbf{Z}, \mathbf{a}_{\pm \infty})$ naturally contains $H_1 (\mathbf{Z}, \mathbf{a})$ as a subgroup, and we will always regard it as such: $H_1 (\mathbf{Z}, \mathbf{a}) \subset H_1 (\mathbf{Z}, \mathbf{a}_{\pm \infty})$. So we can regard H-grading for augmented diagrams as an extension of H-grading for (non-augmented) diagrams. Again we write $\A^{aug}(\ZZ;h)$ for the submodule of $\A^{aug}(\ZZ)$ with $H$-grading $h$ and have a direct sum decomposition $\A^{aug}(\ZZ) = \bigoplus_h \A^{aug} (\ZZ;h)$. Again if $D$ is an augmented diagram from $s$ to $t$ with $H$-grading $h$, the \emph{$H$-data} of $D$ is the triple $(h,s,t)$. Again we write $\A^{aug}(\ZZ;h,s,t) = I(s) \A^{aug}(\ZZ;h) I(t)$ and then $\A^{aug}(\ZZ) = \bigoplus_{h,s,t} \A^{aug}(\ZZ;h,s,t)$.

A $\ZZ$-constrained augmented diagram $\mu$ (not yet symmetrised) from $S$ to $T$ with $H$-grading $h$ has \emph{Maslov} grading again given by $\iota(\mu) = \iota(\mu) = \inv (\mu) - m \left( h , S \right) \in \frac{1}{2} \Z$, where $\inv$ counts inversions/crossings, and $m \colon H_1 (\mathbf{Z}, \mathbf{a}_{\pm \infty}) \times H_0 (\mathbf{a}) \To \frac{1}{2} \Z$ counts local multiplicities of augmented diagrams around places $a_i$ in $S$. (We use $H_0 (\mathbf{a})$ rather than $H_0 (\mathbf{a}_{\pm \infty})$ so that Maslov grading is additive when we glue arc diagrams together. The points $\pm \infty_i$ are not places like the $a_i$.) Again all the diagrams in a symmetrised diagram have the same Maslov grading. When we add a horizontal strand at a $\pm \infty_i$, the fact that we can add the strand means that there is no strand at $\pm \infty_i$ for the horizontal strand to cross; moreover the horizontal strand at $\pm \infty_i$ does not contribute to $m(h,S)$). Thus Maslov grading is well defined on symmetrised $\ZZ$-constrained augmented diagrams. %***NOT TRUE*** Again if we fix $H$-data $(h,s,t)$ then up to a constant Maslov grading is given by number of crossings.

Again the differential $\partial$ respects $H$-data but has Maslov degree $-1$. Maslov index behaves under multiplication as in the non-augmented case. When we have $h \in H_1 (\mathbf{Z}, \mathbf{a}) \subset H_1 (\mathbf{Z}, \mathbf{a}_{\pm \infty})$ then strands do not fly off intervals and we have an isomorphism of differential graded algebras
\[
\A(\ZZ; h,s,t) \cong \A^{aug} (\ZZ;h,s,t).
\]
The isomorphism takes a symmetrised diagram $D \in \A(\ZZ;h,s,t)$ (formally a sum of $2^j$ constrained diagrams) to the element of $\A^{aug}(\ZZ;h,s,t)$ represented by the same diagram (formally a sum of $2^{j+2l}$ constrained diagrams, where $l$ is the number of intervals in $\mathbf{Z}$; all possible horizontal strands at $\pm \infty_i$ are now included). We draw the same diagrams and treat them the same way in both cases.

Accordingly, throughout this paper we regard augmented diagrams as a generalisation of non-augmented diagrams, even though the definition is not formally a generalisation. Alternatively we can regard non-augmented diagrams as augmented diagrams with H-grading zero on exterior steps, in which case augmented diagrams do become a generalisation in a formal sense. %(Since non-augmented diagrams can be subsumed into augmented diagrams, section \ref{sec:strand_diagrams} is, strictly speaking, unnecessary.)

Thus, we drop the ``aug" from our notation and simply write $\A(\ZZ)$ or $\A$ for the augmented strand algebra. When referring to specific H-data $(h,s,t)$ we do not distinguish between $\A$ and $\A^{aug}$, since the two summands are isomorphic when both are defined; and we write $\A(h,s,t)$ for $\A(\ZZ;h,s,t)$ or $\A^{aug}(\ZZ;h,s,t)$, and $\HH(h,s,t)$ for $H(\A(\ZZ;h,s,t))$ or $H(\A^{aug}(\ZZ;h,s,t))$. 

To summarise: (symmetrised constrained augmented strand) diagrams are a generalisation of symmetrised constrained strand diagrams --- generalising the full differential graded algebra structure of strand diagrams, as well as all gradings and idempotents. 

%We will discuss the homology of augmented strand algebras below. Although we have defined augmented diagrams in general, for present purposes we only need augmented diagrams on very particular arc diagrams: namely those with precisely one place on each interval $Z_i$ of $\mathbf{Z}$ (together with the endpoints $+\infty_i, -\infty_i$).

\subsection{Viability}
\label{sec:viability}

%The following notion of \emph{viability} will be crucial throughout this paper.
The idea of viability was already mentioned above (section \ref{sec:strand_diagrams}).
\begin{defn} 
\label{def:viability}
Let $\ZZ = (\mathbf{Z}, \mathbf{a}, M)$ be an arc diagram.
\begin{enumerate}
\item
An element $h \in H_1 (\mathbf{Z}, \mathbf{a})$ or $H_1 (\mathbf{Z}, \mathbf{a}_{\pm \infty})$ is \emph{viable} if $h$ has multiplicity $0$ or $1$ on each step of $\mathbf{Z}$.
\item
A set of H-data $(h,s,t)$ is \emph{viable} if $h$ is viable.
\item
A ($\ZZ$-constrained augmented strand) diagram $\mu$ is \emph{viable} if its H-grading is viable.
\item
A summand $\A(h,s,t)$ or $\HH(h,s,t)$ of $\A$ or $\HH$ is \emph{viable} if $h$ is viable.
\item
An element of $\A$ or $\HH$ is \emph{viable} if it lies in a viable summand
 $\A(h,s,t)$ or $\HH(h,s,t)$.
\end{enumerate}
\end{defn}
Thus a diagram $D \in \A$ (or its homology class in $\HH$) is viable iff its H-grading is viable.

Note that a set of H-data $(h,s,t)$ may be viable, even though no augmented diagram exists with that H-data! This subtle point is important in the sequel, from section \ref{sec:tightness_tensor_products} onwards.

\begin{lem}
\label{lem:viable_crossings_horizontal}
In a viable augmented diagram, every crossing is at a horizontal strand.
\end{lem}

\begin{proof}
Any diagram with a crossing involving two non-horizontal strands has a step covered with multiplicity at least two.
\end{proof}

Thus, whenever we apply the differential $\partial$ to a viable diagram, any crossing resolved involves a dotted horizontal strand at some place $p$ of a pair $P = \{p,q\}$. Resolving that crossing affects the strands at $p$, and makes the dotted horizontal strand at $q$ disappear. Thus the differential acts ``locally" on viable diagrams, each resolution at a specific matched pair. We discuss this idea of ``locality" next.

\subsection{Local diagrams}
\label{sec:local_diagrams}

In the arc diagram $\ZZ = (\mathbf{Z}, \mathbf{a}, M)$, consider cutting the intervals $Z_1, \ldots, Z_l$ of $\mathbf{Z}$ into sub-intervals, each containing precisely one place. This cuts $\ZZ$ into disconnected arc diagrams. There is one connected arc diagram for each matched pair. We call the connected arc diagram so obtained containing the matched pair $P$ the \emph{fragment} of $\ZZ$ at $P$, denoted $\ZZ_P$. We can cut $\mathbf{Z}$ at different points between places, but the results are homeomorphic, so $\ZZ_P$ is well defined up to homeomorphism. 

A fragment $\ZZ_P$ contains just one matched pair consisting of two places; it is the smallest possible nontrivial arc diagram, the only arc diagram up to homeomorphism with one matched pair. %(If the two places lie on the same interval then oriented surgery on $\mathbf{Z}$ at the matched pair yields a circle, violating the definition of an arc diagram.)

Under the correspondence between arc diagrams $\ZZ$ and quadrangulated surfaces $(\Sigma, Q)$ of \cite{mathews_strand_2016}, cutting $\ZZ$ into fragments corresponds to cutting $\Sigma$ into squares.

Let now $D$ be a (symmetrised $\ZZ$-constrained augmented strand) diagram. When we cut $\ZZ$ into fragments, we would like to cut $D$ into fragments also. Note that even if $D$ is not an augmented diagram, strands may fly off the ends of intervals in fragments, so after cutting into fragments the resulting strand diagram may be augmented.

If $D$ has a crossing involving two non-horizontal strands, then problems arise. For one thing, $D$ can be drawn in various ways, so that the crossing appears at various possible locations in $[0,1] \times \mathbf{Z}$; after cutting there is then no well-defined fragment in which the crossing appears. For another, after cutting, more than one strand may fly off the same end of a fragment, which is not permitted in augmented diagrams.

However, if $D$ is \emph{viable} these problems disappear. By lemma \ref{lem:viable_crossings_horizontal} all crossings occur at horizontal strands, so are localised to specific places. Viability also ensures that each interior step of $\mathbf{Z}$ is covered with multiplicity at most 1, so after cutting $\ZZ$, at most one strand flies off each end of an interval. We therefore obtain a well-defined augmented diagram on each fragment.

\begin{defn}
Let $P = \{p,q\}$ be a matched pair of the arc diagram $\ZZ$,
%\begin{enumerate}
%\item 
%Let $\mu$ be a viable $\ZZ$-constrained augmented strand diagram. The \emph{local constrained augmented strand diagram} $\mu_P$ of $\mu$ at $P$ is the augmented diagram obtained on $\ZZ_P$ after cutting $\mu$ into fragments.
%\item
and let $D$ be a viable diagram on $\ZZ$. The \emph{local diagram} $D_P$ of $D$ at $P$ is the diagram obtained on $\ZZ_P$ after cutting $\ZZ$ into fragments.
%\end{enumerate}
It lies in the \emph{local strand algebra} $\A(\ZZ_P)$.
\end{defn}
Note that since a symmetrised diagram $D$ may contain pairs of dotted horizontal arcs at matched pairs, the local diagram $D_P$ may contain a pair of dotted horizontal arcs. %Then $D_P$ is a sum of two $\ZZ_P$-constrained strand diagrams.

When the larger arc diagram $\ZZ$ is understood, we can make the following abbreviations for various augmented strand algebras and summands:
\[
\A_P = \A (\ZZ_P), \quad
\A_P (h,s,t) = \A (\ZZ_P;h,s,t).
\]

We observe that the $H$-data of each $D_P$ is just a restriction of the $H$-data of $D$. Maslov gradings are related by $\iota(D) = \sum_P \iota(D_P)$. We write $(h,s,t)$ for the $H$-data of $D$, and $(h_P, s_P, t_P)$ for the $H$-data of $D_P$.

In addition to cutting diagrams into fragments, we can glue fragments together into diagrams. From augmented diagrams $D_P$ on each $\ZZ_P$, which fit together in the sense that strands flying off intervals connect, we can glue them together to obtain a viable diagram on $\ZZ$. Thus, studying viable diagrams locally is equivalent to studying augmented diagrams on a fragment.

For viable H-data $(h,s,t)$ on $\ZZ$, we thus have
\[
\A (h,s,t) \cong \bigotimes_{\text{matched pairs }P} \A_P (h_P, s_P, t_P).
\]
Since the differential of a diagram is the sum of its various resolutions at its crossings, this is an isomorphism of complexes, or differential $\Z_2$-modules. (Note that the isomorphism holds even if $(h,s,t)$ is not the H-data of any diagram! In this case both sides are zero.) It is also an isomorphism of differential graded algebras. In particular, multiplying two diagrams $D,D'$ on $\ZZ$, and then cutting into fragments, yields the same result as cutting $D,D'$ into fragments, and then multiplying the local diagrams --- provided that it makes sense to cut all the diagrams $D,D'$ and $D D'$ into fragments, i.e. that they are all viable. We prove this now.
\begin{lem}
Let $D,D'$ be viable augmented diagrams on $\ZZ$, with local diagrams $D_P, D'_P$ on each fragment $\ZZ_P$. 
%Suppose that the product $DD'$ is nonzero and viable. 
Then the product $DD'$ is nonzero and viable if and only if each $D_P D'_P$ is nonzero.
In this case $(DD')_P = D_P D'_P$.
\end{lem}
Thus, if under the isomorphism $\A(h,s,t) \cong \bigotimes_P \A_P (h_P, s_P, t_P)$ we have $D = \bigotimes_P D_P$ and $D' = \bigotimes_P D'_P$, then $DD' = \bigotimes_P D_P D'_P$. %Thus multiplication in $\A(h,s,t)$ corresponds to the multiplication in the tensor product $\bigotimes_P \A_P (h_P, s_P, t_P)$ (multiplying each tensor factor separately), and the isomorphism $\A(h,s,t) \cong \bigotimes_P \A_P (h_P, s_P, t_P)$ is in fact an isomorphism of differential graded algebras.

\begin{proof}
Recall the description of multiplication of augmented diagrams in section \ref{sec:augmented_diagrams}. We examine the products $DD'$ and $D_P D'_P$ on fragments $\ZZ_P$. If no strand of $D$ or $D'$ flies off the fragment $\ZZ_P$, then $DD'$ (if nonzero) at $P$ is clearly given by $D_P D'_P$. If strands of both $D,D'$ fly off the top of $\ZZ_P$, then $DD'$ is not viable, and $D_P D'_P = 0$; similarly if strands fly off the bottom. If a strand of $D$ but not $D'$ flies off the top (resp. bottom) of $\ZZ_P$, then as discussed in section \ref{sec:augmented_diagrams}, $D_P D'_P$ is well defined, with a single strand flying off the top (resp. bottom) of $\ZZ_P$, as also does $DD'$ at $P$. The case where a strand of $D'$ but not $D$ flies off $\ZZ_P$ is similar. Gluing together these local results at each matched pair gives the desired result.
\end{proof}

Now for any chain complexes $A,B$ over $\Z_2$ we have $H(A \otimes B) \cong H(A) \otimes H(B)$. (See e.g. \cite[sec. 3.7]{mathews_strings_2017} or  \cite[thm. V.2.1]{hilton_course_1971}.) Thus the homology $H(\A(h,s,t))$ is isomorphic to the tensor product of the $H(\A_P (h_P, s_P, t_P))$, and in fact this isomorphism preserves all gradings. We use the shorthand
\[
\HH_P = H(\A_P), \quad
\HH_P (h_P, s_P, t_P) = H(\A_P (h_P, s_P, t_P))
\]
We call $\HH_P$ the \emph{local homology at $P$}. So we have the following isomorphism, which we often use implicitly in the sequel.
\begin{lem}
\label{lem:homology_decomposition}
For viable $(h,s,t)$ there is an isomorphism of graded $\Z_2$-algebras, respecting H-data and Maslov grading, induced by cutting diagrams into fragments.
\[
\HH(h,s,t) \cong \bigotimes_{\text{matched pairs }P} \HH_P (h_P, s_P, t_P)
\]
\qed
\end{lem}

Any fragment of any arc diagram is homeomorphic to any other, so all local strand algebras are isomorphic. Hence we may speak of \emph{the} local arc diagram or \emph{the} local strand algebra, without reference to any specific matched pair or arc diagram. They are abusively denoted $\ZZ_P$ and $\A_P$ respectively.

\subsection{Terminology for local strand diagrams}
\label{sec:terminology_strand_diagrams}

We now develop systematic terminology to describe diagrams locally. Throughout this section, $P = \{p,q\}$ is a matched pair of an arc diagram $\ZZ = (\mathbf{Z}, \mathbf{a}, M)$, $h \in H_1 (\mathbf{Z}, \mathbf{a}_{\pm \infty})$,
%(which includes $H_1 (\mathbf{Z}, \mathbf{a})$ as a subgroup), 
and $D$ is a diagram with $H$-data $(h,s,t)$.

\begin{defn}[Occupation of places] \
\label{def:occupation_places}
\begin{enumerate}
\item If $h$ has multiplicity $0$ on the steps before and after $p$, then $p$ is \emph{unoccupied} by $h$.
\item If $h$ has multiplicity $1$ on the step before $p$, and $0$ on the step after $p$, then $p$ is \emph{pre-half-occupied} by $h$.
\item If $h$ has multiplicity $0$ on the step before $p$, and $1$ on the step after $p$, then $p$ is \emph{post-half-occupied} by $h$.
\item If $h$ is pre-half-occupied or post-half-occupied, then $p$ is \emph{half-occupied} by $h$.
\item If $h$ has multiplicity $1$ on both steps before and after $p$, then $p$ is \emph{fully occupied} by $h$.
\end{enumerate}
\end{defn}
Although this definition applies to the H-grading $h$, we may equally apply it to a \emph{diagram} $D$. We say that $p$ is unoccupied (resp. pre-half-occupied, post-half-occupied, half-occupied, fully occupied) by $D$, if $p$ is so occupied by its H-grading $h$.

\begin{defn}[Occupation of pairs] \
\label{def:occupation_pairs}
\begin{enumerate}
\item If both $p,q$ are unoccupied by $h$, then $P$ is \emph{unoccupied} by $h$.
\item If $p$ is half-occupied, and $q$ is unoccupied by $h$, then $P$ is \emph{one-half-occupied at $p$} by $h$. Accordingly as $p$ is pre- or post-half-occupied, $P$ is \emph{pre-one-half-occupied} or \emph{post-one-half-occupied}.
\item If both $p,q$ are half-occupied by $h$, then $P$ is \emph{alternately occupied} by $h$.
\item If $p$ is fully occupied, and $q$ is unoccupied by $h$, then $P$ is \emph{once occupied at $p$}.
\item If $p$ is fully occupied and $q$ is half-occupied by $h$, then $P$ is \emph{sesqui-occupied at $p$}. Accordingly as $p$ is pre- or post-half-occupied, $P$ is \emph{pre-sesqui-occupied} or \emph{post-sesqui-occupied}.
\item If $p,q$ are both fully occupied by $h$, then $P$ is \emph{doubly occupied} by $h$.
\end{enumerate}
\end{defn}
Again, we can extend this definition to diagrams: $P$ is unoccupied (resp. sesqui-, pre/post-sesqui-, alternately , one-half-, pre/post-one-half-, doubly occupied) by the diagram $D$, if $P$ is so occupied by its H-grading $h(D)$.

The following definition applies to any set of H-data $(h,s,t)$, i.e. to $h \in H_1 (\mathbf{Z}, \mathbf{a}_{\pm \infty} )$ and idempotents $s,t$.
\begin{defn}[Idempotent terminology] \
\label{def:idempotent_terminology}
%\begin{enumerate}
%\item 
%\begin{enumerate}
%\item If $P \in s$ we say $P$ is \emph{pre-on} or $1*$.
%\item If $P \notin s$ we say $P$ is \emph{pre-off} or $0*$.
%\end{enumerate}
%\item
%\begin{enumerate}
%\item If $P \in t$ we say $P$ is \emph{post-on} or $*1$.
%\item If $P \notin t$ we say $P$ is \emph{post-off} or $*0$. 
%\end{enumerate}
%\item
\begin{enumerate}
\item If $P \notin s$ and $P \notin t$ we say $P$ is \emph{off-off} or \emph{all-off} or $00$.
\item If $P \notin s$ and $P \in t$ we say $P$ is \emph{off-on} or $01$.
\item If $P \in s$ and $P \notin t$ we say $P$ is \emph{on-off} or $10$.
\item If $P \in s$ and $P \in t$ we say $P$ is \emph{on-on} or \emph{all-on} or $11$.
\end{enumerate}
%\end{enumerate}
\end{defn}
(We find this terminology awkward, hence we offer several equally awkward alternatives.)
Again, this definition extends to diagrams. We say $D$ is all-off, off-on, etc., at $P$ if its H-data $(h,s,t)$ is all-off, off-on, etc., at $P$.

We can denote occupation and on/off properties with H-data: for instance, we may say that a pair $P$ is is all-on doubly occupied by $(h,s,t)$, or equivalently that $(h,s,t)$ is all-on doubly occupied at $P$.

Any H-data $(h,s,t)$, including the H-data of a diagram $D$, can be described completely by the terminology of occupation and on/off idempotents. The H-grading $h$ described precisely by the occupation of the various matched pairs. The idempotents $s,t$ are described precisely by the on/off data at each pair.

We can often deduce properties of $D$ simply from its occupation of places, or its on/off/etc properties. For instance, if $p$ is pre-half-occupied in a diagram $D$, then a strand of $D$ must end at $p$, and no strand can begin at $p$; so $D$ is off-on at $P$. %, so $p \notin S$ and $p \in T$. If $p$ is post-half-occupied, $p \in S$ and $p \notin T$. If $p$ is unoccupied, then there is a horizontal strand, or no strand, at $p$, so ($p \in S$ and $p \in T$) or ($p \notin S$ and $p \notin T$). If $p$ is fully occupied, there are either strands beginning and ending at $p$, or no strand beginning or ending at $p$, so ($p \in S$ and $p \in T$) or ($p \notin S$ and $p \notin T$).
%If $P$ is doubly occupied, then strands must run past both $p$ and $q$; we can have strands starting and ending at $p$ (or starting and ending at $q$, but not at both $p$ and $q$), or running $-\infty \mapsto +\infty$ at each interval; so $P$ is all-on or all-off. 
%Such reasoning is often useful.
%Note that the various ``occupation" definitions imply various idempotent properties. If $P$ is doubly occupied, each place $\{p,q\}$ of $P$ is fully occupied, so $P$ is all-on or all-off. Similarly, if $P$ is pre-sesqui-occupied, then $P$ is off-on, and if $P$ is post-sesqui-occupied, then $P$ is on-off. We continue through the various cases: a pre-sesqui-occupied pair must be off-on, and a post-sesqui-occupied pair on-off; a once occupied pair is all-on or all-off; an alternately occupied pair is all-on; a pre-one-half-occupied pair is off-on; a post-one-half-occupied pair is on-off; and finally, an unoccupied pair is all-on or all-off.

\subsection{Tightness of a diagram}
\label{sec:tightness}

We now define the tightness or twistedness of a diagram. Throughout this section $D$ is a viable diagram on an arc diagram $\ZZ$, so $D \in \A = \A(\ZZ)$. If $D$ has no crossings then $\partial D = 0$ so $D$ represents a homology class in $\HH$. Thus the following definition makes sense.
\begin{defn}[Tightness of a diagram]
\label{def:tight_twisted_crossed_diagrams}
A viable diagram $D \in \A$ is:
\begin{enumerate}
\item \emph{tight} if it has no crossings and is nonzero in homology;
\item \emph{twisted} if it has no crossings but is zero in homology
\item \emph{crossed} if it has crossings.
\end{enumerate}
\end{defn}
The significance of the tightness of a diagram will become clearer as we proceed. (We gave a rough description in section \ref{sec:many_types_twisted}.)
%Tight diagrams describe tight contact structures; twisted diagrams describe ``minimally" overtwisted contact structures; crossed diagrams are more degenerate.

\begin{lem}[Local-global tightness of diagram]
\label{lem:tightness_degeneracy}
Let $D$ be a viable diagram.
\begin{enumerate}
\item If $D$ is tight iff for each matched pair $P$, $D_P$ is tight.
\item If $D$ is twisted iff for each matched pair $P$, $D_P$ is tight or twisted, and at least one $D_P$ is twisted.
\item If $D$ is crossed iff for some matched pair $P$, $D_P$ is crossed.
\end{enumerate}
\end{lem}
We could equivalently say that $D$ is $X$ iff $D$ is $X$ at all matched pairs, where $X$ is either of the properties ``tight" or ``tight or twisted".
%We could equivalently say that $D$ is tight iff $D$ is tight at all matched pairs; and that $D$ is tight or twisted iff $D$ is tight or twisted at all matched pairs. 
Thus tightness is a ``local-to-global" property, as is ``(tight or twisted)-ness".

This statement indicates an increasing order of degeneracy: tightness means tight everywhere; being twisted somewhere makes the diagram twisted; and then, being crossed somewhere makes the whole diagram crossed.

\begin{proof}
If $D$ has a crossing, then it is local to some $P$ (lemma \ref{lem:viable_crossings_horizontal}), hence $D_P$ is crossed; the converse is clear, proving (iii). So now assume $D$ and all $D_P$ are crossingless. By lemma \ref{lem:homology_decomposition} $\HH(h,s,t) \cong \bigotimes_P \HH_P (h_P, s_P, t_P)$, which is a tensor product of $\Z_2$ vector spaces. So $D$ is nonzero in homology iff all $D_P$ are nonzero in homology; (i) and (ii) follow. 
\end{proof}

\subsection{Local strand diagrams, local algebras and homology}
%\subsection{Classification of local strand diagrams}
\label{sec:classification_local_diagrams}
\label{sec:local_algebras}

The local arc diagram $\ZZ_P$ is very simple. There is only a small set of possible H-data, all of which are viable; and given H-data on a fragment, the set of possible diagrams is even smaller. We now explicitly list out the possibilities.

There are only 4 steps in $\ZZ_P$, all exterior, and $h$ determines which are covered by strands. Let $P = \{p,q\}$. The idempotents $s$ and $t$ determine whether a strand begins or ends at $P$ --- though whether the strand lies specifically at $p$ or $q$ (or at both, with dotted strands) may be ambiguous. In most cases, but not all, this is enough to determine the diagram completely.

In table \ref{tbl:local_diagrams} we draw all augmented diagrams on $\ZZ_P$ --- or equivalently, all possible local diagrams of a viable diagram. For each set of H-data, described in terms of occupation and on/off terminology, there are no more than two viable local diagrams; specifying tightness then determines a unique diagram, up to relabelling the twin places $p,q$. This gives the following proposition.
\begin{prop}[Classification of local diagrams]
\label{prop:viable_diagram_classification}
Let $D$ be a diagram on $\ZZ_P$. Then the H-data and tightness of $D$ determines $D$ up to relabelling twins, and $D$ is as shown in table \ref{tbl:local_diagrams}.
\end{prop}
Since table \ref{tbl:local_diagrams} shows all diagrams on $\ZZ_P$, together with H-data, proposition \ref{prop:viable_diagram_classification} is proved except for the classification into tight, twisted and crossed diagrams; we prove this shortly.

For those H-data admitting more than one diagram, a diagram can alternatively be specified by Maslov grading (rather than tightness), as shown in table \ref{tbl:local_diagrams}. In all such cases, there is a crossed and a non-crossed diagram; the Maslov grading of the former is greater by $1$ than the latter. 

Hence, if we fix viable H-data $(h,s,t)$ on an arc diagram $\ZZ$, using the isomorphism $\A(h,s,t) \cong \bigotimes_P \A_P (h_P, s_P, t_P)$ of section \ref{sec:local_diagrams}, then up to a constant, the Maslov grading of a diagram $D$ is given by the number of matched pairs at which $D$ is crossed.

\begin{table}
\begin{center}
\begin{minipage}{1\textwidth}
\begin{tabular}{cc|c|c|c}
H-data & & Tight & Twisted & Crossed \\
\hline
Unoccupied  & all-off & 
\begin{tikzpicture}[scale=1]
\strandbackgroundshading
\strandsetupn{}{}
\leftoff
\rightoff
\draw (1.5,0.75) node {$0$};
\end{tikzpicture}
 & & \\
\hline
& all-on & 
\begin{tikzpicture}[scale=1]
\strandbackgroundshading
\strandsetupn{}{}
\lefton
\righton
\dothorizontals
\draw (1.5,0.75) node {$0$};
\end{tikzpicture}
 & & \\

\hline
One-half-occupied & Pre- & 
\begin{tikzpicture}[scale=1]
\strandbackgroundshading
\beforewused
\strandsetupn{}{}
\leftoff
\righton
\useb
\draw (1.5,0.75) node {$0$};
\end{tikzpicture}
 & & \\
\hline
 & Post- & 
\begin{tikzpicture}[scale=1]
\strandbackgroundshading
\afterwused
\strandsetupn{}{}
\lefton
\rightoff
\usea
\draw (1.5,0.75) node {$-\frac{1}{2}$};
\end{tikzpicture}
& & \\

\hline
Alternately occupied & all-on & 
\begin{tikzpicture}[scale=1]
\strandbackgroundshading
\beforevused
\afterwused
\strandsetupn{}{}
\lefton
\righton
\usea
\used
\draw (1.5,0.75) node {$-\frac{1}{2}$};
\end{tikzpicture}
 & & \\

\hline
Once occupied 
 & all-off & 
\begin{tikzpicture}[scale=1]
\strandbackgroundshading
\beforewused
\afterwused
\strandsetupn{}{}
\leftoff
\rightoff
\useab
\draw (1.5,0.75) node {$0$};
\end{tikzpicture}
 & & \\
\hline
& all-on & & 

\begin{tikzpicture}[scale=1]
\strandbackgroundshading
\beforewused
\afterwused
\strandsetup
\lefton
\righton
\usea
\useb
\draw (1.5,1.25) node {$w_p$};
\draw (1.5,0.25) node {$-1$};
\end{tikzpicture}
& 
\begin{tikzpicture}[scale=1]
\strandbackgroundshading
\beforewused
\afterwused
\strandsetup
\lefton
\righton
\useab
\dothorizontals
\draw (1.5, 1.25) node {$c_p$};
\draw (1.5,0.25) node {$0$};
\end{tikzpicture}
\\

\hline
Sesqui-occupied & Pre- & 
\begin{tikzpicture}[scale=1]
\strandbackgroundshading
\beforevused
\beforewused
\afterwused
\strandsetupn{}{}
\leftoff
\righton
\useab
\used
\draw (1.5,0.75) node {$0$};
\end{tikzpicture}
& & \\
\hline
 & Post- & 
\begin{tikzpicture}[scale=1]
\strandbackgroundshading
\aftervused
\beforewused
\afterwused
\strandsetupn{}{}
\lefton
\rightoff
\useab
\usec
\draw (1.5,0.75) node {$-\frac{1}{2}$};
\end{tikzpicture}
 & & \\

\hline
Doubly occupied & all-off & 
\begin{tikzpicture}[scale=1]
\strandbackgroundshading
\beforevused
\aftervused
\beforewused
\afterwused
\strandsetupn{}{}
\leftoff
\rightoff
\useab
\usecd
\draw (1.5,0.75) node {$0$};
\end{tikzpicture}
& & \\
\hline
& all-on & 
\begin{tikzpicture}[scale=1]
\strandbackgroundshading
\beforevused
\aftervused
\beforewused
\afterwused
\strandsetup
\lefton
\righton
\usea
\useb
\usecd
\draw (1.5, 1.25) node {$g_p$};
\draw (1.5,0.25) node {$-1$};
\draw (2,0.75) node {or};
\begin{scope}[xshift=2.7 cm]
\strandbackgroundshading
\beforevused
\aftervused
\beforewused
\afterwused
\strandsetupn{}{}
\lefton
\righton
\useab
\usec
\used
\draw (1.5, 1.25) node {$g_q$};
\draw (1.5,0.25) node {$-1$};
\end{scope}
\end{tikzpicture} 
& & 
\begin{tikzpicture}[scale=1]
\strandbackgroundshading
\beforevused
\aftervused
\beforewused
\afterwused
\strandsetup
\lefton
\righton
\dothorizontals
\useab
\usecd
\draw (1.5, 1.25) node {$c_P$};
\draw (1.5,0.25) node {$0$};
\end{tikzpicture}
\end{tabular}

\caption{Possible local diagrams, classified by $H$-data and tightness. Maslov indices are shown, and some diagrams are named.}
\label{tbl:local_diagrams}
\end{minipage}
\end{center}
\end{table}

%\subsection{Local algebras and homology}

We now describe the local strand algebra $\A_P$, and its homology, explicitly. Let $(h,s,t)$ be the H-data of a diagram on $\ZZ_P$.

Table \ref{tbl:local_diagrams} shows that $(h,s,t)$ determines a diagram in all cases except two: when $(h,s,t)$ is all-on doubly occupied or all-on once occupied. If $P$ is all-on doubly occupied by $(h,s,t)$, then $3$ diagrams are possible; if $P$ is all-on once occupied by $(h,s,t)$, then $2$ diagrams are possible.

These diagrams are important in the sequel, and so we name then. (Our choice of symbols may seem arbitrary, but there is method in the madness: $c$ stands for ``Crossed", $g$ stands for ``tiGht", and $w$ stands for ``tWisted").
\begin{defn} \
\label{def:C_P}
\begin{enumerate}
\item
If $P$ is all-on doubly occupied by $(h,s,t)$:
\begin{enumerate}
\item $c_P$ is the unique crossed diagram;
\item $g_p$ is the unique crossingless diagram with strands beginning and ending at $p$
\item $g_q$ is the unique crossingless diagram with a strand beginning and ending at $q$.
\end{enumerate}
\item
If $P$ is all-on once occupied at $p$ by $(h,s,t)$:
\begin{enumerate}
\item $c_p$ is the unique crossed diagram;
\item $w_p$ is the unique crossingless diagram.
\end{enumerate}
\item
For any other H-data, denote the unique diagram by $u_P$. 
\end{enumerate}
\end{defn}

Define chain complexes $C''_P$, $C'_P$, $C_P$ by
\[
\begin{array}{ccl}
C''_P : & 0 \To \Z_2 \langle c_P \rangle \To \Z_2 \langle g_p, g_q  \rangle \To 0 
& \text{where $\partial c_P = g_p + g_q$ and $\partial g_p = \partial g_q  = 0$} \\
C'_P : & 0 \To \Z_2 \langle c_p \rangle \To \Z_2 \langle w_p \rangle \To 0,
& \text{where $\partial c_p = w_p$ and $\partial w_p = 0$.} \\
C_P : & 0 \To \Z_2 \langle u_P \rangle \To 0.
\end{array}
\]
As a chain complex, up to a shift giving the correct Maslov grading, each summand $\A_P (h,s,t)$ of $\A_P$ is isomorphic to $C''_P, C'_P$ or $C_P$, accordingly as $(h,s,t)$ is all-on doubly occupied, all-on once occupied, or anything else.

%Note that we can regard $C'_P$ as the algebra
%\[
%C'_P = \frac{\Z_2 [c_P]}{(c_P^2)},
%\]
%with $\partial = \frac{\partial}{\partial c_P}$, and $o_P$ corresponding to the constant $1$. Similarly, we can regard $C_A$ as 
%\[
%C_P = \Z_2
%\]
%regarded as constants, with $c_a$ corresponding to the constant $1$. Finding such an description of $C''_P$ however is less straightforward.

Calculating the homology of these complexes is straightforward.
\begin{itemize}
\item
$H(C''_P) \cong \Z_2$, generated by the homology class of $g_p$ or $g_q$ (equal in homology since $\partial c_P = g_p + g_q$).
\item
$H(C'_P) = 0$.
\item
$H(C_P) \cong \Z_2$, generated by the homology class of $u_P$.
\end{itemize}

%We can now complete the proof of proposition \ref{prop:viable_diagram_classification}.
\begin{proof}[Proof of proposition \ref{prop:viable_diagram_classification}]
As observed above, it remains to classify diagrams by tightness. Crossed diagrams are clear. Our calculations now show that the only crossingless diagram which is zero in homology is $w_p$; all other crossingless diagrams are tight.
\end{proof}

\subsection{Homology of strand algebras}
\label{sec:homology_of_strand_algebras}

We can now compute the homology of $\A(h,s,t)$ directly, for any viable H-data $(h,s,t)$ on an arc diagram $\ZZ$.

Let $(h,s,t)$ be the H-data of a viable diagram. From section \ref{sec:local_algebras} and the isomorphism of DGAs $\A(h,s,t) \cong \bigotimes_P \A_P (h_P, s_P, t_P)$ (section \ref{sec:local_diagrams}), we have the following isomorphism of chain complexes or differential $\Z_2$-modules:
\[
\A (h,s,t) \cong
\bigotimes_{\text{$P$ all-on doubly occupied}} C''_P \otimes
\bigotimes_{\text{$P$ all-on once occupied}} C'_P \otimes
\bigotimes_{\text{other $P$}} C_P
%\cong
%\left( C''_P \right)^{\otimes L} \otimes \left( C''_P \right)^{\otimes N} \otimes \left( C_P \right)^{R}.
\]
%Here $L,N,R$ are respectively the numbers of pairs which are all-on doubly occupied, all-on once occupied, or neither. 
Homology decomposes as $\HH(h,s,t) \cong \bigotimes_P \HH_P (h_P, s_P, t_P)$ (lemma \ref{lem:homology_decomposition}), and the calculations of section \ref{sec:local_algebras} now give
\[
\HH(h,s,t) \cong \left\{ \begin{array}{ll}
0 & \text{$\ZZ$ has an all-on once occupied pair,} \\
\Z_2 & \text{otherwise.} \end{array} \right.
\]
Moreover, when there are no all-on once occupied pairs, $\HH(h,s,t) \cong \Z_2$ is generated by the homology class of any crossingless diagram, which are therefore all unique in homology.

This calculation recovers the homology calculation of Lipshitz--Ozsv\'{a}th--Thurston \cite{lipshitz_bimodules_2015}, for viable H-data, and extends it to augmented diagrams. Their calculation says that $\HH(h,s,t)$ is nontrivial iff there exists a crossingless diagram $D$ satisfying two conditions (stated above in section \ref{sec:strand_diagrams}), which we can now translate into our terminology. Condition (i) is that $D$ be viable. Condition (ii) is that if $p$ is fully occupied and $q$ is not fully occupied (i.e. $P$ is once occupied at $p$, or sesqui-occupied), then $P$ is not all-on. But sesqui-occupied local diagrams are never all-on (by reference to table \ref{tbl:local_diagrams} or otherwise), so condition (ii) simply rules out crossingless all-on once occupied local diagrams; this is equivalent to ruling out all-on once occupied pairs.

Because of this calculation, the following definition makes sense.
\begin{defn}
\label{defn:Mhst}
Let $(h,s,t)$ be viable H-data. 
\begin{enumerate}
\item
The \emph{homology class} of $(h,s,t)$, denoted $M_{h,s,t}$, is the unique nonzero homology class in $\HH(h,s,t)$, if it exists; otherwise $M_{h,s,t} = 0$.
\item
The \emph{local homology class} of $(h,s,t)$ at $P$, denoted $M_{h,s,t}^P$, is the unique nonzero local homology class in $\HH_P (h,s,t)$, if it exists; otherwise $M_{h,s,t}^P  = 0$.
\end{enumerate}
\end{defn}
This definition applies to any viable (abstract) set of H-data on $\ZZ$, i.e. $(h,s,t)$ where $h \in H_1 (\mathbf{Z}, \mathbf{a}_{\pm \infty})$ covers each step at most once, and $s,t$ are idempotents. Note that $(h,s,t)$ need not be the H-data of a diagram; in this case $\A(h,s,t) = \HH(h,s,t) = 0$.

The following statement encapsulates the above calculations and discussion.
\begin{prop}
\label{prop:homology_summands}
Let $(h,s,t)$ be viable H-data. Then precisely one of the following is true.
\begin{enumerate}
\item There is a tight diagram with H-data $(h,s,t)$; $(h,s,t)$ is the H-data of a diagram with no all-on once occupied pairs; $\HH(h,s,t) \cong \Z_2$, generated by the homology class $M_{h,s,t}$ of any crossingless diagram with H-data $(h,s,t)$.
\item There is a twisted diagram with H-data $(h,s,t)$; $(h,s,t)$ is the H-data of a diagram with an all-on once occupied pair; $\HH(h,s,t) = 0$ but $\A(h,s,t) \neq 0$. 
\item There is no diagram with H-data $(h,s,t)$; $\A(h,s,t) = 0$.
\end{enumerate}
\end{prop}

\begin{proof}
Clearly if there is no diagram with H-data $(h,s,t)$ then $\A(h,s,t)=0$; so suppose $(h,s,t)$ is the H-data of a diagram, and hence $\A(h,s,t) \neq 0$. In this case a crossingless diagram $D$ with H-data $(h,s,t)$ can always be drawn. If there is an all-on once occupied pair $P$, we calculated $\HH(h,s,t) = 0$, so $D$ is twisted. If there are no all-on once occupied pairs then we calculated $\HH(h,s,t) \cong \Z_2$, so $D$ is tight and generates $\HH(h,s,t)$.
\end{proof}

Proposition \ref{prop:homology_summands} allows us to define the tightness of H-data as follows. 
\begin{defn}[Tightness of H-data]
\label{def:tightness_H-data}
Let $(h,s,t)$ be viable H-data on $\ZZ$.
\begin{enumerate}
\item $(h,s,t)$ is \emph{tight} if there is a tight diagram with H-data $(h,s,t)$.
We denote the set of all viable tight H-data by $\mathbf{g}(\ZZ)$.
\item $(h,s,t)$ is \emph{twisted} if there is a twisted diagram with H-data $(h,s,t)$. We denote the set of all viable twisted H-data on $\ZZ$ by $\mathbf{w}(\ZZ)$.
\item Otherwise, $(h,s,t)$ is \emph{singular}.
\end{enumerate}
\end{defn}
When the arc diagram is understood we simply write $\mathbf{g}$ or $\mathbf{w}$ rather than $\mathbf{g}(\ZZ)$ or $\mathbf{w}(\ZZ)$.

Proposition \ref{prop:homology_summands} in fact gives several equivalent characteristations of tight, twisted or singular H-data. 
%For instance, $(h,s,t)$ is twisted iff it is the H-data of a diagram with an all-on once occupied pair.

Just as for tightness of diagrams, tightness of H-data obeys a ``local-to-global" principle.
\begin{lem}[Local-global tightness of H-data]
\label{lem:local-global_H-data}
Again let $(h,s,t)$ be viable H-data on $\ZZ$.
\begin{enumerate}
\item $(h,s,t)$ is tight iff for all matched pairs $P$, $(h_P,s_P,t_P)$ is tight .
\item $(h,s,t)$ is twisted iff for each matched pair $P$, $(h_P, s_P, t_P)$ is tight or twisted, and at least one $(h_P, s_P, t_P)$ is twisted.
\item $(h,s,t)$ is singular iff for some matched pair $P$, $(h_P, s_P, t_P)$ is singular.
\end{enumerate}
\end{lem}

\begin{proof}
First, $(h,s,t)$ is non-singular iff it is the H-data of a diagram iff each $(h_P, s_P, t_P)$ is the H-data of a diagram, proving (iii). So assume $(h,s,t)$, and hence all $(h_P, s_P, t_P)$, are non-singular, hence tight or twisted. Then $(h,s,t)$ is twisted iff there is an all-on once-occupied pair $P$, in which case this $(h_P, s_P, t_P)$ is twisted. Otherwise, $(h,s,t)$ and all $(h_P, s_P, t_P)$ are tight.
\end{proof}

\subsection{Properties of twisted and crossed diagrams}
\label{sec:tight_twisted_crossed_diagrams}

We now consider some properties of twisted and crossed diagrams. We first consider crossed diagrams.

\begin{lem}[Products of crossed and crossingless diagrams] 
\label{lem:products_crossings}
%\begin{enumerate}
%\item If $D_1 D_2$ is nonzero and crossed, then at least one of $D_1$ or $D_2$ is crossed.
If two diagrams $D_1$ and $D_2$ are crossingless then $D_1 D_2$ is zero or crossingless. 
Hence the submodule of $\A$ generated by crossingless diagrams forms a subalgebra.
%\end{enumerate}
\end{lem}
Note this lemma applies to crossingless diagrams in general (not just viable ones). The product of two crossingless viable diagrams, even though nonzero and (by the lemma) crossingless, may not be viable.

\begin{proof}
If $D_1 D_2$ has a crossing, then one strand starts below and ends above another. The two strands must change their order either in $D_1$ or $D_2$ (even if they are dotted strands); so $D_1$ or $D_2$ has a crossing. This proves the first statement; linearity then gives the second.
\end{proof}

Note that the converse of the above lemma is not true: it is possible to have $D_1 D_2$ crossingless, but $D_1$ or $D_2$ crossed. In fact the product of a crossed diagram with another diagram may be tight. We call this phenomenon \emph{sublimation}; see figure \ref{fig:sublimation}. 
%So the submodule of $\A^{aug}$ generated by diagrams with crossings is not a subalgebra. 
As mentioned in the section \ref{sec:intro_contact_meaning}, sublimation occurs repeatedly in $A_\infty$ operations.
% in homology: twisted fragments are transformed to crossed fragments, which can then be multiplied to be tight. In section \ref{sec:anatomy_tensor_products} we will define a notion of \emph{sublime} accordingly.

\begin{figure}
\begin{center}
\begin{tikzpicture}[scale=1.5]
\strandbackgroundshading
\tstrandbackgroundshading{1}
\beforewused
\afterwused
\taftervused{1}
\strandsetupn{}{}
\tstrandsetup{1}
\tstrandsetup{2}
\lefton
\righton
\trightoff{1}
\useab
\tusec{1}
\dothorizontals
\begin{scope}[xshift = 3.5 cm]
\draw (-0.75, 0.75) node {$=$};
\strandbackgroundshading
\aftervused
\beforewused
\afterwused
\strandsetupn{}{}
\lefton
\rightoff
\useab
\usec
\end{scope}
\end{tikzpicture}
\caption{Sublimation.}
\label{fig:sublimation}
\end{center}
\end{figure}
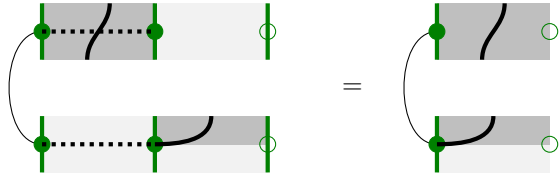

Turning to twisted diagrams, we observe that they are characterised by a specific local diagram $w_p$ (of definition \ref{def:C_P}) at an all-on once occupied pair $P = \{p,q\}$.

Indeed, a viable diagram $D$ is twisted iff each local diagram $D_P$ is tight or twisted, and at least one $D_P$ is twisted (lemma \ref{lem:tightness_degeneracy}); and by the classification of viable local diagrams in proposition \ref{prop:viable_diagram_classification} and table \ref{tbl:local_diagrams}, the only twisted local diagram is $w_p$.

At the pair $P$, one place $p$ must be occupied and its twin $q$ unoccupied, so we can speak of a twisted diagram being twisted not just at a matched pair, but \emph{at a specific place}.
\begin{defn}[Diagram twisted at a place]
\label{def:twisted_at_place}
A viable diagram $D$ such that $D_P$ is twisted at a pair $P = \{p,q\}$, with $p$ occupied, is \emph{twisted at the place $p$}.
\end{defn}

In terms of contact geometry, a contact structure which is ``minimally overtwisted" at a particular square has two bypasses from adjacent squares; these adjacent squares lie around a particular corner of the square, as in figure \ref{fig:twisted_example}.

Note that if $D,D'$ are viable crossingless diagrams, at least one of which is twisted, then their product $DD'$ (if nonzero and viable) is twisted: $DD'$ is crossingless by lemma \ref{lem:products_crossings}, and in homology at least one of $D$ or $D'$ is zero, hence $DD'$ is zero in homology. This corresponds to the contact-geometric phenomenon that a contact manifold containing an overtwisted submanifold is overtwisted.

\subsection{Diagrams representing a homology class}
\label{sec:diagrams_representing_homology}

Let $D$ be a tight diagram with H-data $(h,s,t)$ on $\ZZ$, hence (definition \ref{def:tight_twisted_crossed_diagrams}) nonzero in homology, with homology class $M_{h,s,t}$ (definition \ref{defn:Mhst}), a generator of $\HH(h,s,t)$ (proposition \ref{prop:homology_summands}).

While $D$ determines $M_{h,s,t}$, the diagram $D$ cannot always be recovered from $M_{h,s,t}$. There is ambiguity at all-on doubly occupied pairs, where the local diagrams $g_p, g_q$ (definition \ref{def:C_P}, or see table \ref{tbl:local_diagrams}) are both tight. We say $g_p$ and $g_q$ are related by \emph{strand switching} at the pair $\{p,q\}$; see figure \ref{fig:strand_switching}. 

In fact, $\HH(h,s,t)$ is isomorphic, as an abelian group, to the abelian group generated by tight diagrams with H-data $(h,s,t)$, modulo the subgroup generated by sums of diagrams related by strand switching (these are the boundaries: $\partial c_P = g_p + g_q$). 

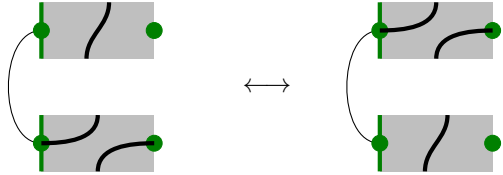
\begin{figure}
\begin{center}
\begin{tikzpicture}[mypersp, scale=1.5]
\begin{scope}[xzp=0]
\strandbackgroundshading
\beforevused
\aftervused
\beforewused
\afterwused
\strandsetupn{}{}
\lefton
\righton
\useab
\usec
\used
\draw (2,0.75) node {$\longleftrightarrow$};
\end{scope}
\begin{scope}[xzp=0, xshift=3 cm]
\strandbackgroundshading
\beforevused
\aftervused
\beforewused
\afterwused
\strandsetupn{}{}
\lefton
\righton
\usea
\useb
\usecd
\end{scope}
\end{tikzpicture}
\caption{Switching strands.}
\label{fig:strand_switching}
\end{center}
\end{figure}

We now determine the number of tight diagrams with given viable H-data $(h,s,t)$. There are none when $(h,s,t)$ is twisted or singular, since then $\HH(h,s,t)=0$ (proposition \ref{prop:homology_summands}). When $(h,s,t)$ is tight, at each pair we have a unique choice of tight local diagram, except at all-on doubly occupied pairs, where we have two tight local diagrams, related by strand switching. For any such choices we can glue local diagrams together into a tight diagram, giving the following statement.
\begin{lem}
\label{lem:selecting_occupied}
Let $(h,s,t)$ be tight viable H-data on the arc diagram $\ZZ$.
Let $L$ be the number of pairs all-on doubly occupied by $(h,s,t)$.
Then there are precisely $2^L$ tight diagrams with $H$-data $(h,s,t)$. Any two such diagrams are related by a sequence of switchings of strands.
\qed
\end{lem}
%(Note that, as per our convention from section \ref{sec:augmented_diagrams}, the diagrams discussed here are tight \emph{symmetrised} $\ZZ$-constrained augmented strand diagrams. So each of the $2^L$ diagrams is in turn a sum of $2^{j+n}$ unsymmetrised diagrams, where $j$ is the number of pairs of horizontal strands, and $n = | \bigcup_{i=1}^l \{-\infty_i, +\infty_i\} \backslash (S \cup T)|$. In the non-augmented case, each diagram is a sum of $2^j$ diagrams.)

%If $h$ is supported on interior steps of $\mathbf{Z}$, then $(h,s,t)$ is the viable $H$-data of a (non-augmented) diagram, and the same proof shows that there are $2^L$ tight (non-augmented) diagrams with $H$-data $(h,s,t)$.

These $2^L$ tight diagrams are precisely the diagrams in $\A$ representing their common homology class $M_{h,s,t}$ in $\HH(h,s,t)$.

\subsection{Dimensions of strand algebras}
\label{sec:dim_algebras}

It will be useful to know the dimension of $\A(h,s,t)$ as a $\Z_2$ vector space, as well as its subspaces of cycles and boundaries. %As the homology $\HH(h,s,t)$ is no bigger than $\Z_2$, and often zero, these dimensions are always close, and often equal.

Throughout this section let $\ZZ$ be an arc diagram and $(h,s,t)$ be viable non-singular H-data on $\ZZ$, with $L$ all-on doubly occupied pairs, and $N$ all-on once occupied pairs. Dimension always refers to the dimension of a $\Z_2$ vector space.
\begin{lem}
\label{lem:dim_Ahst}
The dimension of $\A(h,s,t)$ is $3^L 2^N$.
\end{lem}

\begin{proof}
A basis is given by the diagrams with H-data $(h,s,t)$; these may be specified locally at each matched pair $P$. Table \ref{tbl:local_diagrams} shows that if $P$ is all-on once occupied, then there are 3 choices at $P$; if $P$ is all-on doubly occupied, then there are 2 choices; in any other case there is a unique choice. \end{proof}

Now we refine $\A(h,s,t)$ by Maslov grading. As discussed in section \ref{sec:classification_local_diagrams}, with H-data fixed, the Maslov grading of a diagram $D$ is given, up to a constant, by the number of matched pairs at which $D$ is crossed.

We denote by $\A_n (h,s,t)$ the $\Z_2$ vector subspace of $\A(h,s,t)$ spanned by diagrams with crossings at precisely $n$ matched pairs. Then $\A(h,s,t) = \bigoplus_n \A_n (h,s,t)$ and this is the decomposition by Maslov grading.

\begin{lem}
\label{lem:dim_Anhst}
The dimension of $\A_n (h,s,t)$ is given by %either of the expressions below, which are equal:
\[
\dim \A_n (h,s,t) = \sum_i 2^{L-i} \binom{L}{i} \binom{N}{n-i}
= \sum_k \binom{L}{k} \binom{N+k}{n}.
\]
\end{lem}
For integers $a,b$, we regard $\binom{a}{b}$ as zero when $b < 0$ or $b>a$. The summations are over all integers $i$ or $j$; each has only finitely many nonzero terms.

\begin{proof}
Let $i$ be the number of crossed all-on doubly occupied pairs. Then there are $\binom{L}{i}$ choices for the all-on doubly occupied pairs which are to be crossed. At each crossed all-on doubly occupied pair there is a unique diagram, but at the $L-i$ non-crossed all-on doubly occupied pairs, there are 2 possible diagrams, giving $2^{L-i}$ choices at non-crossed all-on doubly occupied pairs. The other $n-i$ pairs with crossings must be all-on once occupied pairs, and there are $\binom{N}{n-i}$ choices for which all-on once occupied pairs will have crossings. Once these are chosen, all local diagrams are uniquely determined, and all such local diagrams glue into a diagram in $\A_n (h,s,t)$. Summing over all $i$ gives the first equality.

For the second equality, fix a reference diagram $D_0$ with H-data $(h,s,t)$ and no crossings. (Such a diagram always exists locally, by table \ref{tbl:local_diagrams}, and the local diagrams glue together.) Consider a diagram $D$ in $\A_n (h,s,t)$ and let $k$ be the number of all-on doubly occupied pairs at which $D$ and $D_0$ differ. There are $\binom{L}{k}$ ways in which we can choose these $k$ pairs. Now the $n$ pairs with crossings must come from the $k$ all-on doubly occupied pairs just chosen, together with the $N$ all-on once occupied pairs. There are $\binom{N+k}{n}$ ways to choose which of these $N+k$ pairs will be crossed. We now observe that once such choices are made, the diagram $D$ is uniquely determined. For at all-on doubly occupied pairs, the diagram either coincides with $D_0$ (and has no crossing), or differs from $D_0$, and if it differs from $D_0$ then we have chosen it to be crossed or not; these choices correspond to the $3$ possible local diagrams at an all-on doubly occupied pair. At all-on once occupied pairs, the diagram either coincides with $D_0$ (and is uncrossed), or is selected to be crossed; these choices correspond to the $2$ possible local diagrams at all-on once occupied pairs. At any other pair there is a unique local diagram. All such choices determine a diagram in $\A_n (h,s,t)$, and all such diagrams are constructed uniquely by such choices. Thus $\dim \A_n (h,s,t) = \sum_k \binom{L}{k} \binom{N+k}{n}$.
\end{proof}

We remark that it is also possible to prove directly that the two summations are equal. %The above proof also gives a combinatorial proof of an identity of binomial summations. 
Combining lemmas \ref{lem:dim_Ahst} and \ref{lem:dim_Anhst}, we have 
\begin{equation}
\label{eqn:binomial_identity}
\sum_n \sum_i 2^{L-i} \binom{L}{i} \binom{N}{n-i}
= \sum_n \sum_k \binom{L}{k} \binom{N+k}{n}
= 3^L 2^N.
\end{equation}

Next, we consider the dimension of the spaces of \emph{boundaries} and \emph{cycles} in $\A_n (h,s,t)$. Let $B_n (h,s,t)$ and $Z_n (h,s,t)$ respectively denote the $\Z_2$ vector subspaces of $\A_n (h,s,t)$ generated by boundaries and cycles. In other words, for any $n \geq 0$, $\partial \colon \A_{n+1} \To \A_n$ has image $B_n$ and kernel $Z_{n+1}$. When $(h,s,t)$ is twisted, $\A(h,s,t)$ has trivial homology, so $B_n (h,s,t) = Z_n (h,s,t)$ for all $n$.

\begin{lem}
\label{lem:dim_Bn}
The dimension of $B_n (h,s,t)$ is given by %each of the expressions below:
\[
\dim B_n (h,s,t) 
%= \sum_{k=1}^\infty (-1)^{k+1} \dim \A_{n+k} (h,s,t)
= \sum_i  2^{L-i} \binom{L}{i} \binom{N-1}{n-i}
= \sum_k \binom{L}{k} \binom{N+k-1}{n}
\]
\end{lem}
Again, summations are over all integers.

\begin{proof}
We know that nonzero homology only arises from diagrams with no crossings, i.e. with $n = 0$. Hence for all $n \geq 0$ we have $Z_{n+1} = B_{n+1}$, so $\dim B_n = \dim \A_{n+1} - \dim B_{n+1}$. Applying this repeatedly we obtain
\[
\dim B_n = \dim \A_{n+1} - \dim \A_{n+2} + \dim \A_{n+3} - \cdots
= \sum_{k=1}^\infty (-1)^{k+1} \dim \A_{n+k}.
\]
From lemma \ref{lem:dim_Anhst} we have $\dim \A_{n+k} = \sum_i 2^{L-i} \binom{L}{i} \binom{N}{n+K-i}$, and hence the above is equal to 
\begin{align*}
\sum_{k=1}^\infty (-1)^{k+1} \sum_i 2^{L-i} \binom{L}{i} \binom{N}{n+K-i} 
&= \sum_i 2^{L-i} \binom{L}{i} \sum_{k=1}^\infty (-1)^{k+1} \binom{N}{n+k-i} \\
&= \sum_i 2^{L-i} \binom{L}{i} \binom{N-1}{n-i},
\end{align*}
where we used the fact that for any integers $a,b$, $\sum_{k=1}^\infty (-1)^{k+1} \binom{a}{b+k} = \binom{a}{b+1} - \binom{a}{b+2} + \cdots = \binom{a-1}{b}$ (easily established, for instance, by induction). We have now proved the first claimed equality; the second follows from the identity of lemma \ref{lem:dim_Anhst}.
\end{proof}

\subsection{An ideal in the strand algebra}
\label{sec:ideals}

We now introduce an ideal $\F$ in $\A = \A(\ZZ)$, which will be important in the sequel.

\begin{defn} \
\label{def:ideal_F}
%\begin{enumerate}
%\item
%$\E$ is the $\Z_2$-submodule of $\A$ generated by (symmetrised $\ZZ$-constrained) diagrams which are not viable.
%\item
The $\Z_2$-submodule of $\A$ generated by diagrams which are not viable, or have at least one doubly occupied crossed pair, is denoted $\F$.
%\end{enumerate}
\end{defn}
%In other words, a diagram $D$ lies in $\F$ iff it has some step covered by two or more strands, or has a doubly occupied crossed pair. %Clearly $\E \subseteq \F$.

\begin{lem}
%$\E$ and 
$\F$ is a two-sided ideal of $\A$.
\end{lem}

\begin{proof}
First we claim that if $D,D'$ are diagrams where $D$ is not viable, then $DD'$ and $D'D$ are zero or non-viable. For $D$ then has some step covered by two or more strands, so $DD'$ is either zero, or has a step covered by two or more strands, hence is not viable; similarly for $D'D$.

Now suppose $D$ lies in $\F$ and $D'$ is another diagram. If $D$ is not viable, then the previous paragraph shows $DD', D'D \in \F$. So we may assume $D$ is viable and has a doubly occupied crossed pair $P$. After multiplication on the right (resp. left) by $D'$ the result may become non-viable, in which case $DD'$ (resp. $D'D$) is in $\F$. If the result is viable, then $DD'$ (resp. $D'D$) still has a doubly occupied crossed pair at $P$.
\end{proof}

The quotient $\A/\F$ is freely generated as a $\Z_2$-module by viable diagrams without crossed doubly occupied pairs. 
%Indeed, an element of $\A/\F$ is represented uniquely by a sum of viable diagrams without crossed doubly occupied pairs. 
Products can then be taken as in $\A$, unless the result is non-viable or has a crossed doubly occupied pair, in which case the result is zero.

The decomposition $\A \cong \bigoplus_{h,s,t} \A(h,s,t)$ also descends to the quotient $\A/\F$. (However, the differential $\partial$ does not, as it does not preserve $\F$.) Hence we may make the following definitions.

\begin{defn}\
\label{defn:Abar}
\begin{enumerate}
\item
The $\Z_2$-algebra $\AAbar$ is the quotient algebra $\A/\F$.
\item
The $\Z_2$ vector space $\AAbar(h,s,t)$ is the $(h,s,t)$ graded summand of $\AAbar$.
\item
For $x \in \A$, we denote by $\overline{x}$ its image in $\AAbar$ under the quotient map $\A \To \AAbar$.
\item
For a homomorphism $f$ with image in $\A$, we denote by $\fbar$ the homomorphism obtained by composing $f$ with the quotient map $\A \To \AAbar$.
\item
The \emph{standard form representative} $x \in \A$ of an $\overline{x} \in \AAbar$ is the sum of viable diagrams without crossed doubly occupied pairs whose image under the quotient map $\A \To \AAbar$ is $\overline{x}$.
\end{enumerate}
\end{defn}

%Many other, similar, ideals certainly exist in $\A$.
%: for instance, the submodule of $\A$ generated by diagrams which are non-viable or have a doubly occupied pair also forms an ideal.
%And, certain other similar submodules are not ideals. For instance, the submodule of $\A$ generated by diagrams with a crossed doubly occupied pair is not an ideal, since multiplication by another diagram may remove the crossing. However, the result in this case will not be viable.
The quotient $\AAbar$ is useful for our needs. Non-viable diagrams cannot contribute to homology, and although some crossed diagrams can be ``salvaged" into tight diagrams (thus contributing to homology) via sublimation, sublimation does not apply to crossed doubly occupied pairs. Thus $\AAbar$ is generated by diagrams which are ``salvageable" in this sense.

\section{Cycle selection and creation operators}
\label{sec:cycle_selection_creation}

\subsection{Cycle selection homomorphisms}
\label{sec:cycle_selection_homs}

Throughout this section fix an arc diagram $\ZZ$.

As discussed in section \ref{sec:intro_construction}, the construction of an $A_\infty$ structure on $\HH$ begins from a map $f_1 \colon \HH \To \A$ as follows.
\begin{defn}
\label{def:cycle_selection}
A \emph{cycle selection map} is a $\Z_2$-module homomorphism $f \colon \HH \To \A$ which preserves Maslov and H-gradings, and sends each homology class $x \in \HH$ to a cycle in $\A$ which represents $x$.
\end{defn}
Defining a cycle selection map requires finding diagrams representing each homology class --- as discussed in section \ref{sec:diagrams_representing_homology} for individual homology classes. 

The following constraint is a natural one to make, avoiding a proliferation of diagrams.
\begin{defn}[Diagrammatically simple homomorphisms] \
\label{def:diagrammatically_simple}
\begin{enumerate}
\item
A $\Z_2$-module homomorphism $f \colon \A \To \A$ is \emph{diagrammatically simple} if each for each diagram $D$, $f(D)$ is zero, or a single diagram.
\item
A $\Z_2$-module homomorphism $f \colon \HH \To \A$ is \emph{diagrammatically simple} if for all $M \in \HH$ that can be represented by a single diagram, $f(M)$ is zero, or a single diagram.
\end{enumerate}
\end{defn}

Recall (proposition \ref{prop:homology_summands}) that an H-summand $\HH(h,s,t)$ of $\HH$ is nonzero precisely when $(h,s,t)$ is tight, in which case $\HH(h,s,t) \cong \Z_2$, generated by $M_{h,s,t}$ (definition \ref{defn:Mhst}). To define a diagrammatically simple $f \colon \HH \To \A$, we simply need to select, for each tight H-data $(h,s,t) \in \mathbf{g}$ (definition \ref{def:tightness_H-data}), a tight diagram with H-data $(h,s,t)$. For such $(h,s,t)$, by lemma \ref{lem:selecting_occupied} there are precisely $2^L$ diagrams representing $M_{h,s,t}$, where $L$ is the number of pairs all-on doubly occupied by $(h,s,t)$. Making any of these $2^L$ choices for $f_1 (M_{h,s,t})$, and repeating for each $(h,s,t) \in \mathbf{g}$, results in a diagrammatically simple cycle selection map; and all diagrammatically simple cycle selection maps are of this form.

We now formalise the above and systematically describe all diagrammatically simple cycle selection homomorphisms.

As is standard in set theory, given a set $S$ whose elements are sets, a \emph{set choice function} for $S$ assigns to each $x \in S$ an element of $x$. The set of set choice functions for $S$ is naturally in bijection with the direct product of $S$ (i.e. the direct product of the elements of $S$), denoted $\prod S$, and we regard an element of $\prod S$ as a choice function for $S$. If $S$ is empty, we regard $S$ as having a unique choice function, which is the null function.

\begin{defn}[Pair choice function]
\label{def:pair_choice_function}
For $(h,s,t) \in \mathbf{g}$, let $\mathbf{P}_{h,s,t}$ be the set of all-on doubly occupied pairs of $(h,s,t)$. A \emph{pair choice function} for $(h,s,t)$ is a set choice function for $\mathbf{P}_{h,s,t}$. 
\end{defn}
Note $\mathbf{P}_{h,s,t}$ is a set of sets, each with two elements. 
%, so this definition makes sense. 
A pair choice function for $(h,s,t) \in \mathbf{g}$ assigns to each doubly occupied all-on pair $P = \{p,q\}$ of $(h,s,t)$ one of its places $p$ or $q$. If $\left| \mathbf{P}_{h,s,t} \right| = L$ then the number of pair choice functions for $(h,s,t)$ is $\left| \prod \mathbf{P}_{h,s,t} \right| = 2^{|\mathbf{P}_{h,s,t}|} = 2^L$. If $L=0$, then $(h,s,t)$ has a unique (null) pair choice function.

Given a pair choice function $\mathcal{C}(h,s,t)$ for $(h,s,t) \in \mathbf{g}$, we draw a tight diagram with H-data $(h,s,t)$, denoted $D_{\mathcal{C}(h,s,t)}$, as follows. At a matched pair $P$ which is not all-on doubly occupied, we draw the unique tight local diagram with H-data $(h_P, s_P, t_P)$. At an all-on doubly occupied matched pair $P = \{p,q\}$, $\mathcal{C}(h,s,t)$ selects one of the places $p$ or $q$; we denote it $\mathcal{C}(h,s,t)(P)$. There are two possible tight local diagrams $g_p, g_q$ (definition \ref{def:C_P}) with the required H-data at $P$, and we draw $g_{\mathcal{C}(h,s,t)(P)}$, the diagram with strands beginning and ending at $\mathcal{C}(h,s,t)(P)$. Putting these local diagrams together gives $D_{\mathcal{C}(h,s,t)}$.

\begin{defn}[Cycle choice function]
\label{def:cycle_choice_function}
A \emph{cycle choice function} for $\ZZ$ is a function which assigns to each $(h,s,t) \in \mathbf{g}(\ZZ)$ a pair choice function for $(h,s,t)$.
\end{defn}
A cycle choice function can be regarded as an element of the set $\prod_{(h,s,t) \in \mathbf{g}} \prod \mathbf{P}_{h,s,t}$. 

If $\CC$ is a cycle choice function, then we write $\CC(h,s,t)$ for the pair choice function assigned to $(h,s,t) \in \mathbf{g}$; this pair choice function for $(h,s,t)$ then determines a tight diagram $D_{\CC(h,s,t)}$ with H-data $(h,s,t)$ as described above.

Now a cycle choice function $\CC$ determines a map $f^\CC \colon \HH \To \A$ as follows. For $(h,s,t) \in \mathbf{g}$, $\HH(h,s,t)$ contains a single nonzero homology class $M_{h,s,t}$, and we define $f^\CC (M_{h,s,t}) = D_{\CC(h,s,t)}$. Combining such maps over all $(h,s,t) \in \mathbf{g}$ yields a diagrammatically simple cycle selection map $f^\CC \colon \HH \To \A$.
%It is not difficult to see that 
Indeed, all diagrammatically simple cycle selection maps are of this form, and distinct cycle choice functions yield distinct homomorphisms $f^\CC$, giving the following statement.
\begin{lem}
\label{lem:cycle_selection_classification}
Let $f \colon \HH \To \A$ be a cycle selection homomorphism. Then $f$ is diagrammatically simple iff $f=f^\CC$ for a unique cycle choice function $\CC$.
\qed
\end{lem}
In other words, there is a bijective correspondence between diagrammatically simple cycle selection maps, and cycle choice functions.

%\begin{proof}
%It remains to show that a diagrammatically simple cycle selection homomorphism $f$ is of the form $f^\CC$. For such an $f$, and $(h,s,t) \in \mathbf{g}$, $f$ takes $M_{h,s,t}$ to a diagram $f(M_{h,s,t})$ representing $M_{h,s,t}$. At each doubly occupied all-on matched pair $P = \{p,q\}$ of $(h,s,t)$, $f(M_{h,s,t})$ has strands beginning and ending at one of the places $p$ or $q$. Sending each $P$ to the place of $P$ at which strands begin and end gives a pair choice function for $(h,s,t)$, and doing the same over all $(h,s,t) \in \mathbf{g}$ gives a cycle choice function $\CC$. By construction then $f = f^\CC$.
%\end{proof}

We also consider a cycle selection map $f$ in general (not necessarily diagrammatically simple). For $(h,s,t) \in \mathbf{g}$ we have $\HH(h,s,t)$ generated by $M_{h,s,t}$. Then $f(M_{h,s,t})$ need not be a single diagram, but must be a sum of diagrams representing $M_{h,s,t}$, all of the same H-grading and Maslov grading, hence tight diagrams representing $M_{h,s,t}$. Moreover, as $f(M_{h,s,t})$ represents $M_{h,s,t}$, $f(M_{h,s,t})$ must be the sum of an \emph{odd} number of distinct diagrams. Conversely, if for each $(h,s,t) \in \mathbf{g}$ we define $f(M_{h,s,t})$ to be the sum of an odd number of distinct tight diagrams representing $M_{h,s,t}$, we obtain a cycle selection homomorphism.

\subsection{Differences in cycle selection}
\label{sec:diff_cycle_selection}

We now show how the different choices available in cycle selection are related to the ideal $\F$. Fix an arc diagram $\ZZ$ throughout this section.

\begin{lem}
\label{lem:even_num_reps}
Let $D_1, \ldots, D_{2n} \in \A$ be an even number of tight diagrams, all representing the homology class $M \in \HH$. Then we have the following.
\begin{enumerate}
\item $D_1 + \cdots + D_{2n} \in \partial \F$
\item If $g \in \A$ is an element homogeneous in Maslov grading and $H$-data, satisfying $\partial g = D_1 + \cdots + D_{2n}$, then $g \in \F$.
\end{enumerate}
\end{lem}

%\begin{lem}
%Let $D,D' \in \A$ be distinct tight diagrams representing the same homology class $M \in \HH$. Denote the pairs at which $D$ and $D'$ differ by $P_1, \ldots, P_k$ (these pairs are all-on doubly occupied in both $D,D'$ by lemma \ref{lem:selecting_occupied}). Then we have the following.
%\begin{enumerate}
%\item
%\label{lem:homology_reps_diff}
%$D + D' \in \partial \F$. In particular, $D+D' = \partial (F_1 + \cdots + F_k)$, where each $F_i$ is a distinct viable diagram, and each $F_i$ is crossed at $P_i$ and tight at all other pairs of $\ZZ$.
%\item
%\label{lem:tight_difference_in_F}
%If $g \in \A$ is an element homogeneous in Maslov grading and H-data, satisfying $\partial g = D+D'$, then $g \in \F$.
%\end{enumerate}
%\end{lem}
%Homogeneity of $g$ here means that when $g$ is written as a sum of distinct diagrams, each diagram has the same Maslov grading and H-data.

\begin{proof}
We first prove (i) when $n=1$, so take diagrams $D,D'$ which differ by switching strands at some all-on doubly occupied pairs $P_1, \ldots, P_k$ (lemma \ref{lem:selecting_occupied}). We proceed by induction on $k$. When $k=1$, let $F_1$ be the diagram all-on doubly occupied crossed at $P_1$, and equal to $D$ and $D'$ elsewhere.
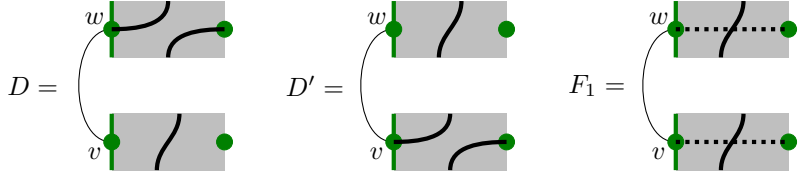
\begin{center}
\begin{tikzpicture}[mypersp, scale=1.5]
\begin{scope}[xzp=0, xshift = 0 cm]
\draw (-0.7,0.75) node {$D = $};
\strandbackgroundshading
\beforevused
\aftervused
\beforewused
\afterwused
\cubestrandsetup
\lefton
\righton
\draw [ultra thick] (0,1.25) to [out=0, in=-90] (0.5,1.5);
\draw [ultra thick] (0.5,1) to [out=90, in=180] (1,1.25);
\draw [ultra thick]  (0.4,0) to [out=90, in=-90] (0.6,0.5);
%\draw [ultra thick, dotted]  (0,0.25) -- (1,0.25); 
%\draw [ultra thick, dotted]  (0,1.25) -- (1,1.25); 
%\draw [ultra thick]  (0.4,1) to [out=90, in=-90] (0.6,1.5);
\end{scope}
\begin{scope}[xzp=0, xshift=2.5 cm]
\draw (-0.7, 0.75) node {$D' = $};
\strandbackgroundshading
\beforevused
\aftervused
\beforewused
\afterwused
\cubestrandsetup
\lefton
\righton
\draw [ultra thick] (0,0.25) to [out=0, in=-90] (0.5,0.5);
\draw [ultra thick] (0.5,0) to [out=90, in=180] (1,0.25);
\draw [ultra thick]  (0.4,1) to [out=90, in=-90] (0.6,1.5);
\end{scope}
\begin{scope}[xzp=0, xshift=5 cm]
\draw (-0.7, 0.75) node {$F_1 = $};
\strandbackgroundshading
\beforevused
\aftervused
\beforewused
\afterwused
\cubestrandsetup
\lefton
\righton
\draw [ultra thick, dotted]  (0,1.25) -- (1,1.25); 
\draw [ultra thick]  (0.4,1) to [out=90, in=-90] (0.6,1.5);
\draw [ultra thick, dotted]  (0,0.25) -- (1,0.25); 
\draw [ultra thick]  (0.4,0) to [out=90, in=-90] (0.6,0.5);
\end{scope}
\end{tikzpicture} 
\end{center}
Then $F_1$ is viable, crossed at $P_1$, hence lies in $\F$, and is tight at all other pairs; so $\partial F_1 = D + D'$, as desired.

Now consider $D,D'$ differing at $k$ pairs. Let $D''$ be obtained from $D$ by switching strands at $P_1$. By induction
\[
D+D'' = \partial F_1, \quad
D'' + D' = \partial \left( F_2 + \cdots + F_k \right)
\]
for some viable diagrams $F_1, \ldots, F_k$, with each $F_i$ crossed at $P_i$ (hence in $\F$) and tight elsewhere.% (hence the $F_i$ are distinct). 
Thus $D+D' = \partial ( F_1 + \cdots + F_k)$, proving (i) when $n=1$. For general $n$, simply split the diagrams $D_1, \ldots, D_{2n}$ into pairs and apply the $n=1$ case.

If $g$ is homogeneous in Maslov grading and H-data and $\partial g = D_1 + \cdots + D_{2n}$, then every diagram $G$ in $g$ has the same Maslov grading and tight H-data as each $F_i$ considered above. Thus $G$ is viable and has precisely one pair with a crossed local diagram. From table \ref{tbl:local_diagrams} we see that crossings can only occur in viable diagrams at pairs which are all-on once occupied or all-on doubly occupied. But having the same viable tight H-data as the $F_i$, $G$ has no all-on once occupied pairs (proposition \ref{prop:homology_summands}). So $G$ has a crossing at an all-on doubly occupied pair, and $G \in \F$. Hence $g \in \F$.
\end{proof}

%Applying the previous lemma to the choices of tight diagrams available for each homology class in a cycle selection homomorphism, we obtain the following.
%\begin{lem}
%Let $f, g \colon \HH \To \A$ be cycle selection homomorphisms. Then $f+g$ is a $\Z_2$-module homomorphism which preserves $H$-grading and Maslov grading and has image entirely in $\partial \F$.
%\end{lem}

%\begin{proof}
%As discussed in section \ref{sec:cycle_selection_homs}, $f(M_{h,s,t})$ is a sum of an odd number of distinct tight diagrams representing $M_{h,s,t}$; similarly for $g(M_{h,s,t})$. Thus $(f+g)(M_{h,s,t})$ is the sum of an even number of tight diagrams representing $M_{h,s,t}$. (Even if some of the diagrams cancel, they cancel in pairs.) By lemma \ref{lem:even_num_reps} then $(f+g)(M_{h,s,t}) \in \partial \F$.
%\end{proof}

\subsection{Creation operators}
\label{sec:creation_operators}

Let $(h,s,t)$ be viable H-data on an arc diagram $\ZZ$, which is all-on once occupied at a pair $P = \{p,q\}$, occupied at $p$. We saw in section \ref{sec:local_algebras} that $\A_P (h_P,s_P,t_P)$ is, as a chain complex, given by $C'_P$ (definition \ref{def:C_P}), which has trivial homology:
\[
0 \To \Z_2 \langle c_p \rangle \stackrel{\partial}{\To} \Z_2 \langle w_p \rangle \To 0, \quad \text{where $\partial c_p = w_p$ and $\partial w_p = 0$.}
\]
Here $c_p$ is the unique local crossed diagram, $w_p$ is the unique local twisted diagram.

There is a chain homotopy $A^* \colon C'_P \To C'_P$ from the identity to $0$; in fact, as $C'_P$ is so simple, there is a unique such $A^*$, given as follows.
\begin{defn}[Local creation operator]
The \emph{creation operator} $A^* \colon C'_P \To C'_P$ is the $\Z_2$-module homomorphism given by $A^* (w_p) = c_p$ and $A^* (c_p) = 0$.
\end{defn}
In other words, $A^*$ inserts a crossing, as in figure \ref{fig:creation_example}. The name $A^*$ references creation operators in physics. We have $A^* \partial + \partial A^* = 1$, a ``Heisenberg relation" or a chain homotopy from the identity to 0.

%If $C'_P$ is regarded as $\Z_2[c_P]/(c_P^2)$, then $A^*$ is multiplication by $c_P$. The creation operator $A^*$ takes an overtwisted diagram and inserts a crossing at $P$; if there is already a crossing then $A^*$ yields zero.

%Still assuming $(h,s,t)$ is viable H-data on $\ZZ$, all-on once occupied at $P$, 
Consider $\A(h,s,t) \cong \bigotimes_{P'} \A_P (h_{P'}, s_{P'}, t_{P'})$ (section \ref{sec:local_diagrams}). We can rewrite this as
\begin{equation}
\label{eqn:tensor_single_out_pair}
\A(h,s,t) \cong \A_P (h_P, s_{P}, t_{P}) \otimes \bigotimes_{P' \neq P} \A_{P'} (h_{P'}, s_{P'}, t_{P'}).
\end{equation}
The first factor contains local diagrams at $P$, and is isomorphic to $C'_P$; the second factor contains local diagrams everywhere else. A diagram $D \in \A(h,s,t)$ is then written as $x \otimes y$, where $x \in \A_P (h_P, s_{P}, t_{P}) \cong C'_{P}$ and $y \in \bigotimes_{P' \neq P} \A_{P'} (h_{P'}, s_{P'}, t_{P'})$. 
\begin{defn}[Creation operator]
\label{def:creation_operator}
Let $P$ be an all-on once occupied pair of the viable H-data $(h,s,t)$. The \emph{creation operator} $A_{P}^* \colon \A(h,s,t) \To \A(h,s,t)$ is given by $A_P^* = A^* \otimes 1$, in the tensor decomposition (\ref{eqn:tensor_single_out_pair}) above.
\end{defn}
In other words, $A_P^*$ inserts a crossing at $P$. Clearly $A_P^*$ is diagrammatically simple (definition \ref{def:diagrammatically_simple}). Note that if $D \in \F$ (i.e. $D$ has a crossed doubly occupied pair: definition \ref{def:ideal_F}), then $A_P^* D \in \F$ also. So $A_P^*$ descends to a map $\Abar_P^* \colon \AAbar(h,s,t) \To \AAbar(h,s,t)$.

\begin{lem}
\label{lem:creation_operators_chain_homotopies}
With $P$ and $(h,s,t)$ as above, $A_P^* \partial + \partial A_P^* = 1$ on $\A(h,s,t)$.
\end{lem}

\begin{proof}
Take a diagram in $\A(h,s,t)$ and write it as $x \otimes y$ according to the decomposition (\ref{eqn:tensor_single_out_pair}) above, so $x = c_p$ or $w_p$. Recalling that $\partial c_p = w_p$, $\partial w_p = 0$, $A^* w_p = c_p$, $A^* c_p = 0$, we have
\begin{align*}
\left( A_P^* \partial + \partial A_P^* \right) \left( w_p \otimes y \right)
&= A_P^* \left( w_p \otimes \partial y \right) + \partial \left( c_p \otimes y \right)
= c_p  \otimes \partial y + w_p \otimes y + c_p \otimes \partial y = w_p \otimes y, \\
\left( A_P^* \partial + \partial A_P^* \right) \left( c_P \otimes y \right)
&= A_P^* \left( w_p \otimes y + c_p \otimes \partial y \right) + 0
= c_p \otimes y.
\end{align*}
%(Alternatively $A^* \partial + \partial A^* = 1$ can be seen by considering the effect on strand diagrams directly.)
\end{proof}
This chain homotopy shows directly that $\HH(h,s,t) = 0$ when there is an all-on once occupied pair (proposition \ref{prop:homology_summands}).

In fact, creation operators are the \emph{only} way to obtain a \emph{diagrammatically simple} (definition \ref{def:diagrammatically_simple}) chain homotopy to the identity on a summand $\A(h,s,t)$.
\begin{lem}
\label{lem:chain_homotopies_creations}
Suppose $(h,s,t)$ is viable and non-singular, and $\int \colon \A(h,s,t) \To \A(h,s,t)$ is a diagrammatically simple $\Z_2$-module homomorphism which has pure Maslov degree, satisfying
\[
\int \partial + \partial \int = 1.
\]
Then $\int = A_P^*$, for some all-on once occupied matched pair $P$ of $(h,s,t)$.
\end{lem}

\begin{proof}
The existence of the chain homotopy $\int$ implies $\HH(h,s,t) = 0$; being non-singular then $(h,s,t)$ is twisted, so there is an all-on once occupied pair. With $(h,s,t)$ fixed, Maslov degree is given, up to a constant, by the number of pairs at which a diagram is crossed (section \ref{sec:classification_local_diagrams}). Since $\int \partial + \partial \int = 1$ and $\partial$ has Maslov degree $-1$, $\int$ has Maslov degree $1$. 

We use the decomposition $\A(h,s,t) \cong \bigotimes_P \A_P (h_P, s_P, t_P)$, noting (section \ref{sec:local_algebras}) that each $\A_P (h_P, s_P, t_P)$ is isomorphic (as a chain complex) to $C_P, C'_P$ or $C''_P$ (definition \ref{def:C_P}).

Take an arbitrary crossingless diagram $D_0$ with H-data $(h,s,t)$. Then $D_0$ is twisted at each all-on once occupied pair. As $\int$ is diagrammatically simple, $\int D_0$ is a single diagram or $0$. Since $\partial D_0 = 0$ and $\int \partial + \partial \int = 1$ we obtain $\partial \int D_0 = D_0$. Thus $\int D_0$ is a single diagram whose differential is $D_0$. The only such diagrams are those obtained from $D_0$ by inserting a crossing at an all-on once occupied pair $P = \{p,p'\}$ (say occupied at $p$). That is, $\int D_0 = A_P^* D_0$.

We claim that for \emph{any} diagram $D$ with H-data $(h,s,t)$, $\int D = A_P^* D$. The proof is by induction on the number $k$ of pairs at which $D$ is crossed (i.e., up to a constant, Maslov grading).

Suppose $D,D'$ are distinct crossingless diagrams with H-data $(h,s,t)$, which differ by switching strands at a single all-on doubly occupied pair $Q = \{q,q'\}$. The argument above shows that $\int D = A_R^* D$ for some all-on once occupied pair $R = \{r,r'\}$ (occupied at $r$), and similarly that $\int D' = A_V^* D'$ for some all-on once occupied pair $V = \{v,v'\}$ (occupied at $v$). We claim $R=V$. To see why, suppose $R \neq V$ and consider $\A(h,s,t)$ as a tensor product. We may write
\begin{gather*}
D = g_q \otimes w_r \otimes w_v \otimes z, \quad D' = g_{q'} \otimes w_r \otimes w_v \otimes z, \\
\int D = g_q \otimes c_r \otimes w_v \otimes z, \quad
\int D' = g_{q'} \otimes w_r \otimes c_v \otimes z,
\end{gather*}
where the first tensor factor is $C''_Q$, the second is $C'_R$, the third is $C'_V$, and the last is the tensor product given by all other matched pairs. Now we consider the diagram $E = c_Q \otimes w_r \otimes w_v \otimes z$, obtained from either $D$ or $D'$ by inserting a pair of crossing dotted horizontal strands at $Q$. We compute
\[
\int \partial E 
= \int \left( D + D' \right)
= g_q \otimes c_r \otimes w_v \otimes z + g_{q'} \otimes w_r \otimes c_v \otimes z,
\]
and hence $\int E$ is a single diagram (by diagrammatic simplicity) whose differential is
\[
\partial \int E
= \left( \int \partial + 1 \right) E
= g_q \otimes c_r \otimes w_v \otimes z + g_{q'} \otimes w_r \otimes c_v \otimes z
+ c_Q \otimes w_r \otimes w_v \otimes z.
\]
The three diagrams on the right respectively have crossings at $R$, $V$ and $Q$. Hence $\int E$ must have crossings at $R, V$ and $Q$, contradicting the fact that $\int$ has Maslov degree $1$. We conclude that $R=V$.

Thus, if $D,D'$ are two crossingless diagrams with H-data $(h,s,t)$, which differ by switching strands at a single all-on doubly occupied pair, then $\int D$ and $\int D'$ are both given by applying a creation operator $A_P^*$ at the same matched pair $P$. Now any two crossingless diagrams with H-data $(h,s,t)$ can be related by switching strands at some all-on doubly occupied pairs. Repeatedly applying this fact, we see that for any crossingless diagram $D$ with $H$-data $(h,s,t)$, $\int D = A_P^* D$. This proves the result when $k=0$.

Now take a $k \geq 0$ and suppose that, for all diagrams $D$ with H-data $(h,s,t)$ and crossings at $\leq k$ pairs, $\int D = A_P^* D$. Consider a diagram $D$ with H-data $(h,s,t)$, crossed at $k+1$ pairs. Then $D = w_p \otimes x$ or $c_p \otimes x$, where the first tensor factor refers to the complex $C'_P$ for the pair $P$, and the second factor refers to the tensor product given by all other matched pairs.

If $D = w_p \otimes x$ then $\partial D = w_p \otimes \partial x$, which contains diagrams crossed at $k$ pairs. By induction then
\[
\int \partial D = A_P^* \partial D = A_P^* \left( w_p \otimes \partial x \right) = c_p \otimes \partial x.
\]
It follows that $\int D$ is a single diagram (by diagrammatic simplicity) whose differential is
\[
\partial \int D = \left( \int \partial + 1 \right) D = c_p \otimes \partial x + w_p \otimes x.
\]
There is only one such diagram, namely $c_p \otimes x$. Thus $\int D = c_p \otimes x = A_P^* \left( w_p \otimes x \right) = A_P^* D$.

If $D = c_p \otimes x$ then $\partial D = w_p \otimes x + c_p \otimes \partial x$ and so by induction $\int \partial D = A_P^* \partial D = c_p \otimes x = D$. We then have $\partial \int D = \int \partial D + D = 0$, so $\int D$ is a single diagram crossed at $k+1 \geq 1$ pairs, or zero, whose differential is zero. Thus $\int D = 0 = A_P^* D$.

Thus, in any case, $\int D = A_P^* D$. By induction then $\int = A_P^*$.
\end{proof}

It is not difficult to see that, if we drop the requirement that $\int$ be diagrammatically simple, the result no longer holds: there are many $\Z_2$-module homomorphisms $\A(h,s,t) \To \A(h,s,t)$ of pure Maslov degree satisfying $\partial \int + \int \partial = 1$, which are not creation operators, as we see next. 
%(For instance, any sum of an odd number of creation operators gives such an $\int$.)

\subsection{Inverting the differential}
\label{sec:inverting_differential}

The following straightforward lemma shows how a creation operator $A_P^*$ ``integrates", i.e. finds partial inverses of the differential $\partial$ (hence the notation $\int$ of section \ref{sec:creation_operators}), as is required in constructing an $A_\infty$ structure (section \ref{sec:intro_construction}).

\begin{lem}
\label{lem:creations_invert_differential}
Suppose the viable H-data $(h,s,t)$ contains an all-on once occupied pair $P$. If $x \in \A(h,s,t)$ is a cycle, then $x = \partial A_P^* x$.
\end{lem}

\begin{proof}
As $x$ is a cycle, $\partial x = 0$. Hence $x = (A_P^* \partial + \partial A_P^*) x = \partial A_P^* x$.
\end{proof}

Recall from section \ref{sec:dim_algebras} the decomposition $\A(h,s,t) = \bigoplus_n \A_n (h,s,t)$ over Maslov grading, where $\A_n (h,s,t)$ contains diagrams with crossings at $n$ pairs, and the subspaces $Z_n (h,s,t)$ of cycles and $B_n (h,s,t)$ of boundaries. We are interested in maps obeying the following property.

\begin{defn}[Inverting differential]
A $\Z_2$-module homomorphism $\int \colon Z_n (h,s,t) \To \A_{n+1} (h,s,t)$ \emph{inverts the differential} if, for all $x \in Z_n (h,s,t)$, the equation $x = \partial \int x$ holds.
\end{defn}
If we have maps inverting the differential on $Z_n (h,s,t)$ for all $n$, of course these can be combined into a map $Z (h,s,t) \To \A (h,s,t)$ of Maslov degree 1 such that $\partial \int = 1$.

Lemma \ref{lem:creations_invert_differential} says that $A^*_P$ (more precisely, its restriction to $Z_n (h,s,t)$) inverts the differential.

When we have viable H-data $(h,s,t)$ with several all-on once occupied matched pairs $P_1, P_2, \ldots$, there are several creation operators $A^*_{P_1}, A^*_{P_2}, \ldots$ on $\A(h,s,t)$, and hence many ways to invert the differential. However, not every operator which inverts the differential is a creation operator. 

For one thing, we can simply choose different creation operator on each Maslov summand. %For each $n$ we may choose an all-on once occupied pair $P_n$ and define $A_n^* \colon Z_n (h,s,t) \To A_{n+1} (h,s,t)$ as the restriction of $A_{P_n}^*$ to $Z_n (h,s,t)$. Then $\bigoplus_n A_n^* \colon Z(h,s,t) \To \A(h,s,t)$ inverts the differential.
For another, we can also replace a creation operator with a sum of an odd number of creation operators. %: if $j \geq 0$ is an integer and $P_1, \ldots, P_{2j+1}$ is are all-on once occupied pairs for $(h,s,t)$, then for any cycle $x \in \A(h,s,t)$ we have
%\[
%x = \partial \left( A_{P_1}^* + \cdots + A_{P_{2j+1}}^* \right) x.
%\]
More fundamentally, however, not every operator $\int \colon Z_n (h,s,t) \To \A_{n+1} (h,s,t)$ inverting the differential is a sum of creation operators.
\begin{prop}
\label{prop:inverting_vs_creation}
Let $(h,s,t)$ be twisted H-data with $L \geq 0$ all-on doubly occupied pairs and $N \geq 1$ all-on once occupied pairs. Then the set of $\Z_2$-module homomorphisms $\int \colon Z_n (h,s,t) \To \A_{n+1} (h,s,t)$ inverting the differential is an affine $\Z_2$ vector space of dimension
\begin{equation}
\label{eqn:dim_inverting_differential}
\left[ \sum_k \binom{L}{k} \binom{N+k-1}{n} \right]
\left[ \sum_k \binom{L}{k} \binom{N+k-1}{n+1} \right],
\end{equation}
while the set of linear combinations of creation operators is a $\Z_2$ vector space of dimension $N$.
\end{prop}
Clearly the expression (\ref{eqn:dim_inverting_differential}) is in general much larger than $N$, so there exist many more maps inverting the differential than linear combinations of creation operators, as mentioned in section \ref{sec:intro_construction}.

For instance, taking $N=1,L=1,n=0$ the dimensions are 1 and 2; taking $N=4,L=0,n=1$, the dimensions are $4$ and $9$; and so on.

\begin{proof}
Since $(h,s,t)$ is twisted, $\HH(h,s,t)=0$ (proposition \ref{prop:homology_summands}), so $Z_n (h,s,t) = B_n (h,s,t)$ for all $n$. Let $P$ be an arbitrarily chosen all-on once occupied pair. 

Let $\mathcal{S}$ be the set of $\Z_2$-module homomorphisms $Z_n (h,s,t) \To \A_{n+1} (h,s,t)$ that invert the differential. Denoting the restriction of $A_P^*$ to $Z_n (h,s,t)$ (somewhat abusively) again by $A_P^*$, by lemma \ref{lem:creations_invert_differential} $A_P^* \in \mathcal{S}$. 

Let $\mathcal{T}$ be the set of $\Z_2$-module homomorphisms $Z_n (h,s,t) \To \A_{n+1} (h,s,t)$ with image in $Z_{n+1} (h,s,t)$. A $\Z_2$-module homomorphism $T \colon Z_n (h,s,t) \To \A_{n+1} (h,s,t)$ lies in $\mathcal{T}$ if and only if $\partial T = 0$. 

We now observe that $T \in \mathcal{T}$ iff $A_P^* + T \in \mathcal{S}$. Indeed, since $\partial A_P^* = 1$ we have $\partial T = 0$ iff $\partial (A_P^* + T) = 1$. Thus $\mathcal{S} = A_P^* + \mathcal{T}$.

Now the space of all linear combinations of creation operators has dimension $N$: it has basis given by %(restrictions of) the 
operators $A_{P'}^*$, over each of the $N$ all-on once occupied pairs $P'$. (They are linearly independent as they affect distinct matched pairs.)

On the other hand, $\mathcal{T}$ is nothing but the set of $\Z_2$-module homomorphisms $Z_n (h,s,t) \To Z_{n+1} (h,s,t)$. Hence $\mathcal{T}$ is a $\Z_2$ vector space of dimension $\dim Z_n (h,s,t) \dim Z_{n+1} (h,s,t)$, and $\mathcal{S}$ is an affine vector space of the same dimension. Lemma \ref{lem:dim_Bn} gives the dimension of $Z_n (h,s,t) = B_n (h,s,t)$ as $\sum_k \binom{L}{k} \binom{N+k-1}{n}$, which yields the desired result.
\end{proof}

%The dimension calculation in the above proof is unnecessary for an existence result; we could have simply given an example of an arc diagram with $L=1$ all-on doubly occupied pair $P = \{p,p'\}$, and $N=1$ all-on once occupied pair $Q = \{q,q'\}$, and shown directly that $\dim Z_0 (h,s,t)  = 2$ (generated by $w_p \otimes g_q \otimes \cdots$ and $w_p \otimes g_{q'} \otimes \cdots$)  and $\dim Z_1 (h,s,t) = 1$ (generated by $(w_p \otimes c_Q + c_p \otimes g_q + c_p \otimes g_{q'}) \otimes \cdots$).  

%\begin{qn}
%For twisted H-data $(h,s,t)$, is every diagrammatically simple $\Z_2$-module homomorphism $Z_n (h,s,t) \To \A_{n+1} (h,s,t)$ inverting the differential a creation operator?
%\end{qn}

The above deals with inverting the differential when $(h,s,t)$ is twisted, i.e. there is at least one all-on once occupied pair. When there are \emph{no} all-on once occupied pairs, i.e. $(h,s,t)$ is tight, the H-data only permits crossings in places which immediately land us in $\F$.

\begin{lem}
\label{lem:no_FOC_pairs_integral_in_F}
Suppose $(h,s,t)$ is tight. If $x \in \A(h,s,t)$ has pure Maslov grading, and $x = \partial f$ for some $f \in \A(h,s,t)$ also of pure Maslov grading, then $f \in \F$.
\end{lem}

\begin{proof}
Since $f$ has pure Maslov grading and $\partial f = x$, $f$ is a sum of viable diagrams, each crossed at one more matched pair than $x$. From proposition \ref{prop:viable_diagram_classification} and table \ref{tbl:local_diagrams}, crossings can only occur at all-on once or doubly occupied pairs. But tight $(h,s,t)$ have none of the former (proposition \ref{prop:homology_summands}), so any crossing occurs at a doubly occupied pair. Hence all diagrams in $f$ lie in $\F$.
\end{proof}

\subsection{Global creation operators}
\label{sec:global_creation}

For any twisted H-data $(h,s,t)$ on the arc diagram $\ZZ$ (i.e. $(h,s,t) \in \mathbf{w}(\ZZ)$: definition \ref{def:tightness_H-data}), there is an all-on once occupied pair $P$ (proposition \ref{prop:homology_summands}), and hence a creation operator $A_P^*$ on $\A(h,s,t)$ (definition \ref{def:creation_operator}) which inverts the differential (lemma \ref{lem:creations_invert_differential}).

We now introduce formalism to piece together such operators into a ``global" operator on all twisted summands. %, by choosing an all-on once occupied pair in each twisted summand.
\begin{defn}[Creation choice function]
\label{def:creation_choice_function}
A \emph{creation choice function} for $\ZZ$ assigns to each $(h,s,t) \in \mathbf{w}(\ZZ)$ one of its all-on once occupied matched pairs.
\end{defn}
If $\CC$ is a creation choice function and $(h,s,t)$ is twisted, we write $\CC(h,s,t)$ for the all-on once occupied matched pair chosen by $\CC$. Hence there is a creation operator
\[
A_{\CC(h,s,t)}^* \colon \A(h,s,t) \To \A(h,s,t).
\]

\begin{defn}
\label{def:global_creation_operator}
Let $\CC$ be a creation choice function for $\ZZ$.  The \emph{creation operator} of $\CC$ is the $\Z_2$-module homomorphism
\[
A_\CC^* \colon \bigoplus_{(h,s,t) \in \mathbf{w}} \A(h,s,t) \To \bigoplus_{(h,s,t) \in \mathbf{w}} \A(h,s,t)
\quad \text{given by} \quad
A_\CC^* = \bigoplus_{(h,s,t) \in \mathbf{w}} A_{\CC(h,s,t)}^*.
\]
\end{defn}

Putting together what we know on each summand, in particular the Heisenberg relation (lemma \ref{lem:creation_operators_chain_homotopies}) and differential  inversion (lemma \ref{lem:creations_invert_differential}), we immediately obtain the following. 
\begin{prop}
\label{prop:creation_operator_Heisenberg}
The creation operator $A_\CC^*$ of a creation choice function $\CC$ preserves H-grading, has Maslov degree $1$, and satisfies 
\[
A_\CC^* \partial + \partial A_\CC^* = 1.
\]
Moreover, for any cycle $x \in \bigoplus_{(h,s,t) \in \mathbf{w}} \A(h,s,t)$, we have $x = \partial A_\CC^* x$.
\qed
\end{prop}

Moreover, applying lemma \ref{lem:chain_homotopies_creations} over all $(h,s,t) \in \mathbf{w}$, we immediately have the following.
\begin{prop}
\label{prop:chain_homotopy_classification}
Suppose $\int \colon \bigoplus_{(h,s,t) \in \mathbf{w}} \A(h,s,t) \To \bigoplus_{(h,s,t) \in \mathbf{w}} \A(h,s,t)$ is a diagrammatically simple $\Z_2$-module homomorphism which preserves H-data, has pure Maslov degree, and satisfies
\[
\int \partial + \partial \int = 1.
\]
Then $\int$ is the creation operator $A_\CC^*$ of a creation choice function $\CC$.
\qed
\end{prop}

\subsection{Cycle selection and creation operators via ordering}
\label{sec:ordering}

In section \ref{sec:cycle_selection_homs} we defined a diagrammatically simple cycle selection homomorphism $f^\CC \colon \HH \To \A$, for any cycle choice function $\CC$. %Recall (definition \ref{def:cycle_choice_function}) that a cycle choice function is a function which assigns to each set of tight H-data $(h,s,t) \in \mathbf{g}$ a pair choice function; and a pair choice function (definition \ref{def:pair_choice_function}) assigns to each all-on doubly occupied pairs for $(h,s,t)$ one of its places. 
Indeed, we saw (lemma \ref{lem:cycle_selection_classification}) that any diagrammatically simple cycle selection homomorphism is of this form.

Quite separately, in section \ref{sec:global_creation} we defined a creation operator $A_\CC^*$, %\colon \bigoplus_{(h,s,t) \in \mathbf{w}} \A(h,s,t) \To \bigoplus_{(h,s,t) \in \mathbf{w}} \A(h,s,t)$,
for any creation choice function $\CC$. 
%Recall (definition \ref{def:creation_choice_function}) that a creation choice function assigns to each viable twisted $(h,s,t) \in \mathbf{w}$ one of its all-on once occupied matched pairs. It is a chain homotopy on $\bigoplus_{(h,s,t) \in \mathbf{w}} \A(h,s,t)$ from the identity to zero which preserves $H$-data and has Maslov degree $1$; and 
Moreover, we showed (proposition \ref{prop:chain_homotopy_classification}) that any diagrammatically simple chain homotopy from the identity to with appropriate gradings is of the form $A_\CC^*$ for some creation choice function $\CC$.

We now discuss a useful method to obtain such ``choice functions", of both types. 
%As usual, let $\ZZ = (\mathbf{Z}, \mathbf{a}, M)$ be an arc diagram.
\begin{defn}
\label{defn:pair_ordering}
Let the pairs of $\ZZ$ be $P_1 = \{p_1, p'_1\}, \ldots, P_k = \{p_k, p'_k\}$. A \emph{pair ordering} on $\ZZ$ consists of a total order on each of the sets
\[
\{ P_1, \ldots, P_k \}, P_1, P_2, \ldots, P_k.
\]
\end{defn}
Thus a pair ordering puts the pairs of $\ZZ$ in some order; and also puts the two places of each pair in some order. We denote a pair ordering by $\preceq$, and use this symbol for each of the total orders involved.

We note that $\ZZ$ comes with several naturally ordered sets that can be used to give a pair ordering. Recall that $\mathbf{Z}$ consists of $l$ intervals $Z_1, \ldots, Z_l$. Each interval is naturally totally ordered. Further, a total ordering of these intervals is implicitly given in listing them as $Z_1, \ldots, Z_l$. Then $\mathbf{Z}$ is totally ordered, and as $\mathbf{a}$ and each $P_i$ are subsets of $\mathbf{Z}$, they inherit a total order. The ordering on places can also be used to obtain an ordering on the set $\{P_1, \ldots, P_k\}$, in various reasonable ways; for instance if $P_i = \{p_i, p'_i\}$ and $P_j = \{p_j, p'_j\}$, we could define $P_i \prec P_j$ when $\min_{\preceq} \{p_i, p'_i\} \prec \min_{\preceq} \{p_j, p'_j\}$. Thus we obtain a pair ordering. But there is nothing natural about this way to order pairs, just as there is nothing natural about the ordering $Z_1, \ldots, Z_l$ of intervals; reordering the $Z_i$ yields a homeomorphic arc diagram, but an entirely different pair ordering.

Nonetheless, once we have a pair ordering, we naturally obtain a cycle choice function and a creation choice function, as follows.
\begin{defn}
\label{def:choice_functions_from_ordering}
Let $\preceq$ be a pair ordering on $\ZZ$.
\begin{enumerate}
\item
The \emph{cycle choice function of $\preceq$}, denoted $\mathcal{CY}^{\preceq}$, assigns to each set of tight H-data $(h,s,t) \in \mathbf{g}(\ZZ)$ the pair choice function on $\mathbf{P}_{h,s,t}$ which chooses from each all-on doubly occupied pair its $\preceq$-minimal place.
\item
The \emph{creation choice function of $\preceq$}, denoted $\mathcal{CR}^{\preceq}$, assigns to each twisted set of H-data $(h,s,t) \in \mathbf{w}(\ZZ)$ its $\preceq$-minimal all-on once occupied matched pair.
\end{enumerate}
\end{defn}
Note the definition of $\mathcal{CY}^\preceq$ uses the ordering on the $P_i$, while the definition of $\mathcal{CR}^\preceq$ uses the ordering on $\{P_1, \ldots, P_k\}$.

Suppose $P_i = \{p_i, p'_i\}$ is an all-on doubly occupied pair for the tight H-data $(h,s,t)$, with $p_i \prec p'_i$. Since $\mathcal{CY}^{\preceq}(h,s,t)$ chooses the $\preceq$-minimal place at each all-on doubly occupied pair, the resulting cycle selection homomorphism $f^{\mathcal{CY}^\preceq}$ always chooses a diagram with strands beginning and ending at $p_i$ rather than $p'_i$.

Now suppose the pairs of $\ZZ$ are ordered as $P_1 \prec P_2 \prec \cdots \prec P_k$. Since $\mathcal{CR}^{\preceq}$ chooses, for given twisted H-data $(h,s,t)$, the all-on once occupied matched pair which is minimal in the ordering of $\{P_1, \ldots, P_k\}$. So the resulting creation operator $A^*_{\mathcal{CR}^\preceq}$ inserts a crossing at $P_1$, if it is all-on once occupied; otherwise at $P_2$, if it is all-on once occupied; and so on. 

Clearly not every cycle choice function arises from a pair ordering, nor does every creation choice function. Nonetheless pair orderings provide a useful method to construct cycle choice functions and creation choice functions, and thus to construct $A_\infty$ structures on $\HH$.

\section{Tensor products of strand diagrams}
\label{sec:tensor_products}

\subsection{Anatomy and terminology for tensor products}
\label{sec:anatomy_tensor_products}

We now consider tensor powers of $\A$ and $\HH$ in some detail. Since $\A$ is freely generated (as a $\Z_2$ vector space) by (symmetrised $\ZZ$-constrained augmented strand) diagrams on $\ZZ$, its tensor power $\A^{\otimes n}$ is freely generated by tensor products $D_1 \otimes \cdots \otimes D_n$ of such diagrams.

The Maslov and H-gradings on $\A$ naturally carry over to $\A^{\otimes n}$, in such a way that the gradings of $D_1 \otimes \cdots \otimes D_n$ agree with those of the product $D_1 \cdots D_n$ (when it is nonzero) in $\A$. Recall (section \ref{sec:strand_diagrams}) that $h(D), \iota(D)$ denote the H- and Maslov gradings of a diagram $D$, and $m \colon H_1 (\mathbf{Z}, \mathbf{a}) \times H_0 (\mathbf{a}) \To \frac{1}{2} \Z$ counts average local multiplicities around places.

Throughout this section, let $D_1, \ldots, D_n$ be diagrams with H-data $(h_1, s_1, t_1), \ldots, (h_n, s_n, t_n)$, and homology classes $M_1, \ldots, M_n$.

We saw in section \ref{sec:strand_diagrams} that $\iota(D_1 D_2) = \iota(D_1) + \iota(D_2) + m(h_2, \partial h_1)$ and $h(D_1 D_2) = h(D_1) + h(D_2)$. Applying this result repeatedly shows that
\[
h \left( D_1 \cdots D_n \right) = \sum_{i=1}^n h(D_i) = \sum_{i=1}^n h_i
\quad \text{and} \quad
\iota \left( D_1 \cdots D_n \right) = \sum_{i=1}^n \iota(D_i) + \sum_{1 \leq j < k \leq n} m \left( h_k, \partial h_j \right),
\]
motivating the following definition.

\begin{defn}[Gradings for tensor products] \
\begin{enumerate}
\item
The \emph{H-grading} of $D_1 \otimes \cdots \otimes D_n$ or $M_1 \otimes \cdots \otimes M_n$ is
\[
h \left( D_1 \otimes \cdots \otimes D_n \right) 
= h \left( M_1 \otimes \cdots \otimes M_n \right)
= \sum_{i=1}^n h_i \in  H_1 ( {\bf Z}, {\bf a}).
\]
\item
The \emph{Maslov grading} of $D_1 \otimes \cdots \otimes D_n$ or $M_1 \otimes \cdots \otimes M_n$ is 
\[
\iota \left( D_1 \otimes \cdots \otimes D_n \right) 
= \iota \left( M_1 \otimes \cdots \otimes M_n \right)
= \sum_{i=1}^n \iota(D_i) + \sum_{1 \leq j < k \leq n} m \left( h_k, \partial h_j \right).
\]
\end{enumerate}
\end{defn}

The notion of viability (section \ref{sec:viability}) also extends usefully to $\A^{\otimes n}$. We add a requirement that idempotents must match. Thus, viability of tensor products is fundamentally a property of \emph{H-data} (not just H-grading).

The definition uses the decomposition of $\A^{\otimes n}$ over H-data: from $\A = \bigoplus_{(h,s,t)} \A(h,s,t)$, we have
\[
\A^{\otimes n} = \bigoplus_{(h_1, s_1, t_1), \ldots, (h_n, s_n, t_n)} 
\Big( \A(h_1, s_1, t_1) \otimes \A(h_2, s_2, t_2) \otimes \cdots \otimes \A(h_n, s_n, t_n) \Big).
\]
There is a similar decomposition of the homology $\HH^{\otimes n}$.

\begin{defn}[Viability for tensor products] \
\label{defn:viability_tensor_products}
\begin{enumerate}
\item
A sequence of H-data $(h_1, s_1, t_1), \ldots, (h_n, s_n, t_n)$ is \emph{viable} if the following conditions hold:
\begin{enumerate}
\item
for each $1 \leq i \leq n-1$, $t_i = s_{i+1}$; and
\item
$h_1 + \cdots + h_n$ is viable (i.e. has multiplicity $0$ or $1$ on each step of $\mathbf{Z}$).
\end{enumerate}
\item
A summand $\A(h_1, s_1, t_1) \otimes \cdots \otimes \A(h_n, s_n, t_n)$ of $\A^{\otimes n}$, or a summand $\HH(h_1, s_1, t_1) \otimes \cdots \otimes \HH(h_n, s_n, t_n)$ of $\HH^{\otimes n}$, is \emph{viable} if the sequence of H-data $(h_1, s_1, t_1)$, $\ldots$, $(h_n, s_n, t_n)$ is viable.
\item
An element of $\A^{\otimes n}$ or $\HH^{\otimes n}$ is \emph{viable} if it lies in a viable summand.
\end{enumerate}
\end{defn}
We refer to the first condition of (i), that all $t_i = s_{i+1}$, as \emph{idempotent matching}. When it fails, we say we have an \emph{idempotent mismatch}. 

Thus the tensor product $D_1 \otimes \cdots \otimes D_n \in \A^{\otimes n}$ or $M_1 \otimes \cdots \otimes M_n \in \HH^{\otimes n}$ is viable iff the sets of H-data $(h_1, s_1, t_1), \ldots, (h_n, s_n, t_n)$ form a viable sequence.

We observe that when $n=1$, the notions of viability discussed above reduce to those previously discussed. The notion of H-data then naturally extends to a tensor product.

\begin{defn}[H-data of tensor product]
\label{def:H-data_tensor_product}
If $D_1 \otimes \cdots \otimes D_n$ or $M_1 \otimes \cdots \otimes M_n$ is viable, its \emph{H-data} is the triple $(h_1 + \cdots + h_n, s_1, t_n)$.
\end{defn}

The notions of occupation of places (definition \ref{def:occupation_places}) and pairs (definition \ref{def:occupation_pairs}) only depend on H-grading; and on/off terminology (definition \ref{def:idempotent_terminology}) only depends on idempotents. Hence these notions naturally extend to tensor products.
\begin{defn}[Occupation of places and pairs by tensor product] \
\begin{enumerate}
\item
A place is unoccupied, pre-half-occupied, post-half-occupied, half-occupied, or fully occupied by $D_1 \otimes \cdots \otimes D_n$ (resp. $M_1 \otimes \cdots \otimes M_n$) if it is so occupied by $h(D_1 \otimes \cdots \otimes D_n)$ (resp. $h(M_1 \otimes \cdots \otimes M_n)$).
\item
A pair $P$ is unoccupied, one-half-occupied, pre/post-one-half-occupied, alternately occupied, sesqui-occupied, pre/post-sesqui occupied, or doubly occupied by $D_1 \otimes \cdots \otimes D_n$ (resp. $M_1 \otimes \cdots \otimes M_n$) if it is so occupied by $h(D_1 \otimes \cdots \otimes D_n)$ (resp. $h(M_1 \otimes \cdots \otimes M_n)$).
\end{enumerate}
\end{defn}

\begin{defn}[Idempotent terminology for tensor products]
A viable tensor product $D_1 \otimes \cdots \otimes D_n$ or $M_1 \otimes \cdots \otimes M_n$ is off-off/00, off-on/01, on-off/10 or on-on/11 at a pair $P$ accordingly as the H-data of $D_1 \otimes \cdots \otimes D_n$ or $M_1 \otimes \cdots \otimes M_n$ is off-off/00, off-on/01, on-off/10 or on-on/11 at $P$.
\end{defn}
Thus we describe the on/off/0/1 status of $D_1 \otimes \cdots \otimes D_n$ at $P$ via definition \ref{def:idempotent_terminology} using $s=s_1$ and $t=t_n$. In other words, we use the idempotents $s_1, t_n$ at the beginning and end of the 	tensor product; the ``interior" idempotents $s_2, \ldots, s_{n-1}$ are irrelevant. (For this reason we avoid the terminology ``all-on" or ``all-off" in this context, and prefer the more awkward but less misleading ``on-on", $11$, etc.)

We can also extend the notion of a local diagram to tensor products.
\begin{defn}[Local tensor product]
\label{def:local_tensor_product}
Let $P$ be a matched pair of $\ZZ$ and $D_1 \otimes \cdots \otimes D_n \in \A(\ZZ)^{\otimes n}$ a viable tensor product. The \emph{local tensor product} $(D_1 \otimes \cdots \otimes D_n)_P$ is the tensor product of local diagrams on $\ZZ_P$
\[
\left( D_1 \otimes \cdots \otimes D_n \right)_P
= \left( D_1 \right)_P \otimes \cdots \otimes \left( D_n \right)_P
\in \A_P^{\otimes n}.
\]
Similarly for homology classes,
\[
\left( M_1 \otimes \cdots \otimes M_n \right)_P
= \left( M_1 \right)_P \otimes \cdots \otimes \left( M_n \right)_P \in \HH_P^{\otimes n}.
\]
\end{defn}

An important property of viability is the following.
\begin{lem} \
\label{lem:nonzero_viable_tensor_product}
\begin{enumerate}
\item 
If the product $D_1 \cdots D_n$ is a nonzero viable diagram, then $D = D_1 \otimes \cdots \otimes D_n$ is viable.
\item
If the product $M_1 \cdots M_n$ is a nonzero homology class, then $M = M_1 \otimes \cdots \otimes M_n$ is viable.
\end{enumerate}
\end{lem}

\begin{proof}
In both cases, if idempotents don't match then the product is zero. In the first case, viability of $D_1 \cdots D_n$ then implies viability of $D$. In the second case, $\HH(h,s,t)$ is only nonzero for viable $(h,s,t)$, so $M_1 \cdots M_n$ nonzero then implies viability of $M$.
\end{proof}
Note that neither of the converses to lemma \ref{lem:nonzero_viable_tensor_product}(i) or (ii) is true: there exist viable $D_1 \otimes \cdots \otimes D_n$ with $D_1 \cdots D_n = 0$, and viable $M_1 \otimes \cdots \otimes M_n$ with $M_1 \cdots M_n = 0$. In fact, $D_1 \otimes \cdots \otimes D_n$ and $M_1 \otimes \cdots \otimes M_n$ may be viable, yet there may not exist any strand diagram with its H-data! See e.g. figure \ref{fig:singular_example}. We introduce the notions of ``critical" and ``singular" to describe these phenomena below.

\subsection{Tightness of tensor products of diagrams}
\label{sec:tightness_tensor_products}

As usual, throughout this section $D_1, \ldots, D_n$ are diagrams on $\ZZ$. We extend our notions of tightness/twistedness to tensor products. Note that if $D_1 \otimes \cdots \otimes D_n$ is viable, and the product $D_1 \cdots D_n$ is nonzero, then $D_1 \cdots D_n$ is a viable diagram, so (definition \ref{def:tight_twisted_crossed_diagrams}) is tight, twisted, or crossed.

\begin{defn}[Tightness of tensor product]
\label{def:tensor_product_tightness}
Suppose $D = D_1 \otimes \cdots \otimes D_n$ is viable, with H-data $(h,s,t)$.
\begin{enumerate}
\item If $D_1 \cdots D_n$ is nonzero and tight, and all $D_i$ are tight, then $D$ is \emph{tight}.
\item If $D_1 \cdots D_n$ is nonzero and tight, but not all $D_i$ are tight, then $D$ is \emph{sublime}.
\item If $D_1 \cdots D_n$ is nonzero and twisted, then $D$ is \emph{twisted}.
\item If $D_1 \cdots D_n$ is nonzero and crossed, then $D$ is \emph{crossed}.
\item If $D_1 \cdots D_n = 0$, but $\A(h,s,t) \neq 0$, then $D$ is \emph{critical}.
\item If $D_1 \cdots D_n = 0$ and $\A(h,s,t) = 0$, then $D$ is \emph{singular}.
\end{enumerate}
\end{defn}
Definition \ref{def:tensor_product_tightness} presents tightness as as a list of things that go increasingly wrong. First a diagram is not tight; then the resulting diagram is twisted; then crossed; then zero; and then its existence is nonsensical. It is clear that any viable $D$ falls into precisely one of these types. % (and this remains true whether we consider $D$ at a single matched pair, or globally, and whether in the augmented or non-augmented case).

Note that this definition applies equally well if we consider $D$ at a single matched pair $P = \{p,p'\}$, or as a whole. So we may say that $D$ is tight (sublime, twisted, etc.) at $P$, meaning that $D_P$ is tight (sublime, twisted, etc.) on $\ZZ_P$. 

The condition that $\A(h,s,t) \neq 0$ is equivalent to the existence of a diagram with H-data $(h,s,t)$. Thus when $D$ is singular, its H-data $(h,s,t)$ is also singular: no diagram exists with this H-data.

When $n=1$, the notions of tightness reduce to those of definition \ref{def:tight_twisted_crossed_diagrams}; the sublime, critical and singular cases do not arise.

When $D$ is twisted at a pair $P = \{p,p'\}$, then $D_1 \cdots D_n$ is twisted at one of the places $p$ or $p'$ (definition \ref{def:twisted_at_place}), and we say $D$ is twisted at $p$ or $p'$ accordingly.

When $D$ is viable, since idempotents match according to $t_i = s_{i+1}$, we can draw $D$ as a sequence of strand diagrams, side by side, where the right hand side of $D_i$ coincides with the left hand side of $D_{i+1}$. Thus we regard $\A^{\otimes n}$ as a ``horizontal" tensor product, and the local decomposition $\A = \bigotimes \A_P$ as a ``vertical" tensor product.

For instance, figure \ref{fig:sesqui_tensor_product} depicts a viable tensor product of 6 diagrams, at a pair which is sesqui-occupied. 
%It may easily be included into a tensor product of non-augmented diagrams on a larger arc diagram.

\begin{figure}
\begin{center}
\begin{tikzpicture}[scale=1.5]
\strandbackgroundshading
\tstrandbackgroundshading{1}
\tstrandbackgroundshading{2}
\tstrandbackgroundshading{3}
\tstrandbackgroundshading{4}
\tstrandbackgroundshading{5}
\afterwused
\tbeforewused{2}
\taftervused{5}
\strandsetup
\tstrandsetup{1}
\tstrandsetup{2}
\tstrandsetup{3}
\tstrandsetup{4}
\tstrandsetup{5}
\lefton
\rightoff
\trightoff{1}
\trighton{2}
\trighton{3}
\trighton{4}
\trightoff{5}
\usea
\tuseb{2}
\tdothorizontals{3}
\tdothorizontals{4}
\tusec{5}
\end{tikzpicture}
\caption{A sesqui-occupied tight tensor product of 6 diagrams.}
\label{fig:sesqui_tensor_product}
\end{center}
\end{figure}
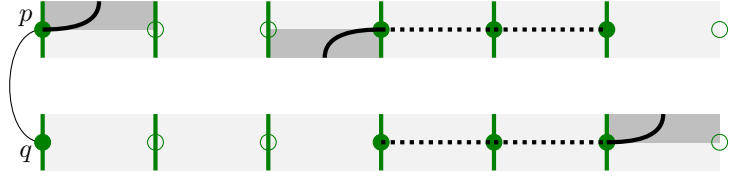

The following lemma generalises lemma \ref{lem:tightness_degeneracy}.
\begin{lem}[Local-global tightness of tensor product]
Let $D = D_1 \otimes \cdots \otimes D_n$ be viable.
\label{lem:tensor_product_local_tightness}
\begin{enumerate}
\item $D$ is tight iff $D$ is tight at all matched pairs.
\item $D$ is sublime iff $D$ is tight or sublime at all matched pairs, and sublime at $\geq 1$ matched pair.
\item $D$ is twisted iff $D$ is tight, sublime or twisted at all matched pairs, and twisted at $\geq 1$ matched pair.
\item $D$ is crossed iff $D$ is tight, sublime, twisted or crossed at all matched pairs, and crossed at $\geq 1$ matched pair.
\item $D$ is critical iff $D$ is tight, sublime, twisted, crossed or critical at all matched pairs, and critical at $\geq 1$ matched pair.
\item $D$ is singular iff $D$ is singular at $\geq 1$ matched pair.
\end{enumerate}
\end{lem}
Thus decomposing tensor products of diagrams according to matched pairs, we see that there is an increasing order of degeneracy
\[
\text{tight} < \text{sublime} < \text{twisted} < \text{crossed} < \text{critical} < \text{singular},
\]
and the tightness of a tensor product of diagrams is given by its ``most degenerate" local tensor product. 

An equivalent ``local-to-global" formulation (as in section \ref{sec:tightness}) is that $D$ is $X$ iff $D$ is $X$ at all matched pairs, where $X$ is any of the following ascending properties: tight; tight or sublime; tight, sublime or twisted; etc.

\begin{proof}
Let $D$ have H-data $(h,s,t)$. First, $D$ is singular iff $\A(h,s,t)=0$, iff $(h,s,t)$ is singular, iff some $(h_P, s_P, t_P)$ is singular (lemma \ref{lem:local-global_H-data}), iff some $\A_P (h_P, s_P, t_P)=0$, iff some $D_P$ is singular. So we may now assume $D$ and all $D_P$ are non-singular, hence diagrams exist with H-data $(h,s,t)$.

If $D$ is critical then $D_1 \cdots D_n = \bigotimes_P (D_1 \cdots D_n)_P = 0$ so some $(D_1 \cdots D_n)_P = 0$, hence $D$ is critical at some matched pair. Conversely, if $D$ is critical at $P$ then $(D_1 \cdots D_n)_P = 0$ so $D_1 \cdots D_n = 0$ and $D$ is critical.

If $D$ is crossed then $D_1 \cdots D_n$ is nonzero and crossed. By viability (lemma \ref{lem:viable_crossings_horizontal}), each crossing occurs at some matched pair $P$, hence $D$ is crossed at $P$. Moreover, from above no matched pair is critical; hence $D$ is tight, sublime, twisted or crossed at each pair. Conversely, suppose $D$ is crossed at $P$ and is not critical at any pair. Then $D_1 \cdots D_n$ is crossed at $P$, and nonzero at every matched pair, so $D_1 \cdots D_n$ is nonzero and crossed, hence $D$ is crossed.

If $D$ is twisted then $D_1 \cdots D_n$ is nonzero and twisted; hence (lemma \ref{lem:tightness_degeneracy}) $D_1 \cdots D_n$ is twisted at some pair $P$ but not crossed at any pair. Thus $D$ is twisted at $P$, but from above, $D$ is not critical or crossed at any matched pair; hence $D$ is tight, sublime or twisted at each matched pair. Conversely, suppose $D$ is tight, sublime, or twisted at each matched pair, and twisted at some matched pair $P$. Then $D_1 \cdots D_n$ is nonzero and crossingless at each matched pair, and twisted at $P$. Thus $D_1 \cdots D_n$ is nonzero and twisted, so $D$ is twisted.

If $D$ is sublime then $D_1 \cdots D_n$ is nonzero and tight, but some $D_i$ is not tight. Since the product $D_1 \cdots D_n$ is tight, it is tight at each matched pair (lemma \ref{lem:tightness_degeneracy}), so each $D_P$ is tight or sublime. But if all $D_P$ are tight, then all $(D_i)_P$ are tight, so all $D_i$ are tight. Hence at least one $D_P$ must be sublime. Conversely, suppose $D$ is tight or sublime at all matched pairs, and sublime at some matched pair. Then $D_1 \cdots D_n$ is tight at each pair, hence tight. However if all $D_i$ are tight, then all $(D_i)_P$ are tight, so all $D_P$ are tight. Thus some $D_i$ is not tight, so $D$ is sublime.

If $D$ is tight then the product $D_1 \cdots D_n$ is tight, and each $D_i$ is tight. Hence $D_1 \cdots D_n$ and each $D_i$ are tight at each matched pair. Thus $D$ is tight at each matched pair (lemma \ref{lem:tightness_degeneracy}). Conversely, if $D$ is tight at each matched pair then $D_1 \cdots D_n$ is tight at each matched pair, as is each $D_i$. Thus $D_1 \cdots D_n$ is tight, as are the $D_i$. So $D$ is tight.
\end{proof}

\subsection{Sub-tensor-products, extension and contraction}
\label{sec:sub_extension_contraction}

It is useful to consider the following notions regarding tensor products.

\begin{defn} 
Let $D_1, \cdots, D_n$ be diagrams, with homology classes $M_1, \cdots, M_n$. Consider the tensor products $D = D_1 \otimes \cdots \otimes D_n \in \A^{\otimes n}$ and $M = M_1 \otimes \cdots \otimes M_n \in \HH^{\otimes n}$.
\begin{enumerate}
\item
A \emph{sub-tensor-product} of $D$ is a tensor product $D' = D_i \otimes D_{i+1} \otimes \cdots \otimes D_{j-1} \otimes D_j$, where $1 \leq i \leq j \leq n$. 
\item
A \emph{sub-tensor-product} of $M$ is a tensor product $M' = M_i \otimes M_{i+1} \otimes \cdots \otimes M_{j-1} \otimes M_j$, where $1 \leq i \leq j \leq n$.
\end{enumerate}
\end{defn}
Clearly if $D$ (resp. $M$) is a viable tensor product, then any sub-tensor-product $D'$ (resp. $M'$) is also viable.

Now a diagram is an idempotent iff all its strands are horizontal. Idempotents can be inserted into a tensor product of strand diagrams to ``extend" it, as in the following straightforward statement, which also gives a method to ``contract" it.
\begin{lem}[Extending and contracting tensor products]
\label{lem:extending_contracting}
Let $D = D_1 \otimes \cdots \otimes D_n$ be a viable tensor product of diagrams. Let $M_i$ be the homology class of $D_i$, and let $M = M_1 \otimes \cdots \otimes M_n$. Let $D_i^*$ be the unique idempotent diagram consisting of dotted horizontal strands at all places of $t_i = s_{i+1}$, and let $M_i^*$ be its homology class.
\begin{enumerate}
\item
\begin{enumerate}
\item
The tensor product $D' = D_1 \otimes \cdots \otimes D_i \otimes D_i^* \otimes D_{i+1} \otimes \cdots \otimes D_n$ is also viable.
\item
Suppose that for some $1 \leq i \leq j \leq n$, the product $D_i D_{i+1} \cdots D_j$ is nonzero. Then $D'' = D_1 \otimes \cdots \otimes D_{i-1} \otimes (D_i \cdots D_j) \otimes D_{j+1} \otimes \cdots \otimes D_n$ is also viable.
\end{enumerate}
\item
\begin{enumerate}
\item
Suppose all $M_i$ are nonzero. Then $M' = M_1 \otimes \cdots M_i \otimes M_i^* \otimes M_{i+1} \otimes \cdots \otimes M_n$ is also viable.
\item
Suppose that for some $1 \leq i \leq j \leq n$, the product $M_i M_{i+1} \cdots M_j$ is nonzero. Then $M'' = M_1 \otimes \cdots \otimes M_{i-1} \otimes (M_i \cdots M_j) \otimes M_{j+1} \otimes \cdots \otimes M_n$ is also viable.
\end{enumerate}
\end{enumerate}
\qed
\end{lem}

\begin{defn} \
\label{def:extending_contracting}
\begin{enumerate}
\item
In lemma \ref{lem:extending_contracting}(i), we say $D'$ is obtained from $D$ by \emph{extension} by $D^*$, and $D''$ is obtained from $D$ by \emph{contraction} of $D_i \otimes \cdots \otimes D_j$. 
\item
In lemma \ref{lem:extending_contracting}(ii), we say $M'$ is obtained from $M$ by \emph{extension} by $M^*$, and $M''$ is obtained from $M$ by \emph{contraction} of $M_i \otimes \cdots \otimes M_j$.
\item
We say one tensor product (of diagrams, or homology classes) is obtained from another by \emph{extension-contraction}, if it is obtained by some sequence of extensions and contractions.
\end{enumerate}
\end{defn}
Observe that extension and contraction of a tensor product preserve H-data and Maslov grading.

%Note that both lemma \ref{lem:extending_contracting} and definition \ref{def:extending_contracting} apply to tensor products of diagrams and their homology classes in general (not just local diagrams). 

%Thus, starting from the viable tensor product of diagrams $D = D_1 \otimes \cdots \otimes D_n$, we may repeatedly extend by idempotents between any $D_i$ and $D_{i+1}$, and repeatedly contract any $D_i \otimes D_{i+1} \otimes \cdots \otimes D_j$ when the product $D_i D_{i+1} \cdots D_j$ is nonzero; the resulting tensor product, obtained by extension-contraction, is still viable. 

Figure \ref{fig:extension-contraction} shows an extension-contraction of figure \ref{fig:sesqui_tensor_product}.
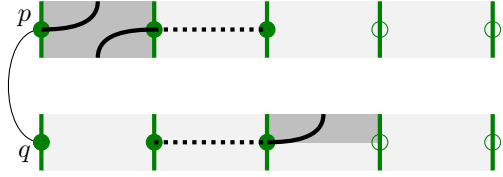
\begin{figure}
\begin{center}
\begin{tikzpicture}[scale=1.5]
\strandbackgroundshading
\tstrandbackgroundshading{1}
\tstrandbackgroundshading{2}
\tstrandbackgroundshading{3}
\afterwused
\beforewused
\taftervused{2}
\strandsetup
\tstrandsetup{1}
\tstrandsetup{2}
\tstrandsetup{3}
\tstrandsetup{4}
\lefton
\righton
\trighton{1}
\trightoff{2}
\trightoff{3}
\usea
\useb
\tdothorizontals{1}
\tusec{2}
\end{tikzpicture}
\caption{An extension-contraction of figure \ref{fig:sesqui_tensor_product}.}
\label{fig:extension-contraction}
\end{center}
\end{figure}

Note that extensions by idempotents may be reversed by contraction, and contractions of idempotents may be reversed by extension. But a contraction involving more than one factor with non-horizontal strands cannot be reversed by extension; hence the following definition.
\begin{defn}
\label{def:trivial_contraction}
If two or more of $D_i, D_{i+1}, \ldots, D_j$ (resp. $M_i, M_{i+1}, \ldots, M_j$) contain non-horizontal strands (i.e. are not idempotents), then contraction of $D_i \otimes \cdots \otimes D_j$ in $D$ (resp. of $M_i \otimes \cdots \otimes M_j$ in $M$) is \emph{nontrivial}. Otherwise, the contraction is \emph{trivial}.
\end{defn}
Thus, in a trivial contraction, either all of $D_i, \ldots, D_j$ are idempotents, as is their product; or precisely one diagram $D_k$ among $D_i, \ldots, D_j$ has non-horizontal strands, in which case $D_i \cdots D_j = D_k$.

\subsection{Sublime and singular tensor products}
\label{sec:sublime_singular}

%For the next few sections, we focus on diagrams, rather than their homology classes; we take up homology classes of diagrams again in section \ref{sec:tensor_product_homology}.

We now collect a couple of useful facts about sublime and singular tensor products.

\begin{lem}[Sublime contains crossed]
\label{lem:sublime_contains_crossed}
If the viable tensor product $D= D_1 \otimes \cdots \otimes D_n$ is sublime, then some $D_i$ is 11 once occupied crossed at some matched pair.
\end{lem}
In other words, any sublimation arises by the multiplication of a crossed diagram by another diagram (or diagrams) to undo the crossing, eventually arriving at a tight diagram, as in figure \ref{fig:sublimation}.

\begin{proof}
Assume for contradiction that all $D_i$ are crossingless. Since all the $D_i$, as well as $D_1 \cdots D_n$ are crossingless ($D$ sublime implies $D_1 \cdots D_n$ is tight), all these diagrams have homology classes. In homology, $D_1 \cdots D_n$ is nonzero, (definition \ref{def:tight_twisted_crossed_diagrams}), hence so are all $D_i$. Thus all $D_i$ are tight, contradicting $D$ being sublime; hence  some $D_i$ is crossed at some pair $P$. From table \ref{tbl:local_diagrams}, $P$ is 11 once occupied or 11 doubly occupied by $D_i$. But in the latter case, all steps at $P$ are occupied by $D_i$, so by viability all other $D_j$ are idempotent at $P$, so $D_1 \cdots D_n$ cannot be tight, contradicting $D$ being sublime.
\end{proof}

A singular tensor product $D = D_1 \otimes \cdots \otimes D_n$ is rather pathological: its H-data, even though arising from a viable tensor product of diagrams, is not the H-data of any single diagram. 
%It may not be clear that singular tensor products exist. 
Lemma \ref{lem:tensor_product_local_tightness} says $D$ is singular iff some $D_P$ is singular. As we now show, $D_P$ must take a very specific form.
\begin{lem}[Structure of singular tensor products]
\label{lem:singular_structure}
Let $P = \{p,q\}$ be a matched pair, and let $D = D_1 \otimes \cdots \otimes D_n$ be a viable singular tensor product of local diagrams on $\ZZ_P$. Then $P$ is 00 alternately occupied by $D$. Precisely two factors $D_i, D_j$ (with $i < j$) contain non-horizontal strands; both are tight. All other $D_k$ are idempotents. After possibly relabelling $p$ and $q$, $P$ is pre-one-half-occupied at $p$ by $D_i$, and post-one-half-occupied at $q$ by $D_j$.
\end{lem}
In other words, any singular local tensor product is an extension (definition \ref{def:extending_contracting}) of figure \ref{fig:singular_example}. For the diagram $D_1 \otimes D_2$ of the figure, $D_1 D_2 = 0$ and there is no diagram with its H-data: there is no such thing as a 00 alternately occupied pair in a strand diagram, as proposition \ref{prop:viable_diagram_classification} and table \ref{tbl:local_diagrams} make clear.

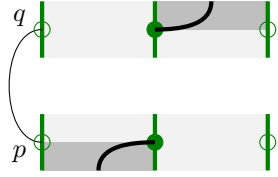
\begin{figure}
\begin{center}
\begin{tikzpicture}[scale=1.5]
\strandbackgroundshading
\tstrandbackgroundshading{1}
\beforevused
\tafterwused{1}
\strandsetupn{p}{q}
\tstrandsetup{1}
\tstrandsetup{2}
\leftoff
\righton
\trightoff{1}
\used
\tusea{1}
\end{tikzpicture}
\caption{A singular tensor product. By lemma \ref{lem:singular_structure}, up to extension, all local singular tensor products are of this form.}
\label{fig:singular_example}
\end{center}
\end{figure}

\begin{proof}
Let $D$ have H-data $(h,s,t)$. For a viable diagram $D_0$ on $\ZZ_P$, we observe that if $D_0$ covers an even number of the 4 steps of $\ZZ_P$, then $P$ is $00$ or $11$; and if $D_0$ covers an odd number of these steps, then $P$ is $01$ or $10$. This follows, for instance, from proposition \ref{prop:viable_diagram_classification}, from inspecting table \ref{tbl:local_diagrams}, or by analysing directly how strands must proceed in such a diagram. Applying this observation to each $D_i$ in $D$, we see that if $h$ covers an even number of steps of $\ZZ_P$, then $P$ is $00$ or $11$; and if $h$ covers an odd number of steps of $\ZZ_P$, then $P$ is $01$ or $10$.

Moreover, all steps covered by $h$ cannot be covered by a single $D_i$. For then all other $D_j$ are idempotents, so $D_i$ has the H-data $(h,s,t)$ of $D$, contradicting $D$ being singular. In particular, $h$ must cover at least two steps of $P$.

If $h$ covers 2 steps, then all possible H-data $(h,s,t)$ satisfying the conditions above appear in table \ref{tbl:local_diagrams}, hence are non-singular, except if $P$ is 00 alternately occupied. In this case, the 2 steps covered must be covered by two diagrams $D_i, D_j$, where $i<j$; all other diagrams must be idempotents. Then $P$ must be 01 one-half-occupied by $D_i$, hence pre-one-half-occupied, and 10 one-half-occupied by $D_j$, hence post-one-half-occupied. Thus $D$ has the structure claimed.

If $h$ covers 3 steps, then the only possible H-data not appearing in table \ref{tbl:local_diagrams} are where $P$ is 10 pre-sesqui-occupied or 01 post-sesqui-occupied. We consider the first case; the second is similar. Without loss of generality suppose $p$ is pre-half-occupied and $q$ is fully occupied. If the 3 steps are covered by two diagrams $D_i, D_j$, where $D_i$ covers one step and $D_j$ covers two steps, then $P$ is one-half-occupied by $D_i$. Moreover, by our initial observation, $D_j$ is 00 or 11, so by viability $D_i$ must be 10, hence $P$ is post-one-half-occupied by $D_i$. Thus both $p$ and $q$ are pre-half-occupied by $D_j$, but there is no diagram which does so. If the three steps are covered by three diagrams, then $P$ is pre-one-half-occupied by two diagrams (which must be 01) and post-one-half-occupied by one diagram (which must be 10), and all other diagrams are idempotents. But there is no way to combine the idempotent data 01, 01, 10 of these three diagrams viably so that $P$ is 10 in the tensor product. Hence no such $D$ exists.

If $h$ covers all 4 steps, all possible H-data already appear in table \ref{tbl:local_diagrams} so $D$ cannot be singular.
\end{proof}

\subsection{Enumeration of local tensor products}
\label{sec:enumerate_local_tensor_products}

We now enumerate all viable local tensor products of diagrams. So let $D = D_1 \otimes \cdots \otimes D_n$ be a viable tensor product of diagrams on $\ZZ_P$, where $P = \{p,q\}$. %As already seen in the proof of lemma \ref{lem:tightness_local_factors}, in any viable $D = D_1 \otimes \cdots \otimes D_n$, there are only 4 steps to cover at $P$. 
Viability implies that each of the 4 steps of $\ZZ_P$ is covered at most once. Hence at most $4$ of $D_1, \ldots, D_n$ can contain non-horizontal strands; the rest are idempotents. 

Enumerating the possible such $D$ is assisted by the following lemma, describing the tightness of the $D_i$.
\begin{lem}[Tightness of local sub-tensor-products and tensor factors]
\label{lem:tightness_local_factors}
Let $D$ be a viable tensor product of local diagrams on $\ZZ_P$.
\begin{enumerate}
\item If $D$ is tight, then each $D_i$ is tight.
\item If $D$ is sublime, then one $D_i$ is crossed 11 once occupied, and for the remaining factors $D_j$:
\begin{enumerate}
\item one $D_j$ is twisted, and all other $D_j$ are idempotents; or
\item one or two $D_j$ are tight with non-horizontal strands, and all other $D_j$ are idempotents.
\end{enumerate}
\item If $D$ is twisted, then either:
\begin{enumerate}
\item precisely two $D_i$ are tight, and all other $D_j$ are idempotents; or
\item precisely one $D_i$ is twisted, and all other $D_j$ are idempotents.
\end{enumerate}
\item If $D$ is crossed, then precisely one or two $D_i$ are crossed, and all other factors $D_j$ are idempotents.
\item If $D$ is critical, then no $D_i$ are crossed, $W$ of the $D_i$ are twisted, $G$ of the $D_i$ are tight with non-horizontal strands, and all other $D_i$ are idempotents, where $2W+G \leq 4$ and $W+G \geq 2$.
\item If $D$ is singular, then precisely two $D_i$ are tight, and all other $D_j$ are idempotents.
\end{enumerate}
\end{lem}
%Note that a diagram is an idempotent iff all its strands are (dotted) horizontal.
In the critical case, the inequalities on $W$ and $G$ imply $(W,G) \in \{(2,0), (1,1), (1,2), (0,2),(0,3), (0,4)\}$. (In fact the case $(W,G) = (0,2)$ never arises; such tensor products turn out to be singular.)

\begin{proof}
Part (i) follows by definition \ref{def:tensor_product_tightness}.

If $D$ is sublime then by lemma \ref{lem:sublime_contains_crossed} some $D_i$ is crossed 11 once occupied. Each crossed $D_i$ thus covers exactly $2$ of the $4$ steps in $\ZZ_P$, so there are at most two crossed $D_i$. If two factors $D_i, D_j$ are crossed, then by viability $P$ all other factors are idempotents and $D_1 \cdots D_n$ is crossed, contradicting $D$ being sublime. So there is one crossed diagram $D_i$. Only 2 steps of $\ZZ_P$ are not covered by $D_i$, and so there are at most 2 other factors $D_j$ with non-horizontal strands, which are tight or twisted. A twisted $D_j$ would cover both the remaining steps, the only possibilities are (a) and (b) as claimed.

If $D$ is twisted then (definition \ref{def:tensor_product_tightness}) $D_1 \cdots D_n$ is twisted, hence (table \ref{tbl:local_diagrams}) only two steps of $\ZZ_P$ are covered. Thus at most $2$ of the $D_i$ are not idempotents. If one $D_i$ is non-idempotent, then $D_i$ is twisted. If two $D_i$ are non-idempotent, then each must cover one step, hence both are tight.

If $D_1 \cdots D_n$ is crossed, then at least one $D_i$ is crossed (lemma \ref{lem:products_crossings}); and by viability, as any crossed diagram covers at least two steps, there are at most two crossed $D_i$. If there are two crossed factors, then they cover all steps, so all other factors are idempotents. If only one $D_i$ is crossed, we observe that any viable multiplication of $D_i$ with any tight or twisted diagram results in a tight diagram, so all other factors must be idempotents.

For (v), we first claim no $D_i$ is crossed. As noted above at most two $D_i$ are crossed. If two $D_i$ are crossed then all other factors are idempotents, so that $D_1 \cdots D_n$ is nonzero crossed; if one $D_i$ is crossed, then any viable product of $D_i$ with a tight or twisted diagram is nonzero; either way contradicting criticality of $D$, proving the claim. Hence each non-idempotent $D_i$ is twisted or tight. Each twisted factor covers exactly 2 steps (table \ref{tbl:local_diagrams}); each tight factor covers at least 1 step. These factors altogether cover $\geq 2W+G$ steps. Since $\ZZ_P$ has 4 steps, $2W+G \leq 4$.  On the other hand $W+G \geq 2$ since there must be at least 2 non-idempotent factors; otherwise the single non-idempotent $D_i = D_1 \cdots D_n \neq 0$, contradicting criticality.

Lemma \ref{lem:singular_structure} gives the final part.
\end{proof}

Using the structure provided by lemma \ref{lem:tightness_local_factors}, or otherwise, we can enumerate viable tensor products of diagrams on $\ZZ_P$ and obtain the following generalisation of proposition \ref{prop:viable_diagram_classification}.
\begin{prop}[Classification of viable local tensor products]
\label{prop:viable_tensor_product_classification}
Let $D = D_1 \otimes \cdots \otimes D_n$ be a viable tensor product of diagrams on $\ZZ_P$. Then $D$ is an extension-contraction of a diagram shown in table \ref{tbl:local_tensor_products}, and its H-data and tightness are as shown.
\qed
\end{prop}

%Consider $D$ at the pair $P$.
%\begin{enumerate}
%\item If $P$ is unoccupied, then $P$ is 00 or 11, and $D$ is tight at $P$.
%\item If $P$ is one-half-occupied, then $P$ is either 01 or 10, and $D$ is tight at $P$.
%\begin{enumerate}
%\item If $P$ is pre-one-half-occupied, then $P$ is 01.
%\item If $P$ is post-one-half-occupied, then $P$ is 10.
%\end{enumerate}
%\item If $P$ is alternately occupied, then $P$ is 00 or 11.
%\begin{enumerate}
%\item If $P$ is 00, then $D$ is singular at $P$.
%\item If $P$ is 11, then $D$ is tight at $P$.
%\end{enumerate}
%\item If $P$ is once occupied, then $P$ is 00 or 11.
%\begin{enumerate}
%\item If $P$ is 00, then $D$ is tight at $P$.
%\item If $P$ is 11, then $D$ is twisted or crossed at $P$.
%\end{enumerate}
%\item If $P$ is sesqui-occupied, then $P$ is 01 or 10.
%\begin{enumerate}
%\item If $P$ is pre-sesqui-occupied, then $P$ is 01, and $D$ is tight, sublime, or critical.
%\item If $P$ is post-sesqui-ocupied, then $P$ is 10, and $D$ is tight, sublime, or critical.
%\end{enumerate}
%\item If $P$ is doubly occupied, then $P$ is 00 or 11.
%\begin{enumerate}
%\item If $P$ is 00, then $D$ is tight, sublime, or critical.
%\item If $P$ is 11, then $D$ is tight, sublime, crossed, or critical.
%\end{enumerate}
%\end{enumerate}
%In all cases, $D_P$ is an extension-contraction of a diagram shown in table \ref{tbl:local_tensor_products}.
%\qed
%\end{prop}
Note that in this proposition, $D$ may be an extension-contraction of more than one of the possibilities: a contraction of a sublime tensor product may coincide with the contraction of a tight tensor product.

Table \ref{tbl:local_tensor_products} also shows Maslov gradings with each local tensor product. As mentioned in section \ref{sec:sub_extension_contraction}, Maslov grading is preserved under extension and contraction. Observe that, for any given viable H-data, if there is a critical tensor product, then there is also a tight tensor product, and the Maslov grading of the latter is $1$ greater than the former.

%\begin{tabular}{cc|c|c|c|c|c}
%$H$-data & & Tight & Resolutive & Crossed & Overtwisted & Critical \\
%\hline 
%Unoccupied 
% & $00$ & X & & & & \\
% & $11$ & X & & & & \\
%One-half-occupied 
% & Pre- & X & & & &\\
% & Post- & X & & & &\\
%Alternately occupied & ($11$) & X & & & & \\
%Once occupied 
% & $00$ & X & & &  & \\
% & $11$ & & & X & X & \\
%Sesqui-occupied 
% & Pre- & X & X & & & X \\
% & Post- & X & X & & & X \\
%Doubly occupied 
% & $00$ & X & X & & & X \\
% & $11$ & X & X & X & & X \\
%\end{tabular}

\begin{table}
\begin{center}
\begin{minipage}{1\textwidth}

\begin{tabular}{>{\centering\arraybackslash}m{2.2cm}|>{\centering\arraybackslash}m{2.1cm}|>{\centering\arraybackslash}m{2.3cm}|>{\centering\arraybackslash}m{1.1cm}|>{\centering\arraybackslash}m{1.1cm}|>{\centering\arraybackslash}m{2.1cm}|>{\centering\arraybackslash}m{1.1cm}} 
%\begin{tabular}{c|c|c|c|c|c}
$H$-data & Tight & Sublime & Twisted & Crossed & Critical & Singular \\
\hline 
\begin{tabular}{c} Unoccupied  \\ $00$ \end{tabular}
& 
\begin{tikzpicture}[xscale=0.48, yscale=0.8]
\strandbackgroundshading
\strandsetupn{}{}
\leftoff
\rightoff
\draw (1.5,0.75) node {$0$};
\end{tikzpicture}
 & & & & & \\
\hline
\begin{tabular}{c} Unoccupied  \\ $11$ \end{tabular}
 & 
\begin{tikzpicture}[xscale=0.48, yscale=0.8]
\strandbackgroundshading
\strandsetupn{}{}
\lefton
\righton
\dothorizontals
\draw (1.5,0.75) node {$0$};
\end{tikzpicture}
& & & & & \\
\hline
\begin{tabular}{c} Pre-one-half- \\ occupied 01 \end{tabular}
 & 
\begin{tikzpicture}[xscale=0.48, yscale=0.8]
\strandbackgroundshading
\beforewused
\strandsetupn{}{}
\leftoff
\righton
\useb
\draw (1.5,0.75) node {$0$};
\end{tikzpicture}
& & & & &\\
\hline
\begin{tabular}{c} Post-one-half- \\ occupied 10 \end{tabular}
 & 
\begin{tikzpicture}[xscale=0.48, yscale=0.8]
\strandbackgroundshading
\afterwused
\strandsetupn{}{}
\lefton
\rightoff
\usea
\draw (1.5,0.75) node {$-\frac{1}{2}$};
\end{tikzpicture}
& & & & &\\
\hline
\begin{tabular}{c} Alternately \\ occupied \\ 00 \end{tabular}
& & & & & &
\begin{tikzpicture}[xscale=0.48, yscale=0.8]
\strandbackgroundshading
\tstrandbackgroundshading{1}
\beforevused
\tafterwused{1}
\strandsetupn{}{}
\tstrandsetup{1} 
\tstrandsetup{2}
\leftoff
\righton
\trightoff{1}
\used
\tusea{1}
\draw (2.5,0.75) node {$-\frac{1}{2}$};
%(0.75,-0.3) 
\end{tikzpicture}
\\
\hline
\begin{tabular}{c} Alternately \\ occupied \\ 11 \end{tabular}
& 
\begin{tikzpicture}[xscale=0.48, yscale=0.8]
\strandbackgroundshading
\tstrandbackgroundshading{1}
\afterwused
\tbeforevused{1}
\strandsetupn{}{}
\tstrandsetup{1}
\tstrandsetup{2}
\lefton
\rightoff
\trighton{1}
\usea
\tused{1}
\draw (2.5,0.75) node {$-\frac{1}{2}$};
\end{tikzpicture}
%-1/2 + 0 + 0 = -1/2
& & & & & \\
\hline
\begin{tabular}{c} Once \\ occupied \\ $00$ \end{tabular}
 & 
\begin{tikzpicture}[xscale=0.48, yscale=0.8]
\strandbackgroundshading
\tstrandbackgroundshading{1}
\beforewused
\tafterwused{1}
\strandsetupn{}{}
\tstrandsetup{1}
\tstrandsetup{2}
\leftoff
\righton
\trightoff{1}
\useb
\tusea{1}
\draw (2.5,0.75) node {$0$};
\end{tikzpicture}
%0 + 1/2 -1/2
 & & & & & \\
\hline
\begin{tabular}{c} Once \\ occupied \\ $11$ \end{tabular}
  & & & 
\begin{tikzpicture}[xscale=0.48, yscale=0.8]
\strandbackgroundshading
\tstrandbackgroundshading{1}
\afterwused
\tbeforewused{1}
\strandsetupn{}{}
\tstrandsetup{1}
\tstrandsetup{2}
\lefton
\rightoff
\trighton{1}
\usea
\tuseb{1}
\draw (0.75,-0.3) node {$-1$};
\end{tikzpicture}
% -1/2 -1/2 + 0
& 
\begin{tikzpicture}[xscale=0.48, yscale=0.8]
\strandbackgroundshading
\beforewused
\afterwused
\strandsetupn{}{}
\lefton
\righton
\useab
\dothorizontals
\draw (0.5,-0.3) node {$0$};
\end{tikzpicture}
& & \\
\hline
\begin{tabular}{c} Pre-sesqui- \\ occupied 01 \end{tabular}
 & 
\begin{tikzpicture}[xscale=0.48, yscale=0.8]
\strandbackgroundshading
\tstrandbackgroundshading{1}
\tstrandbackgroundshading{2}
\beforewused
\tafterwused{1}
\tbeforevused{2}
\strandsetupn{}{}
\tstrandsetup{1}
\tstrandsetup{2}
\tstrandsetup{3}
\leftoff
\righton
\trightoff{1}
\trighton{2}
\useb
\tusea{1}
\tused{2}
\draw (1.5,-0.3) node {$0$};
\end{tikzpicture}
% 0 - 1/2 + 0
% + 1/2 + 0
% + 0
& 
\begin{tikzpicture}[xscale=0.48, yscale=0.8]
\strandbackgroundshading
\tstrandbackgroundshading{1}
\beforevused
\tbeforewused{1}
\tafterwused{1}
\strandsetupn{}{}
\tstrandsetup{1}
\tstrandsetup{2}
\leftoff
\righton
\trighton{1}
\used
\tuseab{1}
\tdothorizontals{1}
\draw (1,-0.3) node {$0$};
\end{tikzpicture}
% 0 + 0 + 0
& & & 
\begin{tikzpicture}[xscale=0.48, yscale=0.8]
\strandbackgroundshading
\tstrandbackgroundshading{1}
\tstrandbackgroundshading{2}
\beforevused
\tafterwused{1}
\tbeforewused{2}
\strandsetupn{}{}
\tstrandsetup{1}
\tstrandsetup{2}
\tstrandsetup{3}
\leftoff
\righton
\trightoff{1}
\trighton{2}
\used
\tusea{1}
\tuseb{2}
\draw (1.5,-0.3) node {$-1$};
\end{tikzpicture}
% 0 - 1/2 + 0
% + 0 - 1/2
% + 0
& \\
\hline
\begin{tabular}{c} Post-sesqui- \\ occupied 10 \end{tabular}
& 
\begin{tikzpicture}[xscale=0.48, yscale=0.8]
\strandbackgroundshading
\tstrandbackgroundshading{1}
\tstrandbackgroundshading{2}
\aftervused
\tbeforewused{1}
\tafterwused{2}
\strandsetupn{}{}
\tstrandsetup{1}
\tstrandsetup{2}
\tstrandsetup{3}
\lefton
\rightoff
\trighton{1}
\trightoff{2}
\usec
\tuseb{1}
\tusea{2}
\draw (1.5,-0.3) node {$-\frac{1}{2}$};
\end{tikzpicture}
% -1/2 + 0 - 1/2
% + 0 + 1/2
% + 0
 & 
\begin{tikzpicture}[xscale=0.48, yscale=0.8]
\strandbackgroundshading
\tstrandbackgroundshading{1}
\beforewused
\afterwused
\taftervused{1}
\strandsetupn{}{}
\tstrandsetup{1}
\tstrandsetup{2}
\lefton
\righton
\trightoff{1}
\useab
\dothorizontals
\tusec{1}
\draw (1,-0.3) node {$-\frac{1}{2}$};
\end{tikzpicture}
% 0 + 0 - 1/2
 & & & 
\begin{tikzpicture}[xscale=0.48, yscale=0.8]
\strandbackgroundshading
\tstrandbackgroundshading{1}
\tstrandbackgroundshading{2}
\afterwused
\tbeforewused{1}
\taftervused{2}
\strandsetupn{}{}
\tstrandsetup{1}
\tstrandsetup{2}
\tstrandsetup{3}
\lefton
\rightoff
\trighton{1}
\trightoff{2}
\usea
\tuseb{1}
\tusec{2}
\draw (1.5,-0.3) node {$-\frac{3}{2}$};
\end{tikzpicture}
% -1/2 + 0 - 1/2
% -1/2 + 0
% + 0
 & \\
\hline
\begin{tabular}{c} Doubly \\ occupied \\ $00$ \end{tabular}
 & 
\begin{tikzpicture}[xscale=0.48, yscale=0.8]
\strandbackgroundshading
\tstrandbackgroundshading{1}
\tstrandbackgroundshading{2}
\tstrandbackgroundshading{3}
\beforewused
\tafterwused{1}
\tbeforevused{2}
\taftervused{3}
\strandsetupn{}{}
\tstrandsetup{1}
\tstrandsetup{2}
\tstrandsetup{3}
\tstrandsetup{4}
\leftoff
\righton
\trightoff{1}
\trighton{2}
\trightoff{3}
\useb
\tusea{1}
\tused{2}
\tusec{3}
\draw (2,-0.3) node {$0$};
\end{tikzpicture}
% 0 - 1/2 + 0 - 1/2
% + 1/2 + 0 + 1/2
% + 0 + 0
% + 0
 & 
\begin{tikzpicture}[xscale=0.48, yscale=0.8]
\strandbackgroundshading
\tstrandbackgroundshading{1}
\tstrandbackgroundshading{2}
\beforewused
\tbeforevused{1}
\taftervused{1}
\tafterwused{2}
\strandsetupn{}{}
\tstrandsetup{1}
\tstrandsetup{2}
\tstrandsetup{3}
\leftoff
\righton
\trighton{1}
\trightoff{2}
\useb
\tusecd{1}
\tdothorizontals{1}
\tusea{2}
\draw (1.5,-0.3) node {$0$};
\end{tikzpicture}
% 0 + 0 - 1/2
% + 0 + 0
% + 1/2
 & & & 
\begin{tikzpicture}[xscale=0.48, yscale=0.8]
\strandbackgroundshading
\tstrandbackgroundshading{1}
\tstrandbackgroundshading{2}
\tstrandbackgroundshading{3}
\beforewused
\taftervused{1}
\tbeforevused{2}
\tafterwused{3}
\strandsetupn{}{}
\tstrandsetup{1}
\tstrandsetup{2}
\tstrandsetup{3}
\tstrandsetup{4}
\leftoff
\righton
\trightoff{1}
\trighton{2}
\trightoff{3}
\useb
\tusec{1}
\tused{2}
\tusea{3}
\draw (2,-0.3) node {$-1$};
\end{tikzpicture}
% 0 - 1/2 + 0 - 1/2
% + 0 - 1/2 + 0
% + 0 + 0
% + 1/2
 & \\
\hline
\begin{tabular}{c} Doubly \\ occupied \\ $11$ \end{tabular}
 & 
\begin{tikzpicture}[xscale=0.48, yscale=0.8]
\strandbackgroundshading
\tstrandbackgroundshading{1}
\tstrandbackgroundshading{2}
\tstrandbackgroundshading{3}
\afterwused
\tbeforevused{1}
\taftervused{2}
\tbeforewused{3}
\strandsetupn{}{}
\tstrandsetup{1}
\tstrandsetup{2}
\tstrandsetup{3}
\tstrandsetup{4}
\lefton
\rightoff
\trighton{1}
\trightoff{2}
\trighton{3}
\usea
\tused{1}
\tusec{2}
\tuseb{3}
\draw (2,-0.3) node {$-1$};
\end{tikzpicture}
% -1/2 + 0 -1/2 + 0
% + 0 + 1/2 + 0
% + 0 + 0
% - 1/2
 & 
\begin{tabular}{c}
\begin{tikzpicture}[xscale=0.48, yscale=0.8]
\strandbackgroundshading
\tstrandbackgroundshading{1}
\tstrandbackgroundshading{2}
\beforevused
\aftervused
\tafterwused{1}
\tbeforewused{2}
\strandsetupn{}{}
\tstrandsetup{1}
\tstrandsetup{2}
\tstrandsetup{3}
\lefton
\righton
\trightoff{1}
\trighton{2}
\usecd
\dothorizontals
\tusea{1}
\tuseb{2}
\draw (3.5,0.75) node {$-1$};
\end{tikzpicture}
\\ 
\begin{tikzpicture}[xscale=0.48, yscale=0.8]
\strandbackgroundshading
\tstrandbackgroundshading{1}
\tstrandbackgroundshading{2}
\afterwused
\tbeforewused{1}
\tbeforevused{2}
\taftervused{2}
\strandsetupn{}{}
\tstrandsetup{1}
\tstrandsetup{2}
\tstrandsetup{3}
\lefton
\rightoff
\trighton{1}
\trighton{2}
\usea
\tuseb{1}
\tusecd{2}
\tdothorizontals{2}
\end{tikzpicture}
\end{tabular}
% 0 - 1/2 + 0
% + 0 - 1/2
% + 0
%$-1$
% - 1/2 + 0 + 0
% - 1/2 + 0
% + 0
& & 
\begin{tikzpicture}[xscale=0.48, yscale=0.8]
\strandbackgroundshading
\tstrandbackgroundshading{1}
\beforewused
\afterwused
\tbeforevused{1}
\taftervused{1}
\strandsetupn{}{}
\tstrandsetup{1}
\tstrandsetup{2}
\lefton
\righton
\trighton{1}
\useab
\dothorizontals
\tusecd{1}
\tdothorizontals{1}
\draw (1,-0.3) node {$0$};
\end{tikzpicture}
% 0 + 0 + 0
  & 
\begin{tikzpicture}[xscale=0.48, yscale=0.8]
\strandbackgroundshading
\tstrandbackgroundshading{1}
\tstrandbackgroundshading{2}
\tstrandbackgroundshading{3}
\afterwused
\tbeforewused{1}
\taftervused{2}
\tbeforevused{3}
\strandsetupn{}{}
\tstrandsetup{1}
\tstrandsetup{2}
\tstrandsetup{3}
\tstrandsetup{4}
\lefton
\rightoff
\trighton{1}
\trightoff{2}
\trighton{3}
\usea
\tuseb{1}
\tusec{2}
\tused{3}
\draw (2,-0.3) node {$-2$};
\end{tikzpicture}
% - 1/2 + 0 - 1/2 + 0
% - 1/2 + 0 - 1/2
% + 0 + 0
% + 0
& \\
\hline
\end{tabular}

\caption{Possible local tensor products, by H-data and tightness. Maslov gradings also shown.}
\label{tbl:local_tensor_products}
\end{minipage}
\end{center}
\end{table}

\subsection{Tightness of local and sub-tensor products}
\label{sec:tightness_local_sub}

In the sequel we need to know about the behaviour of tightness when we decompose or extend or contract tensor products. We saw in lemma \ref{lem:tensor_product_local_tightness} that when we consider tensor products of diagrams locally (i.e. decompose ``vertically"), tightness behaves in an ordered way. However, when we decompose according to the ``horizontal" tensor product, tightness is not so well behaved. 

By proposition \ref{prop:viable_tensor_product_classification}, a viable local tensor product is an extension-contraction of one shown in table \ref{tbl:local_tensor_products}. We can then enumerate the tightness of sub-tensor-products in each case, and obtain the following result.
\begin{lem}[Tightness of local sub-tensor-products]
\label{lem:local_sub-tensor-product_tightness}
Let $D$ be a viable tensor product of diagrams on $\ZZ_P$, and let $D'$ be a sub-tensor-product of $D$. Then the possible tightness types of $D$ and $D'$ are as shown in table \ref{tbl:local_sub-tensor-product_tightness}.
\qed
\end{lem}

\begin{table}
\begin{center}
\begin{tabular}{x{1.5cm}| c|c|c|c|c|c}
\diag{.1em}{1.5cm}{$D$}{$D'$} & Tight & Sublime & Twisted & Crossed & Critical & Singular  \\ \hline
Tight 	& X & 	&		&		&		& \\ \hline
Sublime & X	&	X	&	X	&	X	&		& \\ \hline
Twisted &	X	&		&	X &		&		& \\ \hline
Crossed & X	&		&		&	X &		& \\ \hline 
Critical&	X	&		&	X	&		&	X	&	X \\ \hline
Singular&	X	&		&		&		&		&	X
\end{tabular}
\caption{Possible tightness types of a viable local tensor product $D$ and a sub-tensor-product $D'$.}
\label{tbl:local_sub-tensor-product_tightness}
\end{center}
\end{table}

Thus, for instance, if $D$ is tight then $D'$ is also tight; if $D'$ is sublime then $D$ is also sublime; and if $D'$ is critical then $D$ is also critical.

We also have a similar ``global" result about the possible tightness types of a tensor product $D$ and sub-tensor product $D'$, on a general arc diagram.
\begin{lem}[Tightness of sub-tensor-products]
\label{lem:sub-tensor-product_tightness}
Let $D = D_1 \otimes \cdots \otimes D_n$ be a viable tensor product of diagrams on an arc diagram $\ZZ$, and let $D'$ be a sub-tensor-product of $D$. 
\begin{enumerate}
\item
If $D$ is tight, then $D'$ is tight.
\item
If $D'$ is critical or singular, then $D$ is critical or singular.
\end{enumerate}
Every combination of tightness types not ruled out by these implications is possible.
%Then the possible tightness types of $D$ and $D'$ are as shown in table \ref{tbl:sub-tensor-product_tightness}.
\end{lem}

%\begin{table}
%\begin{center}
%\begin{tabular}{x{1.5cm}| c|c|c|c|c|c}
%\diag{.1em}{1.5cm}{$D$}{$D'$} & Tight & Sublime & Twisted & Crossed & Critical & Singular  \\ \hline
%Tight 	& X & 	&		&		&		& \\ \hline
%Sublime & X	&	X	&	X	&	X	&		& \\ \hline
%Twisted &	X	&	X	&	X &	X	&		& \\ \hline
%Crossed & X	&	X	&	X	&	X &		& \\ \hline 
%Critical&	X	&	X	&	X	&	X	&	X	&	X \\ \hline
%Singular&	X	&	X	&	X	&	X	&	X	&	X
%\end{tabular}
%\caption{Possible tightness types of a viable tensor product $D$ and a sub-tensor-product $D'$, in general.}
%\label{tbl:sub-tensor-product_tightness}
%\end{center}
%\end{table}
	
\begin{proof}
If $D$ is tight, then $D_1 \cdots D_n$ is tight (definition \ref{def:tensor_product_tightness}), hence nonzero in homology (definition \ref{def:tight_twisted_crossed_diagrams}), hence any $D_i \cdots D_j$ is nonzero in homology, hence tight, hence $D'$ is tight.

If $D'$ is critical or singular then $D_i \cdots D_j = 0$ (definition \ref{def:tensor_product_tightness}), so $D_1 \cdots D_n = 0$, so $D$ is critical or singular.

We show some of the remaining possibilities in figures. Figures \ref{fig:sublime_tensor_example} and \ref{fig:crossed_tensor_example} show examples of sublime, twisted, crossed and critical tensor products containing many types of sub-tensor-products. The small number remaining are omitted.
\end{proof}

In figures such as \ref{fig:sublime_tensor_example}(left) we only draw the local tensor products at one matched pair; and in figures \ref{fig:twisted_tensor_example}(right) and \ref{fig:crossed_tensor_example} we only draw the local tensor products at two matched pairs. These can easily be extended to figures of tensor products of non-augmented diagrams on connected arc diagrams if desired.

\begin{figure}
\begin{center}
\begin{tikzpicture}[scale=1]
\strandbackgroundshading
\tstrandbackgroundshading{1}
\tstrandbackgroundshading{2}
\afterwused
\tbeforewused{1}
\tbeforevused{2}
\taftervused{2}
\strandsetupn{}{}
\tstrandsetup{1}
\tstrandsetup{2}
\tstrandsetup{3}
\lefton
\rightoff
\trighton{1}
\trighton{2}
\usea
\tuseb{1}
\tusecd{2}
\tdothorizontals{2}

\begin{scope}[xshift=6 cm, yshift = 1 cm]
\strandbackgroundshading
\tstrandbackgroundshading{1}
\tstrandbackgroundshading{2}
\tbeforewused{2}
\tafterwused{2}
\strandsetupn{}{}
\tstrandsetup{1}
\tstrandsetup{2}
\tstrandsetup{3}
\lefton
\righton
\trighton{1}
\trighton{2}
\dothorizontals
\tdothorizontals{1}
\tusea{2}
\tuseb{2}
\end{scope}
\begin{scope}[xshift = 6 cm, yshift=-1 cm]
\strandbackgroundshading
\tstrandbackgroundshading{1}
\tstrandbackgroundshading{2}
\beforewused
\afterwused
\taftervused{1}
\strandsetupn{}{}
\tstrandsetup{1}
\tstrandsetup{2}
\tstrandsetup{3}
\lefton
\righton
\trightoff{1}
\trightoff{2}
\dothorizontals
\useab
\tusec{1}
\end{scope}
\end{tikzpicture}

\caption{Left: A sublime tensor product $D_1 \otimes D_2 \otimes D_3$ containing tight ($D_1, D_2$), sublime ($D_2 \otimes D_3$, $D_1 \otimes D_2 \otimes D_3$), twisted ($D_1 \otimes D_2$) and crossed ($D_3$) sub-tensor-products. 
Right: A twisted tensor product $D_1 \otimes D_2 \otimes D_3$ containing tight ($D_2$), sublime ($D_1 \otimes D_2$), twisted ($D_3$, $D_2 \otimes D_3$, $D_1 \otimes D_2 \otimes D_3$) and crossed ($D_1$) sub-tensor-products.}
\label{fig:sublime_tensor_example}
\label{fig:twisted_tensor_example}
\end{center}
\end{figure}
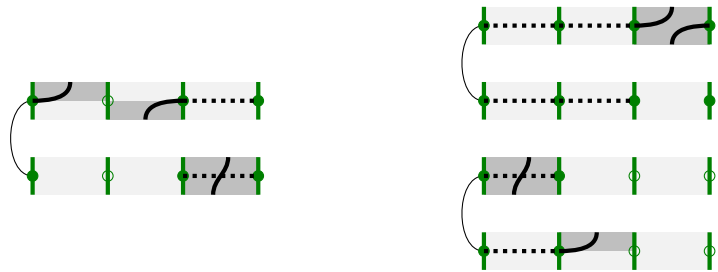

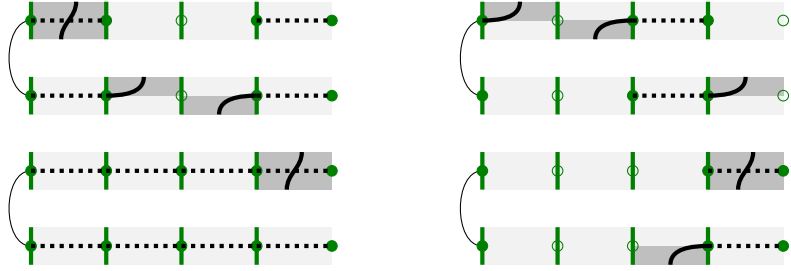
\begin{figure}
\begin{center}
\begin{tikzpicture}[scale=1]
\strandbackgroundshading
\tstrandbackgroundshading{1}
\tstrandbackgroundshading{2}
\tstrandbackgroundshading{3}
\beforewused
\afterwused
\taftervused{1}
\tbeforevused{2}
\strandsetupn{}{}
\tstrandsetup{1}
\tstrandsetup{2}
\tstrandsetup{3}
\lefton
\righton
\trightoff{1}
\trighton{2}
\trighton{3}
\dothorizontals
\useab
\tusec{1}
\tused{2}
\tdothorizontals{3}
\begin{scope}[yshift=-2 cm]
\strandbackgroundshading
\tstrandbackgroundshading{1}
\tstrandbackgroundshading{2}
\tstrandbackgroundshading{3}
\tafterwused{3}
\tbeforewused{3}
\strandsetupn{}{}
\tstrandsetup{1}
\tstrandsetup{2}
\tstrandsetup{3}
\lefton
\righton
\trighton{1}
\trighton{2}
\trighton{3}
\dothorizontals
\tdothorizontals{1}
\tdothorizontals{2}
\tdothorizontals{3}
\tuseab{3}
\end{scope}

\begin{scope}[xshift = 6 cm]
\strandbackgroundshading
\tstrandbackgroundshading{1}
\tstrandbackgroundshading{2}
\tstrandbackgroundshading{3}
\afterwused
\tbeforewused{1}
\taftervused{3}
\strandsetupn{}{}
\tstrandsetup{1}
\tstrandsetup{2}
\tstrandsetup{3}
\lefton
\rightoff
\trighton{1}
\trighton{2}
\trightoff{3}
\usea
\tuseb{1}
\tdothorizontals{2}
\tusec{3}
\end{scope}
\begin{scope}[xshift = 6 cm, yshift=-2 cm]
\strandbackgroundshading
\tstrandbackgroundshading{1}
\tstrandbackgroundshading{2}
\tstrandbackgroundshading{3}
\tbeforevused{2}
\tafterwused{3}
\tbeforewused{3}
\strandsetupn{}{}
\tstrandsetup{1}
\tstrandsetup{2}
\tstrandsetup{3}
\lefton
\rightoff
\trightoff{1}
\trighton{2}
\trighton{3}
\tused{2}
\tdothorizontals{3}
\tuseab{3}
\end{scope}
\end{tikzpicture}

\caption{Left: A crossed tensor product $D_1 \otimes D_2 \otimes D_3 \otimes D_4$ containing tight ($D_2$, $D_3$), sublime ($D_1 \otimes D_2$, $D_1 \otimes D_2 \otimes D_3$), twisted ($D_2 \otimes D_3$) and crossed ($D_1$, $D_4$, $D_3 \otimes D_4$, $D_2 \otimes D_3 \otimes D_4$, $D_1 \otimes D_2 \otimes D_3 \otimes D_4$) sub-tensor-products.
Right: A critical tensor product $D_1 \otimes D_2 \otimes D_3 \otimes D_4$ containing tight ($D_1, D_2, D_3, D_2 \otimes D_3$), sublime ($D_3 \otimes D_4$), twisted ($D_1 \otimes D_2$, $D_1 \otimes D_2 \otimes D_3$), crossed ($D_4$), critical ($D_1 \otimes D_2 \otimes D_3 \otimes D_4$) and singular ($D_2 \otimes D_3 \otimes D_4$) sub-tensor-products.}
\label{fig:crossed_tensor_example}
\end{center}
\end{figure}

On the other hand, extension-contraction preserves most, but not all, types of tightness. The only subtlety is sublimation: sublime tensor products may become tight.
\begin{lem}
\label{lem:tightness_extension-contraction}
Suppose $D'$ is obtained from $D = D_1 \otimes \cdots \otimes D_n$ by extension-contraction. Then $D$ and $D'$ have the same tightness, or $D$ is sublime and $D'$ is tight.
\end{lem}

\begin{proof}
Under extension or contraction, the product $D_1 \cdots D_n$ of a tensor product $D_1 \otimes \cdots \otimes D_n$ remains invariant, as does H-data.

Singularity of a tensor product is defined by reference only to H-data; hence hence $D$ is singular iff $D'$ is singular. 

Now assume $D,D'$ are not singular. The tightness properties ``tight or sublime", ``twisted", ``crossed" and ``critical" of $D$ are defined by the properties of the product $D_1 \cdots D_n$ (i.e. whether $D_1 \cdots D_n$ is tight, twisted, crossed or zero respectively); hence these tightness properties are preserved under extension-contraction.

It remains to prove that if $D$ is tight then $D'$ is tight. In this case, any sub-tensor-product of $D'$ is tight (lemma \ref{lem:sub-tensor-product_tightness}), and hence for any $1 \leq i \leq j \leq n$ the product $D_i \cdots D_j$ is tight. Thus in any extension-contraction $D'$ of $D$, the product of any sub-tensor product is tight; so $D'$ is tight.
\end{proof}

\subsection{Tensor products of homology classes}
\label{sec:tensor_product_homology}

We now turn to $\HH^{\otimes n}$.
%As we have seen (e.g. proposition \ref{prop:homology_summands}), a given set of viable H-data $(h,s,t)$ is twisted if $(h,s,t)$ has an all-on once occupied pair, in which case $\HH(h,s,t) = 0$; otherwise $(h,s,t)$ is tight, and $\HH(h,s,t) \cong \Z_2$, generated by $M_{h,s,t}$ (definition \ref{defn:Mhst}), represented by any tight diagram with H-data $(h,s,t)$. Moreover, when $(h,s,t)$ is tight, the tight diagrams with H-data $(h,s,t)$ are precisely the crossingless diagrams.
Let $M = M_1 \otimes \cdots \otimes M_n$ be a viable tensor product of nonzero homology classes of diagrams on an arc diagram $\ZZ$, where $M_i$ has H-data $(h_i, s_i, t_i)$ and is represented by a tight diagram $D_i$.

%Let $M_1, \ldots, M_n$ be nonzero homology classes, where $M_i$ has H-data $(h_i, s_i, t_i)$. So each $(h_i, s_i, t_i)$ is tight and $M_i = M_{h_i, s_i, t_i}$. We consider the tensor product $M = M_1 \otimes \cdots \otimes M_n \in \HH^{\otimes n}$. Recall (definition \ref{defn:viability_tensor_products}) $M$ is viable iff each $t_i = s_{i+1}$ and $h_1 + \cdots + h_n$ is viable.

%Choose a diagram $D_i$ representing each $M_i$ (necessarily crossingless and tight). Then $D = D_1 \otimes \cdots \otimes D_n$ is a viable tensor product of diagrams. 

Since each $D_i$ is tight, no $D_i$ is crossed, so $D$ is neither crossed (lemma \ref{lem:products_crossings}) nor sublime (lemma \ref{lem:sublime_contains_crossed}). That only leaves the possibilities in the following definition.
\begin{defn}
\label{def:tightness_tensor_homology}
Suppose $M = M_1 \otimes \cdots \otimes M_n$ is viable, and let $D_i$ be a diagram representing $M_i$. Then $M$ is \emph{tight, twisted, critical} or \emph{singular} accordingly as $D = D_1 \otimes \cdots \otimes D_n$ is tight, twisted, critical or singular.
\end{defn}
We can also speak of $M$ being tight, twisted, critical or singular at a matched pair $P$, or twisted at a place $p$, as for $D$.

However, there may be multiple choices for the $D_i$; so we check that tightness is well defined.
\begin{lem}
\label{lem:homology_tensor_tightness_well_defined}
Let $D_i, D'_i$ be diagrams representing $M_i$, and let $D = D_1 \otimes \cdots D_n$, $D' = D'_1 \otimes \cdots \otimes D'_n$. Then $D$ and $D'$ have the same tightness.
\end{lem}

\begin{proof}
Consider a matched pair $P$. If $D_i$ and $D'_i$ differ at $P$, then $P$ is all-on doubly occupied by $D_i$. In this case $D_i$ covers all four steps of the local arc diagram $\ZZ_P$, as does $D'_i$; so every $D_j$ and $D'_j$ with $j \neq i$ is idempotent at $P$. Thus both $D$ and $D'$ are tight at $P$. Hence $D$ and $D'$ have the same tightness at each matched pair. By lemma \ref{lem:tensor_product_local_tightness} then $D$ and $D'$ have the same tightness.
\end{proof}

Now we observe from table \ref{tbl:local_tensor_products} that any local tight, twisted, critical or singular tensor product, for any viable H-grading, can be constructed using only tight diagrams. So the table of possible local tensor products of homology classes of diagrams is precisely given by the tight, twisted, critical and singular columns of table \ref{tbl:local_tensor_products}, and we have the following proposition.
%, using the notion of extension-contraction from definition \ref{def:extending_contracting}(ii).
\begin{prop}[Classification of viable local tensor products of homology classes]
\label{prop:homology_tensor_product_classification}
Let $M = M_1 \otimes \cdots \otimes M_n$ be a viable tensor product of nonzero homology classes of diagrams on $\ZZ_P$. Then $M$ is an extension-contraction of a tensor product of homology classes of diagrams shown in the tight, twisted, critical or singular columns of table \ref{tbl:local_tensor_products}, and its H-data and tightness are as shown.
\qed
\end{prop}

Strictly speaking, table \ref{tbl:local_tensor_products} shows tensor products of diagrams; proposition \ref{prop:homology_tensor_product_classification} refers to their homology classes. In practice when drawing or referring to homology classes, we draw and refer to diagrams (necessarily tight) representing them.

The following observation from the classification of proposition \ref{prop:homology_tensor_product_classification} and table \ref{tbl:local_tensor_products} allows us to say something about tightness merely from H-data.
\begin{lem}
\label{lem:homology_tensor_product_tightness_H-data}
Let $M = M_1 \otimes \cdots \otimes M_n$ be viable on $\ZZ_P$, with H-data $(h,s,t)$.
\begin{enumerate}
\item
$M$ is tight or critical at $P$ iff $(h,s,t)$ is tight at $P$.
\item
$M$ is twisted at $P$ iff $(h,s,t)$ is twisted at $P$.
\item
$M$ is is singular at $P$ iff $(h,s,t)$ is singular at $P$.
\end{enumerate}
\qed
\end{lem}

We can distinguish tightness in $\HH^{\otimes n}$ by the following result.
\begin{lem}[Characterising tightness of tensor product of homology classes]
\label{lem:homology_tensor_tightness_characterisation}
Suppose $M = M_1 \otimes \cdots \otimes M_n$ is a viable tensor product of nonzero homology classes of diagrams on an arc diagram $\ZZ$, with H-data $(h,s,t)$, and let $D_i$ be a diagram representing $M_i$.
\begin{enumerate}
\item
$M$ is tight iff $M_1 \cdots M_n \neq 0$.
\item
$M$ is twisted iff $M_1 \cdots M_n = 0$ but $D_1 \cdots D_n \neq 0$.
\item
$M$ is critical iff $D_1 \cdots D_n = 0$, but $\A(h,s,t) \neq 0$.
\item
$M$ is singular iff $\A(h,s,t) = 0$
\end{enumerate}
\end{lem}
Like definition \ref{def:tensor_product_tightness}, lemma \ref{lem:homology_tensor_tightness_characterisation} presents tightness as a list of things that go increasingly wrong. Note $\A(h,s,t) = 0$ implies $D_1 \cdots D_n = 0$ implies $M_1 \cdots M_n = 0$, so precisely one of these cases applies.

Recalling the isomorphism between $\HH$ and the contact category, $M = M_1 \otimes \cdots \otimes M_n$ describes the stacking of tight cubulated contact structures on a thickened surface $\Sigma \times [0,1]$. Cases (ii) through (iv) describe overtwisted structures, in increasing order of degeneracy. In case (ii) the stacked contact cubes above each individual square of the quadrangulation remain tight, but the overall contact structure is overtwisted (as in figure \ref{fig:twisted_example}); in case (iii) the contact cube above some square becomes overtwisted; in case (iv) the contact cube above some square is overtwisted, even when restricted to the boundary of the cube.

\begin{proof}
Let $D = D_1 \otimes \cdots \otimes D_n$, so by definition \ref{def:tightness_tensor_homology} (and lemma \ref{lem:homology_tensor_tightness_well_defined}), $M$ and $D$ have the same tightness, and it is sufficient to consider the tightness of $D$.

If $D$ is tight then $D_1 \cdots D_n$ is tight, so $M_1 \cdots M_n \neq 0$ (definition \ref{def:tight_twisted_crossed_diagrams}).  Conversely, if $M_1 \cdots M_n \neq 0$ then all $M_i \neq 0$, and these nonzero homology classes are represented by the diagrams $D_1 \cdots D_n$ and $D_i$, which are tight, so $D$ is tight.

If $D$ is twisted then $D_1 \cdots D_n$ is twisted, hence nonzero, but its homology class $M_1 \cdots M_n = 0$ (definition \ref{def:tight_twisted_crossed_diagrams} again). Conversely, if $M_1 \cdots M_n = 0$ but the diagram $D_1 \cdots D_n$ is nonzero, then $D_1 \cdots D_n$ is not tight; it is also not crossed, since no $D_i$ is crossed (lemma \ref{lem:products_crossings}); hence it is twisted. Thus $D$ is twisted.

The characterisations of critical and singular follow immediately from definition \ref{def:tensor_product_tightness}.
\end{proof}

Lemma \ref{lem:tensor_product_local_tightness} applied to homology immediately gives the following.
\begin{lem}[Local-global tightness in $\HH^{\otimes n}$]
\label{lem:homology_tensor_tightness_local}
Let $M = M_1 \otimes \cdots \otimes M_n$ be viable, where $M_i$ is represented by diagram $D_i$.
\begin{enumerate}
\item
$M$ is tight iff $M$ is tight at all matched pairs.
\item
$M$ is twisted iff $M$ is tight or twisted at each matched pair, and twisted at $\geq 1$ matched pair.
\item
$M$ is critical iff $M$ is tight, twisted or critical at all matched pairs, and critical at $\geq 1$ matched pair.
\item
$M$ is singular iff $M$ is singular at $\geq 1$ matched pair.
\end{enumerate}
\qed
\end{lem}
Thus the properties ``tight", ```tight or twisted", ``tight, twisted or critical" and ``not singular" are all local-to-global properties of tensor products of homology classes.

We also consider contractions and extensions.
\begin{lem} \
\label{lem:homology_contractions}
Let $M = M_1 \otimes \cdots \otimes M_n$ be a viable tensor product of diagrams on $\ZZ_P$.
\begin{enumerate}
\item
If $M$ is tight, then for all $1 \leq i \leq j \leq n$, the product $M_i \cdots M_j$ is nonzero, so $M_1 \otimes \cdots \otimes M_{i-1} \otimes (M_i \cdots M_j) \otimes M_{j+1} \otimes M_n$ is a contraction of $M$.
\item
If $M$ is twisted, critical or singular, then any contraction of $M$ is trivial. Moreover, $M$ is an extension of a tensor product of homology classes shown in the twisted, critical of singular columns of table \ref{tbl:local_tensor_products}.
\end{enumerate}
\end{lem}
Recall trivial and nontrivial contractions were defined in section \ref{sec:sub_extension_contraction} (definition \ref{def:trivial_contraction}).

\begin{proof}
If $M$ is tight, then (lemma \ref{lem:homology_tensor_tightness_characterisation}) $M_1 \cdots M_n \neq 0$; so any $M_i \cdots M_j \neq 0$.

If $M$ is twisted, critical or singular, then %(lemma \ref{lem:homology_tensor_tightness_local}) $M$ is twisted, critical or singular at some matched pair $P$. 
by proposition \ref{prop:homology_tensor_product_classification}, $M$ is an extension-contraction of a tensor product shown in the appropriate column of table \ref{tbl:local_tensor_products}. We observe that multiplying any two consecutive diagrams in any of these tensor products yields a twisted or zero diagram, which is zero in homology. Thus no nontrivial contraction exists.
\end{proof}

The following fact about critical tensor products will be useful in the sequel.
\begin{lem}[``It takes 3 to be critical"]
\label{lem:it_takes_3}
Suppose $M = M_1 \otimes \cdots \otimes M_n$ is viable and critical on an arc diagram $\ZZ$. Then $n \geq 3$.
\end{lem}

\begin{proof}
By lemma \ref{lem:homology_tensor_tightness_local}, some local tensor product $M_P$ is critical. By lemma \ref{lem:homology_contractions}, $M_P$ is an extension of a critical diagrams in table \ref{tbl:local_tensor_products}, and all such diagrams have at least 3 factors.
\end{proof}
%While it takes 3 to be critical, the seemingly worse, singular, case \emph{can} arise with $n=2$! Nonetheless, any viable tensor product of two nonzero homology classes is tight, twisted or singular. 

We also consider how tightness behaves under taking sub-tensor-products of local homology classes. The following lemma is immediate from applying lemma \ref{lem:local_sub-tensor-product_tightness} to homology (recalling definition \ref{def:tightness_tensor_homology} of tightness). Effectively we simply cross out the sublime and crossed rows and columns of table \ref{tbl:local_sub-tensor-product_tightness}.
\begin{lem}[Tightness of local sub-tensor-products of homology classes]
\label{lem:local_tightness_homology_sub-tensor-product}
Let $M$ be a viable tensor product of nonzero homology classes of diagrams on $\ZZ_P$, and let $M'$ be a sub-tensor-product. Then the possible tightness types of $M$ and $M'$ are as shown in table \ref{tbl:local_homology_sub-tensor-product_tightness}.
\qed
\end{lem}

\begin{table}
\begin{center}
\begin{tabular}{x{1.5cm}| c|c|c|c}
\diag{.1em}{1.5cm}{$M$}{$M'$} & Tight & Twisted & Critical & Singular  \\ \hline
Tight 	& X &		&		& \\ \hline
Twisted &	X	&	X &		& \\ \hline
Critical&	X	&	X	&	X	&	X \\ \hline
Singular&	X	&		&		&	X
\end{tabular}
\caption{Possible tightness types of a viable local tensor product of homology classes $M$ and a sub-tensor-product $M'$.}
\label{tbl:local_homology_sub-tensor-product_tightness}
\end{center}
\end{table}

Thus, for instance, if $M$ is tight, then $M'$ is tight; In this case $M$ corresponds to a tight contact manifold and $M'$ to a contact submanifold. Similarly, if $M'$ is critical, then $M$ is critical. %; and if $M'$ is critical or twisted, then $M$ is critical or twisted. 

Considering the tightness of sub-tensor-products globally, we need the following statement, which is immediate from lemma \ref{lem:sub-tensor-product_tightness}.
\begin{lem}[Tightness of global sub-tensor-products of homology classes]
\label{lem:tightness_homology_sub-tensor-product}
Let $M$ be a viable tensor product of nonzero homology classes of diagrams on an arc diagram $\ZZ$, and let $M'$ be a sub-tensor-product.
\begin{enumerate}
\item If $M$ is tight, then $M'$ is tight.
\item If $M'$ is critical or singular, then $M$ is critical or singular.
\end{enumerate}
\qed
\end{lem}
Note the contrapositive of (ii): if $M$ is tight or twisted, then $M'$ is tight or twisted.

%We can now consider tightness of sub-tensor-products globally. The following result is immediate from lemma \ref{lem:sub-tensor-product_tightness}.
%\begin{lem}
%\label{lem:tightness_homology_sub-tensor-product}
%Let $M$ be a viable tensor product of nonzero homology classes of diagrams, and let $M'$ a sub-tensor-product. Then the possible tightness types of $M$ and $M'$ are as shown in table \ref{tbl:homology_sub-tensor-product_tightness}.
%\qed
%\end{lem}

%\begin{table}
%\begin{center}
%\begin{tabular}{x{1.5cm}| c|c|c|c}
%\diag{.1em}{1.5cm}{$M$}{$M'$} & Tight & Twisted & Critical & Singular  \\ \hline
%Tight 	& X &		&		& \\ \hline
%Twisted &	X	&	X &		& \\ \hline
%Critical&	X	&	X	&	X	&	X \\ \hline
%Singular&	X	&	X	&	X	&	X
%\end{tabular}
%\caption{Possible tightness types of a general viable tensor product of homology classes $M$ and a sub-tensor-product $M'$.}
%\label{tbl:homology_sub-tensor-product_tightness}
%\end{center}
%\end{table}

%In particular, if $M$ is tight, then $M'$ is tight; if $M$ is twisted, then $M'$ is tight or twisted; and if $M'$ is critical or singular, then $M$ is critical or singular.

\subsection{Generalised contraction}
\label{sec:generalised_contraction}

It is useful to generalise the notion of contraction discussed above. 
Let $M = M_1 \otimes \cdots \otimes M_n$ be a viable tensor product of nonzero homology classes of diagrams. A contraction of $M$ replaces a sub-tensor-product $M' = M_i \otimes \cdots \otimes M_j$ with $M_i \cdots M_j$ provided that this product is nonzero (definition \ref{def:extending_contracting}). Recalling (proposition \ref{prop:homology_summands}) that for any given tight H-data $(h,s,t)$ there is a unique nonzero homology class, we observe $M_i \cdots M_j$ is the unique homology class of diagram with the H-data of $M'$. This leads to the following generalisation.
\begin{defn}
\label{def:H-contraction}
Let $M = M_1 \otimes \cdots \otimes M_n$ be a viable tensor product of nonzero homology classes of diagrams on an arc diagram $\ZZ$. Suppose a sub-tensor-product $M_i \otimes \cdots \otimes M_j$ has tight H-data, and let $M^*$ be the unique nonzero homology class of a diagram with this H-data.

Then we say $M' = M_1 \otimes \cdots \otimes M_{i-1} \otimes M^* \otimes M_{j+1} \otimes \cdots \otimes M_n$ is obtained from $M$ by \emph{H-contraction} of $M_i \otimes \cdots \otimes M_j$.
\end{defn}
Thus H-contraction generalises contraction. Moreover if $M'$ is obtained from $M$ by H-contraction, then $M'$ is viable, and has the same H-data as $M$.

Tightness locally behaves rather nicely under H-contraction.
\begin{lem}
\label{lem:tightness_H-contraction}
Let $M$ be a viable tensor product of nonzero homology classes of diagrams on $\ZZ_P$. Suppose $M'$ is obtained from $M$ by H-contraction.
\begin{enumerate}
\item $M$ is tight or critical iff $M'$ is tight or critical. Moreover,
\begin{enumerate}
\item if $M$ is tight, then $M'$ is tight;
\item if $M'$ is critical, then $M$ is critical.
\end{enumerate}
\item $M$ is twisted iff $M'$ is twisted.
\item $M$ is singular iff $M'$ is singular.
\end{enumerate}
\end{lem}

\begin{proof}
The H-data of $M$ is tight, twisted or singular accordingly as $M$ is respectively (tight or critical), twisted or singular (lemma \ref{lem:homology_tensor_product_tightness_H-data}). Since H-contraction preserves H-data, these tightness properties are also preserved.

If $M$ is tight, then we replace $M_i \otimes \cdots \otimes M_j$ with $M_i \cdots M_j$ (lemma \ref{lem:homology_contractions}), so we have a bona fide contraction (definition \ref{def:extending_contracting}), and $M'$ is tight: the product of the factors in both $M$ and $M'$ is $M_1 \cdots M_n$ (lemma \ref{lem:homology_tensor_tightness_characterisation}).

If $M'$ is critical, then, by proposition \ref{prop:homology_tensor_product_classification} and lemma \ref{lem:homology_contractions}, $M'$ is an extension of one of the tensor products shown in the critical column of table \ref{tbl:local_tensor_products}. Thus each tensor factor of $M'$ covers at most one step of $\ZZ_P$. Since $M$ is obtained from $M'$ by replacing a tensor factor of $M'$ with $M_i \otimes \cdots \otimes M_j$, in a way that preserves H-data, each tensor factor of $M$ also covers at most one step of $\ZZ_P$; in fact $M$ is an extension of $M'$. So $M$ is critical.
\end{proof}

\section{Constructing A-infinity structures}
\label{sec:construction}

\subsection{The construction}
\label{sec:constructing_operations}

We now describe $A_\infty$ structures on $\HH$. %We first give an overview of the construction.

Recall the discussion of section \ref{sec:intro_construction}. Regard $\A$ as an $A_\infty$ algebra with trivial $n$-ary operations for $n \geq 3$. The homology $\HH$ can be regarded as a DGA with trivial differential. We construct an $A_\infty$ structure $X$ on $\HH$ extending this DGA structure, i.e. with $X_1 = 0$ and $X_2$ being multiplication, together with a morphism of $A_\infty$ algebras $f \colon \HH \To \A$, consisting of  maps $f_n \colon \HH^{\otimes  n} \To \A$, extending a cycle selection map $f_1 \colon \HH \hookrightarrow \A$. Kadeishvili's construction \cite{kadeishvili_homology_1980} builds maps $X_n, f_n$ and auxiliary maps $U_n$ inductively over $n$. All these maps preserve H-data; $X_n$ and $U_n$ have Maslov grading $n-2$, and $f_n$ has Maslov grading $n-1$. At each stage there we only need to construct $f_n$ explicitly, then $U_{n+1}$ and $X_{n+1}$ are determined by equations (\ref{eqn:Un_def}) and (\ref{eqn:Xn_def}). 

As it turns out, we only need to construct maps $\fbar_n, \Ubar_n \colon \HH^{\otimes n} \To \AAbar$ (definition \ref{defn:Abar}) into the quotient $\AAbar = \A/\F$. 

To construct the cycle selection homomorphism $f_1$, as discussed in section \ref{sec:cycle_selection_homs}, we use a cycle choice function $\mathcal{CY}$. And as discussed in section \ref{sec:ordering}, a cycle choice function can be constructed from a pair ordering on the arc diagram $\ZZ$.

To construct $f_n$ for $n \geq 2$, we need to solve equation (\ref{eqn:fn_eqn}): $f_1 X_n - U_n = \partial f_n$. Since all maps preserve H-data, this can be done separately on each H-summand. On twisted summands, defining $f_n$ amounts to inverting the differential, as discussed in section \ref{sec:inverting_differential}. We use the creation operators of section \ref{sec:creation_operators}, choosing appropriate creation operators on each summand using a creation choice function, as discussed in section \ref{sec:global_creation}. On other summands, it turns out that no choice is necessary, once we project to $\AAbar$; we can take $\fbar_n = 0$ in this case.

Our construction will satisfy the following condition on the maps $f_n$. The idea is that if $f_1 X_n - U_n = 0$, then it reasonable to say that $f_n$ ``ought" to be zero as well. (The constant of integration is most naturally zero!) 

%[*** COMMENTED OUT THE PREVIOUS VERSION'S DEFINITION OF STRONGLY BALANCED... UNCOMMENT IT IF WE NEED IT***]
%As a slightly stronger condition requires that if $\partial \fbar_n = 0$ (i.e. $\partial f_n \in \F$) then $\fbar_n = 0$.
\begin{defn}
\label{def:balanced}
%\begin{enumerate}
%\item
Suppose that for all $M$, if $(f_1 X_n - U_n) (M) = 0$ then $f_n (M) = 0$. In this case we say $f_n$ is \emph{balanced}.
%\item
%Suppose $f_n$ is balanced, and also that if $(f_1 X_n - U_n) (M_1 \otimes \cdots \otimes M_n) \in \F$ then $f_n (M_1 \otimes \cdots \otimes M_n) \in \F$. In this case we say $f_n$ is \emph{strongly balanced}.
%\end{enumerate}
%\end{enumerate}
\end{defn}

We now state and prove the result.
\begin{thm}
\label{thm:main_thm}
Let $\ZZ$ be an arc diagram and let $M_i$ be nonzero homology classes of diagrams on $\ZZ$. Let $\mathcal{CY}$ be a cycle choice function for $\ZZ$, and let $\mathcal{CR}$ be a creation choice function. Then there is an $A_\infty$ structure $X$ on $\HH(\ZZ)$ extending its DGA structure, and a morphism of $A_\infty$ algebras $f \colon \HH(\ZZ) \To \A(\ZZ)$ extending  the cycle selection  map $f_1 = f^{\mathcal{CY}}$, such that the following conditions hold.
\begin{enumerate}
\item
If $M = M_1 \otimes \cdots \otimes M_n$ is not viable, then $\fbar_n (M) = 0$ and $X_n (M) = 0$; and if $M$ has an idempotent mismatch then $f_n (M)=0$.
\item
The maps $X_n \colon \HH(\ZZ)^{\otimes n} \To \HH(\ZZ)$ of $X$, and the maps $f_n \colon \HH(\ZZ)^{\otimes n} \To \A(\ZZ)$ of $f$, all preserve H-data; moreover $X_n$ has Maslov grading $n-2$ and $f_n$ has Maslov grading $n-1$.
\item
Each map $f_n$ is balanced.
\item
For $n \geq 2$, on each twisted H-summand, $f_n = A_{\mathcal{CR}}^* \circ (f_1 X_n - U_n)$, where $U_n$ is defined by equation (\ref{eqn:Un_def}) and $X_n$ is defined by equation (\ref{eqn:Xn_def}) from section \ref{sec:intro_construction}.
%\begin{align*}
%U_n \left( M_1 \otimes \cdots \otimes M_n \right) &= \sum_{j=1}^{n-1} f_j \left( M_1 \otimes \cdots \otimes M_j \right) f_{n-j} \left( M_{j+1} \otimes \cdots \otimes M_n \right) \\
%& \quad + \sum_{k=0}^{n-2} \sum_{j=2}^{n-1} f_{n-j+1} \left( M_1 \otimes \cdots \otimes M_k \otimes X_j \left( M_{k+1} \otimes \cdots \otimes M_{k+j} \right) \otimes \cdots \otimes M_n \right), \\
%X_n &= [U_n].
%\end{align*}
\end{enumerate}
The $A_\infty$-operations $X_n$ satisfying these conditions are unique. The maps $f_n$ are uniquely defined modulo $\F$.
\end{thm}
Recall that when $M$ is singular, there are no diagrams with its H-data $(h,s,t)$, and $\A(h,s,t) = 0$ (lemma \ref{lem:homology_tensor_tightness_characterisation}). So condition (i), that $f_n$ and $X_n$ preserve H-data, implies that when $M$ is singular, $f_n (M)$ and $X_n (M)$ are both zero.

The uniqueness statement means that, although the $f_n$ are not uniquely determined by the conditions of the theorem, after composing with the quotient map $\A \To \AAbar$ to obtain $\fbar_n \colon \HH^{\otimes n} \To \AAbar$, the maps $\fbar_n$ \emph{are} uniquely determined.

Recall from section \ref{sec:ordering} that a pair ordering $\preceq$ (definition \ref{defn:pair_ordering}) determines a cycle choice function $\mathcal{CY}^\preceq$ and a creation choice function $\mathcal{CR}^\preceq$ (definition \ref{def:choice_functions_from_ordering}). Hence we immediately obtain the following corollary.
\begin{cor}
\label{cor:main_cor}
Let $\preceq$ be a pair ordering on an arc diagram $\ZZ$. Then there is an $A_\infty$ structure $X$ on $\HH$ extending its DGA structure, and a morphism of $A_\infty$ algebras $f \colon \HH \To \A$ extending the cycle selection map $f^{\mathcal{CY}^\preceq}$, satisfying the conditions of theorem \ref{thm:main_thm}, such that on twisted summands for $n \geq 2$, $f_n = A_{\mathcal{CR}^\preceq}^* \circ \left( f_1 X_n - U_n \right)$.
\qed
\end{cor}
Corollary \ref{cor:main_cor} is a precise form of theorem \ref{thm:first_thm}.

\begin{proof}[Proof of theorem \ref{thm:main_thm}]
We follow the method described above. At level 1, equations (\ref{eqn:Un_def}), (\ref{eqn:Xn_def}) and (\ref{eqn:fn_eqn}) require $U_1 = 0$, $X_1 = 0$ and $\partial f_1 = 0$. The last equation is satisfied by $f_1 = f^{\mathcal{CY}}$. Since diagrams with non-viable H-data are zero in homology, $f_1 = 0$ for such diagrams.

Now suppose we have constructed all operations at level $<n$ as required; we construct $U_n, X_n, f_n$.

We define $U_n$ by equation (\ref{eqn:Un_def}). Then $U_n$ Maslov grading $n-2$. As the $f_j$ are not uniquely defined, neither is $U_n$. However, all the $\fbar_j$ are uniquely defined, and since equation \ref{eqn:Un_def} expresses $U_n$ as a sum of products of values of $f_j$, $\Ubar_n$ is also uniquely defined. Since the $f_j$ (and multiplication in $\A$) preserve H-data, $U_n$ does also.

We define $X_n$ by equation (\ref{eqn:Xn_def}); then $X_n$ respects gradings as required. As in \cite{kadeishvili_homology_1980}, $U_n$ is a cycle and $X_n$ is its homology class, so $X_n$ is well defined. Now all diagrams in $\F$ are non-viable or have crossings, and such diagrams do not contribute to homology. Thus $X_n (M)$ is determined completely by $\Ubar_n (M)$, which is uniquely defined; hence $X_n (M)$ is uniquely defined.

To define $f_n$, we must solve equation (\ref{eqn:fn_eqn}) for each viable $M = M_1 \otimes \cdots \otimes M_n$:
\[
\partial f_n \left( M \right) = \left( f_1 X_n - U_n \right) \left( M \right).
\]
We now consider various cases for $M$.

First, suppose $M = M_1 \otimes \cdots \otimes M_n$ is non-viable because of an idempotent mismatch. Then each term in $U_n (M)$ from (\ref{eqn:Un_def}) is zero: by induction, $f_i$ and $X_i$ for $i<n$ are zero on tensor products with mismatches, and the product of two mismatched diagrams is zero. Thus $U_n (M) = 0$, and by (\ref{eqn:Xn_def}) then $X_n (M) = 0$. We set $f_n (M) = 0$ as required by the balanced condition; equation (\ref{eqn:fn_eqn}) is then satisfied. 

Next, suppose $M$ is non-viable but has no idempotent mismatch, hence has some step covered more than once. As $U_n$ preserves H-data then $U_n (M)$ is a sum of non-viable diagrams, so $U_n (M) \in \F$ and $\Ubar_n (M) = 0$. Then $X_n (M) = 0$, and equation (\ref{eqn:fn_eqn}) then requires $\partial f_n (M) = U_n (M)$. If $U_n (M) = 0$ then we set $f_n (M) = 0$, satisfying the balanced condition; otherwise we choose $f_n (M)$ arbitrarily to be any solution to this equation with the same H-data as $M$, and of pure Maslov grading (necessarily 1 greater than $M$). Then $f_n (M) \in \F$, being a sum of non-viable diagrams. Thus $f_n (M)$ is not uniquely determined, but $\fbar_n (M)$ is uniquely determined, indeed $\fbar_n (M) = 0$.

If $M$ is singular, then as there are no diagrams with the H-data of $M$, $U_n (M)$, $X_n (M)$ and $f_n (M)$ are all zero, and all required conditions are satisfied.

So we may now assume $M$ is viable and non-singular; its H-data $(h,s,t)$ (definition \ref{def:H-data_tensor_product}) is thus tight or twisted (definition \ref{def:tightness_H-data}) . 

If $(h,s,t)$ is twisted, then as required we take $f_n = A^*_{\mathcal{CR}} \circ (f_1 X_n - U_n)$. Then $f_n$ is balanced. Since $(f_1 X_n - U_n)(M)$ is a boundary, hence a cycle (in fact boundaries and cycles are the same since $\HH(h,s,t) = 0$), we have
\[
\partial f_n (M) = \partial A^*_{\mathcal{CR}} \left( f_1 X_n (M) - U_n (M) \right)
= \left( f_1 X_n - U_n \right) (M),
\]
inverting the differential, by lemma \ref{lem:creations_invert_differential} and the discussion of section \ref{sec:global_creation}.

If $(h,s,t)$ is tight, by lemma \ref{lem:no_FOC_pairs_integral_in_F}, any $f_n (M)$ of the required Maslov grading and  satisfying $\partial f_n (M) = f_1 X_n (M) - U_n (M)$ lies in $\F$. We choose $f_n (M)$ to be zero if $f_1 X_n - U_n = 0$ (satisfying the balanced condition), otherwise to be any solution to this equation with the same H-data as $M$, and pure Maslov grading. Then $f_n (M)$ is not uniquely determined, but $\fbar_n (M) = 0$.

This defines $f_n$ and $X_n$ satisfying the required conditions, with the uniqueness claimed. Having followed Kadeishvili's construction, the $X_n$ form an $A_\infty$ structure on $\HH$, and the $f_n$ form a morphism of $A_\infty$ algebras $\HH \To \A$. 
\end{proof}

Note it follows from this proof that whenever $M$ has tight H-data then $\fbar_n (M) = 0$.

\subsection{Shorthand notation}
\label{sec:shorthand}

For convenience, we use some shorthand for viable nonzero tensor products in $\A^{\otimes n}$ and $\HH^{\otimes n}$. The shorthand is essentially a stylised version of our previous diagrams.

Let $M = M_1 \otimes \cdots \otimes M_n$ be a viable tensor product of nonzero homology classes of diagrams on an arc diagram $\ZZ$. A shorthand diagram represents $M$ by an array of data. Each row refers to a matched pair $P$ of $\ZZ$. The $n$ columns refer to $M_1, \ldots, M_n$. In the row for $P = \{p,p'\}$ and the column of $M_i$, we write which of the four steps of $\ZZ_P$ are covered by $M_i$. Along the row for $P$, between the columns we draw a hollow or solid circle indicating whether $P$ is contained in the corresponding idempotent (``on or off"). This is well defined since $M$ is viable.

Such notation specifies $M$ completely, since it specifies the H-data of each $M_i$.

We denote the steps before and after a place $p$ by $p_-$ and $p_+$ respectively. We write $p_\pm$ to indicate that both $p_+$ and $p_-$ are covered. 

Figure \ref{fig:shorthand_notation} shows an example. 
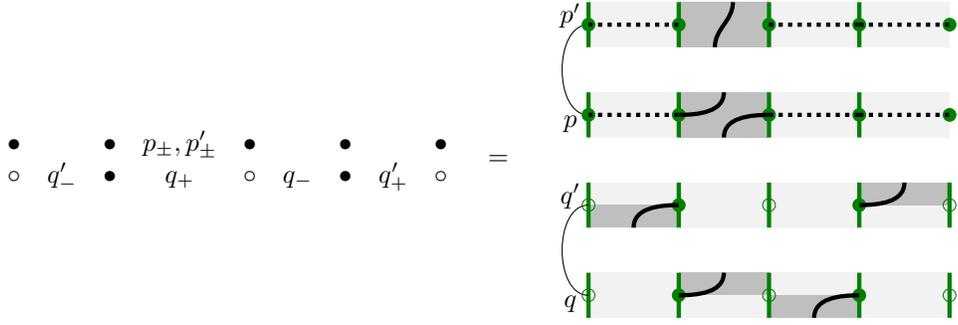
\begin{figure}
\begin{center}
\begin{tikzpicture}[scale=1.2]
%\begin{scope}[xshift=0 cm]
%\draw (-1.8,0.75) node {$\begin{array}{ccccc}
%\on & p_+ & \off & p_- & \on \end{array}$};
%\draw (-0.5,0.75) node {$=$};
%\strandbackgroundshading
%\tstrandbackgroundshading{1}
%\afterwused
%\tbeforewused{1}
%\strandsetupn{p'}{p}
%\tstrandsetup{1}
%\tstrandsetup{2}
%\lefton
%\rightoff
%\trighton{1}
%\usea
%\tuseb{1}
%\end{scope}
%\begin{scope}[xshift = 3cm]
%\draw (-1.5,0.75) node {$\begin{array}{ccc}
%\on & p_\pm & \on\end{array}$};
%\draw (-0.5,0.75) node {$=$};
%\strandbackgroundshading
%\afterwused
%\beforewused
%\strandsetup
%\tstrandsetup{1}
%\lefton
%\usea
%\useb
%\end{scope}

%\begin{scope}[yshift = -2 cm]
%\draw (-2.5,0.75) node {$\begin{array}{ccccccccc}
%\on & p_+ & \off & p_- & \on & q_+ & \off & q_- & \on
%\end{array}$};
%\draw (-0.5,0.75) node {$=$};
%\strandbackgroundshading
%\tstrandbackgroundshading{1}
%\tstrandbackgroundshading{2}
%\tstrandbackgroundshading{3}
%\afterwused
%\tbeforewused{1}
%\taftervused{2}
%\tbeforevused{3}
%\strandsetup
%\tstrandsetup{1}
%\tstrandsetup{2}
%\tstrandsetup{3}
%\tstrandsetup{4}
%\lefton
%\rightoff
%\trighton{1}
%\trightoff{2}
%\trighton{3}
%\usea
%\tuseb{1}
%\tusec{2}
%\tused{3}
%\end{scope}

\begin{scope}[yshift = -2 cm]
\draw (-4,-0.25) node {$\begin{array}{ccccccccc}
\on &  			& \on & p_\pm, p'_\pm & \on 	&  		& \on 	&  			& \on \\
\off & q'_-	& \on & q_+						& \off	&	q_-	&	\on		& q'_+ 	& \off
\end{array}$};
\draw (-1,-0.25) node {$=$};
\strandbackgroundshading
\tstrandbackgroundshading{1}
\tstrandbackgroundshading{2}
\tstrandbackgroundshading{3}
\tbeforevused{1}
\taftervused{1}
\tbeforewused{1}
\tafterwused{1}
\strandsetupn{p}{p'}
\tstrandsetup{1}
\tstrandsetup{2}
\tstrandsetup{3}
\lefton
\righton
\trighton{1}
\trighton{2}
\trighton{3}
\dothorizontals
\tuseab{1}
\tused{1}
\tusec{1}
\tdothorizontals{2}
\tdothorizontals{3}
\end{scope}
\begin{scope}[yshift=-4 cm]
\strandbackgroundshading
\tstrandbackgroundshading{1}
\tstrandbackgroundshading{2}
\tstrandbackgroundshading{3}
\beforewused
\taftervused{1}
\tbeforevused{2}
\tafterwused{3}
\strandsetupn{q}{q'}
\tstrandsetup{1}
\tstrandsetup{2}
\tstrandsetup{3}
\tstrandsetup{4}
\leftoff
\righton
\trightoff{1}
\trighton{2}
\trightoff{3}
\useb
\tusec{1}
\tused{2}
\tusea{3}
\end{scope}

\end{tikzpicture}
\caption{Shorthand notation for viable nonzero tensor products of homology classes of diagrams.}
\label{fig:shorthand_notation}
\end{center}
\end{figure}
%As in previous examples, we usually only write the situation at certain matched pairs. 
(It is always possible to include further matched pairs, each with a tight tensor product of homology classes of diagrams, so as to obtain a tensor product of homology classes of non-augmented diagrams on a connected arc diagram.) %In any case, we always assume the tensor product is tight at any matched pairs not shown. 

Occasionally, when the idempotents can be inferred from the H-grading of each $M_i$, we omit the circles in the notation.

We use a similar notation to write elements of $\AAbar^{\otimes n}$. Let $D = D_1 \otimes \cdots \otimes D_n$ be a viable tensor product of diagrams on $\ZZ$, which does not lie in $\F$. Again we write an array where each row refers to a matched pair $P$, and each column refers to a $D_i$, and in each row we have hollow or solid circle to signify idempotents. However, a diagram is not always specified by its H-data, so we use some of the notation of definition \ref{def:C_P}. When there is a unique tight local diagram with the H-data, we simply write which steps are covered. Otherwise, we use the notation $g_p, c_p, w_p$ for 11 doubly occupied tight, 11 once occupied crossed, and 11 once occupied twisted pairs. (As we compute in $\AAbar$, all-on doubly occupied crossed diagrams do not arise. See further section \ref{sec:choices_of_maps}.)

When the $A_\infty$ operations defined by a pair ordering, as in corollary \ref{cor:main_cor}, we may order the pairs upwards in our array (just as they are ordered along the intervals of $\mathbf{Z}$). A creation operator then always applies at the all-on once occupied pair which is lowest in our shorthand notation.

We adopt notation where each matched pair is denoted by a capital letter, and its two places by the corresponding lower case letter, the latter under $\preceq$ being primed. Thus we always write pairs as $P = \{p,p'\}$, $Q = \{q,q'\}$, etc., where $p \prec p'$, $q \prec q'$, etc. Then a cycle choice function always selects a cycle with strands at a place with an unprimed label. 

When a tensor product $M = M_1 \otimes \cdots \otimes M_n$ is twisted at a place $p$ of a pair $P = \{p,p'\}$, it is 11 once occupied at $P$, with $p$ fully occupied (proposition \ref{prop:homology_tensor_product_classification} and table \ref{tbl:local_tensor_products}), and the two steps $p_+, p_-$ are covered by some $M_i$ and $M_j$, with $i<j$. Thus, in our shorthand, across the row corresponding to $P$ we see $p_+$ in one column, then $p_-$ in another column, in that order.

Similarly, if $P$ is critical, then it is sesqui-occupied or doubly occupied. Looking across the row corresponding to $P$ we see one of the following sequences, appearing in order, in distinct columns (possibly after relabelling $p$ and $p'$):
\begin{itemize}
\item Pre-sesqui-occupied: $p'_-, p_+, p_-$
\item Post-sesqui-occupied: $p_+, p_-, p'_+$
\item 00 doubly occupied: $p_-, p'_+, p'_-, p_+$
\item 11 doubly occupied: $p'_+, p'_-, p_+, p_-$
\end{itemize}

\subsection{Low-level maps}
\label{sec:low-level_ex}

Now that we have constructed $A_\infty$ structures on $\HH$, we consider low-level maps explicitly. We assume $A_\infty$ structures are constructed from a cycle choice function $\mathcal{CY}$ and a creation choice function $\mathcal{CR}$, as in theorem \ref{thm:main_thm}. 

%As discussed in section \ref{sec:constructing_operations}, we only need the projection $\fbar_n$ of $f_n$ to $\AAbar$ in order to compute the $X_n$, and the $\fbar_n$ are uniquely defined by the cycle choice and creation choice functions (or pair ordering).

Level 1 maps are straightforward ($X_1 = 0$ and $f_1 = f^{\mathcal{CY}}$), as is multiplication $X_2$. We consider $\fbar_2$.

Let $M = M_1 \otimes M_2$ be a tensor product of nonzero homology classes of diagrams, with H-data $(h,s,t)$. By theorem \ref{thm:main_thm} (and subsequent discussion), if $M$ is non-viable or singular, then $\fbar_n(M)$ and $X_n (M)$ are both zero; so we assume $M$ is viable and non-singular. Then $(h,s,t)$ is tight or twisted.

When $(h,s,t)$ is tight, $\fbar_2 (M) = 0$. So suppose $(h,s,t)$ is twisted, hence has at least one all-on once occupied pair. Clearly $M$ is then not tight; in fact, $M$ cannot be critical either, since it takes 3 to be critical (lemma \ref{lem:it_takes_3}). So $M$ is twisted, hence is tight or twisted at each matched pair (lemma \ref{lem:homology_tensor_tightness_local}). In particular, $M$ is twisted at each all-on once occupied pair, and tight at each other pair.

Since $U_2 (M_1 \otimes M_2) = f_1 (M_1) f_1 (M_2)$, we have
\begin{equation}
\label{eqn:f2_description}
f_2 (M_1 \otimes M_2) = A^*_{\mathcal{CR}} \Big( f_1 (M_1 M_2) + f_1 (M_1) f_1 (M_2) \Big).
\end{equation}
Since $M$ is twisted, $M_1 M_2 = 0$ and $f_1 (M_1) f_1 (M_2) \neq 0$ (lemma \ref{lem:homology_tensor_tightness_characterisation}). The nonzero product $f_1 (M_1) f_1 (M_2)$ is clearly not tight, and it is also not crossed (being the product of two crossingless diagrams: lemma \ref{lem:products_crossings}), hence it is twisted, hence tight or twisted at each pair (lemma \ref{lem:tightness_degeneracy}). At each all-on once occupied pair $f_1 (M_1) f_1 (M_2)$ cannot be tight, so must be twisted; and at each other pair, it	 must be tight. We then have
\begin{equation}
\label{eqn:nonzero_f2_description}
f_2 (M_1 \otimes M_2) = A^*_{\mathcal{CR}^\preceq} \left( f_1 (M_1) f_1 (M_2) \right),
\end{equation}
where $A^*_{\mathcal{CR}^\preceq}$ adds a crossing at the $\preceq$-minimal all-on once occupied pair of $f_1 (M_1) f_1 (M_2)$.

Thus, $f_2 (M_1 \otimes M_2)$ is given by a single diagram, which is crossed at the $\preceq$-minimal all-on once occupied pair of $M_1 \otimes M_2$, and is elsewhere given by $f_1 (M_1) f_1 (M_2)$. The idea is shown in figure \ref{fig:f2_effect}. In this way, $f_2$ turns one pair from twisted to crossed.

\begin{figure}
\begin{center}
\begin{tikzpicture}[scale=1.5]
\draw (-2.5, 0.75) node {$f_2 \Big( \begin{array}{ccccc}
\on & p'_+ & \off & p'_- & \on \end{array} \Big) = $};
\draw (-0.5,0.75) node {$f_2 \Bigg($};
\strandbackgroundshading
\tstrandbackgroundshading{1}
\afterwused
\tbeforewused{1}
\strandsetupn{p}{p'}
\tstrandsetup{1}
\lefton
\rightoff
\trighton{1}
\usea
\tuseb{1}
\draw (2.5, 0.75) node {$\Bigg) =$};
\begin{scope}[xshift=3.5 cm]
\strandbackgroundshading
\beforewused
\afterwused
\strandsetupn{p}{p'}
\lefton
\righton
\dothorizontals
\useab
\end{scope}
\begin{scope}[xshift=5 cm]
\draw (0.5,0.75) node {$= \begin{array}{ccc} \on & c_{p'} & \on \end{array}$};
\end{scope}
\end{tikzpicture}
\caption{The effect of $f_2$.}
\label{fig:f2_effect}
\end{center}
\end{figure}

\section{Properties of A-infinity structures}
\label{sec:A-infinity_structures_general}

We now consider $A_\infty$ structures on $\HH$ constructed by Kadeishvili's method in general --- not just those defined by the construction in theorem \ref{thm:main_thm}. Using this general method, no creation operators or creation choice functions are used.

Throughout this section, then, we consider an $A_\infty$ structure $X$ on $\HH$, with operations $X_n \colon \HH^{\otimes n} \To \HH$, and a morphism of $A_\infty$ algebras $f \colon \HH \To \A$, with maps $f_n \colon \HH^{\otimes n} \To A$, constructed by Kadeishvili's method. So there are also auxiliary maps $U_n \colon \HH^{\otimes n} \To \A$, and we assume the maps $U_n, X_n, f_n$ satisfy equations (\ref{eqn:Un_def}), (\ref{eqn:Xn_def}) and (\ref{eqn:fn_eqn}). We assume all $U_n, X_n, f_n$ preserve H-data and have Maslov degree $n-2$, $n-2$, $n-1$ respectively. Thus, each $f_n$ inverts the differential in $\partial f_n = f_1 X_n - U_n$, but not necessarily by a creation operator.

As usual, throughout this section, $M = M_1 \otimes \cdots \otimes M_n$ denotes a tensor product of nonzero homology classes of diagrams.

\subsection{Non-viable input}
\label{sec:H-data_viability}

We have already seen that, if a tensor product of homology classes of diagrams $M_1 \otimes \cdots \otimes M_n$ is not viable, then their product is zero (lemma \ref{lem:nonzero_viable_tensor_product}). We now show that other operations are zero as well.

\begin{lem} \
\label{lem:fn_Xn_H-data_viability}
If all $f_n$ are balanced, and $M$ is not viable, then $\fbar_n (M)$ and $X_n (M)$ are both zero. More precisely, we have the following.
\begin{enumerate}
\item
If $M$ is not viable because some step is covered twice, then $\fbar_n (M)$ and $X_n (M)$ are both zero.
\item
If all $f_n$ are balanced, and $M$ is not viable because of an idempotent mismatch, then $f_n (M)$ and $X_n (M)$ are both zero.
\end{enumerate}
\end{lem}
%(Note this lemma applies even if the $f_n$ are not diagrammatically simple.)

\begin{proof}
First suppose $M$ has some step covered twice. As $X_n$ preserves H-grading, and there are no tight diagrams with such H-grading, $X_n (M)=0$. As $f_n$ preserves H-grading, $f_n (M)$ is a sum of non-viable diagrams, hence lies in $\F$, so $\fbar_n (M) = 0$.

We now show (ii) by induction on $n$. When $n=1$ there is nothing to prove.
Suppose the result is true for all $f_i$ with $i<n$; now assume $M$ is mismatched, and we will show $f_n (M) = 0$.

Consider $U_n (M)$. In a term of the form $f_j (M_1 \otimes \cdots \otimes M_j) f_{n-j} (M_{j+1} \otimes \cdots \otimes M_n)$, if the mismatch occurs within $M_1 \otimes \cdots M_j$ or $M_{j+1} \otimes \cdots \otimes M_n$ then by induction the term is zero; otherwise it occurs between $M_j$ and $M_{j+1}$, in which case the product is zero. In a term of the form $f_{n-j+1} ( M_1 \otimes \cdots \otimes M_k \otimes X_j (M_{k+1} \otimes \cdots \otimes M_{k+j} ) \otimes \cdots \otimes M_n)$, if the mismatch occurs within $M_{k+1} \otimes \cdots \otimes M_{k+j}$ then the $X_j$ term is zero, hence the whole term is zero; otherwise it occurs within the $f_{n-j+1}$ term and again we have zero. Thus $U_n (M) = 0$, so $X_n (M)=0$. Then $\partial f_n (M) = (f_1 X_n - U_n)(M) = 0$, and since $f_n$ is balanced then $f_n (M) = 0$ also.
\end{proof}

\subsection{Equivalent choices of maps}
\label{sec:choices_of_maps}

In the proof of theorem \ref{thm:main_thm}, we saw that although there might be many choices available for the $f_n$ on tight summands, such choices had no effect on the resulting $X_n$. 

In a similar vein, we now show that, in applying Kadeishvili's construction in general (i.e. without creation operators), the choices available for the maps $\fbar_n$ do not depend on any previous choices.
\begin{lem}
\label{lem:F_persistence}
Suppose that $f_i$ are defined for all $i < n$, $U_i$ and $X_i$ are defined for all $i \leq n$, and that two functions $a, b \colon H^{\otimes n} \To A$ satisfy
\[
\overline{a} = \overline{b}
\quad \text{and} \quad
\partial a = \partial b = f_1 X_n - U_n.
\]
Whether we choose $f_n = a$ or $b$, for all $N>n$, the choices for each $\fbar_N$ are identical.
\end{lem}
Let us be more explicit. Taking $f_n = a$ we define $U$ and $X$ maps at level $n+1$, which we denote $U_{n+1}^a, X_{n+1}^a$. Then we have a set of choices $\mathcal{S}_{n+1}^a = \{ \fbar \; \mid \; \partial f = f_1 X_{n+1}^a - U_{n+1}^a \}$ for $\fbar_{n+1}$. On the other hand, taking $f_n = b$ we define $U_{n+1}^b, X_{n+1}^b$ and have another set of choices $\mathcal{S}^b_{n+1} = \{ \fbar \; \mid \; \partial f = f_1 X^b_{n+1} - U^b_{n+1} \}$ for $\fbar_{n+1}$. Lemma \ref{lem:F_persistence} says that $\mathcal{S}^a_{n+1} = \mathcal{S}^b_{n+1}$. Moreover, after taking $f_n = a$ and making arbitrary choices $f_{n+1}^a, \ldots, f_{N-1}^a$ using Kadeishvili's construction, obtaining maps $U_{n+1}^a, X_{n+1}^a, \ldots, U_{N}^a, X_{N}^a$, we obtain a set of choices $\mathcal{S}^a_N = \{ \fbar \; \mid \; \partial f = f_1 X_N^a - U_N^a \}$ for $\fbar_N$; after taking $f_n = b$ and making arbitrary choices $f_{n+1}^b \ldots, f_{N-1}^b$ and obtaining maps $U_{n+1}^b, X_{n+1}^b, \ldots, U_N^b, X_N^b$, we have choices $\mathcal{S}^b_N = \{ \fbar \; \mid \; \partial f = f_1 X_N - U_N \}$ for $\fbar_N$. Lemma \ref{lem:F_persistence} says, more generally, that these too are equal: $\mathcal{S}^a_N = \mathcal{S}^b_N$. 

\begin{proof}
Let $M$ have H-data $(h,s,t)$ (definition \ref{def:H-data_tensor_product}).
When $(h,s,t)$ is not viable (i.e. covers some step at least twice: definition \ref{def:viability}), there is only one choice for $\fbar_n (M)$, namely $0$, by lemma \ref{lem:fn_Xn_H-data_viability}. And when $(h,s,t)$ is singular, $\fbar_n (M) = 0$ as there are no available diagrams. Hence we need only consider $\fbar_n (M)$ when $(h,s,t)$ is viable and non-singular, hence tight or twisted (proposition \ref{prop:homology_summands}, definition \ref{def:tightness_H-data}).

Since $\overline{a} = \overline{b}$, $a - b$ takes values in $\F$. As $\F$ is an ideal, $U_{n+1}^a (M)$ and $U^b_{n+1} (M)$ differ by values in $\F$. Diagrams in $\F$ do not contribute to homology, as they have crossings, so $[U^a_{n+1} (M)] = [U^b_{n+1} (M)]$. It follows that $X^a_{n+1} (M)= X^b_{n+1}(M)$; we simply write $X_{n+1}(M)$ in either case. Moreover, $U^a_{n+1}(M) - U^b_{n+1}(M)$ is a boundary; let $U^a_{n+1}(M) - U^b_{n+1}(M) = \partial g_{n+1}$. As $U^a_{n+1}(M) - U^b_{n+1}(M) \in \F$, there is such a $g_{n+1}$ in $\F$: if $(h,s,t)$ is twisted we can use a creation operator; if $(h,s,t)$ is tight then any crossing occurs at an all-on doubly occupied pair, so any diagram with crossings lies in $\F$. Then to define $f_{n+1}(M)$ we must solve
\[
f_1 X_{n+1}(M) - U^a_{n+1}(M) = \partial f^a_{n+1}(M)
\quad \text{or} \quad
f_1 X_{n+1}(M) - U^b_{n+1}(M) = \partial f^b_{n+1}(M).
\]
Now observe that $f^a_{n+1}(M)$ is a solution of the first equation iff $f^a_{n+1}(M) + g_{n+1}$ is a solution of the second equation. As $g_{n+1} \in \F$, the possible $\fbar^a_{n+1}(M)$ and $\fbar^b_{n+1}(M)$ are identical.

Thus, the possible choices for $f_{n+1}$ differ by values in $\F$. The possible choices for $U_{n+2}$ then differ by values in $\F$, and the argument proceeds by induction, giving the desired result.
\end{proof}

Thus in the construction of the maps $f_n$ and $X_n$, it is sufficient to consider $\fbar$ rather than $f$ at each level. So we may effectively compute in $\A/\F = \AAbar$.

\subsection{First properties of A-infinity operations}
\label{sec:first_properties_A-infinity}

We now have some preliminary results giving some description of $\fbar_n$ and $X_n$.

\begin{lem}
\label{lem:fn_crossings}
Suppose all $f_k$ are balanced. For any $n \geq 2$ and $M = M_1 \otimes \cdots \otimes M_n$, $f_n (M)$ is a sum of crossed diagrams.
\end{lem}
This includes a sum of no crossed diagrams, when $f_n (M) = 0$.
%(Note the above lemma applies even if the $f_n$ are not diagrammatically simple.)

\begin{proof}
Recall $f_n$ preserves H-data and has Maslov grading $n-1 \geq 1$. For fixed H-data, Maslov grading is (up to a constant) given by the number of matched pairs with crossings (section \ref{sec:classification_local_diagrams}). So each diagram in $f_n (M)$ must contain at least one crossing.
\end{proof}

\begin{lem}
\label{lem:Xn_directly}
Suppose all $f_k$ are balanced. Then $X_n (M)$ is represented by the sum of all crossingless diagrams in the sum
\[
\sum_{j=1}^{n-1} \fbar_j (M_1 \otimes \cdots \otimes M_j) \fbar_{n-j} (M_{j+1} \otimes \cdots \otimes M_n),
\]
where we write elements of $\AAbar$ in standard form.
\end{lem}
Recall (definition \ref{defn:Abar}(v)) that the standard form of an element of $\AAbar$ is a sum of viable diagrams without crossed doubly occupied pairs. For $n=1$ the result reduces to $X_1 = 0$.

\begin{proof}
By construction, $X_n (M)$ is the homology class of $U_n (M)$. Consider the terms of (\ref{eqn:Un_def}) defining $U_n (M)$. Terms of the form $f_{\bullet} (M_1 \otimes \cdots \otimes X_\bullet ( \cdots ) \otimes \cdots \otimes M_n)$ only contain crossed diagrams (lemma \ref{lem:fn_crossings}) hence do not  contribute to homology. Thus, $X_n (M)$ is represented by the sum of tight diagrams of the form $f_j (M_1 \otimes \cdots \otimes M_j) f_{n-j} (M_{j+1} \otimes \cdots \otimes M_n)$. Since diagrams in the ideal $\F$ do not contribute to homology, $X_n$ is represented by the sum claimed in standard form.
\end{proof}

Lemma \ref{lem:Xn_directly} allows us to calculate $X_n (M)$ directly from $\fbar_j$ and $\fbar_{n-j}$. Diagrams in $\fbar_j$ or $\fbar_{n-j}$ usually contain crossings (lemma \ref{lem:fn_crossings}), but the crossings may multiply out to give a tight result. Such $\fbar_j \otimes \fbar_{n-j}$ are sublime. Sublimation is therefore ubiquitous in the operations $X_n$, arising in any nonzero $X_n (M)$.

\subsection{Conditions for nontrivial A-infinity operations}
\label{sec:nontrivial_conditions}

We now find some necessary conditions for $\fbar_n$ or $X_n$ to be nonzero.

\begin{thm}
\label{thm:fn_structure}
Suppose all $f_k$ are balanced. Let $n \geq 2$, let $M_1, \ldots, M_n$ be nonzero homology classes of tight diagrams, and let $M = M_1 \otimes \cdots \otimes M_n$. If $\fbar_n (M) \neq 0$, then the following statements hold.
\begin{enumerate}
\item
In $M$ there are $l$ matched pairs which are twisted, and $m$ matched pairs which are critical, where $l+m \geq n-1$ and $m \leq n-2$. All other matched pairs are tight.
\item
$\fbar_n (M)$ is represented by a sum of diagrams, where each diagram $D$ satisfies the following conditions.
\begin{enumerate}
\item All $m$ of the critical matched pairs in $M$ become tight in $D$.
\item Precisely $n-m-1$ of the $l$ twisted matched pairs in $M$ become crossed in $D$; the other $l-n+m+1$ twisted matched pairs in $M$ remain twisted in $D$.
\item All tight matched pairs in $M$ remain tight in $D$.
\end{enumerate}
\end{enumerate}
\end{thm}

Note in particular that the conditions in (i) imply that $l>0$, so the H-data of $M$ is twisted. Hence if $M$ has tight H-data then $\fbar_n (M) = 0$.

\begin{proof}
First, by lemma \ref{lem:fn_Xn_H-data_viability}, $M$ is viable. %(Note this lemma requires all $f_k$ to be balanced.)

Write $\fbar_n (M)$ in standard form (definition \ref{defn:Abar}(v)), as a sum of distinct diagrams without crossed doubly occupied pairs. Let $D$ be one of these diagrams. As $\fbar_n$ respects H-grading and has Maslov grading $n-1$, $D$ is viable, with $h(D) = h(M)$, and $\iota(D) = \iota(M) + n-1$. From tables \ref{tbl:local_diagrams} and \ref{tbl:local_tensor_products}, at each matched pair the Maslov grading can increase by at most $1$; hence there are precisely $n-1$ matched pairs at which $D$ has a higher Maslov index than $M$.

At each matched pair $P$, $D$ must give a viable local diagram which respects local H-data. There are no such diagrams for singular pairs; hence all matched pairs of $M$ are tight, twisted, or critical.

Consider a matched pair $P$ where $M$ is critical. From table \ref{tbl:local_tensor_products}, $P$ is sesqui-occupied or doubly occupied by $M$. Every all-on doubly occupied pair must remain uncrossed in $D$ (by assumption of standard form). From table \ref{tbl:local_diagrams}, any viable local diagram at a sesqui-occupied or doubly occupied matched pair, which is not crossed all-on doubly occupied, must be tight. So $D_P$ is tight at $P$, and (again by reference to the table) $\iota(D_P) = \iota(M_P) + 1$.

Now consider a matched pair $P$ where $M$ is tight. Then the local H-data of $M$ at $P$ is tight. We observe from table \ref{tbl:local_diagrams} that, with crossed all-on doubly occupied local diagrams ruled out, any viable local diagram with tight H-data must be tight. Thus $D_P$ is tight, and hence $\iota(D_P) = \iota(M_P)$.

Now $\iota(D) = \iota(M) + n-1$, and $m$ of this increase is accounted for at critical matched pairs. The remaining increase of $n-1-m$ must arise at the $l$ pairs where $M$ is twisted. From table \ref{tbl:local_tensor_products}, we observe that these are precisely the pairs where $M$ is all-on once occupied. At such pairs, two viable local diagrams are possible: a tight and a crossed diagram. Crossings can thus be inserted at such pairs to increase the Maslov index; they must be inserted at $n-1-m$ such pairs for $D$ to have the correct Maslov index; so $n-1-m \leq l$. The remaining $l+m-n+1$ pairs must remain twisted in $D$. 

The diagram $D$ thus has precisely $n-m-1$ crossings. But by lemma \ref{lem:fn_crossings}, $D$ must have at least one crossing. Thus $n-m-1 \geq 1$.
\end{proof}

\begin{thm}
\label{thm:Xn_structure}
Suppose all $f_k$ are balanced. If $X_n (M) \neq 0$, then the following statements hold.
\begin{enumerate}
\item
$M$ has precisely $n-2$ critical matched pairs, and all other matched pairs are tight.
\item
$X_n (M)$ is the unique homology class of tight diagram with the H-data of $M$.
\end{enumerate}
\end{thm}

In particular, if $X_n (M) \neq 0$, then $M$ has tight H-data.

For $n=1$ this result just says $X_1 = 0$; for $n=2$, $X_2$ being multiplication, it follows from lemmas \ref{lem:homology_tensor_tightness_characterisation} and \ref{lem:homology_tensor_tightness_local}.

\begin{proof}
By lemma \ref{lem:fn_Xn_H-data_viability}, $M$ is viable. Since $X_n$ respects H-data, let $M$ and $X_n (M)$ have H-data $(h,s,t)$. Since there is at most one nonzero homology class with fixed H-data, (ii) follows immediately.

As $X_n (M) \in \HH(h,s,t)$ is nonzero, the H-data $(h,s,t)$ is tight (proposition \ref{prop:homology_summands}; definition \ref{def:tightness_H-data}), and thus tight at each matched pair (lemma \ref{lem:local-global_H-data}). In particular, $M$ has no 11 once occupied or 00 alternately occupied pairs. From proposition \ref{prop:homology_tensor_product_classification} and table \ref{tbl:local_tensor_products}, $M$ is tight or critical at each pair.

When $M$ is tight at a pair $P$, $M_1 \cdots M_n$ is tight at $P$ (lemma \ref{lem:homology_tensor_tightness_characterisation}). As $X_n (M)$ is given at $P$ by the unique tight diagram with the H-data of $M$ then $X_n (M)_P = (M_1 \cdots M_n)_P$. In this case $X_n (M)$ has the same Maslov index as $M$ at $P$.

On the other hand, when $M$ is critical at $P$, $X_n (M)$ must still be given at $P$ by the unique tight diagram with the same H-data. Inspecting table \ref{tbl:local_tensor_products} (and recalling that multiplication in $\HH$ has zero Maslov grading), we observe that $\iota(X_n (M)_P) = \iota(M_P) + 1$.

Since Maslov index is additive over matched pairs (section \ref{sec:local_diagrams}), the difference in Maslov indices is given by the number of matched pairs at which $M$ is critical. Since $X_n$ has Maslov index $n-2$, $M$ has precisely $n-2$ critical pairs.
\end{proof}

When the maps $f_n$ and $X_n$ are defined from a pair ordering, as in corollary \ref{cor:main_cor} or theorem \ref{thm:first_thm}, theorems \ref{thm:fn_structure} and \ref{thm:Xn_structure} respectively yield parts (i) and (ii) of theorem \ref{thm:second_thm}.

Now we show that it's not possible to have $\fbar_n$ nonzero simultaneously with $X_n$, and more.

%We can now easily make the following observation.
%\begin{lem}
%\label{lem:fn_Xn}
%Let $n \geq 2$ and suppose all $f_k$ for $2 \leq k \leq n$ are balanced. Let $M_1, \cdots, M_n$ be homology classes of tight diagrams. If $\fbar_n (M_1 \otimes \cdots \otimes M_n) \neq 0$, then $X_n (M_1 \otimes \cdots \otimes M_n) = 0$ and $M_1 \cdots M_n = 0$.
%\end{lem}

%\begin{proof}
%From lemma \ref{lem:fn_structure}, $M_1 \otimes \cdots \otimes M_n$ has at least $n-1 \geq 1$ matched pairs which are overtwisted or critical. Such matched pairs are not tight, so $M_1 \cdots M_n = 0$.

%Also from lemma \ref{lem:fn_structure}, $M_1 \otimes \cdots \otimes M_n$ has at most $n-2$ critical matched pairs. Hence there is at least one overtwisted matched pair. But by theorem \ref{thm:Xn_structure}, if $X_n (a_1 \otimes \cdots \otimes a_n) \neq 0$ then all matched pairs are critical or tight. Thus $X_n = 0$.
%\end{proof}

\begin{lem}
\label{lem:fn_Xn}
Suppose that for all $k \geq 1$, $f_k$ is balanced. Let $n \geq 2$. If $X_n (M) \neq 0$ or $M_1 \cdots M_n \neq 0$ then $\fbar_n (M) = 0$.
\end{lem}
Contrapositively, if $\fbar_n (M) \neq 0$ then $X_n (M) = 0$ and $M_1 \cdots M_n = 0$.

\begin{proof}
Let $M$ have H-data $(h,s,t)$. 
By the comment after theorem \ref{thm:Xn_structure}, if $X_n (M) \neq 0$, then $(h,s,t)$ is tight. And if $M_1 \cdots M_n \neq 0$, then $M$ is tight, so again $(h,s,t)$ is tight. But by the comment after theorem \ref{thm:fn_structure}, if $(h,s,t)$ is tight then $\fbar_n (M) = 0$.

%First suppose $X_n (M) \neq 0$. Since $X_n$ respects H-data, $X_n (M)$ is the unique nonzero homology class with H-data $(h,s,t)$. Each of $f_1 X_n (M)$ and $U_n (M)$ is the sum of an odd number of tight diagrams representing this homology class. Thus $\partial f_n (M) = (f_1 X_n - U_n)(M)$ is the sum of an even number of tight diagrams all representing $X_n (M)$. By lemma \ref{lem:even_num_reps}(ii) then $f_n (M) \in \F$, so $\fbar_n (M) = 0$.

%Now suppose $M_1 \cdots M_n \neq 0$. Then $M$ is tight, hence tight at every matched pair (lemmas \ref{lem:homology_tensor_tightness_characterisation}, \ref{lem:homology_tensor_tightness_local}). Now $f_n (M)$ has the same H-data as $M_1 \cdots M_n$, and Maslov grading greater by $n-1$. Hence each diagram in $f_n (M)$ has crossings at $n-1 \geq 1$ pairs (section \ref{sec:classification_local_diagrams}); but since its H-data is tight, inspecting table \ref{tbl:local_diagrams}, each crossing is at an all-on doubly occupied pair. Thus $f_n (M) \in \F$ and $\fbar_n (M) = 0$.
\end{proof}

Theorem \ref{thm:Xn_structure} places stringent necessary conditions on a tensor product $M_1 \otimes \cdots \otimes M_n$ to yield a nonzero result under $X_n$. However, we will see in section \ref{sec:nec_not_suff} that these conditions are not sufficient.

%In other words, it is possible to have an A-infinity structure $X$ on $\HH$, and homology classes of tight diagrams $M_1, \ldots, M_n$, with $M = M_1 \otimes \cdots \otimes M_n$ viable, with precisely $n-2$ critical matched pairs and all other matched pairs tight, and yet have $X_n (M) = 0$. This is possible even when the $H$-data of $M$ is tight.

\subsection{Levels 1, 2 and 3 in general}
\label{sec:low_levels}

We now describe some properties of the maps $f_n, X_n$ at levels 1, 2 and 3. Unlike the discussion of section \ref{sec:low-level_ex}, we do not assume the $A_\infty$ structure derives from creation operators, as in section \ref{sec:construction} and theorem \ref{thm:main_thm} and corollary \ref{cor:main_cor}. We assume that Kadeishvili's construction is used, and the maps $U_n, X_n, f_n$ preserve H-data  and have appropriate Maslov gradings. We will additionally now assume that the $f_n$ are balanced.

Level 1 is again straightforward. By construction $X_1 = 0$, and $f_1$ is a cycle selection homomorphism. If $f_1$ is diagrammatically simple (definition \ref{def:diagrammatically_simple}) then it arises from a cycle choice function (lemma \ref{lem:cycle_selection_classification}). In general, for each tight $(h,s,t)$, $f_1$ maps $M_{h,s,t}$ to the sum of an odd number of tight diagrams representing $M_{h,s,t}$. %(There are $2^L$ such tight diagrams, where $L$ is the number of all-on doubly occupied matched pairs in $(h,s,t)$: lemma \ref{lem:selecting_occupied}.) %Any two choices for $f_1$ differ by a homomorphism $\HH \To \A$ with image in $\partial \F$.

Now consider level 2; let $M = M_1 \otimes M_2$. By construction $X_2$ is multiplication.  As for $f_2$, we have the following.

\begin{lem}
\label{lem:f2_description}
Suppose that $f_1$ and $f_2$ are balanced. Then $\fbar_2 (M) \neq 0$ if and only if $M$ is viable and has at least one twisted matched pair. Then $\fbar_2 (M)$ in standard form is the sum of an odd number of diagrams, each with a single crossing at a matched pair where $M$ is twisted, and elsewhere tight or twisted in agreement with $M$.
\end{lem}

\begin{proof}
If $\fbar_2 (M) \neq 0$, then $M$ is viable (lemma \ref{lem:fn_Xn_H-data_viability}), and theorem \ref{thm:fn_structure} says that $M$ has no critical matched pairs (impossible at level 2 in any case: lemma \ref{lem:it_takes_3}), and at least one twisted (i.e. all-on once occupied) matched pair. Moreover, $\fbar_2 (M)$ is represented by a sum of diagrams, each of which has precisely one crossing at a twisted matched pair. Also $M_1 M_2 = 0$ (from the twisted matched pair, or lemma \ref{lem:fn_Xn}).

Conversely, suppose $M = M_1 \otimes M_2$ is viable and has at least one twisted matched pair. Then $X_2 (M) = 0$, and $U_2 (M) = f_1 (M_1) f_1 (M_2)$ is the sum of an odd number of diagrams, each tight and twisted at the same matched pairs as $M$; the diagrams differ only by strand switching at all-on doubly occupied pairs (section \ref{sec:diagrams_representing_homology}). Since $\partial f_2 (M) = f_1 (M_1) f_1 (M_2)$ and $f_2$ respects H-data and has Maslov grading $1$, $f_2 (M)$ is a sum of diagrams, each of which has a crossing at precisely one matched pair, and at each other pair is tight or twisted in agreement with $M$. In standard form (definition \ref{defn:Abar}) $\fbar_2 (M)$ is then given by omitting diagrams with crossings at doubly occupied pairs, so that each crossing is at a matched pair where $M$ is twisted. The differential of each remaining diagram is a single diagram, but the differential of each omitted diagram is a sum of two diagrams, so $\fbar_2 (M)$ in standard form is the sum of an odd number of diagrams.
\end{proof}

When the $A_\infty$ structure is defined by creation operators, as in section \ref{sec:construction} and theorem \ref{thm:main_thm}, then any nonzero $f_2 (M)$ is a single diagram, as described in section \ref{sec:low-level_ex}.
%given by equation (\ref{eqn:f2_description}) or (\ref{eqn:nonzero_f2_description}) from 

We now consider $X_3$; the case is illustrative, showing the role of critical and sublime tensor products of diagrams. Let $M = M_1 \otimes M_2 \otimes M_3$ and suppose $X_3 (M) \neq 0$. Then $M$ is viable (lemma \ref{lem:fn_Xn_H-data_viability}) and by theorem \ref{thm:Xn_structure}, $M$ has precisely one critical matched pair; all other matched pairs are tight. By lemma \ref{lem:Xn_directly}, $X_3 (M)$ is represented by the sum of all crossingless diagrams in
\[
\fbar_1 (M_1) \fbar_2 (M_2 \otimes M_3) + \fbar_2 (M_1 \otimes M_2) \fbar_1 (M_3).
\]
Each diagram in an $\fbar_2 (M_i \otimes M_{i+1})$ term has a crossing at precisely one matched pair $P$, where $M_i \otimes M_{i+1}$ is twisted; since such $P$ cannot be tight in $M$ (lemma \ref{lem:local_tightness_homology_sub-tensor-product}, table \ref{tbl:local_homology_sub-tensor-product_tightness}), $P$ is the critical matched pair of $M$. Multiplying this diagram by the third $M_j$ must then produce a tight diagram. There are two cases for the tightness of the various tensor products:
\begin{itemize}
\item $M_1 \otimes M_2$ twisted; each diagram $D$ in $\fbar_2 (M_1 \otimes M_2)$ crossed; $M_3$ and each diagram $D'$ in $\fbar_1 (M_3)$ tight; each $D \otimes D'$ sublime; and $M_1 \otimes M_2 \otimes M_3$ critical.
\item $M_2 \otimes M_3$ twisted; each diagram $D'$ in $\fbar_2 (M_2 \otimes M_3)$ crossed; $M_1$ and each diagram $D$ in $\fbar_1 (M_1)$ tight; each $D \otimes D'$ sublime; and $M_1 \otimes M_2 \otimes M_3$ critical.

In this case, the situation at $P$ is shown in figure \ref{fig:X3_effect}.
\end{itemize}
We note that these two cases are mutually exclusive: only one of the terms $f_2 (M_1 \otimes M_2) f_1 (M_3)$ or $f_1 (M_1) f_2 (M_2 \otimes M_3)$ can be nonzero. For instance, in the first case $M_2 \otimes M_3$ is singular, and in the second case $M_1 \otimes M_2$ is singular.

Thus, to obtain a nonzero result for $X_3$, we start with a critical tensor product $M_1 \otimes M_2 \otimes M_3$; then a twisted sub-tensor-product (i.e. $M_2 \otimes M_3$ in figure \ref{fig:X3_effect}) combines via $f_2$ into a crossed diagram, yielding a sublime tensor product (i.e. $f_1 (M_1) \otimes f_2 (M_2 \otimes M_3)$ in figure \ref{fig:X3_effect}); and then these are multiplied to give a tight result. This process
%, with tensor products at a matched pair progressing from critical, to twisted, to sublime, to tight, 
is the process depicted in figure \ref{fig:A-infinity-reordering}; it occurs generally in Kadeishvili's construction, without any need for creation operators.

\begin{figure}
\begin{center}
\begin{tikzpicture}[scale=1.5]
\draw (-3, 0.75) node {$X_3 \Big( \begin{array}{ccccccc}
\off & p_- & \on & p'_+ & \off & p'_- & \on \end{array} \Big) = $};
\draw (-0.8,0.75) node {$X_3 \Bigg($};
\strandbackgroundshading
\tstrandbackgroundshading{1}
\tstrandbackgroundshading{2}
\beforevused
\tafterwused{1}
\tbeforewused{2}
\strandsetupn{p}{p'}
\tstrandsetup{1}
\tstrandsetup{2}
\leftoff
\righton
\trightoff{1}
\trighton{2}
\used
\tusea{1}
\tuseb{2}
\draw (3.4, 0.75) node {$\Bigg) =$};
\begin{scope}[xshift=4.2 cm]
\strandbackgroundshading
\beforewused
\afterwused
\beforevused
\strandsetupn{p}{p'}
\leftoff
\righton
\useab
\used
\end{scope}

\end{tikzpicture}
\caption{An example of $X_3 (M_1 \otimes M_2 \otimes M_3)$, where $M_1 \otimes M_2 \otimes M_3$ is critical, $M_1 \otimes M_2$ is singular, and $M_2 \otimes M_3$ is twisted. Moreover, $\fbar_2 (M_2 \otimes M_3)$ is crossed, and $\fbar_1 (M_1) \otimes \fbar_2 (M_2 \otimes M_3)$ is sublime.}
\label{fig:X3_effect}
\end{center}
\end{figure}

We can prove a converse, and give necessary and sufficient conditions for $X_3 \neq 0$.
\begin{prop}
\label{prop:X3_description}
Suppose $f_1$ and $f_2$ are balanced. Then $X_3 (M)$ is nonzero if and only if $M$ is viable, critical at precisely one matched pair $P$, and tight at all other matched pairs.
\end{prop}

\begin{proof}
We only need prove that if the conditions on $M$ hold, then $X_3 (M) \neq 0$. By lemma \ref{lem:homology_contractions}, $M_P$ is an extension of a tensor product shown in the critical column of table \ref{tbl:local_tensor_products}; but $M_P$ has 3 factors, so $M_P$ is exactly the critical 01 pre-sesqui-occupied or 10 post-sesqui-occupied tensor product shown there. We consider the first case; the second case is similar.

So suppose $M_P$ is critical 01 pre-sesqui-occupied. We first observe that $\fbar_1 (M_1)$ is the sum of an odd number of tight diagrams, all representing $M_1$, and differing by strand switching at all-on doubly occupied pairs (lemma \ref{lem:selecting_occupied}).

By lemma \ref{lem:f2_description}, $\fbar_2 (M_2 \otimes M_3)$ is the sum of an odd number of diagrams, each with a single crossing at a pair where $M_2 \otimes M_3$ is twisted. Let $D$ be one of these diagrams; let its crossing be at the a pair $P'$. By lemma \ref{lem:local_tightness_homology_sub-tensor-product} and table \ref{tbl:local_homology_sub-tensor-product_tightness} then $P'$ cannot be tight in $M$, so $P'=P$. Thus $\fbar_2 (M_2, \otimes M_3)$ is represented by the sum of an odd number of diagrams, each of which has a crossing at $P$ and is tight elsewhere. These diagrams differ by strand switching at all-on doubly occupied pairs.

Then $\fbar_1 (M_1) \otimes \fbar_2 (M_2 \otimes M_3)$ is the sum of an odd number of sublime tensor products of diagrams, and $\fbar_1 (M_1) \fbar_2 (M_2 \otimes M_3)$ is the sum of an odd number of tight diagrams. Moreover, in this case $M_1 \otimes M_2$ is singular so $\fbar_2 (M_1 \otimes M_2) = 0$.

%Similarly, in the second case $\fbar_2 (M_1 \otimes M_2)$ is crossed, $\fbar_1 (M_1)$ is tight, $\fbar_2 (M_1 \otimes M_2) \otimes \fbar_1 (M_3)$ is sublime, and $\fbar_2 (M_1 \otimes M_2) \fbar_1 (M_3)$ is tight. Further $\fbar_2 (M_2 \otimes M_3) = 0$.

%In either case, 

Thus $\fbar_1 (M_1) \fbar_2 (M_2 \otimes M_3) + \fbar_2 (M_1 \otimes M_2) \fbar_1 (M_3)$ is the sum of an odd number of tight diagrams, all related by strand switching. By lemma \ref{lem:Xn_directly} $X_3 (M)$ is the homology class of any one (or the sum of all) of these diagrams. Thus $X_3 (M) \neq 0$.
\end{proof}

Thus, if there are sufficiently few critical matched pairs in $M$, we may be able to guarantee that $X_n (M) \neq 0$. In section \ref{sec:nontrivial_higher_operations} we give some results in this direction, giving sufficient conditions for $X_n$ and $\fbar_n$ to be nonzero.

\section{Further examples and computations}
\label{sec:examples}

We now calculate some further examples and prove some further results, for low-level $A_\infty$ maps.

Recall the previous calculations along these lines. In section \ref{sec:low-level_ex} we discussed the operations $f_n$ and $X_n$ for $n \leq 2$, when $A_\infty$ operations are defined by cycle choice and creation choice functions, as in theorem \ref{thm:main_thm}. And in section \ref{sec:low_levels} we again discussed low-level maps, especially $\fbar_2$ and $X_3$, for $A_\infty$ operations obtained more generally using Kadeishvili's method.

In this section we consider $A_\infty$ operations defined by a pair ordering $\preceq$, as in corollary \ref{cor:main_cor}, and consider maps at level 3 and 4, using the shorthand notation of section \ref{sec:shorthand}.

As always, let $M = M_1 \otimes \cdots \otimes M_n$ denote a tensor product of nonzero homology classes of diagrams. We assume $M$ is viable, necessary for nonzero results (lemma \ref{lem:fn_Xn_H-data_viability}); let $M$ have H-data $(h,s,t)$. We work with $\fbar_n$ and $\Ubar_n$; this loses no generality for calculating $X_n$ (lemma \ref{lem:F_persistence}).

\subsection{Level 3}
\label{sec:level_3}

Consider the operation $\Ubar_3$, given by
\[
\Ubar_3 (M) = \fbar_1 (M_2) \fbar_2 (M_2 \otimes M_3) + \fbar_2 (M_1 \otimes M_2) \fbar_1 (M_3) + \fbar_2 (M_1 M_2 \otimes M_3) + \fbar_2 (M_1 \otimes M_2 M_3).
\]
As in lemma \ref{lem:Xn_directly} (and seen in section \ref{sec:low_levels}), the last two terms cannot contribute to $X_3$, since they yield crossed diagrams (lemma \ref{lem:fn_crossings}). But in general all four terms can be nonzero; indeed, some terms may be equal and cancel. 
%each diagram has at most one crossing. 
Some examples are shown in figures \ref{fig:U3_eg1} and \ref{fig:U3_eg2}. These examples illustrate shorthand calculations alongside the standard notation. Each $\fbar_2$ is calculated using section \ref{sec:low-level_ex} and equation (\ref{eqn:nonzero_f2_description}).

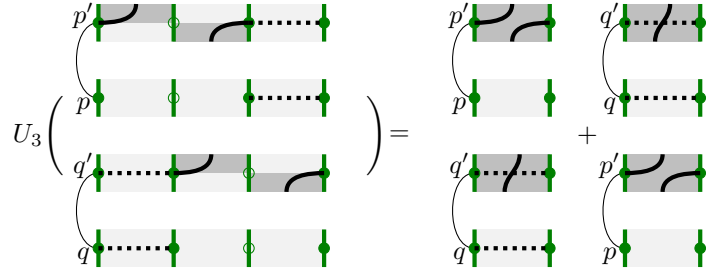
\begin{figure}
\begin{center}
\begin{align*}
U_3 &\Big( \begin{array}{ccccccc}
\on	&	p'_+	&	\off	&	p'_-	&	\on		&				&	\on \\
\on & 			&	\on		&	q'_+	&	\off	&	q'_-	&	\on
\end{array} \Big) \\
&= 
\fbar_1 \left( \begin{array}{ccc}
\on	&	p'_+	&	\off \\
\on &				& \on
\end{array} \right)
\fbar_2 \left( \begin{array}{ccccc}
\off	&	p'_-	&	\on		&				&	\on \\
\on		&	q'_+	&	\off	&	q'_-	&	\on
\end{array} \right)
+
\fbar_2 \left( \begin{array}{ccccc}
\on & p'_+ 	& \off 	& p'_- & 	\on \\
\on & 			& \on		& q'_+	&	\off
\end{array} \right)
\fbar_1 \left( \begin{array}{ccc}
\on		&				&	\on \\
\off	& q'_-	& \on
\end{array} \right) \\
&= 
\left( \begin{array}{ccc}
\on	&	p'_+	&	\off \\
\on &				& \on
\end{array} \right)
\left( \begin{array}{ccc}
\off	&	p'_-		& \on \\
\on		&	c_{q'}	&	\on
\end{array} \right)
+
\left( \begin{array}{ccc}
\on	&	c_{p'}	&	\on \\
\on	&	q'_+		&	\off
\end{array} \right)
\left( \begin{array}{ccc}
\on		&				&	\on \\
\off	&	q'_-	&	\on
\end{array} \right) \\
&= \left( \begin{array}{ccc}
\on		&	w_{p'}	&	\on \\
\on		&	c_{q'}	&	\on
\end{array} \right) 
+
\left( \begin{array}{ccc}
\on		&	c_{p'}	&	\on \\
\on		&	w_{q'}	&	\on
\end{array} \right)
\end{align*}
\end{center}
\begin{center}
\begin{tikzpicture}[scale=1]
\begin{scope}[xshift = -3 cm]
\draw (-0.8,-0.25) node {$U_3 \Bigg( $};
\strandbackgroundshading
\tstrandbackgroundshading{1}
\tstrandbackgroundshading{2}
\afterwused
\tbeforewused{1}
\strandsetupn{p}{p'}
\tstrandsetup{1}
\tstrandsetup{2}
\tstrandsetup{3}
\lefton
\rightoff
\trighton{1}
\trighton{2}
\usea
\tuseb{1}
\tdothorizontals{2}
\draw (3.7, -0.25) node {$\Bigg)$};
\end{scope}
\begin{scope}[xshift = -3 cm, yshift = -2 cm]
\strandbackgroundshading
\tstrandbackgroundshading{1}
\tstrandbackgroundshading{2}
\tafterwused{1}
\tbeforewused{2}
\strandsetupn{q}{q'}
\tstrandsetup{1}
\tstrandsetup{2}
\tstrandsetup{3}
\lefton
\righton
\trightoff{1}
\trighton{2}
\dothorizontals
\tusea{1}
\tuseb{2}
\end{scope}
\begin{comment}
\begin{scope}[xshift = 2 cm, yshift = 0 cm]
\draw (-0.7, -0.25) node {$=$};
\strandbackgroundshading
\tstrandbackgroundshading{1}
\afterwused
\tbeforewused{1}
\strandsetupn{p}{p'}
\tstrandsetup{1}
\tstrandsetup{2}
\lefton
\rightoff
\trighton{1}
\usea
\tuseb{1}
\end{scope}
\begin{scope}[xshift = 2 cm, yshift = -2 cm]
\strandbackgroundshading
\tstrandbackgroundshading{1}
\tafterwused{1}
\tbeforewused{1}
\strandsetupn{q}{q'}
\tstrandsetup{1}
\tstrandsetup{2}
\lefton
\righton
\trighton{1}
\dothorizontals
\tuseab{1}
\tdothorizontals{1}
\end{scope}
\begin{scope}[xshift = 5 cm, yshift = 0 cm]
\draw (-0.5, -0.25) node {$+$};
\strandbackgroundshading
\tstrandbackgroundshading{1}
\beforewused
\afterwused
\strandsetupn{p}{p'}
\tstrandsetup{1}
\tstrandsetup{2}
\lefton
\righton
\trighton{1}
\dothorizontals
\useab
\tdothorizontals{1}
\end{scope}
\begin{scope}[xshift = 5 cm, yshift = -2 cm]
\strandbackgroundshading
\tstrandbackgroundshading{1}
\afterwused
\tbeforewused{1}
\strandsetupn{q}{q'}
\tstrandsetup{1}
\tstrandsetup{2}
\lefton
\rightoff
\trighton{1}
\usea
\tuseb{1}
\end{scope}
\end{comment}
\begin{scope}[xshift = 2 cm, yshift = 0 cm] %8.5 
\draw (-1, -0.25) node {$=$};
\strandbackgroundshading
\beforewused
\afterwused
\strandsetupn{p}{p'}
\tstrandsetup{1}
\lefton
\righton
\usea
\useb
\end{scope}
\begin{scope}[xshift = 2 cm, yshift = -2 cm] %8.5
\strandbackgroundshading
\beforewused
\afterwused
\strandsetupn{q}{q'}
\tstrandsetup{1}
\lefton
\righton
\useab
\dothorizontals
\end{scope}
\begin{scope}[xshift = 4 cm, yshift = 0 cm] %10.5
\draw (-0.5, -0.25) node {$+$};
\strandbackgroundshading
\beforewused
\afterwused
\strandsetupn{q}{q'}
\tstrandsetup{1}
\lefton
\righton
\useab
\dothorizontals
\end{scope}
\begin{scope}[xshift = 4 cm, yshift = -2 cm] %10.5 
\strandbackgroundshading
\beforewused
\afterwused
\strandsetupn{p}{p'}
\tstrandsetup{1}
\lefton
\righton
\usea
\useb
\end{scope}
\end{tikzpicture}
\caption{An example of $U_3 (M_1 \otimes M_2 \otimes M_3)$, where $M_1 \otimes M_2 \otimes M_3$ is twisted. In this case both $M_1 M_2$ and $M_2 M_3$ are zero, and the two terms $f_1 (M_1) f_2 (M_2 \otimes M_3)$ and $f_1 (M_1 \otimes M_2) f_1 (M_3)$ both contribute to $U_3$. The sub-tensor-products $M_1 \otimes M_2$ and $M_2 \otimes M_3$ are also twisted, and the two nonzero terms of $U_3$ respectively apply an $f_2$ to insert a crossing in each. The calculation is in shorthand. The result is also shown in standard notation.}
\label{fig:U3_eg1}
\end{center}
\end{figure}

\begin{figure}
\begin{center}
\begin{align*}
\Ubar_3 \left(
\begin{array}{ccccccc}
\on & p'_+ & \off & & \off & p'_- & \on
\end{array} \right) &=
\fbar_2 \left( \begin{array}{ccccc}
\on	&	p'_+	&	\off	&		&
\end{array} \right)
+
\fbar_2 \left( \begin{array}{ccccc}
\on	&	p'_+	&	\off	&		&
\end{array} \right)
= 0
\end{align*}

\begin{tikzpicture}[scale=1]
\draw (-0.8,0.75) node {$U_3 \Bigg( $};
\strandbackgroundshading
\tstrandbackgroundshading{1}
\tstrandbackgroundshading{2}
\afterwused
\tbeforewused{2}
\strandsetupn{p}{p'}
\tstrandsetup{1}
\tstrandsetup{2}
\tstrandsetup{3}
\lefton
\rightoff
\trightoff{1}
\trighton{2}
\usea
\tuseb{2}
\draw (3.5, 0.75) node {$\Bigg)$};
\begin{scope}[xshift = 5.2 cm]
\draw (-1, 0.75) node {$= \fbar_2 \Bigg( $};
\strandbackgroundshading
\tstrandbackgroundshading{1}
\afterwused
\tbeforewused{1}
\strandsetupn{p}{p'}
\tstrandsetup{1}
\tstrandsetup{2}
\lefton
\rightoff
\trighton{1}
\usea
\tuseb{1}
\draw (2.5,0.75) node {$\Bigg)$};
\end{scope}
\begin{scope}[xshift= 9.3 cm, yshift= 0 cm]
\draw (-1,0.75) node {$+ \fbar_2 \Bigg( $};
\strandbackgroundshading
\tstrandbackgroundshading{1}
\afterwused
\tbeforewused{1}
\strandsetupn{p}{p'}
\tstrandsetup{1}
\tstrandsetup{2}
\lefton
\rightoff
\trighton{1}
\usea
\tuseb{1}
\draw (2.7, 0.75) node {$\Bigg) = 0$};
\end{scope}
\end{tikzpicture}
\caption{An example of $U_3 (M_1 \otimes M_2 \otimes M_3)$. In this case both $f_2 (M_1 \otimes M_2)$ and $f_2 (M_2 \otimes M_3)$ are zero. The two terms $f_2 (M_1 M_2 \otimes M_3)$ and $f_2 (M_1 \otimes M_2 M_3)$ are both nonzero, but cancel out to give zero.}
\label{fig:U3_eg2}
\end{center}
\end{figure}

Continuing to $\fbar_3$, we know that when $\fbar_3 (M) \neq 0$ then $M$ is viable (lemma \ref{lem:fn_Xn_H-data_viability}), $X_3 (M) = 0$ and $M_1 M_2 M_3 = 0$ (lemma \ref{lem:fn_Xn}). Moreover, $M$ has no singular matched pairs, $l$ twisted matched pairs, and $m$ critical matched pairs, where $m \leq 1$ and $l+m \geq 2$ (theorem \ref{thm:fn_structure}). It follows that $l \geq 1$, so $(h,s,t)$ is twisted, so by theorem \ref{thm:main_thm}(iv) then $\fbar_3 (M) = \Abar^*_{\mathcal{CR}^\preceq} \circ \Ubar_3 (M)$, where $A^*_{\mathcal{CR}^\preceq}$ is the creation operator of the creation choice function $\mathcal{CR}^\preceq$ (definition \ref{def:global_creation_operator}) of the pair ordering $\preceq$ (definition \ref{def:choice_functions_from_ordering}). 
\begin{equation}
\label{eqn:f3_4_terms}
\fbar_3 (M)
= \Abar^*_{\mathcal{CR}^\preceq} \Big( \fbar_1 (M_1) \fbar_2 (M_2 \otimes M_3) + \fbar_2 (M_1 \otimes M_2) \fbar_1 (M_3) + \fbar_2 (M_1 M_2 \otimes M_3) + \fbar_2 (M_1 \otimes M_2 M_3) \Big)
\end{equation}
%, and $\Abar^*_{\mathcal{CR}^\preceq}$ is its image in $\AAbar$ (well-defined as discussed in section \ref{sec:creation_operators}).

Each of the four terms in equation (\ref{eqn:f3_4_terms}) consists of at most one diagram in standard form (definition \ref{defn:Abar}). Since diagrams in $\Ubar_3 (M)$ may have a crossing, a diagram in $\fbar_3 (M)$ may have up to two crossings.

Now $M$ has $m \leq 1$ critical matched pairs. If $m=0$ then all pairs are tight or twisted, and any diagram in $\fbar_3$ above has precisely two crossings. If $m=1$, then the critical pair $P$ must eventually have a tight local diagram to yield a nonzero result, so the diagram at $P$ becomes crossed by an $\fbar_2$ and then sublimates; hence any diagram in $\fbar_3$ has one crossing.

%It is not difficult to compute many values of $\fbar_3$. We can effectively focus on twisted or critical matched pairs, since $\fbar_3 (M)$, when nonzero, simply multiplies out twisted pairs; and any singular pair ensures a zero result.

We find that, in order to obtain a nonzero result for $\fbar_3 (M)$, the local diagrams at twisted or critical matched pairs must be ``distributed" across $M_1, M_2, M_3$. For twisted pairs we make this precise in the following statement.

Recall (section \ref{sec:tensor_product_homology}) that if $M$ is twisted at a matched pair $P = \{p,p'\}$, then one place $p$ or $p'$ is occupied, and accordingly $M$ is twisted at $p$ or $p'$. If $M$ is twisted at $p$, then the two steps $p_+, p_-$ around $p$ are covered by some $M_i$ and $M_j$ respectively, with $i<j$, as in the twisted column of table \ref{tbl:local_tensor_products}.
\begin{lem}
\label{lem:f3_rows_same}
Consider an $A_\infty$ structure defined by a pair ordering $\preceq$. 

Suppose $M = M_1 \otimes M_2 \otimes M_3$ is viable, twisted at precisely two places $p,q$ of matched pairs $P = \{p,p'\}$ and $Q = \{q,q'\}$, with all other pairs tight. Moreover, suppose that $p_+, q_+$ are both covered by the same $M_i$, and $p_-, q_-$ are both covered by the same $M_j$.

Then $X_3 (M)$, $\Ubar_3 (M)$ and $\fbar_3 (M)$ are all zero.
\end{lem}

We can denote this result for $\fbar_3$ by
\[
\fbar_3 \left(
\begin{array}{cccccc}
\on & q_+ & \off & q_- & \on  & \on \\
\on & p_+ & \off & p_- & \on  & \on
\end{array} \right) =  
\fbar_3 \left(
\begin{array}{cccccc}
\on & q_+ & \off & \off & q_- & \on  \\
\on & p_+ & \off & \off & p_- & \on
\end{array} \right) =  
\fbar_3 \left(
\begin{array}{cccccc}
\on &  \on & q_+ & \off & q_- & \on \\
\on &  \on & p_+ & \off & p_- & \on
\end{array} \right) = 0.
\]

\begin{proof}
There are three possibilities for $i$ and $j$: $(i,j) = (1,2), (1,3)$ or $(2,3)$. In all cases $X_3 (M) = 0$ as there are no critical matched pairs (theorem \ref{thm:Xn_structure}). Suppose without loss of generality that $P \preceq Q$, so creation operators introduce crossings at $P$ in preference to $Q$.

First suppose $(i,j) = (1,2)$. Then $M_1 M_2 = 0$ (being twisted) and $M_2 M_3 \neq 0$ (being tight), so $\fbar_2 (M_2 \otimes M_3) = 0$ (lemma \ref{lem:fn_Xn}). Thus
\[
%\begin{align*}
\Ubar_3 (M) = \fbar_2 (M_1 \otimes M_2) \fbar_1 (M_3) + \fbar_2 (M_1 \otimes M_2 M_3) %\\
%&= \Abar^*_{\mathcal{CR}^\preceq} \left( \fbar_1 (M_1) \fbar_1 (M_2) \right) \fbar_1 (M_3) + \Abar^*_{\mathcal{CR}^\preceq} \left( \fbar_1 (M_1) \fbar_1 (M_2 M_3) \right)
%\end{align*}
\]
Now $\fbar_2 (M_1 \otimes M_2) = \Abar_{\mathcal{CR}^{\preceq}} \left( \fbar_1 (M_1) \fbar_1 (M_2) \right)$ is (in standard form) the diagram obtained from $\fbar_1 (M_1) \fbar_1 (M_2)$  by inserting a crossing at $P$. Similarly $\fbar_2 (M_1 \otimes M_2 M_3)$ (in standard form) is obtained from $\fbar_1 (M_1) \fbar_1 (M_2 M_3)$ by inserting a crossing at $P$. Since the diagrams $\fbar_2 (M_1 \otimes M_2 M_3)$ and $\fbar_2 (M_1 \otimes M_2) \fbar_1 (M_3)$ have the same H-data, are crossed at $P$, twisted at $Q$, elsewhere tight, and have the same strands at all-on doubly occupied pairs (chosen by the same cycle selection function of $\preceq$), they are equal. Thus $\Ubar_3 (M) = 0$ and $\fbar_3 (M) = \Abar_{\mathcal{CR}^\preceq} \circ \Ubar_3 (M) = 0$.

%In shorthand:
%\begin{align*}
%\Ubar_3 \left(
%\begin{array}{cccccc}
%\on & q_+ & \off & q_- & \on & \on \\
%\on & p_+ & \off & p_- & \on & \on
%\end{array} \right) 
%&=
%\fbar_2 \left( \begin{array}{ccccc}
%\on & q_+ & \off & q_- & \on \\
%\on & p_+ & \off & p_- & \on 
%\end{array} \right)
%\fbar_1 \left( \begin{array}{cc}
%\on & \on \\
%\on & \on 
%\end{array} \right) +
%\fbar_2 \left( \begin{array}{ccccc}
%\on & q_+ & \off & q_- & \on \\
%\on & p_+ & \off & p_- & \on 
%\end{array} \right) \\
%&= 
%\left( \begin{array}{ccc}
%\on & w_q & \on \\
%\on & c_p & \on
%\end{array} \right)
%\left( \begin{array}{cc}
%\on & \on \\
%\on & \on 
%\end{array} \right)
%+
%\left( \begin{array}{ccc}
%\on & w_q & \on \\
%\on & c_p & \on 
%\end{array} \right)
%=
%2 
%\left( \begin{array}{ccc}
%\on & w_q & \on \\
%\on & c_p & \on 
%\end{array} \right)
%= 0
%\end{align*}
The case $(i,j) = (2,3)$ is similar. 

Finally suppose $(i,j) = (1,3)$. Then $M_1 M_2$ and $M_2 M_3$ are nonzero, so $\fbar_2 (M_1 \otimes M_2) = \fbar_2 (M_2 \otimes M_3) = 0$ (lemma \ref{lem:fn_Xn}). The remaining two terms of $\Ubar_3 (M)$ are $\fbar_2 (M_1 M_2 \otimes M_3)$ and $\fbar_2 (M_1 \otimes M_2 M_3)$, both of which are crossed at $p$, twisted at $q$, and equal elsewhere, so again $\Ubar_3$ and $\fbar_3$ are zero.
\end{proof}

The following lemma, together with lemma \ref{lem:f3_rows_same} and the general result of theorem \ref{thm:fn_structure}, completely calculates $\fbar_3 (M)$ when $M$ has two non-tight matched pairs.
\begin{lem}
\label{lem:f3_computations}
Consider an $A_\infty$ structure defined by a pair ordering $\preceq$. 

Suppose $M = M_1 \otimes M_2 \otimes M_3$ is viable, has two matched pairs $P = \{p,p'\} \prec Q = \{q,q'\}$ which are twisted or critical, and all other matched pairs tight, in one of the arrangements depicted below.

Then $\fbar_3 (M)$ is zero or nonzero as shown. If nonzero, it is given by a single diagram in $\AAbar$, with the H-data of $M$, which is crossed at each twisted matched pair of $M$, and elsewhere tight.
\end{lem}

Nonzero:

\begin{gather*}
\left( \begin{array}{c|c|c}
q'_- & q_+ & q_- \\
p_+ & p_- &
\end{array} \right)
\
\left( \begin{array}{c|c|c}
 & q_+ & q_- \\
p_+ & p_- & p'_+
\end{array} \right)
\
\left( \begin{array}{c|c|c}
q_+ & & q_- \\
p'_- & p_+ & p_-
\end{array} \right)
\
\left( \begin{array}{c|c|c}
q_+ & q_- & q'_+ \\
p_+ & & p_-
\end{array} \right)
\
\left( \begin{array}{c|c|c}
q_+ & q_- & \\
 & p_+ & p_-
\end{array} \right)
\\
\left( \begin{array}{c|c|c}
q_+ & q_- &  \\
p'_- & p_+ & p_-
\end{array} \right)
\
\left( \begin{array}{c|c|c}
q_+ & q_- & q'_+ \\
 & p_+ & p_-
\end{array} \right)
\
\left( \begin{array}{c|c|c}
q'_- & q_+ & q_- \\
p_+ & & p_-
\end{array} \right)
\
\left( \begin{array}{c|c|c}
q_+ & & q_- \\
p_+ & p_- & p'_+
\end{array} \right)
\
\left( \begin{array}{c|c|c}
 & q_+ & q_- \\
p_+ & p_- & 
\end{array} \right)
\\
\left( \begin{array}{c|c|c}
q_+ & q_- & \\
p_+ & p_- & p'_+
\end{array} \right)
\
\left( \begin{array}{c|c|c}
q_+ & q_- & \\
p_+ & & p_-
\end{array} \right)
\
\left( \begin{array}{c|c|c}
 & q_+ & q_- \\
p_+ & & p_-
\end{array} \right)
\
\left( \begin{array}{c|c|c}
 & q_+ & q_- \\
p'_- & p_+ & p_-
\end{array} \right)
\end{gather*} \\

Zero:

\begin{gather*}
\left( \begin{array}{c|c|c}
q_+ & q_- & q'_+ \\
p_+ & p_- & 
\end{array} \right)
\
\left( \begin{array}{c|c|c}
q_+ & & q_- \\
p_+ & p_- & 
\end{array} \right)
\
\left( \begin{array}{c|c|c}
q_+ & & q_- \\
 & p_+ & p_-
\end{array} \right)
\
\left( \begin{array}{c|c|c}
q'_- & q_+ & q_- \\
 & p_+ & p_-
\end{array} \right)
\end{gather*}

(Circles denoting idempotents are omitted; they can be inferred since each nontrivial local diagram covers at most one step.)

The conclusion that, if $\fbar_3 (M)$ is nonzero, then it is as claimed, follows purely from grading considerations: $\fbar_3$ has Maslov grading $2$, but the Maslov index can only be increased at non-tight pairs. There are only two non-tight matched pairs, so the Maslov index must be increased by $1$ at each. A twisted pair must become crossed, and a critical pair must become tight.
\begin{proof}
In the cases depicted in the first four diagrams in the first two rows above, we have a critical and a twisted pair, and $M_1 M_2 = M_2 M_3 = 0$. In each of these cases one of $M_1 \otimes M_2$ or $M_2 \otimes M_3$ is singular, and the other is twisted. Then precisely one of $\fbar_2 (M_1 \otimes M_2)$ or $\fbar_2 (M_2 \otimes M_3)$ is nonzero, and $\fbar_2$ introduces a crossing at the twisted matched pair. Then the multiplication $\fbar_2 (M_1 \otimes M_2) \fbar_1 (M_3)$ or $\fbar_1 (M_1) \fbar_2 (M_2 \otimes M_3)$ is tight at one pair and twisted at the other; and in fact this diagram is $\Ubar_3 (M)$. Applying a creation operator, we obtain $\fbar_3 (M)$ as a single diagram with a single crossing.

In the cases depicted at the end of the first and second rows, again $M_1 M_2 = M_2 M_3 = 0$, and both $\fbar_2 (M_1 \otimes M_2)$ and $\fbar_2 (M_2 \otimes M_3)$ are nonzero, each with a single crossed pair. So $\fbar_2 (M_1 \otimes M_2) \fbar_1 (M_3)$ and $\fbar_1 (M_1) \fbar_2 (M_2 \otimes M_3)$ are both nonzero, one crossed at $p$ and twisted at $q$, the other crossed at $q$ and twisted at $p$. The creation operator $\Abar_{\mathcal{CR}^\preceq}$ sends the former to zero, and introduces a crossing at $p$ into the latter. Thus $\fbar_3 (M)$ is given by a single diagram, crossed at both $p$ and $q$, as desired.

The other cases can be calculated by similar reasoning.
\end{proof}

\subsection{Level 4}
\label{sec:nec_not_suff}

We now compute two examples at level 4, illustrating some interesting phenomena. As usual, let $M = M_1 \otimes \cdots \otimes M_n$ denote a tensor product of nonzero homology classes of diagrams, with H-data $(h,s,t)$.

%We still assume the $A_\infty$ structure is defined from a pair ordering.

Our first example shows that the necessary conditions for $X_n$ to be nonzero in theorem \ref{thm:Xn_structure} are not sufficient. It is an $M$ with precisely $2$ critical matched pairs, and all other matched pairs tight --- and in fact one can find a tight diagram with the same H-data --- but with $X_4 (M) = 0$. 

Letting $P = \{p,p'\}$, $Q = \{q,q'\}$ be matched pairs with $P \prec Q$ as usual, we can compute
\[
X_4 \left( \begin{array}{cccccccc}
\on & q_+ & \off & q_- & \on & q'_+ & \off & \off \\
\on & p_+ & \off & p_- & \on & p'_+ & \off & \off
\end{array} \right) = 0,
\]
since in this case $\fbar_3 (M_1 \otimes M_2 \otimes M_3) = 0$ (theorem \ref{thm:fn_structure}; there are two critical pairs), $\fbar_3 (M_2 \otimes M_3 \otimes M_4) = 0$ (since $M_2 \otimes M_3 \otimes M_4$ is singular), and $\fbar_2 (M_3 \otimes M_4) = 0$ (lemma \ref{lem:fn_Xn}; as $M_3 M_4 \neq 0$).

One can also compute that the following are zero:
\begin{gather*}
X_4 \left( \begin{array}{ccccccccc}
\off & q'_- & \on & q_+ & \off & q_- & \on & q'_+ & \off \\
\on  &      & \on & p_+ & \off & p_- & \on & p'_+ & \off
\end{array} \right),
\quad 
X_4 \left( \begin{array}{ccccccccc}
\on & q_+ & \off & q_- & \on & q'_+ & \off &      & \off \\
\on & p_+ & \off & p_- & \on &      & \on  & p'_+ & \off
\end{array} \right), \\
X_4 \left( \begin{array}{ccccccccc}
\on & q_+ & \off &     & \off & q_- & \on & q'_+ & \off \\
\on & p_+ & \off & p_- & \on  &     & \on & p'_+ & \off
\end{array} \right), \quad
X_4 \left( \begin{array}{ccccccccc}
\on & q_+ & \off & q_- & \on & q'_+ & \off &      & \off \\
\on & p_+ & \off & p_- & \on & p'_+ & \off & p'_- & \on
\end{array} \right).
\end{gather*}

Our second example shows that $\fbar_n$ is not diagrammatically simple (as might appear from small cases). We have four matched pairs $P \prec Q \prec R \prec S$, with $P = \{p,p'\}$, $Q = \{q,q'\}$, $R = \{r,r'\}$, $S = \{s,s'\}$, and we claim that
\[
\fbar_4 \left( \begin{array}{ccccccccc}
\on & s_+ & \off	& s_-	& \on 	&			&	\on		&   	& \on \\
\on	& r_+	&	\off	&			&	\off	& r_-	&	\on		&			& \on \\
\on	&			&	\on		&			&	\on		& q_+	&	\off	&	q_-	&	\on \\
\on	&			& \on		& p_+	&	\off	&			&	\off	& p_- & \on \\
\end{array} \right) = 
\left( \begin{array}{c}
c_s \\
w_r \\
c_q \\
c_p
\end{array} \right)
+
\left( \begin{array}{c}
c_s \\
c_r \\
w_q \\
c_p \end{array} \right).
\]
Observe that, as there are no critical pairs, any $X_k$ term with $k>2$ is zero (theorem \ref{thm:Xn_structure}). Moreover, $M_1 M_2 = M_3 M_4 = 0$. Thus $\fbar_4 (M) = \Abar^*_P \circ \Ubar_4 (M)$, and 
\begin{align*}
\Ubar_4 (M) &= \fbar_1 (M_1) \fbar_3 (M_2 \otimes M_3 \otimes M_4) + \fbar_2 (M_1 \otimes M_2) \fbar_2 (M_3 \otimes M_4) \\
&\quad + \fbar_3 (M_1 \otimes M_2 \otimes M_3) \fbar_1 (M_4) + \fbar_3 (M_1 \otimes M_2 M_3 \otimes M_4).
\end{align*}
Now $M_2 \otimes M_3 \otimes M_4$ is twisted at $P$ and $Q$, and tight at $R$ and $S$; $\fbar_3 (M_2 \otimes M_3 \otimes M_4)$ is then given by lemma \ref{lem:f3_computations} and (in standard form) is a nonzero diagram. The same applies to $\fbar_3 (M_1 \otimes M_2 \otimes M_3)$. As $M_1 \otimes M_2$ and $M_3 \otimes M_4$ are twisted at a single matched pair, and tight elsewhere, $\fbar_2 (M_1 \otimes M_2)$ and $\fbar_2 (M_3 \otimes M_4)$ are also both given by single nonzero diagrams (section \ref{sec:low-level_ex}), each with a single crossing.

We now have all terms in $\Ubar_4 (M)$ except $\fbar_3 (M_1 \otimes M_2 M_3 \otimes M_4)$. To this end we note that $M_1 M_2 M_3 = M_2 M_3 M_4 = 0$, so $\Ubar_3 (M_1 \otimes M_2 M_3 \otimes M_4) = \fbar_1 (M_1) \fbar_2 (M_2 M_3 \otimes M_4) + \fbar_2 (M_1 \otimes M_2 M_3) \fbar_1 (M_4)$. Since $M_2 M_3 \otimes M_4$ is twisted at $P$ and $Q$, the creation operator inserts a crossing at $P$; and since $M_1 \otimes M_2 M_3$ is twisted at $R$ and $S$, the creation operator inserts a crossing at $R$. Hence
\begin{align*}
\fbar_3 (M_1 \otimes M_2 M_3 \otimes M_4) 
&= %\Abar^*_P \circ \Ubar_3 (M_1 \otimes M_2 M_3 \otimes M_4) \\
%&= 
\Abar^*_P \left( \fbar_1 (M_1) \fbar_2 (M_2 M_3 \otimes M_4) + \fbar_2 (M_1 \otimes M_2 M_3 ) \fbar_1 (M_4) \right) \\
&= \Abar_P^* \left[
\left( \begin{array}{ccc}
\on & s_+ & \off \\
\on & r_+ & \off \\
\on &  & \on \\
\on &  & \on
\end{array} \right)
\left( \begin{array}{ccc}
\off & s_- & \on \\
\off & r_- & \on \\
\on  & w_q & \on \\
\on  & c_p & \on
\end{array} \right)
+
\left( \begin{array}{ccc}
\on  & w_s & \on \\
\on  & c_r & \on \\
\on  & q_+ & \off \\
\on  & p_+ & \off
\end{array} \right)
\left( \begin{array}{ccc}
\on  &     & \on \\
\on  &     & \on \\
\off & q_- & \on \\
\off & p_- & \on
\end{array} \right)
\right] \\
&= \Abar_P^* \left[
\left( \begin{array}{ccc}
\on & w_s & \on \\
\on & w_r & \on \\
\on & w_q & \on \\
\on & c_p & \on
\end{array} \right)
+
\left( \begin{array}{ccc}
\on & w_s & \on \\
\on & c_r & \on \\
\on & w_q & \on \\
\on & w_p & \on
\end{array} \right)
\right] 
= \left( \begin{array}{ccc}
\on & w_s & \on \\
\on & c_r & \on \\
\on & w_q & \on \\
\on & c_p & \on
\end{array} \right), \\
%\end{align*}
%Thus
%\begin{align*}
\Ubar_4 (M) &=
%\left( \begin{array}{c}
%s_+ \\
%r_+ \\
% \\
%\\
%\end{array} \right)
%\left( \begin{array}{c}
%s_- \\
%r_- \\
%c_q \\
%c_p
%\end{array} \right)
%+
%\left( \begin{array}{c}
%c_s \\
%r_+ \\
% \\
%p_+
%\end{array} \right)
%\left( \begin{array}{c}
%\\
%r_- \\
%c_q \\
%p_-
%\end{array} \right)
%+
%\left( \begin{array}{c}
%c_s \\
%c_r \\
%q_+ \\
%p_+
%\end{array} \right)
%\left( \begin{array}{c}
% \\
% \\
%q_- \\
%p_-
%\end{array} \right)
%+
%\left( \begin{array}{c}
%w_s \\
%c_r \\
%w_q \\
%c_p
%\end{array} \right) \\
%&=
\left( \begin{array}{c}
w_s \\
w_r \\
c_q \\
c_p 
\end{array} \right)
+
\left( \begin{array}{c}
c_s \\
w_r \\
c_q \\
w_p
\end{array} \right)
+
\left( \begin{array}{c}
c_s \\
c_r \\
w_q \\
w_p
\end{array} \right)
+
\left( \begin{array}{c}
w_s \\
c_r \\
w_q \\
c_p
\end{array} \right)
\end{align*} 
so that, applying $\Abar^*_P$, $\fbar_4 (M)$ has the claimed form.

\section{Nontrivial higher operations}
\label{sec:nontrivial_higher_operations}

In this section we only consider $A_\infty$ structures arising from a pair ordering $\preceq$.

Although we have various necessary conditions for $X_n$ or $\fbar_n$ to be nonzero (viability, theorems \ref{thm:fn_structure} and \ref{thm:Xn_structure}, lemma \ref{lem:fn_Xn}, lemma \ref{lem:f3_rows_same}), we do not yet have conditions which are sufficient to ensure $X_n$ or $\fbar_n$ are nonzero --- whether the operations are defined via a pair ordering, or by Kadeishvili's construction more generally.

We have some results at low levels. For instance, $X_2 (M_1 \otimes M_2)$ is nonzero if and only of $M_1 \otimes M_2$ is tight, essentially by definition. Proposition \ref{prop:X3_description} shows that the necessary conditions of theorem \ref{thm:Xn_structure} for $X_3$ to be nonzero are also sufficient. However, the $X_4$ example of section \ref{sec:nec_not_suff} shows that these conditions are not sufficient for $X_4$ to be zero.

Indeed, the $\fbar_3$ examples of section \ref{sec:level_3} (particularly lemma \ref{lem:f3_computations}) show that even the question of whether $\fbar_3$ is zero or nonzero can be rather subtle. The $\fbar_4$ example of section \ref{sec:nec_not_suff} there shows that matters do not get simpler at higher levels.

In this section we prove some sufficient conditions for $\fbar_n$ and $X_n$ to be nonzero. They are, however, far from being necessary conditions.

As usual, throughout this section $M = M_1 \otimes \cdots \otimes M_n$ always denotes a tensor product of nonzero homology classes of diagrams

\subsection{Operation trees}
\label{sec:operation_trees}

Lemma \ref{lem:f3_rows_same} and some of the level 3 and 4 examples show that, even though a tensor product $M_1 \otimes \cdots \otimes M_n$ might have the right number of critical and twisted matched pairs, the steps of these pairs must be covered by the $M_i$ in a way that is appropriately ``horizontally distributed". %In other words, the steps of the various critical and twisted matched pairs must be covered by various $M_i$ in an appropriately distributed way.

To this end, we study rooted trees describing the order in which operations are performed.
\begin{defn}[Operation tree]
\label{def:operation_tree}
An \emph{operation tree} for $\HH^{\otimes n}$ is a rooted plane binary tree with $n$ leaves, ordered from left to right, and with each vertex $v$ labelled by a viable tensor product of homology classes of diagrams $M_v$, so that the following conditions are satisfied.
\begin{enumerate}
\item
Each leaf is labelled with a nonzero homology class of diagram in $\HH$.
\item
Each vertex is labelled with the tensor product of the labels on its ordered children.
\end{enumerate}
If the root vertex is labelled $M$, we say $\T$ is an \emph{operation tree for $M$}.
\end{defn}

Thus, if the leaves are labelled $M_1, \ldots, M_n$ in order, then the root vertex is labelled $M_1 \otimes \cdots \otimes M_n$. See figure \ref{fig:operation_tree} for an example. 

%An operation tree describes a way of combining $M_1 \otimes \cdots \otimes M_n$ into a single element by some sequence of binary operations. At each stage, the tensor product must remain viable.

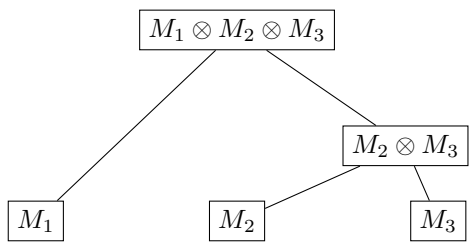
\begin{figure}
\begin{center}
\begin{tikzpicture}
\node[draw, rectangle] (root) {$M_1 \otimes M_2 \otimes M_3$};
%\node[above=0.2cm of root] {$a_1 \otimes a_2 \otimes a_3$};
\node[draw, rectangle, below left=2cm and 1cm of root] (1) {$M_1$} ;
\draw[-] (root) to (1);
\node [draw, rectangle, below right=1cm and 0.1cm of root] (23) {$M_2 \otimes M_3$};
\draw[-] (root) to (23);
\node [draw, rectangle, below=2cm of root] (2) {$M_2$};
\draw[-] (23) to (2);
\node[draw, rectangle, below right=2cm and 1cm of root] (3) {$M_3$};
\draw[-] (23) to (3);
\end{tikzpicture}
\caption{An operation tree.}
\label{fig:operation_tree}
\end{center}
\end{figure}

It will also be useful to consider a certain type of subtree, as in the following definition.
\begin{defn}[Subtree below $v$]
\label{def:subtree_below}
Let $\T$ be an operation tree, and $v$ a vertex of $\T$. The \emph{operation subtree of $\T$ below $v$} is the subtree $\T_v$ of $\T$ with root vertex $v$, consisting of all edges and vertices below $v$, and all vertex labels inherited from $\T$.
\end{defn}
Clearly $\T_v$ is also an operation tree.

\subsection{Validity and distributivity}
\label{sec:valid_distributive}

If $\T$ is an operation tree for $M$, each vertex of $\T$ is labelled by a sub-tensor-product $M_v$ of $M$. The various labels $M_v$ may have different types of tightness, depending on how the various steps around each matched pair are covered.

Singular tensor products should be avoided, and so we make the following definition.
\begin{defn}
\label{def:valid}
Let $\T$ be an operation tree for $\HH^{\otimes n}$.
\begin{enumerate}
\item 
A vertex of $\T$ is \emph{valid} if its label is non-singular.
\item
The operation tree $\T$ is \emph{valid} if it is valid at all of its vertices.
\end{enumerate}
\end{defn}

Thus, in a valid operation tree for $M$, each vertex label is tight, twisted or critical. (Note $M$ may have singular sub-tensor-products, but they do not appear as vertex labels.) Equivalently, each label is tight, twisted or critical at all matched pairs (lemma \ref{lem:homology_tensor_tightness_local}).

%The following lemma states some basic properties of validity.
\begin{lem}
\label{lem:subtrees_valid}
Let $\T$ be a valid operation tree for $M$, and $v$ a vertex of $\T$. Then the operation subtree $\T_v$ of $\T$ below $v$ is valid.
\end{lem}

\begin{proof}
Each label is non-singular in $\T$, hence also non-singular in $\T_v$.
\end{proof}

Nonzero $A_\infty$ operations require carefully regulated numbers of twisted and critical matched pairs, as required by theorems \ref{thm:fn_structure} and \ref{thm:Xn_structure}. Hence we make the following definition.

\begin{defn}
\label{def:distributive}
Let $\T$ be a valid operation tree.  
\begin{enumerate}
\item
A vertex of $\T$ with $k$ leaves, labelled $M$, is \emph{distributive} if there are at least $k-2$ matched pairs at which $M$ is twisted or critical.
\item
The tree $\T$ is \emph{distributive} if every vertex of $\T$ is distributive.
\end{enumerate} 
\end{defn}
%At a vertex $v$ which is distributive, its label $M$ has at least $k-2$ twisted or critical matched pairs, and all other matched pairs are tight. There are no singular matched pairs, by validity of $\T$.

\subsection{Joining and grafting trees}
\label{sec:join_graft}

We now consider some methods to combine operation trees into larger trees.

The first operation, \emph{joining}, places two existing operation trees below a new root vertex.
\begin{defn}
\label{def:joining_trees}
Let $\T', \T''$ be operation trees for $M', M''$, where $M' \otimes M''$ is viable. Let $v',v''$ be the root vertices of $\T',\T''$ respectively. The \emph{join} of $\T'$ and $\T''$ is the tree $\T$ obtained by placing $\T'$ and $\T''$ below $v_0$, so that $v',v''$ are the left and right children of $\T$. The root vertex $v_0$ is labelled $M' \otimes M''$, and each other vertex inherits its label from $\T'$ or $\T''$.
\end{defn}
Clearly, the join of two operation trees is again an operation tree; note that this requires the assumption that $M' \otimes M''$ be viable. Figure \ref{fig:joining_trees} shows an example.

Under certain circumstances, joining trees preserves validity and distributivity.
\begin{lem}
\label{lem:join_valid_distributive}
Let $\T', \T''$ be operation trees for $M' = M_1 \otimes \cdots \otimes M_j$ and $M'' = M_{j+1} \otimes \cdots \otimes M_n$, and let $\T$ be their join. Suppose $\T'$ and $\T''$ are valid and distributive, and one of the following conditions holds:
\begin{enumerate}
\item $X_n (M' \otimes M'') \neq 0$;
\item $\fbar_n (M' \otimes M'') \neq 0$; or
\item $\fbar_j (M') \fbar_{n-j}(M'') \neq 0$, and $M$ contains no 11 doubly occupied pairs.
\end{enumerate}
Then $\T$ is also valid and distributive.
\end{lem}
Note that if $X_n (M' \otimes M'')$ or $\fbar_n (M' \otimes M'')$ 
is nonzero, then $M' \otimes M''$ is certainly viable, so that $\T$ is a well defined operation tree.

\begin{proof}
Each non-root vertex of $\T$ retains its label from $\T'$ or $\T''$. So if $\T', \T''$ are valid (resp. distributive), then $\T$ is valid (resp. distributive) at these vertices. So we only need consider the root vertex $v_0$ of $\T$, which is labelled with $M = M' \otimes M''$.

If $X_n (M) \neq 0$, then by theorem \ref{thm:Xn_structure}, $M$ has precisely $n-2$ matched pairs which are critical, and all other matched pairs are tight. If $\fbar_n (M) \neq 0$, then by theorem \ref{thm:fn_structure}, $M$ has at least $n-1$ matched pairs which are twisted or critical, and all other matched pairs are tight. 

If $\fbar_j (M') \fbar_{n-j}(M'') \neq 0$ then there are at least $j-1$ matched pairs at which $M'$ is twisted or critical, and at least $n-j-1$ pairs at which $M''$ is twisted or critical (theorem \ref{thm:fn_structure}). If any of these pairs coincide, then $M$ has a 11 doubly occupied pair; if these are ruled out, then $M$ has at least $(i-1)+(n-i-1)=n-2$ pairs at which it is twisted or critical.

In each case, $M$ is not singular, and the number of critical or twisted matched pairs is $\geq n-2$. Thus $v_0$ is valid and distributive, and hence so also is $\T$.
\end{proof}

\begin{figure}
\begin{center}
\begin{tikzpicture}
\node[draw, rectangle] (123) {$\begin{array}{c} M_1 \otimes M_2 \\ \quad \otimes M_3 \end{array}$};
\node[draw, rectangle, below left=1.5cm and 0.1cm of 123] (1) {$M_1$} ;
\draw[-] (123) to (1);
\node [draw, rectangle, below right=0.5cm and 0.1cm of 123] (23) {$M_2 \otimes M_3$};
\draw[-] (123) to (23);
\node [draw, rectangle, below=1.5cm of 123] (2) {$M_2$};
\draw[-] (23) to (2);
\node[draw, rectangle, below right=1.5cm and 1cm of 123] (3) {$M_3$};
\draw[-] (23) to (3);
\begin{scope}[xshift = 5 cm, yshift = 0.5 cm]
\node[draw, rectangle] (45) {$M_4 \otimes M_5$};
\node[draw, rectangle, below left=0.5cm and 0.1cm of 45] (4) {$M_4$};
\draw[-] (45) to (4);
\node[draw, rectangle, below right=0.5cm and 0.1cm of 45] (5) {$M_5$};
\draw[-] (45) to (5);
\end{scope}
\begin{scope}[xshift = 9 cm]
\node[draw, rectangle] (root) {$\begin{array}{c} M_1 \otimes M_2 \otimes M_3 \\ \quad \otimes M_4 \otimes M_5 \end{array}$};
\node[draw, rectangle, below left=1 cm and 0.1cm of root] (123) {$\begin{array}{c} M_1 \otimes M_2 \\ \quad \otimes M_3 \end{array}$};
\draw[-] (root) to (123);
\node[draw, rectangle, below left=1.5cm and 0.1cm of 123] (1) {$M_1$} ;
\draw[-] (123) to (1);
\node [draw, rectangle, below right=0.5cm and 0.1cm of 123] (23) {$M_2 \otimes M_3$};
\draw[-] (123) to (23);
\node [draw, rectangle, below=1.5cm of 123] (2) {$M_2$};
\draw[-] (23) to (2);
\node[draw, rectangle, below right=1.5cm and 1cm of 123] (3) {$M_3$};
\draw[-] (23) to (3);
\node[draw, rectangle, below right=1cm and 0.1cm of root] (45) {$M_4 \otimes M_5$};
\draw[-] (root) to (45);
\node[draw, rectangle, below left=0.5cm and 0.1cm of 45] (4) {$M_4$};
\draw[-] (45) to (4);
\node[draw, rectangle, below right=0.5cm and 0.1cm of 45] (5) {$M_5$};
\draw[-] (45) to (5);

\end{scope}
\end{tikzpicture}
\caption{Operation trees $\T', \T'', \T$ (left to right), where $\T$ is the join of $\T'$ and $\T''$.}
\label{fig:joining_trees}
\end{center}
\end{figure}
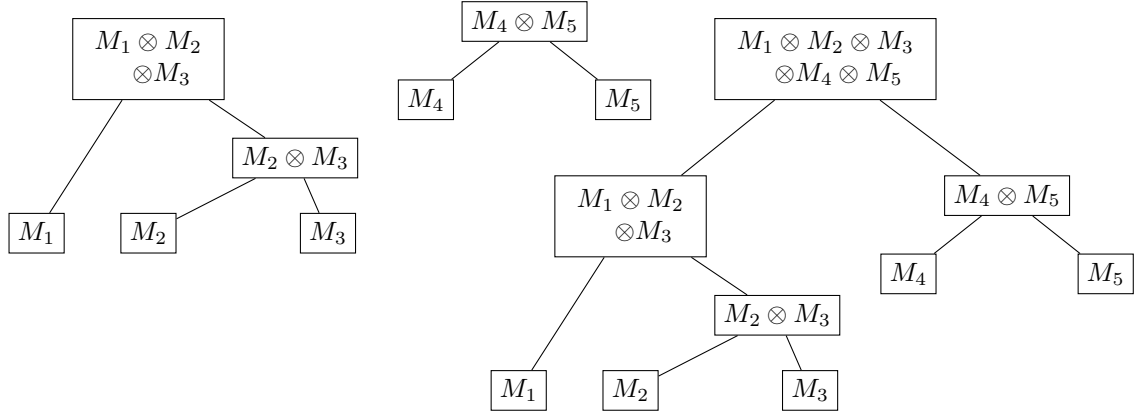

The second operation, \emph{grafting}, implants a tree at a leaf of an existing tree.
\begin{defn}
Let $\T', \T''$ be operation trees for $M' = M_1 \otimes \cdots \otimes M_n$ and $N' = N_1 \otimes \cdots \otimes N_j$, and suppose $N'$ and $M_k$ have the same H-data.

The \emph{grafting of $\T''$ onto $\T'$ at position $k$} is the tree $\T$ obtained by identifying the $k$'th leaf of $\T'$ with the root vertex of $\T''$. The vertices of $\T'$ are relabelled by replacing every instance of $M_k$ with the tensor product $N_1 \otimes \cdots \otimes N_j$; other labels are inherited from $\T''$.
\end{defn}
%(Note that the labelling on the identified vertex is $N_1 \otimes \cdots \otimes N_j$, which may be regarded as inherited from $\T''$ or from $\T'$ by replacing $M_k$ with $N_1 \otimes \cdots \otimes N_j$.)

Figure \ref{fig:grafting_trees} shows an example. Thus $\T$ is an operation tree for the tensor product
\[
M = M_1 \otimes \cdots \otimes M_{k-1} \otimes N_1 \otimes \cdots \otimes N_j \otimes M_{k+1} \otimes \cdots M_n.
\]
The assumption that $N'$ and $M_k$ share the same H-data ensures $M$ is viable. 

\begin{figure}
\begin{center}
\begin{tikzpicture}
\node[draw, rectangle] (M123) {$\begin{array}{c} M_1 \otimes M_2 \\ \quad \otimes M_3 \end{array}$};
\node[draw, rectangle, below left=1.5cm and 0.1cm of M123] (M1) {$M_1$} ;
\draw[-] (M123) to (M1);
\node [draw, rectangle, below right=0.5cm and 0.1cm of M123] (M23) {$M_2 \otimes M_3$};
\draw[-] (M123) to (M23);
\node [draw, rectangle, below=1.5cm of M123] (M2) {$M_2$};
\draw[-] (M23) to (M2);
\node[draw, rectangle, below right=1.5cm and 1cm of M123] (M3) {$M_3$};
\draw[-] (M23) to (M3);
\begin{scope}[xshift = 5 cm, yshift = 0.5 cm]
\node[draw, rectangle] (N12) {$N_1 \otimes N_2$};
\node[draw, rectangle, below left=0.5cm and 0.1cm of N12] (N1) {$N_1$};
\draw[-] (N12) to (N1);
\node[draw, rectangle, below right=0.5cm and 0.1cm of N12] (N2) {$N_2$};
\draw[-] (N12) to (N2);
\end{scope}
\begin{scope}[xshift = 9 cm]
\node[draw, rectangle] (M123) {$\begin{array}{c} M_1 \otimes N_1 \otimes N_2 \\ \quad \otimes M_3 \end{array}$};
\node[draw, rectangle, below left=1.5cm and 0.1cm of M123] (M1) {$M_1$} ;
\draw[-] (M123) to (M1);
\node [draw, rectangle, below right=0.5cm and 0.1cm of M123] (M23) {$N_1 \otimes N_2 \otimes M_3$};
\draw[-] (M123) to (M23);
\node [draw, rectangle, below=1.5cm of M123] (N12) {$N_1 \otimes N_2$};
\draw[-] (M23) to (N12);
\node[draw, rectangle, below right=1.5cm and 1cm of M123] (M3) {$M_3$};
\draw[-] (M23) to (M3);
\node[draw, rectangle, below left=0.5cm and 0.1cm of N12] (N1) {$N_1$};
\draw[-] (N12) to (N1);
\node[draw, rectangle, below right=0.5cm and 0.1cm of N12] (N2) {$N_2$};
\draw[-] (N12) to (N2);
\end{scope}
\end{tikzpicture}
\caption{Operation trees $\T', \T'', \T$, where $\T$ is the grafting of $\T''$ onto $\T'$ at position 2.}
\label{fig:grafting_trees}
\end{center}
\end{figure}
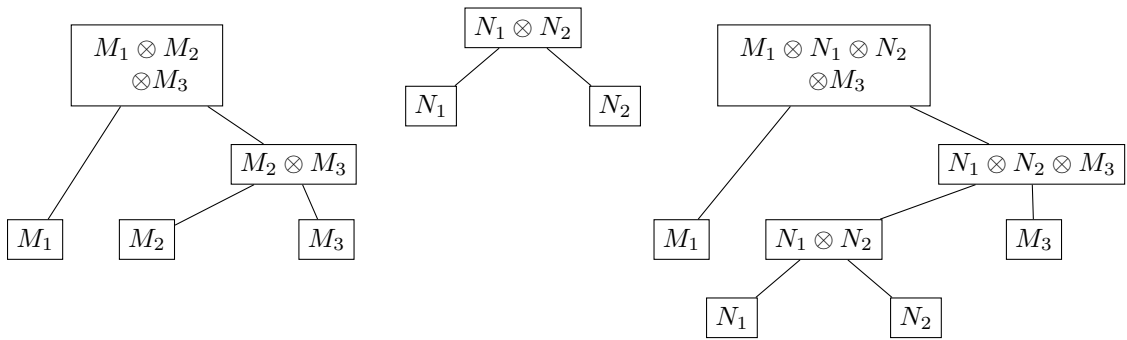

As with joining, under certain circumstances grafting preserves validity and distributivity.
\begin{lem}
\label{lem:graft_valid_distributive}
Let $\T', \T''$ be operation trees for $M' = M_1 = M_1 \otimes \cdots \otimes M_n$ and $N' = N_1 \otimes \cdots \otimes N_j$ respectively. Suppose $X_j (N') = M_k$, and let $\T$ be the grafting of $\T''$ onto $\T'$ at position $k$.

If $\T'$ and $\T''$ are valid and distributive, then $\T$ is also valid and distributive.
\end{lem}
Note $X_j (N') = M_k$ implies $N'$ and $M_k$ have equal H-data, so $\T$ is a well defined operation tree.

\begin{proof}
Each vertex of $\T''$ retains its label, hence validity and distributivity are satisfied. At vertices of $\T'$ which retain their label, the same applies. Thus we only need consider vertices of $\T'$ whose labels are changed in $\T$, i.e. those whose label involves $M_k$.

Let $v$ be a vertex of $\T'$ with $l$ leaves, labelled $M'_v = M_u \otimes \cdots \otimes M_k \otimes \cdots \otimes M_{u+l-1}$; the label in $\T$ is thus $M_v = M_u \otimes \cdots \otimes (N_1 \otimes \cdots \otimes N_j) \otimes \cdots \otimes M_{u+l-1}$. Since $\T'$ is valid, $M'_v$ is non-singular. Since $M_v$ and $M'_v$ have the same H-data, $L$ is non-singular; so $v$ is valid. 

It remains to show that $v$ is distributive. Since $\T'$ is distributive, $M'_v$ has at least $l-2$ matched pairs which are twisted or critical. Now $M_k = X_j (N')$ implies that $M_k$ is the unique nonzero homology class of diagram with the H-data of $N'$, so $M'_v$ is obtained from $M_v$ by an H-contraction (definition \ref{def:H-contraction}). By lemma \ref{lem:tightness_H-contraction}, if $M'_v$ is critical at a matched pair $P$, then $M_v$ is critical at $P$; and if $M'_v$ is twisted at $P$, then $M_v$ is twisted at $P$. Hence $M_v$ has at least as many twisted and critical matched pairs as $M'_v$.
\end{proof}

\subsection{Nonzero operations require trees}
\label{sec:trees_required}

As we now show, a valid distributive operation tree for $M$ is a necessary condition for $X_n(M)$ or $\fbar_n (M)$ to be nontrivial.
\begin{prop}
\label{prop:nonzero_implies_trees}
Consider an $A_\infty$ structure on $\HH$ arising from a pair ordering.
%Let $n \geq 2$, and let $M = M_1 \otimes \cdots \otimes M_n$.
If $X_n (M) \neq 0$ or $\fbar_n (M) \neq 0$, then there is a valid distributive operation tree for $M$.
\end{prop}

Proposition \ref{prop:nonzero_implies_trees} is a precise version of proposition \ref{prop:intro_nonzero_trees}.

\begin{proof}
First note that as $X_n (M)$ or $\fbar_n (M) \neq 0$, $M$ is viable (lemma \ref{lem:fn_Xn_H-data_viability}).

When $n=1$, the valid and distributive conditions are trivial. 

%When $n=2$ there is only one operation tree $\T$ for $M$, which is valid iff $M_1, M_2$ and $M_1 \otimes M_2$ are non-singular; the distributive condition is vacuous. If $M_1 M_2 \neq 0$ then clearly $M_1, M_2, M_1 \otimes M_2$ are tight; if $\fbar_2 (M_1 \otimes M_2) \neq 0$ then $M_1 \otimes M_2$ has all matched pairs tight or twisted (by theorem \ref{thm:fn_structure}), hence $M_1, M_2, M_1 \otimes M_2$ are all tight or twisted at each matched pair (lemma \ref{lem:local_tightness_homology_sub-tensor-product}), hence $M_1, M_2, M_1 \otimes M_2$ are all tight or twisted (lemma \ref{lem:homology_tensor_tightness_local}), hence not singular. So $\T$ is valid and distributive for $M$. 

Now suppose the statement holds for all $X_k$ and $\fbar_k$ for $k<n$, and consider $X_n$ and $\fbar_n$.

Suppose $X_n (M) \neq 0$. By lemma \ref{lem:Xn_directly}, $X_n (M)$ is represented by the sum of crossingless diagrams in $\fbar_j (M_1 \otimes \cdots \otimes M_j) \fbar_{n-j} ( M_{j+1} \otimes \cdots \otimes M_n)$, so some $\fbar_j (M_1 \otimes \cdots \otimes M_j)$ and $\fbar_{n-j} (M_{j+1} \otimes \cdots \otimes M_n)$ are nonzero. By induction there are valid distributive operation trees $\T'$ for $M_1 \otimes \cdots \otimes M_i$ and $\T''$ for $M_{i+1} \otimes \cdots \otimes M_n$. Now let $\T$ be the join of $\T'$ and $\T''$; this is well defined as $M$ is viable. Since $\T'$ and $\T''$ are valid and distributive, by lemma \ref{lem:join_valid_distributive} so is $\T$.

Now suppose $\fbar_n (M) \neq 0$. Then $X_n (M) = 0$ (lemma \ref{lem:fn_Xn}), and $M$ has all matched pairs tight, twisted or critical, with at least one matched pair twisted (theorem \ref{thm:fn_structure}). Thus $\fbar_n (M) = A^*_{\mathcal{CR}^\preceq} \Ubar_n (M)$, and hence $\Ubar_n (M) \neq 0$. From equation (\ref{eqn:Un_def}) then some term of the form $\fbar_j (M_1 \otimes \cdots \otimes M_j) \fbar_{n-j} (M_{j+1} \otimes \cdots \otimes a_M)$ or $\fbar_{n-j+1} (M_1 \otimes \cdots \otimes M_k \otimes X_j (M_{k+1} \otimes \cdots \otimes M_{k+j}) \otimes \cdots \otimes M_n)$ is nonzero. We consider the two cases separately.

In the first case, by induction, there are operation trees $\T'$ for $M_1 \otimes \cdots \otimes M_j$, and $\T''$ for $M_{j+1} \otimes \cdots \otimes M_n$, which are valid and distributive. Let $\T$ be the join of $\T'$ and $\T''$; as $M$ is viable, $\T$ is well defined. By lemma \ref{lem:join_valid_distributive} again, $\T$ is valid and distributive.

In the second case induction gives operation trees $\T'$ for $M_1 \otimes \cdots \otimes M_k \otimes X_j (M_{k+1} \otimes \cdots \otimes M_{k+j} ) \otimes \cdots \otimes M_n$, and $\T''$ for $M_{k+1} \otimes \cdots \otimes M_{k+j}$, which are valid and distributive. Let $\T$ be the grafting of $\T''$ onto $\T'$ at position $k+1$. This is clearly a well-defined operation tree, and by lemma \ref{lem:graft_valid_distributive}, $\T$ is valid and distributive.
%In any case, we have a valid distributive operation tree $\T$ for $M$. The result follows by induction.
\end{proof}

\subsection{Local trees}
\label{sec:local_trees}

Let $\T$ be an operation tree for $M$. We now consider $M$ at a single matched pair $P$, and use this to construct ``localised" versions of $\T$. We will define a \emph{local operation tree}, which has the same underlying tree, and a \emph{reduced local operation tree}, whose underlying tree is obtained by contracting ``extraneous" vertices.

Recall the local tensor product of $M = M_1 \otimes \cdots \otimes M_n$ at $P$ is given by $M_P = (M_1)_P \otimes \cdots \otimes (M_n)_P$ (definition \ref{def:local_tensor_product}).
\begin{defn}
The \emph{local operation tree} $\widetilde{\T_P}$ of $\T$ at $P$ is obtained from $\T$ by replacing each $M_i$ with $(M_i)_P$ in each vertex label.
\end{defn}
Since each $M_i$ is nonzero, so is each $(M_i)_P$; and since $M$ is viable, so is $M_P$. Each vertex label remains viable and the tensor product of its children; so $\widetilde{\T_P}$ is indeed an operation tree for $M_P$.

By proposition \ref{prop:homology_tensor_product_classification}, $M_P$ is an extension-contraction of one of the tensor products shown in the tight, twisted, critical or singular columns of table \ref{tbl:local_tensor_products}. So there are at most 4 tensor factors of $M$ which have non-horizontal strands at $P$, i.e. which cover one or more of the 4 steps around $P$.
\begin{defn}
A tensor factor $M_i$ of $M$ which has a non-horizontal strand at a matched pair $P$ is called \emph{$P$-active}. The corresponding leaves of $\T$ are called \emph{$P$-active leaves}.
\end{defn}
For each $P$, $\T$ has at most $4$ $P$-active leaves. These are precisely the leaves of $\T_P$ labelled by non-idempotent diagrams.

Now we reduce $\widetilde{\T_P}$ to remove non-active leaves and factors. Consider a non-$P$-active factor $M_v$ of $M$, and the corresponding leaf $v$ in $\widetilde{\T_P}$. Then $(M_v)_P$ is idempotent, so deleting it as a factor from $M_P$ leaves a tensor product which is still viable. (Indeed, such a deletion is a trivial contraction: definition \ref{def:trivial_contraction}). We delete $(M_v)_P$ from all labels on $\widetilde{\T}_P$, and we delete the leaf $v$ and its incident edge. This leaves a degree-2 vertex, which we smooth (i.e. we delete the degree-2 vertex and combine the two adjacent edges into a single edge). We then have a binary planar tree. (If the root vertex is smoothed, precisely one of its children remains; that child becomes the root.) It is an operation tree for $(M_1)_P \otimes \cdots \otimes \widehat{(M_v)_P} \otimes \cdots \otimes (M_n)_P$, where the hat denotes a deleted factor.

Repeating the process for all non-active factors, we obtain an operation tree $\T_P$ for $(M_{i_1})_P \otimes \cdots \otimes (M_{i_k})_P$, where the $M_{i_j}$ are the $P$-active factors of $M$. Note $0 \leq k \leq 4$; if $k=0$, $\T_P$ is the empty tree.
\begin{defn}
The operation tree $\T_P$ is called the \emph{reduced local operation tree} of $\T$ at $P$.
\end{defn}

The operation tree $\T_P$ does not depend on the order in which the non-active factors are deleted; in fact it can also be constructed ``at once", as follows.  The $P$-active leaves of $\widetilde{\T}_P$ have a lowest common ancestor $v_0$ in $\T$. Take the edges and vertices along shortest paths in $\T$ from each $P$-active leaf to $v_0$. The union of these edges and vertices is a planar subtree of $\widetilde{\T}_P$ with root $v_0$ and leaves labelled by $M_{i_j}$. Smoothing degree-2 vertices in this subtree and labelling vertices appropriately yields $\T_P$.

Note that the vertices of $\T_P$ can be regarded as a subset of the vertices of $\T$ or $\widetilde{\T}_P$: namely, those vertices which are not deleted or smoothed as we remove non-$P$-active factors.

Local operation trees are useful because of the following fact, a ``local-to-global" law for validity.
\begin{lem}
\label{lem:local_validity}
Let $\T$ be an operation tree. The following are equivalent:
\begin{enumerate}
\item $\T$ is valid.
\item For all matched pairs $P$, the local operation tree $\widetilde{\T}_P$ is valid.
\item For all matched pairs $P$, the reduced local operation tree $\T_P$ is valid.
\end{enumerate}
\end{lem}

\begin{proof}
By lemma \ref{lem:homology_tensor_tightness_local}, a tensor product of homology classes of diagrams is non-singular if and only if it is non-singular at all its matched pairs. Since the labels on the operation trees $\widetilde{\T}_P$ are precisely the labels on $\T$, localised to $P$, (i) and (ii) are equivalent.

As mentioned above, deleting a non-$P$-active leaf from $\widetilde{\T_P}$, corresponding to a non-$P$-active factor $M_i$, produces a trivial contraction on vertex labels. Thus if all vertex labels were non-singular in $\widetilde{\T_P}$, then they remain non-singular. Conversely, if all the ``new" vertex labels are non-singular after deletion, their ``old" labels (being obtained by extension from the ``new" ones --- even at the smoothed vertices) were also non-singular. The deleted vertex was labelled by a single idempotent diagram, which is non-singular. After deleting all non-$P$-active leaves, $\widetilde{\T_P}$ is valid if and only if $\T_P$ is valid.
\end{proof}

\subsection{Climbing a tree}
\label{sec:climbing_tree}

Let $\T$ be a reduced local operation tree. Then $\T$ has no more than 4 leaves, so there are not many possible trees. Indeed, the number of rooted planar binary trees with $1,2,3,4,n$ leaves is $1,1,2,5, \frac{1}{n+1} \binom{2n}{n}$.

The tensor products arising in reduced local operation trees are also small in number. If $M$ is the tensor product labelling the root of $M$, then $M$ is a viable tensor product of homology classes of diagrams on the arc diagram $\ZZ_P$ consisting of a single matched pair. As $\T$ is a reduced local operation tree, $M$ has no idempotents, i.e. every tensor factor of $M$ has non-horizontal strands. Thus (proposition \ref{prop:homology_tensor_product_classification}) $M$ is one of the tensor products shown in the tight, twisted, critical or singular columns of table \ref{tbl:local_tensor_products}, or (in the tight case) a contraction thereof.

We ask: for each such tensor product $M$, which of the possible operation trees on $M$ is valid?

If $M$ is tight or twisted, then any sub-tensor product is tight or twisted (lemma \ref{lem:tightness_homology_sub-tensor-product}), and in particular non-singular, so any operation tree for $M$ is valid. And of course if $M$ is singular, then any operation tree for $M$ is invalid, since its root vertex has singular label $M$.

When $M$ is critical, some but not all operation trees are valid. By examining the possible cases in the critical column of table \ref{tbl:local_tensor_products}, we observe the following.
\begin{itemize}
\item
When $M$ is critical and $P$ is sesqui-occupied, precisely 1 of the 2 operation trees are valid.
\item
When $M$ is critical and $P$ is 00 doubly occupied, precisely 2 of the 5 operation trees are valid.
\item
When $M$ is critical and $P$ is 11 doubly occupied, precisely 3 of the 5 operation trees are valid.
\end{itemize}
Tables \ref{tbl:sesqui-occupied_validity} and \ref{tbl:doubly-occupied_validity} illustrate the valid and invalid trees.

\begin{table}
\begin{center}

\begin{tabular}{>{\centering\arraybackslash}m{4 cm}|>{\centering\arraybackslash}m{5cm}|>{\centering\arraybackslash}m{5cm}} 

$M$ & Valid operation trees & Invalid operation trees \\
\hline

\begin{tikzpicture}[scale=0.9]
\strandbackgroundshading
\tstrandbackgroundshading{1}
\tstrandbackgroundshading{2}
\beforevused
\tafterwused{1}
\tbeforewused{2}
\strandsetupn{}{}
\tstrandsetup{1}
\tstrandsetup{2}
\tstrandsetup{3}
\leftoff
\righton
\trightoff{1}
\trighton{2}
\used
\tusea{1}
\tuseb{2}
\end{tikzpicture}
&
\begin{tikzpicture}[scale=0.55]
\draw [draw=none] (0,0.2) circle (5pt);
\draw (0,0) -- (-1,-1);
\draw (0,0) -- (1,-1);
\draw (1,-1) -- (0,-2);
\draw (1,-1) -- (2,-2);
\fill [draw=black!30!green, fill=black!30!green] (0,0) circle (5pt);
\fill (-1,-1) circle (5pt);
\fill [draw=blue, fill=blue] (1,-1) circle (5pt);
\fill (0,-2) circle (5pt);
\fill (2,-2) circle (5pt);
\end{tikzpicture}
&
\begin{tikzpicture}[scale=0.55]
\draw [draw=none] (0,0.2) circle (5pt);
\draw (0,0) -- (-1,-1);
\draw (0,0) -- (1,-1);
\draw (-1,-1) -- (-2,-2);
\draw (-1,-1) -- (0,-2);
\fill [draw=black!30!green, fill=black!30!green] (0,0) circle (5pt);
\fill [draw=red, fill=red] (-1,-1) circle (5pt);
\fill (1,-1) circle (5pt);
\fill (-2,-2) circle (5pt);
\fill (0,-2) circle (5pt);
\end{tikzpicture}
\\
\hline

\begin{tikzpicture}[scale=0.9]
\strandbackgroundshading
\tstrandbackgroundshading{1}
\tstrandbackgroundshading{2}
\afterwused
\tbeforewused{1}
\taftervused{2}
\strandsetupn{}{}
\tstrandsetup{1}
\tstrandsetup{2}
\tstrandsetup{3}
\lefton
\rightoff
\trighton{1}
\trightoff{2}
\usea
\tuseb{1}
\tusec{2}
\end{tikzpicture}
&
\begin{tikzpicture}[scale=0.55]
\draw [draw=none] (0,0.2) circle (3pt);
\draw (0,0) -- (-1,-1);
\draw (0,0) -- (1,-1);
\draw (-1,-1) -- (-2,-2);
\draw (-1,-1) -- (0,-2);
\fill [draw=black!30!green, fill=black!30!green] (0,0) circle (5pt);
\fill [draw=blue, fill=blue] (-1,-1) circle (5pt);
\fill (1,-1) circle (5pt);
\fill (-2,-2) circle (5pt);
\fill (0,-2) circle (5pt);
\end{tikzpicture}
&
\begin{tikzpicture}[scale=0.55]
\draw [draw=none] (0,0.2) circle (5pt);
\draw (0,0) -- (-1,-1);
\draw (0,0) -- (1,-1);
\draw (1,-1) -- (0,-2);
\draw (1,-1) -- (2,-2);
\fill [draw=black!30!green, fill=black!30!green] (0,0) circle (5pt);
\fill (-1,-1) circle (5pt);
\fill [draw=red, fill=red] (1,-1) circle (5pt);
\fill (0,-2) circle (5pt);
\fill (2,-2) circle (5pt);
\end{tikzpicture}
\end{tabular}

\caption{Validity of operation trees on sesqui-occupied local critical tensor products. Red, green, blue, black vertices respectively indicate singular, critical, twisted and tight labels.}
\label{tbl:sesqui-occupied_validity}
\end{center}
\end{table}

\begin{table}
\begin{center}

\begin{tabular}{>{\centering\arraybackslash}m{2.7 cm}|>{\centering\arraybackslash}m{5.5cm}|>{\centering\arraybackslash}m{5.5cm}} 

$M$ & Valid operation trees & Invalid operation trees \\
\hline
\flushleft
\begin{tikzpicture}[xscale=0.6, yscale=0.8]
\strandbackgroundshading
\tstrandbackgroundshading{1}
\tstrandbackgroundshading{2}
\tstrandbackgroundshading{3}
\beforewused
\taftervused{1}
\tbeforevused{2}
\tafterwused{3}
\strandsetupn{}{}
\tstrandsetup{1}
\tstrandsetup{2}
\tstrandsetup{3}
\tstrandsetup{4}
\leftoff
\righton
\trightoff{1}
\trighton{2}
\trightoff{3}
\useb
\tusec{1}
\tused{2}
\tusea{3}
\end{tikzpicture}
&
\begin{tikzpicture}[scale=0.35]
\draw [draw=none] (0,0.2) circle (6pt);
\draw (0,0) -- (-1,-1);
\draw (0,0) -- (1,-1);
\draw (-1,-1) -- (-2,-2);
\draw (-1,-1) -- (0,-2);
\draw (0,-2) -- (-1,-3);
\draw (0,-2) -- (1,-3);
\fill [draw=black!30!green, fill=black!30!green] (0,0) circle (6pt);
\fill [draw=black!30!green, fill=black!30!green] (-1,-1) circle (6pt);
\fill (1,-1) circle (6pt);
\fill (-2,-2) circle (6pt);
\fill [draw=blue, fill=blue] (0,-2) circle (6pt);
\fill (-1,-3) circle (6pt);
\fill (1,-3) circle (6pt);
\end{tikzpicture}
\hspace{0.2 cm}
\begin{tikzpicture}[scale=0.35]
\draw [draw=none] (0,0.2) circle (6pt);
\draw (0,0) -- (-1,-1);
\draw (0,0) -- (1,-1);
\draw (1,-1) -- (0,-2);
\draw (1,-1) -- (2,-2);
\draw (0,-2) -- (-1,-3);
\draw (0,-2) -- (1,-3);
\fill [draw=black!30!green, fill=black!30!green] (0,0) circle (6pt);
\fill (-1,-1) circle (6pt);
\fill [draw=black!30!green, fill=black!30!green] (1,-1) circle (6pt);
\fill (0,-2) circle (6pt);
\fill [draw=blue, fill=blue] (2,-2) circle (6pt);
\fill (-1,-3) circle (6pt);
\fill (1,-3) circle (6pt);
\end{tikzpicture}
&
\begin{tikzpicture}[scale=0.35]
\draw [draw=none] (0,0.2) circle (6pt);
\draw (0,0) -- (-1,-1);
\draw (0,0) -- (1,-1);
\draw (-1,-1) -- (-2,-2);
\draw (-1,-1) -- (0,-2);
\draw (-2,-2) -- (-3,-3);
\draw (-2,-2) -- (-1,-3);
\fill [draw=black!30!green, fill=black!30!green] (0,0) circle (6pt);
\fill [draw=black!30!green, fill=black!30!green] (-1,-1) circle (6pt);
\fill (1,-1) circle (6pt);
\fill [draw=red, fill=red] (-2,-2) circle (6pt);
\fill (0,-2) circle (6pt);
\fill (-3,-3) circle (6pt);
\fill (-1,-3) circle (6pt);
\end{tikzpicture}
\hspace{0.1 cm}
\begin{tikzpicture}[scale=0.35]
\draw [draw=none] (0,0.2) circle (6pt);
\draw (0,0) -- (-1.2,-1);
\draw (0,0) -- (1.2,-1);
\draw (-1.2,-1) -- (-2.2,-2);
\draw (-1.2,-1) -- (-0.3,-2);
\draw (1.2,-1) -- (0.3,-2);
\draw (1.2,-1) -- (2.2,-2);
\fill [draw=black!30!green, fill=black!30!green] (0,0) circle (6pt);
\fill [draw=red, fill=red] (-1.2,-1) circle (6pt);
\fill [draw=red, fill=red] (1.2,-1) circle (6pt);
\fill (-2.2,-2) circle (6pt);
\fill (-0.3,-2) circle (6pt);
\fill (0.3,-2) circle (6pt);
\fill (2.2,-2) circle (6pt);
\end{tikzpicture}
\hspace{0.1 cm}
\begin{tikzpicture}[scale=0.35]
\draw [draw=none] (0,0.2) circle (6pt);
\draw (0,0) -- (-1,-1);
\draw (0,0) -- (1,-1);
\draw (1,-1) -- (0,-2);
\draw (1,-1) -- (2,-2);
\draw (2,-2) -- (1,-3);
\draw (2,-2) -- (3,-3);
\fill [draw=black!30!green, fill=black!30!green] (0,0) circle (6pt);
\fill (-1,-1) circle (6pt);
\fill [draw=black!30!green, fill=black!30!green] (1,-1) circle (6pt);
\fill (0,-2) circle (6pt);
\fill [draw=red, fill=red] (2,-2) circle (6pt);
\fill (1,-3) circle (6pt);
\fill (3,-3) circle (6pt);
\end{tikzpicture}
\\
\hline
\begin{tikzpicture}[xscale=0.6, yscale=0.8]
\strandbackgroundshading
\tstrandbackgroundshading{1}
\tstrandbackgroundshading{2}
\tstrandbackgroundshading{3}
\afterwused
\tbeforewused{1}
\taftervused{2}
\tbeforevused{3}
\strandsetupn{}{}
\tstrandsetup{1}
\tstrandsetup{2}
\tstrandsetup{3}
\tstrandsetup{4}
\lefton
\rightoff
\trighton{1}
\trightoff{2}
\trighton{3}
\usea
\tuseb{1}
\tusec{2}
\tused{3}
\end{tikzpicture}
&
\begin{tikzpicture}[scale=0.35]
\draw [draw=none] (0,0.2) circle (6pt);
\draw (0,0) -- (-1,-1);
\draw (0,0) -- (1,-1);
\draw (-1,-1) -- (-2,-2);
\draw (-1,-1) -- (0,-2);
\draw (-2,-2) -- (-3,-3);
\draw (-2,-2) -- (-1,-3);
\fill [draw=black!30!green, fill=black!30!green] (0,0) circle (6pt);
\fill [draw=black!30!green, fill=black!30!green] (-1,-1) circle (6pt);
\fill (1,-1) circle (6pt);
\fill [draw=blue, fill=blue] (-2,-2) circle (6pt);
\fill (0,-2) circle (6pt);
\fill (-3,-3) circle (6pt);
\fill (-1,-3) circle (6pt);
\end{tikzpicture}
\hspace{0.1 cm}
\begin{tikzpicture}[scale=0.35]
\draw [draw=none] (0,0.2) circle (6pt);
\draw (0,0) -- (-1.2,-1);
\draw (0,0) -- (1.2,-1);
\draw (-1.2,-1) -- (-2.2,-2);
\draw (-1.2,-1) -- (-0.3,-2);
\draw (1.2,-1) -- (0.3,-2);
\draw (1.2,-1) -- (2.2,-2);
\fill [draw=black!30!green, fill=black!30!green] (0,0) circle (6pt);
\fill [draw=blue, fill=blue] (-1.2,-1) circle (6pt);
\fill [draw=blue, fill=blue] (1.2,-1) circle (6pt);
\fill (-2.2,-2) circle (6pt);
\fill (-0.3,-2) circle (6pt);
\fill (0.3,-2) circle (6pt);
\fill (2.2,-2) circle (6pt);
\end{tikzpicture}
\hspace{0.1 cm}
\begin{tikzpicture}[scale=0.35]
\draw [draw=none] (0,0.2) circle (6pt);
\draw (0,0) -- (-1,-1);
\draw (0,0) -- (1,-1);
\draw (1,-1) -- (0,-2);
\draw (1,-1) -- (2,-2);
\draw (2,-2) -- (1,-3);
\draw (2,-2) -- (3,-3);
\fill [draw=black!30!green, fill=black!30!green] (0,0) circle (6pt);
\fill (-1,-1) circle (6pt);
\fill [draw=black!30!green, fill=black!30!green] (1,-1) circle (6pt);
\fill (0,-2) circle (6pt);
\fill [draw=blue, fill=blue] (2,-2) circle (6pt);
\fill (1,-3) circle (6pt);
\fill (3,-3) circle (6pt);
\end{tikzpicture}
&
\begin{tikzpicture}[scale=0.35]
\draw [draw=none] (0,0.2) circle (6pt);
\draw (0,0) -- (-1,-1);
\draw (0,0) -- (1,-1);
\draw (-1,-1) -- (-2,-2);
\draw (-1,-1) -- (0,-2);
\draw (0,-2) -- (-1,-3);
\draw (0,-2) -- (1,-3);
\fill [draw=black!30!green, fill=black!30!green] (0,0) circle (6pt);
\fill [draw=black!30!green, fill=black!30!green] (-1,-1) circle (6pt);
\fill (1,-1) circle (6pt);
\fill (-2,-2) circle (6pt);
\fill [draw=red, fill=red] (0,-2) circle (6pt);
\fill (-1,-3) circle (6pt);
\fill (1,-3) circle (6pt);
\end{tikzpicture}
\hspace{0.2 cm}
\begin{tikzpicture}[scale=0.35]
\draw [draw=none] (0,0.2) circle (6pt);
\draw (0,0) -- (-1,-1);
\draw (0,0) -- (1,-1);
\draw (1,-1) -- (0,-2);
\draw (1,-1) -- (2,-2);
\draw (0,-2) -- (-1,-3);
\draw (0,-2) -- (1,-3);
\fill [draw=black!30!green, fill=black!30!green] (0,0) circle (6pt);
\fill (-1,-1) circle (6pt);
\fill [draw=black!30!green, fill=black!30!green] (1,-1) circle (6pt);
\fill [draw=red, fill=red] (0,-2) circle (6pt);
\fill (2,-2) circle (6pt);
\fill (-1,-3) circle (6pt);
\fill (1,-3) circle (6pt);
\end{tikzpicture}
\end{tabular}

\caption{Validity of operation trees on doubly occupied local critical tensor products. Red, green, blue, black vertices respectively indicate singular, critical, twisted and tight labels.}
\label{tbl:doubly-occupied_validity}
\end{center}
\end{table}

Starting from the leaves of $\T$, which are all tight, we can climb $\T$, observing how tightness behaves as the (homology classes of) diagrams labelling the vertices are joined into tensor products.

We observe that whenever there is a singular or twisted vertex label, it occurs when two adjacent diagrams are joined into a singular tensor product. Also, we never see both a twisted vertex label and a singular vertex label. This leads to the following statement.
\begin{lem}
\label{lem:valid_twisted_vertex}
Let $\T$ be an operation tree for a viable tensor product of diagrams $M$. Then $\T$ is valid if and only if for every non-tight matched pair $P$ of $M$, $\T_P$ has a twisted vertex label.
\end{lem}

\begin{proof}
By lemma \ref{lem:local_validity}, the validity of $\T$ is equivalent to the the validity of all the $\T_P$. Since $M$ is viable, at each matched pair $M$ is tight, twisted, critical or singular. As discussed above, if $M_P$ is tight at $P$ then $\T_P$ is valid. So it remains to check that when $M_P$ is twisted, critical, or singular, $\T_P$ is valid if and only if $\T_P$ has a twisted vertex label.

If $M_P$ is twisted then again $\T_P$ is valid, and moreover $\T_P$ has a twisted vertex label at the root. % (a trivial contraction of $M_P$). 
If $M_P$ is critical then, from tables \ref{tbl:sesqui-occupied_validity} and \ref{tbl:doubly-occupied_validity}, $\T_P$ is valid if and only if there is a twisted vertex label. And if $M$ is singular, then $\T_P$ is invalid, and moreover $M_P$ must be an extension of the singular example in table \ref{tbl:local_tensor_products} (lemma \ref{lem:homology_contractions}), so $\T_P$ must be the unique rooted binary planar tree with two leaves; the two leaf labels are tight, and the root label is singular, so there is no twisted vertex label. Thus in each case $\T_P$ is valid if and only if it has a twisted vertex label.
\end{proof}

\subsection{Strong validity}
\label{sec:strong_validity}

We saw above that when $M$ is valid, then at every non-tight $P = \{p,p'\}$, the reduced local operation tree $\T_P$ has a twisted vertex label. But in fact, in almost every case, there is \emph{precisely} one twisted vertex label. The only exception is when $M_P$ is 11 doubly occupied and critical (i.e. the second row of table \ref{tbl:doubly-occupied_validity}), and $\T_P$ is the unique rooted planar binary tree of depth 2 (i.e. the second valid operation tree shown). This particular operation tree can lead to the multiplication of a diagram crossed at $p$, with a diagram crossed at $p'$, producing a diagram in $\F$. To avoid it, we introduce a ``strong" form of validity.

\begin{lem}
\label{lem:strong_validity_conditions}
Let $\T$ be an operation tree for $M$. The following are equivalent.
\begin{enumerate}
\item
For every non-tight matched pair $P$ of $M$, there is a unique lowest vertex of $\T$ among those whose label is twisted at $P$.
\item
The operation tree $\T$ is valid, and for each non-tight matched pair $P$ of $M$, there is a unique lowest vertex of $\T$ among those whose label is not tight at $P$.
\item
For each non-tight matched pair $P$ of $M$, there is a unique lowest vertex of $\widetilde{\T_P}$ among those whose label is twisted.
\item
For every non-tight matched pair $P$ of $M$, $\T_P$ has a unique twisted vertex label.
\end{enumerate}
\end{lem}

\begin{defn}
The operation tree $\T$ is \emph{strongly valid} if the conditions of lemma \ref{lem:strong_validity_conditions} hold. 
\end{defn}

Comparing lemmas \ref{lem:valid_twisted_vertex} and \ref{lem:strong_validity_conditions}(iv), it is clear that strong validity implies validity.

\begin{proof}[Proof of lemma \ref{lem:strong_validity_conditions}]
First we show equivalence of (i) and (ii). If $\T$ is not valid, then (i) fails by lemma \ref{lem:valid_twisted_vertex}, and (ii) obviously fails. So assume $\T$ is valid. We show that a vertex $v$ of $\T$, with label $M_v$, is lowest among those with labels twisted at $P$, if and only if it is lowest among those with labels non-tight at $P$.

If $v$ is lowest among vertices with label twisted at $P$, then the children of $v$ have labels which are sub-tensor-products of $M_v$ non-twisted  at $P$. Hence by lemma \ref{lem:local_tightness_homology_sub-tensor-product} and table \ref{tbl:local_homology_sub-tensor-product_tightness}, the labels on these children are tight at $P$. All descendants of these children have tight labels at $P$ also, again by lemma \ref{lem:local_tightness_homology_sub-tensor-product} and table \ref{tbl:local_homology_sub-tensor-product_tightness}. So $v$ is lowest among vertices of $\T$ with labels non-tight at $P$.

Conversely, if $v$ is lowest among those with labels non-tight at $P$, then all descendants of $v$ have tight labels at $P$. 
%In particular, the children of $v$ both have labels tight at $P$. 
Then $(M_v)_P$ is the tensor product of the tight labels of its children: it cannot be critical, by lemma \ref{lem:it_takes_3}, and cannot be singular, since $\T$ is valid. So $v$ is a lowest vertex among those with label twisted at $P$. Thus (i) and (ii) are equivalent.

Condition (iii) is just a reformulation of (i).

To see equivalence of (iii) and (iv), recall how $\T_P$ is obtained from $\widetilde{\T_P}$. If the label $M_v$ on a leaf $v$ of $\widetilde{\T_P}$ is idempotent, then we delete $v$ and its incident edge, delete $M_v$ from all labels, and smooth the resulting degree-2 vertex $w$. Since $M_v$ has only horizontal strands, deleting $M_v$ from a label yields a trivial contraction (definition \ref{def:trivial_contraction}), which does not change the tightness of the label.

Now $w$ has two children $v$ and $x$ in $\widetilde{\T_P}$. Since $M_v$ is tight (being an idempotent), and $M_w$ is an extension of $M_x$ (by the horizontal strands of $M_v$), so $M_w$ and $M_x$ have the same tightness. In particular, neither $v$ nor $w$ cannot be lowest among those with twisted label. After deleting $v$ and all instances of $M_v$ in labels, the label on $w$ is the same as the label on $x$. After smoothing $w$, every remaining vertex has children and descendants with twisted labels if and only if it had them in $\widetilde{\T_P}$. Thus any vertex which was lowest among those with twisted labels was not $v$ or $w$, so remains as a vertex, and remains lowest among those with twisted labels. So the set of lowest vertices with twisted labels is preserved. 

Repeating this process we eventually arrive at $\T_P$. So $\widetilde{\T_P}$ has a unique lowest vertex among those with twisted labels, if and only if the same is true for $\T_P$. From the examination of reduced local operation trees in section \ref{sec:climbing_tree}, we observe that a reduced local operation tree has a unique lowest vertex with twisted label if and only if it has a unique vertex with a twisted label. Thus (iii) and (iv) are equivalent.
\end{proof}

The above discussion also immediately implies the following.
\begin{lem}
\label{lem:valid_not_strongly}
Suppose $\T$ an operation tree for $M$ which is valid but not strongly valid. Then $M$ has a matched pair which is 11 doubly occupied and critical.
\qed
\end{lem}

By lemma \ref{lem:strong_validity_conditions}(i), the following map is well defined.
\begin{defn}
\label{def:pairs_to_vertices}
Let $\T$ be a strongly valid operation tree for $M$. The function
\[
V_\T \colon \{ \text{Non-tight matched pairs of $M$} \} \To \{ \text{Non-leaf vertices of $\T$} \}
\]
sends a matched pair $P$ to the lowest vertex of $\T$ whose label is twisted at $P$.
\end{defn}

By the argument in the proof of lemma \ref{lem:strong_validity_conditions} (that (i) and (ii) are equivalent), $V_\T (P)$ is also the lowest vertex of $\T$ whose label is not tight at $P$.

\begin{lem}
\label{lem:strongly_valid_non-tight_ancestors}
Let $\T$ be a strongly valid operation tree for $M$, and let $P$ be a non-tight matched pair of $M$. Then the vertices of $\T$ whose labels are non-tight at $P$ are precisely $V_\T (P)$ and its ancestors.
%In particular, if $v$ is vertex of $\T$ whose label $M_v$ is not tight at $P$, but whose children have labels tight at $P$, then %$M_v$ is twisted at $P$, and 
%$v = V_\T (P)$.
\end{lem}

\begin{proof}
Let the label on $V_\T (P)$ be $M'$. If $v$ is an ancestor of $V_\T (P)$, labelled $M_v$, then $M'$ is a sub-tensor-product of $M_v$. As $M'$ is not tight at $P$, by lemma \ref{lem:local_tightness_homology_sub-tensor-product} $M_v$ is not tight at $P$.

Conversely, suppose a vertex $v_0$ of $\T$ has label non-tight at $P$. Either $v_0$ is a lowest such vertex, or $v_0$ has a child $v_1$ whose label is also not tight at $P$. If the latter, then $v_1$ is either a lowest such vertex, or has a child whose label is non-tight at $P$. In this way, we eventually arrive at a descendant $v_*$ of $v_0$, which is lowest amongst those whose labels are not tight at $P$. By the comment after definition \ref{def:pairs_to_vertices} then $v_* = V_\T (P)$, so $v_0$ is $V_\T (P)$ or one of its ancestors.
\end{proof}

Strong validity shares many of the properties of validity. The following results generalise lemmas \ref{lem:local_validity} and \ref{lem:subtrees_valid}.
\begin{lem}
\label{lem:strongly_valid_local-global}
Let $\T$ be an operation tree. The following are equivalent:
\begin{enumerate}
\item 
$\T$ is strongly valid.
\item 
For all matched pairs $P$, the local operation tree $\widetilde{\T_P}$ is strongly valid.
\item
For all matched pairs $P$, the reduced local operation tree $\T_P$ is strongly valid.
\end{enumerate}
\end{lem}

\begin{proof}
By definition (i) is equivalent to (iii). As discussed in the proof of lemma \ref{lem:strong_validity_conditions}, strong validity is preserved under the deletion operations which transform $\widetilde{\T_P}$ into $\T_P$, hence (ii) and (iii) are equivalent. (This also follows from the equivalence of characterisations (iii) and (iv) in lemma \ref{lem:strong_validity_conditions}.)
\end{proof}

\begin{lem}
\label{lem:subtree_strongly_valid}
Let $\T$ be a strongly valid operation tree for $M$, and let $v$ be a vertex of $\T$ labelled by $M_v$. Let $\T_v$ be the operation subtree of $\T$ below $v$. Then the following hold.
\begin{enumerate}
\item
$\T_v$ is a strongly valid operation tree for $M_v$.
\item
The function $V_{\T_v}$ is a restriction of the function $V_\T$.
\end{enumerate}
\end{lem}
Note that $M_v$ is a sub-tensor-product of $M$, so by lemma \ref{lem:tightness_homology_sub-tensor-product}, a matched pair which is non-tight in $M_v$ is also non-tight in $M$. Hence the domain of $V_{\T_v}$ is a subset of the domain of $V_\T$, so the assertion of (ii) makes sense.

\begin{proof}
Let $P$ be a matched pair, and consider the local operation trees $\widetilde{\T_P}$ for $M_P$, and $\widetilde{(\T_v)_P}$ for $(M_v)_P$. To prove (i), we show that if $(M_v)_P$ is not tight, then $\widetilde{(\T_v)_P}$ has a unique lowest vertex with twisted label (lemma \ref{lem:strong_validity_conditions}(iii)); and to prove (ii), we show that this vertex is also the unique lowest vertex with twisted label in $\widetilde{\T_P}$.

So suppose $(M_v)_P$ is not tight. It is also not singular: as $\T$ is strongly valid, $\T$ is valid, so by lemma \ref{lem:subtrees_valid} $\T_v$ is valid; hence $M_v$ is non-singular, so $(M_v)_P$ is also non-singular (lemma \ref{lem:homology_tensor_tightness_local}(iv)). Thus $(M_v)_P$ is twisted or critical. 
By lemma \ref{lem:local_validity}(ii), $\widetilde{(\T_v)_P}$ is valid; being an operation tree for the non-tight $(M_v)_P$, by lemma \ref{lem:valid_twisted_vertex}, $\widetilde{(\T_v)_P}$ has a vertex with a twisted label. 

Now $\widetilde{(\T_v)_P}$ is the operation subtree of $\widetilde{\T_P}$ below $v$, with the the same vertex labels, consisting of everything in $\widetilde{\T_P}$ from $v$ down. Thus, any lowest vertex with twisted label in $\widetilde{(\T_v)_P}$ is also a lowest vertex in $\widetilde{\T_P}$ with twisted label. As $(M_v)_P$ is twisted or critical, and is a sub-tensor-product of $M_P$, then $M_P$ is also twisted or critical (lemma \ref{lem:local_tightness_homology_sub-tensor-product}). By strong validity of $\T$ and lemma \ref{lem:strong_validity_conditions}(iii), there is a unique lowest vertex in $\widetilde{\T_P}$ with twisted label. As $\widetilde{(\T_v)_P}$ has a vertex with twisted label, the unique lowest vertex in $\widetilde{\T_P}$ with twisted label lies in $\widetilde{(\T_v)_P}$, and it is also the unique lowest vertex in $\widetilde{(\T_v)_P}$ with twisted label.
\end{proof}

Finally, strong validity implies the following nice separation property of non-tight matched pairs.
\begin{lem}
\label{lem:disjoint_subtrees_tightness}
Let $\T$ be a strongly valid operation tree. Let $v, w$ be vertices of $\T$, with labels $M_v, M_w$ respectively, such that the operation subtrees $\T_v, \T_w$ below $v,w$ are disjoint.

For any matched pair $P$, at least one of $M_v, M_w$ is tight at $P$.
\end{lem}

The disjointness of $\T_v, \T_w$ is equivalent to neither of $v,w$ being a descendant of the other.
\begin{proof}
Suppose to the contrary that both $(M_v)_P, (M_w)_P$ are not tight. By lemma \ref{lem:subtree_strongly_valid}, $\T_v$ and $\T_w$ are strongly valid, so there is a unique lowest vertex $x_v$ in $\widetilde{(\T_v)_P}$ with twisted label, and a unique lowest vertex $x_w$ in $\widetilde{(\T_w)_P}$ with twisted label. But then $x_v, x_w$ are two distinct vertices of $\widetilde{\T_P}$ which are lowest vertices with twisted labels, contradicting strong validity of $\T$.
\end{proof}

% local operation trees $\widetilde{(\T_v)_P}$ and $\widetilde{(\T_w)_P}$ are disjoint subtrees (with the same labels) of $\widetilde{\T_P}$; and the inclusion preserves children. Hence 

\subsection{Transplantation and branch shifts}
\label{sec:transplantation}

We now define two further methods to modify operation trees.

The first method, \emph{transplantation}, replaces an operation subtree (definition \ref{def:subtree_below}) with another tree.
\begin{defn}
\label{def:transplant}
Let $\T$ be an operation tree, and let $\T_v$ be the operation subtree below a non-root vertex $v$, labelled $M'$. Let $\T'$ be another operation tree for $M'$. Then removing $\T_v$ from $\T$ and replacing it with $\T'$ gives an operation tree $\mathcal{U}$. We say $\mathcal{U}$ is obtained from $\T$ by \emph{transplanting} $\T'$ for $\T_v$.
\end{defn}
It is easily verified $\mathcal{U}$ is in fact an operation tree; viability of labels in $\T$ and $\T'$ implies viability of labels in $\mathcal{U}$. If $\T$ is an operation tree for $M$, then $\mathcal{U}$ is also an operation tree for $M$. So $\T$ and $\mathcal{U}$ describe operations on the same inputs, but the operations under $v$ are rearranged.

Note that transplantation is quite different from grafting (section \ref{sec:join_graft}). Grafting adds to an operation tree below a leaf, while transplantation replaces part of an operation tree. Grafting adds new leaves with new labels, requiring relabelling throughout the tree, while leaf labels are unchanged under transplantation.

\begin{lem}
\label{lem:transplant_valid}
Suppose $\U$ is obtained from $\T$ by transplanting $\T'$ for $\T_v$.
\begin{enumerate}
\item If $\T$ and $\T'$ are valid, then $\U$ is also valid.
\item If $\T$ and $\T'$ are strongly valid, then $\U$ is also strongly valid.
\end{enumerate}
\end{lem}

\begin{proof}
All labels on vertices of $\U$ are inherited from $\T$ or $\T''$. If both $\T$ and $\T'$ are valid, then all labels are non-singular, so $\U$ is valid.

Now suppose $\T$ and $\T'$ are strongly valid. Let $M, M'$ be the labels on the root vertex of $\T$, and $v$, respectively, and let $P$ be a matched pair at which $M$ is not tight. By strong validity of $\T$, there is a unique vertex $w$ of $\T$ lowest among those with labels non-tight at $P$ (lemma \ref{lem:strong_validity_conditions}(ii)). Moreover, the vertices of $\T$ with labels non-tight at $P$ are precisely the ancestors of $w$ (lemma \ref{lem:strongly_valid_non-tight_ancestors}). 

If $w$ is not a vertex of $\T_v$, then $v$ is not an ancestor of $w$, so the label $M'$ of $v$ is tight at $P$. Every vertex label in $\T'$ is a sub-tensor-product of $M'$, hence tight at $P$ (lemma \ref{lem:local_tightness_homology_sub-tensor-product}). So the vertices of $\U$ with labels non-tight at $P$ are precisely the vertices of $\T$ with labels non-tight at $P$, and hence there is a unique lowest such vertex, namely $w$.

If $w$ is a vertex of $\T_v$, then $v$ is an ancestor of $w$, so the label $M'$ of $v$ is non-tight at $P$. Since $\T'$ is strongly valid, there is a unique lowest vertex $w'$ of $\T'$ with label non-tight at $P$ (lemma \ref{lem:strong_validity_conditions}(ii) again), and the set of vertices of $\T'$ whose labels are non-tight at $P$ are precisely the ancestors of $w'$ (lemma \ref{lem:strongly_valid_non-tight_ancestors} again). Thus in $\U$, the set of vertices whose labels are non-tight at $P$ are the ancestors of $w'$ in $\T'$, together with the ancestors of $v$ in $\T$ --- in other words, the ancestors of $w'$ in $\U$.

In any case, there is a unique vertex in $\U$ lowest among those with labels non-tight at $P$, so by lemma \ref{lem:strong_validity_conditions}(ii) once more, $\U$ is strongly valid.
\end{proof}

The second method, a \emph{branch shift}, rearranges an operation tree in a way corresponding to a modification $((AB)C) \leftrightarrow (A(BC))$.

Given an operation tree $\T$, denote the left and right children of the root vertex $v$ by $v_L$ and $v_R$, the left and right children of $v_L$ by $v_{LL}$ and $v_{LR}$, and generally for any word $w$ in $L$ and $R$, $v_w$ denotes the descendant of $v$ obtained by successively taking left or right children according to $w$ (if it exists). 

\begin{defn}
\label{defn:branch_shift}
The operation tree $\T'$ is defined by
\[
\T'_L = \T_{LL}, \quad
\T'_{RL} = \T_{LR}, \quad
\T'_{RR} = \T_R.
\]
We say the operation trees $\T$ and $\T'$ are related by a \emph{branch shift}.
\end{defn}
The vertex labels on $\T'$ are either inherited from $\T$, or determined by the fact that each vertex is labelled with the tensor product of its children's labels.

Let $\T_1, \T_2, \T_3$ respectively denote $\T_{LL}, \T_{LR}, \T_R$; let $N_1, N_2, N_3$ be the vertex labels on $v_{LL}, v_{LR}, v_R$ respectively; let the root vertex of $\T'$ be $v'$, and denote its vertices by $v'_w$ for words $w$ in $L$ and $R$. Then in $\T$, $v_L$ is labelled $N_1 \otimes N_2$; and in $\T'$, $v'_R$ is labelled $N_2 \otimes N_3$. The viability of labels in $\T$ ensures the viability of labels in $\T'$, so both $\T$ and $\T'$ are operation trees for $N = N_1 \otimes N_2 \otimes N_3$. Observe that upon reversing left and right, $\T$ is obtained from $\T'$ in the same way. See figure \ref{fig:branch_shift}.

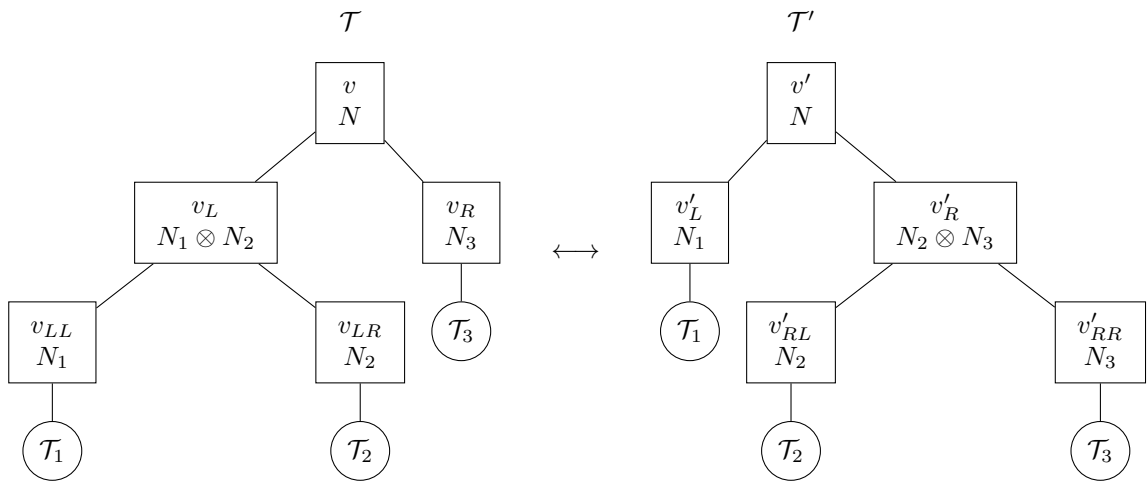
\begin{figure}
\begin{center}
\begin{tikzpicture}
\node[draw, rectangle] (root){$\begin{array}{c} v \\ N \end{array}$};
\node[above=0.3cm of root] {$\T$};
\node[draw, rectangle, below left=0.5cm and 0.5cm of root] (L) {$\begin{array}{c} v_L \\ N_1 \otimes N_2 \end{array}$} ;
\draw[-] (root) to (L);
\node [draw, rectangle, below right=0.5cm and 0.5cm of root] (R) {$\begin{array}{c} v_R \\ N_3 \end{array}$};
\draw[-] (root) to (R);
\node [draw, rectangle, below left=0.5cm and 0.5cm of L] (LL) {$\begin{array}{c} v_{LL} \\ N_1 \end{array}$};
\draw[-] (L) to (LL);
\node[draw, rectangle, below right=0.5cm and 0.5cm of L] (LR) {$\begin{array}{c} v_{LR} \\ N_2 \end{array}$};
\draw[-] (L) to (LR);
\node[draw, circle, below=0.5cm of LL] (T1) {$\T_1$};
\draw[-] (LL) to (T1);
\node[draw, circle, below=0.5cm of LR] (T2) {$\T_2$};
\draw[-] (LR) to (T2);
\node[draw, circle, below=0.5cm of R] (T3) {$\T_3$};
\draw[-] (R) to (T3);
\draw (3, -2) node {$\longleftrightarrow$};
\begin{scope}[xshift = 6cm]
\node[draw, rectangle] (root){$\begin{array}{c} v' \\ N \end{array}$};
\node[above=0.3cm of root] {$\T'$};
\node[draw, rectangle, below left=0.5cm and 0.5cm of root] (L) {$\begin{array}{c} v'_L \\ N_1 \end{array}$} ;
\draw[-] (root) to (L);
\node [draw, rectangle, below right=0.5cm and 0.5cm of root] (R) {$\begin{array}{c} v'_R \\ N_2 \otimes N_3 \end{array}$};
\draw[-] (root) to (R);
\node [draw, rectangle, below left=0.5cm and 0.5cm of R] (RL) {$\begin{array}{c} v'_{RL} \\ N_2 \end{array}$};
\draw[-] (R) to (RL);
\node[draw, rectangle, below right=0.5cm and 0.5cm of R] (RR) {$\begin{array}{c} v'_{RR} \\ N_3 \end{array}$};
\draw[-] (R) to (RR);
\node[draw, circle, below=0.5cm of L] (T1) {$\T_1$};
\draw[-] (L) to (T1);
\node[draw, circle, below=0.5cm of RL] (T2) {$\T_2$};
\draw[-] (RL) to (T2);
\node[draw, circle, below=0.5cm of RR] (T3) {$\T_3$};
\draw[-] (RR) to (T3);
\end{scope}

\end{tikzpicture}
\caption{A branch shift.}
\label{fig:branch_shift}
\end{center}
\end{figure}

All labels in $\T'$ appear in $\T$, with one exception. Thus if $\T$ is valid, then we only have one label to check for validity of $\T'$, giving the following.
\begin{lem} \
\begin{enumerate}
\item
Suppose $\T$ is valid. Then $\T'$ is valid if and only if $N_2 \otimes N_3$ is non-singular. 
\item
Suppose $\T'$ is valid. Then $\T$ is valid if and only if $N_1 \otimes N_2$ is non-singular.
\end{enumerate}
\qed
\end{lem}

\subsection{Strict distributivity}
\label{sec:strict_distributive}

We now strengthen our notion of distributivity (definition \ref{def:distributive}).

\begin{defn} \
\label{def:strictly_distributive}
Let $\T$ be a valid operation tree for $M = M_1 \otimes \cdots \otimes M_n$.
\begin{enumerate}
\item
Let $v$ be a vertex of $\T$ with $k$ leaves, labelled $M_v$. Then $v$ is \emph{strictly distributive} if there are exactly $k-1$ matched pairs at which $M_v$ is twisted or critical.
\item
The tree $\T$ is \emph{strictly $f$-distributive} if it is strictly distributive at each vertex.
\item
The tree $\T$ is \emph{strictly $X$-distributive} if it is strictly distributive at each non-root vertex, and there are precisely $n-2$ matched pairs at which $M$ is twisted or critical.
\end{enumerate}
\end{defn}
Recall distributivity (definition \ref{def:distributive}) at $v$ requires at least $k-2$ twisted or critical matched pairs at $v$; the strict requirement is that there are precisely $k-1$ such pairs. Note that definition \ref{def:strictly_distributive} requires $\T$ to be valid, so no labels are singular.

%An obvious but useful property is the following.
\begin{lem}
\label{lem:subtrees_f-distributive}
Let $\T$ be a valid strictly $f$- or $X$-distributive operation tree, and let $v$ be a non-root vertex. Then the operation subtree $\T_v$ of $\T$ below $v$ is strictly $f$-distributive.
\end{lem}

\begin{proof}
By lemma \ref{lem:subtrees_valid} $\T_v$ is valid, and every vertex of $\T_v$, being a non-root vertex of $\T$, is strictly distributive.
\end{proof}

Strict distributivity imposes strong conditions on the function $V_\T$ (definition \ref{def:pairs_to_vertices}). % of section \ref{sec:strong_validity}.
\begin{lem}
\label{lem:V_bijection}
Let $\T$ be an operation tree for $M = M_1 \otimes \cdots \otimes M_n$.
\begin{enumerate}
\item
If $\T$ is strongly valid and strictly $f$-distributive, then $V_\T$ is a bijection between non-tight matched pairs of $M$ and non-leaf vertices of $\T$.
\item
If $\T$ is strongly valid and strictly $X$-distributive, then $V_\T$ is a bijection between non-tight matched pairs of $M$ and non-leaf non-root vertices of $\T$.
\end{enumerate}
\end{lem}
Since $M$ has $n$ tensor factors, $\T$ has $n$ leaves, hence $n-1$ non-leaf vertices and $n-2$ non-leaf non-root vertices. Strict $f$-distributivity (resp. $X$-distributivity) requires that $M$ has precisely $n-1$ (resp. $n-2$) non-tight matched pairs. So in each case the claimed bijective sets have the same size.

\begin{proof}
When $n=1$, if $\T$ is strictly $f$-distributive, then $M$ has no twisted or critical matched pairs (i.e. is tight), and $\T$ has no non-leaf vertices. When $n=2$, if $\T$ is strictly $X$-distributive, again $M$ is tight, and $\T$ has no non-leaf non-root vertices. In both cases $V_\T$ is a bijection between empty sets.

We now proceed by induction on $n$. So suppose the result is true for operation trees for $M = M_1 \otimes \cdots \otimes M_k$, where $k<n$, and consider an operation tree $\T$ for $M = M_1 \otimes \cdots M_n$ which is strongly valid and strictly $f$-distributive or strictly $X$-distributive.

Let $v_0$ be the root vertex of $\T$, and let $v_L$ and $v_R$ be its left and right children; let their labels be $M_L = M_1 \otimes \cdots \otimes M_i$ and $M_R = M_{i+1} \otimes \cdots \otimes M_n$ respectively. Let $\T_L, \T_R$ be the operation subtrees of $\T$ below $v_L$ and $v_R$ (definition \ref{def:subtree_below}). %So $\T$ is the join of $\T_L$ and $\T_R$ (definition \ref{def:joining_trees}).
Now $\T_L$ and $\T_R$ are strongly valid (lemma \ref{lem:subtree_strongly_valid}) and strictly $f$-distributive (lemma \ref{lem:subtrees_f-distributive}), so by induction we have bijections
\begin{gather*}
V_L \colon \{ \text{non-tight matched pairs of $M_L$} \} \To \{ \text{non-leaf vertices of $\T_L$} \} \\
V_R \colon \{ \text{non-tight matched pairs of $M_R$} \} \To \{ \text{non-leaf vertices of $\T_R$} \},
\end{gather*} 
which by lemma \ref{lem:subtree_strongly_valid}(ii) are restrictions of $V_\T$. Moreover, since $\T_L$ and $\T_R$ are disjoint, and $\T$ is strongly valid, lemma \ref{lem:disjoint_subtrees_tightness} says that the %non-tight matched pairs of $M_L$ and $M_R$ are disjoint. In other words, the 
domains of $V_L$ and $V_R$ are disjoint. It's also clear that the ranges of $V_L$ and $V_R$ are disjoint; their union consists of all non-leaf non-root vertices of $\T$. The domains (and ranges) of $V_L$ and $V_R$ have cardinalities $i-1$ and $n-i-1$ respectively.

Since non-tight matched pairs in $M_L$ or $M_R$ are non-tight in $M$ (lemma \ref{lem:tightness_homology_sub-tensor-product}), the non-tight matched pairs in $M_L$ and $M_R$ form precisely $(i-1)+(n-i-1) = n-2$ non-tight matched pairs of $M$.

If $M$ is strictly $X$-distributive, then these are all the matched pairs in $M$, and $V_\T$ is the disjoint union of $V_L$ and $V_R$, hence a bijection as claimed.

If $M$ is strictly $f$-distributive, then $M$ has precisely $n-1$ non-tight matched pairs. So there is precisely one non-tight matched pair $P_0$ in $M$ which is tight in $M_L$ and $M_R$. Since $P_0$ is tight in both $M_L$ and $M_R$, but non-tight in $M$, $v_0$ is the lowest vertex of $\T$ whose label is non-tight at $P_0$, so $V_\T (P_0) = v_0$. This, together with $V_L$ and $V_R$, defines $V_\T$; we conclude $V$ is a bijection. 
\end{proof}

\subsection{Guaranteed nonzero results}
\label{sec:guaranteed_nonzero}

We now show that, in certain cases, $X_n$ and $\fbar_n$ must be nonzero, and compute their values.

\begin{thm}
\label{thm:nonzero_f}
Consider an $A_\infty$ structure on $\HH$ arising from a pair ordering $\preceq$. 
Suppose $M$ is viable and satisfies the following conditions.
\begin{enumerate}
\item
Every valid and distributive operation tree for $M$ is %strongly valid and
strictly $f$-distributive, and such a tree exists.
\item
No matched pair of $M$ is on-on doubly occupied.
\end{enumerate}
Then $\fbar_n (M) \neq 0$. Moreover $\fbar_n (M)$ is given by a single diagram $D$, which %has the same H-data as $M$, Maslov index $\iota(D) = \iota(M) + n-1$, and at each  matched pair $P$:
%\begin{enumerate}
%\item $D$ is not all-on doubly occupied crossed.
is tight at all matched pairs where $M$ is tight or critical, and 
crossed at all matched pairs where $M$ is twisted.
%\end{enumerate}
\end{thm}
Theorem \ref{thm:nonzero_f} is a precise version of theorem \ref{thm:third_thm}(i).
It explicitly describes $\fbar_n (M) = D$, which is determined by its H-data, and tightness at each matched pair. There is in fact no choice in constructing $D$, since choices only exist at 11 doubly occupied pairs, which are explicitly ruled out.

The description of $D$ follows entirely from Maslov index considerations.
The existence of a valid and strictly $f$-distributive tree for $M$ implies that $M$ is twisted or critical at precisely $n-1$ matched pairs, and tight at all other matched pairs. The Maslov index can only increase by 1 at each non-tight matched pair. Since $\fbar_n$ has Maslov grading $n-1$, Maslov grading must increase at every non-tight matched pair: from twisted to crossed, and from critical to tight.

When $n=1$, condition (i) says that $M=M_1$ is tight (there is only one possible operation tree), and the conclusion is that $\fbar_1 (M)$ is a tight diagram representing $M_1$.

When $n=2$, condition (i) says that $M=M_1 \otimes M_2$ has precisely one non-tight matched pair $P$ (again there is only one possible operation tree), which must be twisted (lemma \ref{lem:it_takes_3}), and all other matched pairs tight. The conclusion is that $\fbar_2 (M)$ is a single diagram $D$ twisted at $P$ and elsewhere tight, in agreement with the discussion of section \ref{sec:low-level_ex}.

\begin{thm}
\label{thm:nonzero_X}
Consider an $A_\infty$ structure on $\HH$ arising from a pair ordering $\preceq$. Suppose $M$ is viable and satisfies the following conditions.
\begin{enumerate}
\item
Every valid and distributive operation tree for $M$ is strictly $X$-distributive, and such a tree exists.
\item
No matched pair of $M$ is twisted or on-on doubly occupied.
\end{enumerate}
Then $X_n (M)$ is nonzero, and is the homology class of the unique tight diagram with the H-data of $M$.
\end{thm}
Theorem \ref{thm:nonzero_X} is a precise version of theorem \ref{thm:third_thm}(ii). The description of $X_n (M)$ follows entirely from the fact that $X_n$ preserves H-data. The uniqueness claim in the theorem makes sense: since $M$ has no 11 doubly occupied pairs, there is only one tight diagram with the same H-data as $M$.

%The appropriate notion of a ``strictly" distributive tree is slightly changed, since $X$ is given by a product of $f$'s, and hence at the root vertex cannot satisfy the strict distributivity condition. 

The exclusion of twisted matched pairs is necessary, since they preclude the existence of a tight diagram (or by theorem \ref{thm:Xn_structure}). The exclusion of 11 doubly occupied pairs is a more heavy-handed assumption, but is necessary for our proof; moreover it cannot be removed because of the example of figure \ref{fig:11_doubly_occupied_example}. In this example, $M = M_1 \otimes \cdots \otimes M_5$ is viable, has no twisted matched pairs, and has two valid distributive operation trees,  both of which are strongly valid and strictly $X$-distributive. However it also has a 11 doubly occupied matched pair $P = \{p,p'\}$, and $X_5 (M) = 0$. 

\begin{figure}
\begin{center}
\[
X_5 \left( \begin{array}{c|c|c|c|c}
p_+ & p_- & & p'_+ & p'_- \\
q'_- & & q_+ & q_- & \\
 & r_+ & r_- & & r'_- \end{array} \right) = 0
\]

\begin{tikzpicture}[scale=0.35]
\draw [draw=none] (0,0.2) circle (6pt);
\draw (1.2,1) -- (0,0);
\draw (1.2,1) -- (4.6,-2);
\draw (0,0) -- (-1.2,-1);
\draw (0,0) -- (1.2,-1);
\draw (-1.2,-1) -- (-2.2,-2);
\draw (-1.2,-1) -- (-0.3,-2);
\draw (1.2,-1) -- (0.3,-2);
\draw (1.2,-1) -- (2.2,-2);
\fill (1.2, 1) circle (6pt);
\fill (0,0) circle (6pt);
\fill (-1.2,-1) circle (6pt);
\fill (1.2,-1) circle (6pt);
\fill (-2.2,-2) circle (6pt);
\fill (-0.3,-2) circle (6pt);
\fill (0.3,-2) circle (6pt);
\fill (2.2,-2) circle (6pt);
\fill (4.6,-2) circle (6pt);
\end{tikzpicture}
\hspace{1 cm}
\begin{tikzpicture}[scale=0.35]
\draw [draw=none] (0,0.2) circle (6pt);
\draw (-1.2,1) -- (0,0);
\draw (-1.2,1) -- (-4.6,-2);
\draw (0,0) -- (-1.2,-1);
\draw (0,0) -- (1.2,-1);
\draw (-1.2,-1) -- (-2.2,-2);
\draw (-1.2,-1) -- (-0.3,-2);
\draw (1.2,-1) -- (0.3,-2);
\draw (1.2,-1) -- (2.2,-2);
\fill (-1.2,1) circle (6pt);
\fill (0,0) circle (6pt);
\fill (-1.2,-1) circle (6pt);
\fill (1.2,-1) circle (6pt);
\fill (-2.2,-2) circle (6pt);
\fill (-0.3,-2) circle (6pt);
\fill (0.3,-2) circle (6pt);
\fill (2.2,-2) circle (6pt);
\fill (-4.6,-2) circle (6pt);
\end{tikzpicture}

\caption{This tensor product $M = M_1 \otimes \cdots \otimes M_5$ (shown in shorthand) has a 11 doubly occupied matched pair $P = \{p,p'\}$, and two valid distributive operation trees as shown, both of which are strongly valid and strictly $X$-distributive. However $\fbar_1 \fbar_4 = \fbar_4 \fbar_1 \neq 0$ and $\fbar_2 \fbar_3 = \fbar_3 \fbar_2 = 0$, so $X_5 = 0$.}
\label{fig:11_doubly_occupied_example}
\end{center}
\end{figure}

While the conditions of theorems \ref{thm:nonzero_f} and \ref{thm:nonzero_X} may seem rather restrictive, they do show that $\fbar_n$ and $X_n$ are nonzero in many cases. For instance, in the $\fbar_3$ examples of lemma \ref{lem:f3_computations}, the first two lines (i.e. 10 out of 14 examples) can be shown to be nonzero directly from theorem \ref{thm:nonzero_f}. It follows from theorem \ref{thm:nonzero_X} that $X_4$ of all the following tensor products are nonzero. 

\begin{gather*}
% This is the first row of my examples page
%X_4 
\left( \begin{array}{c|c|c|c}
p_+ & p_- & p'_+ & \\
q_+ & & q_- & q'_+  
\end{array} \right)
\quad
%X_4 
\left( \begin{array}{c|c|c|c}
p_+ & p_- & p'_+ & \\
 & q_+ & q_- & q'_+ 
\end{array} \right)
\quad
%X_4
\left( \begin{array}{c|c|c|c}
p_+ & p_- & & p'_+ \\
 & q_+ & q_- & q'_+ 
\end{array} \right)
\quad
%X_4 
\left( \begin{array}{c|c|c|c}
p_+ & & p_- & p'_+ \\
q_+ & q_- & q'_+ & 
\end{array} \right) \\
% This is the final row of my X4 nonzero always examples, 3rd from bottom of page... the first entry doesn't apply
%X_4 
\left( \begin{array}{c|c|c|c}
p'_- & p_+ & p_- & p'_+ \\
q_+ & q_- & & q'_+ 
\end{array} \right)
\quad
%X_4 
\left( \begin{array}{c|c|c|c}
p'_- & p_+ & p_- & p'_+ \\
q_+ & & q_- & q'_+ 
\end{array} \right)
\quad
%X_4 
\left( \begin{array}{c|c|c|c}
p'_- & p_+ & p_- & p'_+ \\
q'_- & q_+ & & q_- 
\end{array} \right)
\quad
%X_4 
\left( \begin{array}{c|c|c|c}
p'_- & p_+ & p_- & p'_+ \\
q'_- & & q_+ & q_- 
\end{array} \right)
\end{gather*}

The hypotheses of theorems \ref{thm:nonzero_f} and \ref{thm:nonzero_X} essentially mandate that in each operation described by an operation tree, only one matched pair can be affected.

We first need a preliminary lemma.
\begin{lem}[Plenty of trees]
\label{lem:lotsa_trees}
Consider an $A_\infty$ structure on $\HH$ defined by a pair ordering.
Suppose $M = M_1 \otimes \cdots \otimes M_n$ is viable. Further suppose that every valid and distributive operation tree for $M$ is strongly valid and strictly $f$-distributive, and at least one such tree exists.

Let $P_0 = \{ p_0 , p'_0 \}$ be a matched pair at which $M$ is twisted. Then there exists a strongly valid, strictly $f$-distributive operation tree $\T$ for $M$ such that $V_\T (P)$ is the root vertex $v_0$ of $\T$.
\end{lem}
Let us say something about what lemma \ref{lem:lotsa_trees} means. At $P_0$, $M$ is twisted and hence an extension of the twisted tensor product of table \ref{tbl:local_tensor_products}  (lemma \ref{lem:homology_contractions}). So two steps of $P_0$ are covered, say $p_{0+}$ and $p_{0-}$, by some $M_i$ and $M_j$ respectively for some $i<j$. Now a sub-tensor-product $M'$ of $M$ labelling a non-root vertex of $\T$ is twisted or tight accordingly as $M'$ contains both $M_i$ and $M_j$, or does not. Lemma \ref{lem:lotsa_trees} guarantees the existence of a tree such that all labels on non-root vertices are tight at $P_0$. In other words, $M_i$ and $M_j$ never appear together in any label in $\T$ except at the root vertex $v_0$; as we work our way up the tree, combining tensor factors, $M_i$ and $M_j$ are only combined at the final step, at $v_0$. Since $P_0$ only becomes twisted at $v_0$, $v_0$ is the lowest vertex of $\T$ whose label is twisted at $P$, and $V_\T (P_0) = v_0$.

Since we can find such a tree for each all-on once occupied pair $P$, this gives us ``plenty of trees", which we need for the proof of theorem \ref{thm:nonzero_f}.

Note the hypotheses of lemma \ref{lem:lotsa_trees} are weaker than those of theorem \ref{thm:nonzero_f}. If $M$ satisfies the hypotheses of theorem \ref{thm:nonzero_f}, then every valid distributive operation tree for $M$ is strictly $f$-distributive; but as there are no 11 doubly occupied pairs, any such tree is strongly valid (lemma \ref{lem:valid_not_strongly}), so $M$ satisfies the hypotheses of lemma \ref{lem:lotsa_trees}.

The following lemma captures an argument we will use repeatedly. The terms in square brackets may be included or not.
\begin{lem}
\label{lem:induction_shortcut}
Let $M$ be a viable tensor product of nonzero homology classes of diagrams, which has one of the following two properties.
\begin{enumerate}
\item
Every valid and distributive operation tree for $M$ is [strongly valid and] strictly $f$-distributive, and such a tree exists.
\item
Every valid and distributive operation tree for $M$ is [strongly valid and] strictly $X$-distributive, and such a tree exists.
\end{enumerate}
Let $\T$ be an operation tree for $M$ of the type guaranteed by the condition, and let $v$ be a non-root vertex of $\T$, with label $M_v$. Then $M_v$ satisfies condition (i).
\end{lem}

\begin{proof}
Let $\T'$ be a valid distributive operation tree for $M_v$. Then we can transplant $\T'$ for the operation subtree $\T_v$ of $\T$ below $v$ to obtain an operation tree $\U$ for $M$, which is valid (lemma \ref{lem:transplant_valid}) and distributive (since distributive at each vertex: definition \ref{def:distributive}). By assumption then $\U$ is [strongly valid  and] strictly $f$- or $X$-distributive, so its subtree $\T'$ is also [strongly valid (lemma \ref{lem:subtree_strongly_valid}) and] strictly $f$-distributive (lemma \ref{lem:subtrees_f-distributive}). Finally, $\T_v$ demonstrates that such a tree exists.
\end{proof}

\begin{proof}[Proof of lemma \ref{lem:lotsa_trees}]
When $n=1$ the statement is vacuous: $M=M_1$ is tight, the unique operation tree is strongly valid and strictly $f$-distributive, and $V_\T$ is a bijection between empty sets. 
%When $n=2$ there is only one operation tree $\T$ for $M = M_1 \otimes M_2$, which by assumption is strongly valid and strictly $f$-distributive. Thus $M$ has a unique non-tight matched pair $P_0$, which is twisted since it takes 3 to be critical (lemma \ref{lem:it_takes_3}). At $P_0$ then $M$ must be exactly the twisted tensor product shown in table \ref{tbl:local_tensor_products}. Then $V_\T$ is a bijection between $P_0$ and the root vertex, so the result holds for $n=2$.
Proceeding by induction on $n$, consider a general $n$, and suppose the result holds for all smaller values of $n$.

Let $\T$ be a strongly valid and strictly $f$-distributive operation tree for $M$, which exists by assumption. By strict $f$-distributivity at $v_0$, there are precisely $n-1$ matched pairs at which $M$ is non-tight (i.e. twisted or critical). Let the two children of $v_0$ be $v_L$ and $v_R$, with labels $M_L = M_1 \otimes \cdots \otimes M_m$ and $M_R = M_{m+1} \otimes \cdots \otimes M_n$ respectively. Let $\T_L, \T_R$ be the operation subtrees of $\T$ below $v_L, v_R$ respectively. Then $\T_L, \T_R$ are strongly valid (lemma \ref{lem:subtree_strongly_valid}) and strictly $f$-distributive (lemma \ref{lem:subtrees_f-distributive}) operation trees for $M_L, M_R$ respectively. 

By lemma \ref{lem:V_bijection}, $V_\T$, $V_{\T_L}$ and $V_{\T_R}$ are all bijections, between sets of size $n-1$, $m-1$ and $n-m-1$, respectively; moreover $V_{\T_L}$ and $V_{\T_R}$ are restrictions of $V_\T$ (lemma \ref{lem:subtree_strongly_valid}) with disjoint domains (lemma \ref{lem:disjoint_subtrees_tightness}). Hence there is a unique matched pair $P_1$ such that $V_\T (P_1) = v_0$. Then $P_1$ is twisted in $M$ (definition \ref{def:pairs_to_vertices}), but tight in every other tensor product labelling a vertex. 

%Since $\T_L$ and $\T_R$ are strictly $f$-distributive, $M_L$ is non-tight at precisely $m-1$ matched pairs, and $M_R$ is non-tight at precisely $n-m-1$ pairs. Since $\T_L$ and $\T_R$ are disjoint subtrees of the strongly valid $\T$, by lemma \ref{lem:disjoint_subtrees_tightness}, the non-tight matched pairs of $M_L$ are disjoint from the non-tight matched pairs of $M_R$. So there are precisely $(m-1)+(n-m-1)=n-2$ matched pairs which are non-tight in $M_L$ or $M_R$. By lemma \ref{lem:local_tightness_homology_sub-tensor-product} (or table \ref{tbl:local_homology_sub-tensor-product_tightness}), these matched pairs are also non-tight in $M$. 

If $P_1 = P_0$ then we are done; so suppose that $P_1$ and $P_0$ are distinct. %We will show how to modify $\T$ to obtain another strongly valid and strictly $f$-distributive operation tree for $M$, in which $P_0$ corresponds to the root vertex.
% P_0 = A.... P_1 = B
Then $V_\T (P_0) \neq V_\T (P_1) = v_0$, so $V_\T (P_0)$ is a vertex of $\T_L$ or $\T_R$. Suppose $V_\T (P_0)$ lies in $\T_L$; the $\T_R$ case is similar.

By lemma \ref{lem:induction_shortcut}(i), $M_L$ satisfies the hypotheses of this lemma. By induction there then exists a strongly valid, strictly $f$-distributive operation tree $\T'_L$ for $M_L$, such that $V_{\T'_L} (P_0) = v_L$. Transplanting this $\T'_L$ for $\T_L$ yields a strongly valid (by lemma \ref{lem:transplant_valid}) and strictly $f$-distributive (since strictly distributive at each vertex: definition \ref{def:strictly_distributive}) operation tree $\T'$ for $M$. Moreover, $V_{\T'_L}$ is a restriction of $V_{\T'}$ (lemma \ref{lem:subtree_strongly_valid}), so $V_{\T'} (P_0) = v_L$, and since $P_1$ is tight in $M_L$ and $M_R$, $V_{\T'} (P_1) = v_0$.

Let the children of $v_L$ be $v_{LL}$ and $v_{LR}$, and let their labels in $\T'$ be $M'_{LL} = M_1 \otimes \cdots \otimes M_k$ and $M'_{LR} = M_{k+1} \otimes \cdots \otimes M_m$. Denote the operation subtrees of $\T'$ (or $\T'_L$) below $v_{LL}, v_{LR}$ respectively by $\T'_{LL}, \T'_{LR}$. These are again strongly valid and strictly $f$-distributive (lemmas \ref{lem:subtree_strongly_valid} and \ref{lem:subtrees_f-distributive}).

% T'_{LL} = R ... T'_{LR} = S
% p --> k

By strict $f$-distributivity, $M'_{LL}$, $M'_{LR}$, $M_R$, $M$ have precisely $k-1$, $m-k-1$, $n-m-1$, $n-1$ non-tight matched pairs respectively. But since $\T'_{LL}, \T'_{LR}$ and $\T_R$ are disjoint subtrees (below $v_{LL}, v_{LR}, v_R$) of the strongly valid $\T'$, the sets of matched pairs at which $M'_{LL}, M'_{LR}, M_R$ are non-tight are also disjoint (lemma \ref{lem:disjoint_subtrees_tightness}). Their union consists of $(k-1)+(m-k-1)+(n-m-1) = n-3$ matched pairs, which remain non-tight in $M$ (lemma \ref{lem:tightness_homology_sub-tensor-product}). The two remaining non-tight matched pairs of $M$ are $P_0$ and $P_1$; these two pairs are tight in each of $M'_{LL}, M'_{LR}, M_R$ since $V_{\T'} (P_0) = v_L$ and $V_{\T'} (P_1) = v_0$.

Now perform a branch shift on $\T'$ (definition \ref{defn:branch_shift}) to obtain an operation tree $\T''$ for $M$. Its root has children $v''_L = v_{LL}$ and $v''_R$, and the children of $v''_R$ are $v''_{RL} = v_{LR}$ and $v''_{RR} = v_R$. Below $v''_L, v''_{RL}, v''_{RR}$ respectively we have $\T''_L = \T'_{LL}$, $\T''_{RL} = \T'_{LR}$, and $\T''_{RR} = \T_R$. The labels on $\T''$ are inherited from $\T'_{LL}, \T'_{LR}, \T'_R$, except that $v''_0$ is labelled $M$ and $v''_R$ is labelled with $M''_R = M_{k+1} \otimes \cdots \otimes M_n = M'_{LR} \otimes M_R$. In particular, $v''_L, v''_{RL}, v''_{RR}$ are respectively labelled with $M''_L = M'_{LL}$, $M''_{RL} = M'_{LR}$ and $M''_{RR} = M_R$.

We claim $\T''$ is valid. If $P$ is a matched pair non-tight in $M$, other than $P_0$ or $P_1$, then $P$ is twisted in the label of $V_{\T'} (P)$ (definition \ref{def:pairs_to_vertices}), which is a vertex of one of $\T'_{LL} = \T''_L$, $\T'_{LR} = \T''_{RL}$, or $\T_R = \T''_{RR}$. And $P_0, P_1$ are twisted in $M$, which is the label of the root. Thus for every matched pair $P$, there is a  vertex of $\T''$ whose label is twisted at $P$. By lemma \ref{lem:valid_twisted_vertex} then $\T''$ is valid. 

We also claim $\T''$ is distributive. Each vertex of $\T''$ which shares a label with a vertex of distributive tree $\T'$ is distributive. The only remaining vertex is $v''_R$, which has label $M''_R = M_{k+1} \otimes \cdots \otimes M_n = M'_{LR} \otimes M_R$. Each of the $(m-k-1)+(n-m-1)=n-k-2$ matched pairs $P$ such that $V_{\T'} (P)$ is a vertex of $\T'_{LR}$ or $\T_R$ is non-tight in $M'_{LR}$ or $M_R$, hence also in $M''_R = M'_{LR} \otimes M_R$ (lemma \ref{lem:tightness_homology_sub-tensor-product}). Since there are $n-k$ leaves below $v''_R$, and there are at least $n-k-2$ matched pairs at which $M''_R$ is twisted or critical, $v''_R$ is distributive, and the claim follows.

Since $\T''$ is valid and distributive, by assumption then $\T''$ is strongly valid and strictly $f$-distributive. 
%So $V_{\T''}$ is a bijection (lemma \ref{lem:V_bijection}). 
Now $P_0$ is twisted in $M$ and satisfies $V_{\T'} (P_0) = v_L$, so $P_0$ is twisted in $M_L = M_1 \otimes \cdots \otimes M_m$, but tight in $M'_{LL} = M_1 \otimes \cdots \otimes M_k$ and $M'_{LR} = M_{k+1} \otimes \cdots \otimes M_m$.
Supposing without loss of generality that $P_0$ is twisted at $p_0$ in $M$, then the step $p_{0+}$ must be covered by one of $M_1, \ldots, M_k$, and the step $p_{0-}$ must be covered by one of $M_{k+1}, \ldots, M_m$, with no steps of $P$ covered by any of $M_{m+1}, \ldots, M_n$. Thus $P_0$ is tight in $M''_R = M_{k+1} \otimes \cdots \otimes M_n$, and in $M''_L = M_1 \otimes \cdots \otimes M_k$, the labels of $v''_L$ and $v''_R$; but $P_0$ is twisted in $M$, the label of $v_0$. So $V_{\T''} (P_0) = v_0$, and $\T''$ is the desired tree. By induction, the proof is complete.
\end{proof}

\begin{proof}[Proof of theorem \ref{thm:nonzero_f}]
We have verified the theorem in small cases, so suppose it is true for all $\fbar_k$ with $k<n$, and consider $\fbar_n$. 

By lemma \ref{lem:valid_not_strongly}, since there are no 11 doubly occupied pairs in $M$, validity and strong validity are equivalent; we use this fact repeatedly. Note that if any sub-tensor-product $M'$ of $M$ contains a 11 doubly occupied pair, then $M$ would contain one too; so validity and strong validity are also equivalent for operation trees of sub-tensor-products of $M$.

Our strategy is to compute $\Ubar_n (M)$ explicitly, and then compute $\fbar_n$, using the construction of corollary \ref{cor:main_cor}. Recall $\Ubar_n (M)$ is a sum of terms of the form $\fbar_i (M_1 \otimes \cdots \otimes M_i) \fbar_{n-i} (M_{i+1} \otimes \cdots \otimes M_n)$, and $\fbar_{n-j+1} (M_1 \otimes \cdots \otimes M_k \otimes X_j (M_{k+1} \otimes \cdots \otimes M_{k+j}) \otimes \cdots \otimes M_n)$.

The latter type of term is easiest to deal with: we claim they are all zero. Suppose to the contrary that $\fbar_{n-j+1} (M_1 \otimes \cdots \otimes M_k \otimes X_j (M_{k+1} \otimes \cdots \otimes M_{k+j} \otimes \cdots \otimes M_n) \neq 0$. Then by proposition \ref{prop:nonzero_implies_trees} there are valid distributive operation trees $\T_X$ for $M_{k+1} \otimes \cdots \otimes M_{k+j}$ and $\T_f$ for $M_1 \otimes \cdots \otimes M_k \otimes X_j (M_{k+1} \otimes \cdots \otimes M_{k+j}) \otimes \cdots \otimes X_n$. Grafting $\T_X$ onto $\T_f$ at position $k+1$ yields (lemma \ref{lem:graft_valid_distributive}) a valid distributive operation tree $\T_{fX}$ for $M$. By assumption then $\T_{fX}$ is strictly $f$-distributive. Applying strict distributivity to the vertex of $\T_{fX}$ corresponding to the root of $\T_X$, then $M_{k+1} \otimes \cdots \otimes M_{k+j}$ is twisted or critical at precisely $j-1$ matched pairs. But by theorem \ref{thm:Xn_structure}, $X_j (M_{k+1} \otimes M_{k+j}) \neq 0$ implies that there are precisely $j-2$ such pairs. This gives a contradiction, so all such terms are zero.

We now consider terms of the form $\fbar_i (M_1 \otimes \cdots \otimes M_i) \fbar_{n-i} (M_{i+1} \otimes \cdots \otimes M_n)$ which are nonzero. We will associate to them matched pairs at which $M$ is twisted and eventually obtain a bijection $F \colon A \To B$ where
\[
A = \{ i \mid \fbar_i (M_1 \otimes \cdots \otimes M_i) \fbar_{n-i} ( M_{i+1} \otimes \cdots \otimes M_n ) \neq 0 \}, \quad
B = \{ P \mid \text{$M$ is twisted at $P$} \}.
\]
So let $M' = M_1 \otimes \cdots \otimes M_i$ and $M'' = M_{i+1} \otimes \cdots \otimes M_n$ and suppose $\fbar_i (M') \fbar_{n-i} (M'') \neq 0$. By proposition \ref{prop:nonzero_implies_trees} there are valid (hence strongly valid: $M'$ and $M''$ are sub-tensor-products of $M$, so their validity and strong validity are equivalent) distributive trees $\T'$ for $M'$ and $\T''$ for $M''$. Joining these trees yields an operation  tree $\T$ for $M$ (definition \ref{def:joining_trees}), which is valid (hence strongly valid) and distributive (lemma \ref{lem:join_valid_distributive}(iii)), hence by hypothesis strictly $f$-distributive. We then have a bijection $V_\T$ between non-tight matched pairs of $M$ and non-leaf vertices of $\T$ (lemma \ref{lem:V_bijection}). Moreover, since $\T', \T''$ are subtrees of $\T$, they are also strictly $f$-distributive (lemma \ref{lem:subtrees_f-distributive}). Thus $M$, $M'$, $M''$ are twisted or critical at $n-1$, $i-1$, $n-i-1$ matched pairs respectively, and tight elsewhere. 

By lemma \ref{lem:induction_shortcut}, any valid distributive tree for $M'$ is strongly valid and strictly $f$-distributive; and similarly for $M''$. And since $M$ has no 11 doubly occupied matched pairs, neither do the sub-tensor-products $M'$ or $M''$. So the hypotheses of the theorem apply to $M'$ and $M''$. By induction then $\fbar_i (M')$ and $\fbar_{n-i} (M'')$ are given by single diagrams as described in the statement. Moreover, as $\T', \T''$ are disjoint subtrees of the strongly valid $\T$, the matched pairs at which $M', M''$ are non-tight are disjoint (lemma \ref{lem:disjoint_subtrees_tightness}). This yields $(i-1)+(n-i-1) = n-2$ matched pairs at which $M'$ or $M''$ is non-tight; such pairs are also non-tight in $M$ (lemma \ref{lem:tightness_homology_sub-tensor-product}). So there is precisely one matched pair $P_i$ at which $M$ is non-tight but $M'$ and $M''$ are tight. Then $V_\T (P_i)$ is the root vertex $v_0$, and $P_i$ is twisted in $M$ (by the comment after definition \ref{def:pairs_to_vertices}, or lemma \ref{lem:it_takes_3}). Indeed, $V_\T (P_i)$ is the root vertex, for any $\T$ arising as the join of valid distributive operation trees for $M'$ and $M''$. Define the function $F \colon A \To B$ by $F(i) = P_i$.

By induction $\fbar_i (M')$ (resp. $\fbar_{n-i} (M'')$) is given by a single diagram which is tight at all matched pairs where $M'$ (resp. $M''$) is tight or critical, and crossed at all matched pairs where $M'$ (resp. $M''$) is twisted. We now describe the diagram representing $\fbar_i (M') \fbar_{n-i} (M'')$, at each matched pair $P$. 

First, suppose $P$ is critical in $M$. Then $P \neq P_i$, so $P$ is non-tight in precisely one of $M'$ or $M''$. %It may be that $P$ is critical in one of $M'$ or $M''$, and tight in the other; or twisted in one and tight in the other. In the first case $\fbar_i (M') \fbar_{n-i} (M'')$ is the product of two tight diagrams at $P$; in the second case it is the product of a crossed and tight diagram at $P$. 
Considering the known description of $\fbar_i (M')$ and $\fbar_{n-i}(M'')$, we examine the various cases in the critical column of table \ref{tbl:local_tensor_products}, of which $M$ is an extension (lemma \ref{lem:homology_contractions}), and how the $P$-active factors can be distributed across $M'$ and  $M''$. We observe that in every case $\fbar_i (M') \fbar_{n-i} (M'')$ is tight at $P$. 

Second, suppose $P$ is a matched pair at which $M$ is twisted, other than $P_i$. Then $P$ is non-tight in precisely one of $M'$ or $M''$. Indeed, there are two $P$-active factors and they are both in $M'$, or both in $M''$. So $\fbar_i (M') \fbar_{n-i} (M'')$ at $P$ is the product of an all-on once occupied crossed diagram, and an idempotent, hence is crossed.

Third, suppose $P$ is tight in $M$. Then $P$ is also tight in $M'$ and $M''$ (lemma \ref{lem:local_tightness_homology_sub-tensor-product}), hence also in $\fbar_i (M')$ and $\fbar_{n-i} (M'')$ (by inductive assumption). So $\fbar_i (M') \fbar_{n-i} (M'')$ at $P$ is given by multiplying factors in a tight tensor product, hence is tight.

Finally, at $P_i$, $M'$ and $M''$ are both tight, but $M$ is twisted. 
%So $M$ at $P_i$ is an extension of the unique twisted entry in table \ref{tbl:local_tensor_products} (proposition \ref{prop:homology_tensor_product_classification}, lemma \ref{lem:homology_contractions}). 
Hence $P_i$ is 11 once occupied by $M$, with one step covered by $M'$, and the other by $M''$; by inductive assumption then $\fbar_i (M')$ and $\fbar_{n-i} (M'')$ are both tight at $P_i$, so $\fbar_i (M') \fbar_{n-i} (M'')$ is twisted at $P_i$.

To summarise: when $\fbar_i (M') \fbar_{n-i} (M'')$ is nonzero, there is a unique matched pair $P_i$ which is non-tight (in fact twisted) in $M$ but tight in $M'$ and $M''$; $V_\T (P_i)$ is the root vertex of $\T$; and $\fbar_i (M') \fbar_{n-i} (M'')$ is given by a single diagram which is twisted at $P_i$, crossed at all other matched pairs which are twisted in $M$, and tight at all other matched pairs. We set $F(i) = P_i$.

%Moreover, taking a strongly valid and strictly $f$-distributive operation tree $\T$ for $M$ as constructed above, the function $V_\T$ provides a bijection between non-tight matched pairs of $M$ and non-leaf vertices of $\T$ (lemma \ref{lem:V_bijection}). If $V_\T (P) = v$, then $v$ is the unique lowest vertex in $\T$ whose label $M_v$ is non-tight at $P$ (lemma \ref{lem:strongly_valid_non-tight_ancestors}).

Now we claim that $F$ is injective. Consider another nonzero term $\fbar_j (M_1 \otimes \cdots \otimes M_j) \fbar_{n-j} (M_{j+1} \otimes \cdots \otimes M_n)$, where $i \neq j$. We consider the case $i<j$; the case $i>j$ is similar.
Applying the same argument as above, we obtain strongly valid and strictly $f$-distributive trees $\T_j, \T'_j, \T''_j$ for $M$, $M'_j = M_1 \otimes \cdots \otimes M_j$ and $M''_j = M_{j+1} \otimes \cdots \otimes M_n$ respectively. %, with $\T_j$ the join of $\T'_j$ and $\T''_j$.
We also obtain the bijection $V_{\T_j}$ between non-tight matched pairs of $M$ and non-leaf vertices of $\T_j$. The matched pair $P_j$ has $V_{\T_j} (P_j)$ as the root of $\T_j$, and $F(j) = P_j$. We will show that $P_i \neq P_j$.

Now in the valid distributive tree $\T$ constructed above for $\fbar_i (M') \fbar_{n-i} (M'')$, let $v$ be the lowest common ancestor of the leaves labelled $M_j$ and $M_{j+1}$. Let $P$ be the matched pair such that $V_\T (P) = v$ (well-defined since $V_\T$ is bijective). Since $i<j$, $v$ is a vertex of $\T''$, hence not the root, so $P \neq P_i$. The label $M_v$ of $v$ is then twisted at $P$ (definition \ref{def:pairs_to_vertices}), say at the place $p$.
 %, and both the children $v_L, v_R$ of $v$ have labels tight at $P	$. Moreover, as $v$ is a vertex of $\T''$, which has leaves labelled $M_{i+1}, \ldots, M_n$, the label on $v_L$ contains only tensor factors $M_k$ with $i+1 \leq k \leq j$, and the label on $v_R$ contains only tensor factors $M_k$ with $j+1 \leq k \leq n$. 
So some $M_a$, with %$i+1 \leq 
$a \leq j$, covers the step $p_+$, and some $M_b$, with $j+1 \leq b$, % \leq n$, 
covers $p_-$. As $M$ contains no 11 doubly occupied pairs, any sub-tensor-product of $M$ which is twisted at $P$ must contain $M_a$ and $M_b$.

Now consider $V_{\T_j} (P)$, a vertex of $\T_j$; call its label $M_\#$. Then $M_\#$ is twisted at $P$ (definition \ref{def:pairs_to_vertices}), so $M_\#$ contains $M_a$ and $M_b$ as tensor factors. But since $a \leq j$ and $b \geq j+1$, $M_\#$ cannot be a sub-tensor-product of $M'_j$ or $M''_j$; thus $M_\# = M$ and $V_{\T_j} (P)$ is the root vertex. Thus $P = P_j$. As $P \neq P_i$ then $P_i \neq P_j$. Thus $F$ is injective.

We now show $F$ is surjective. Take a matched pair $P$ at which $M$ is twisted; we will show $P=P_i$ for some $i$. By lemma \ref{lem:lotsa_trees} (which, as discussed above, has weaker hypotheses than the present theorem) there is a strongly valid, strictly $f$-distributive operation tree $\T^*$ for $M$ such that $V_{\T^*} (P)$ is the root vertex $v_0^*$ of $\T^*$. Let the children of $v_0^*$ be $v_L^*, v_R^*$, with labels $M_L^* = M_1 \otimes \cdots \otimes M_i$ and $M_R^* = M_{i+1} \otimes \cdots \otimes M_n$ respectively. Then by definition of $V_{\T^*}$, $P$ is tight in $M_L^*$ and $M_R^*$. By lemma \ref{lem:induction_shortcut}, $M_L^*$ and $M_R^*$ satisfy condition (i) of the present theorem; and as $M_L^*$ and $M_R^*$ are sub-tensor-products of $M$, which has no 11 doubly occupied pairs, they satisfy condition (ii) also. So by induction $\fbar_i (M_L^*)$ and $\fbar_{n-i} (M_R^*)$ are both nonzero, given by single diagrams as described in the statement. By lemma \ref{lem:disjoint_subtrees_tightness} they are non-tight at disjoint matched pairs. Examining the various possible cases at each matched pair (just as we did for $\fbar_i (M') \fbar_{n-i} (M'')$ a few paragraphs ago), we conclude that $\fbar_i (M_L^*) \fbar_{n-i} (M_R^*) \neq 0$. Since $P$ is non-tight in $M$ but tight in $M_L^*$ and $M_R^*$, we have $P=P_i$. So $F$ is surjective, hence a bijection.

Returning to $\Ubar_n (M)$, we now see that each nonzero term of $\Ubar_n (M)$ is of the form $\fbar_i (M') \fbar_{n-i} (M'')$, and these terms correspond bijectively to the matched pairs $P_i$ at which $M$ is twisted. In fact $\fbar_i (M') \fbar_{n-i} (M'')$ is twisted at $P_i$, and crossed at all other matched pairs where $M$ is twisted.

We also observe that $X_n (M) = 0$, since $M$ has precisely $n-1$ non-tight matched pairs, by theorem \ref{thm:Xn_structure}. Thus, following the construction of corollary \ref{cor:main_cor} and the discussion of section \ref{sec:creation_operators}, 
\[
\fbar_n (M) = \Abar^*_{\mathcal{CR}^\preceq} \Ubar_n (M).
\]
By definition \ref{def:choice_functions_from_ordering}, $A^*_{\mathcal{CR}^\preceq}$ applies a creation operator at $P_{min}$, where $P_{min}$ is the $\preceq$-minimal matched pair among pairs where $M$ is twisted.

We observe that there is precisely one diagram in $\Ubar_n (M)$ which is twisted at $P_{min}$, namely $\fbar_i \fbar_{n-i}$ where $i = F^{-1} (P_{min})$, i.e. where $P_i = P_{min}$. Applying $\Abar^*_{\mathcal{CR}^\preceq} = \Abar^*_{P_{min}}$ inserts a crossing at $P_{min}$ to this diagram. All the other diagrams in $\Ubar_n (M)$ are crossed at $P_{min}$, and applying the creation operator gives zero.

We conclude that $\fbar_n (M)$ is given by a single diagram, crossed at all matched pairs where $M$ is twisted, and tight elsewhere, as desired.
\end{proof}

\begin{proof}[Proof of theorem \ref{thm:nonzero_X}]
As there are no 11 occupied matched pairs, by lemma \ref{lem:valid_not_strongly}, validity and strong validity are equivalent.

By lemma \ref{lem:Xn_directly} (since all the maps $f_k$ in the pair ordering construction are balanced), $X_n (M)$ is represented by the sum of all terms of the form $\fbar_i (M_1 \otimes \cdots \otimes M_i) \fbar_{n-i} (M_{i+1} \otimes \cdots \otimes M_n)$.

Let $\T$ be a valid and strictly $X$-distributive operation tree for $M$, which exists by hypothesis. Let its root vertex be $v_0$, with children $v_L, v_R$ labelled $M_L = M_1 \otimes \cdots \otimes M_i$, $M_R = M_{i+1} \otimes \cdots \otimes M_n$. Let $\T_L, \T_R$ be the subtrees below $v_L, v_R$ respectively.

By lemma \ref{lem:induction_shortcut}, $M_L$ and $M_R$ satisfy condition (i) of theorem \ref{thm:nonzero_f}; and being sub-tensor-products of $M$, which has no 11 doubly occupied pairs, $M_L$ and $M_R$ also satisfy condition (ii). So by theorem \ref{thm:nonzero_f}, $\fbar_i (M_L)$ and $\fbar_{n-i} (M_R)$ are both nonzero, given by single diagrams. Since $\T$ is strictly $X$-distributive, $M_L$ and $M_R$ have $i-1, n-i-1$ non-tight matched pairs respectively. These sets of non-tight matched pairs are distinct by lemma \ref{lem:disjoint_subtrees_tightness}, % since $\T_L, \T_R$ are disjoint subtrees of the strongly valid $\T$, 
and also non-tight in $M$ (lemma \ref{lem:tightness_homology_sub-tensor-product}); hence they provide $n-2$ distinct non-tight matched pairs in $M$. By strict $X$-distributivity of $\T$, $M$ has precisely $n-2$ non-tight matched pairs, so each non-tight matched pair of $M$ is non-tight in precisely one of $M_L$ or $M_R$.

By theorem \ref{thm:nonzero_f}, $\fbar_i (M_L)$ (resp. $\fbar_{n-i} (M_R)$) is crossed at every matched pair where $M_L$ (resp. $M_R$) is twisted, and elsewhere tight. Thus at every non-tight (hence critical; twisted pairs are ruled out by hypothesis) matched pair of $M$, precisely one of $M_L, M_R$ is non-tight (twisted or critical), and the other is tight. If one of $M_L, M_R$ is critical and the other is tight, then $\fbar_i (M_L)$ and $\fbar_{n-i} (M_R)$ are tight, and by reference to table \ref{tbl:local_tensor_products} or otherwise, $\fbar_i (M_l) \otimes \fbar_{n-i} (M_R)$ is tight. If one of $M_L, M_R$ is twisted and the other is tight, then one of $\fbar_i (M_L), \fbar_{n-i} (M_R)$ is crossed, and the other is tight, so again by reference to table \ref{tbl:local_tensor_products} or otherwise, $\fbar_i (M_L) \otimes \fbar_{n-i} (M_R)$ is sublime. Either way, $\fbar_i (M_L) \fbar_{n-i} (M_R)$ is tight at each non-tight matched pair of $M$. At tight matched pairs of $M$, $\fbar_i (M_L)$ and $\fbar_i (M_R)$ are both tight, with tight product. So $\fbar_i (M_L) \fbar_{n-i} (M_R)$ is the unique tight diagram with the same H-data as $M$.

Now let $P$ be a non-tight matched pair of $M$. By assumption, $P$ is critical, but not 11 doubly occupied. Thus, by reference to table \ref{tbl:local_tensor_products}, $P$ is sesqui-occupied or 00 doubly occupied and $M_P$ is an extension of one of the corresponding critical diagrams shown there (proposition \ref{lem:homology_contractions}). In particular, there is precisely one place $p$ of $P$ such that the steps $p_+$ and $p_-$ are covered by some $M_a$ and $M_b$ respectively, where $a<b$. We call these the \emph{principal factors} of $P$. Now if $a \leq i < i+1 \leq b$, then considering the various cases of table \ref{tbl:local_tensor_products}, $P$ is singular in $M_L$ or $M_R$, contradicting validity of $\T$. Thus $a,b$ are both $\leq i$, or both $\geq i+1$. In other words, for any non-tight matched pair of $M$, its principal factors have positions which are both $\leq i$, or both $\geq i+1$; they do not cross the $i$'th position.

On the other hand, we claim that, for any for any $1 \leq j \leq n-1$ with $j \neq i$, there is a non-tight matched pair of $P$ whose principal factors have positions $\leq j$ and $\geq j+1$; they \emph{do} cross the $j$'th position. To see this, let $w$ be the least common ancestor of the leaves labelled $M_j$ and $M_{j+1}$. Then $w$ lies in $\T_L$ or $\T_R$, accordingly as $i>j$ or $i<j$.
%there is a lowest vertex $v$ of $\T_L$ (if $i>j$) or $\T_R$ (if $i<j$) whose label includes $M_j$ and $M_{j+1}$ as tensor factors. 
We suppose $i<j$, so $w \in \T_R$; the $\T_L$ case is similar. Clearly $w$ is neither a leaf nor root, so by lemma \ref{lem:V_bijection}, there is a unique matched pair $P$ such that $V_\T (P) = w$. Let the principal factors of $P$ be $M_a$ and $M_b$, where $a<b$. Letting $M_w$ denote the label of $w$, then $M_w$ is twisted at $P$. Letting $w_L, w_R$ denote the children of $v$, their labels are tight at $P$. The label on $w_L$ contains $M_a$, so by construction $a \leq j$. Similarly the label of $w_R$ contains $M_b$, and $j+1 \leq b$. So the two principal factors have positions with are $\leq j$ and $\geq j+1$ respectively.

It follows that for any $j \neq i$, we must have $\fbar_j (M_1 \otimes \cdots \otimes M_j) \fbar_{n-j} (M_{j+1} \otimes \cdots \otimes M_n) = 0$. For if this product were nonzero, then we could repeat the argument above and find that no non-tight matched pair of $M$ has principal factors whose positions cross the $j$'th position, contradicting the previous paragraph.

We conclude that $X_n (M)$ is the homology class of the single diagram $\fbar_i (M_L) \fbar_{n-i} (M_R)$, which has the desired properties.
\end{proof}

%For the end..
\addcontentsline{toc}{section}{References}

\small

\bibliography{A-infinity_strands}
\bibliographystyle{amsplain}

\end{document}